\documentclass{report}
\usepackage[utf8]{inputenc}
\usepackage[utf8]{inputenc}
\usepackage{amsmath}

\pdfoutput=1
\usepackage{setspace}
\usepackage[left=3.4cm, right=3.4cm, bottom=3cm, top=3cm]{geometry}
\usepackage{hyperref}
\usepackage[raggedright]{titlesec}
\counterwithin*{chapter}{part}
\usepackage{cleveref}
\usepackage{amssymb}
\usepackage{footnotebackref}
\usepackage{mathtools}
\usepackage{mathabx,graphicx}
\usepackage{tikz-cd}
\usepackage{amsthm}
\usepackage{enumitem}
\usepackage{eqparbox}

\usepackage{hyperref, cleveref}
\usepackage{graphicx}
\usepackage{graphics}
\usepackage{mathrsfs}
\usepackage{cleveref}
\usepackage{thmtools}
\newcommand{\RNum}[1]{\uppercase\expandafter{\romannumeral #1\relax}}
\usepackage{bbm}
\newlist{steps}{enumerate}{1}
\setlist[steps, 1]{label = Step \arabic*:}
\theoremstyle{plain}
\newtheorem{theorem}{Theorem}[section]
\newtheorem{definition}[theorem]{Definition}

\newtheorem{proposition}[theorem]{Proposition}
\newtheorem{lemma}[theorem]{Lemma}
\newtheorem{remark}[theorem]{Remark}
\newtheorem{corollary}[theorem]{Corollary}
\theoremstyle{definition}
\newtheorem{exmp}[theorem]{Example}
\newtheorem{rmk}[theorem]{Remark}
\newtheorem{conjecture}[theorem]{Conjecture}
\numberwithin{equation}{section}
\newtheorem*{conjecture*}{Conjecture}
\usepackage{rotating}
\usepackage{comment}
\usepackage[nottoc,notlot,notlof]{tocbibind} 

\newtheorem{thm}{theorem}[section]

\usepackage{blindtext}
\makeatletter\@addtoreset{chapter}{part}
\title{THE VAN EST MAP ON GEOMETRIC STACKS}
\author{\vspace{15mm}\\by\vspace{28mm}\\Joshua Lackman\vspace{50mm}\\A thesis submitted in conformity with the requirements\\ for the degree of Doctor of Philosophy\\
Graduate Department of Mathematics\\University of Toronto\vspace{22mm}\\\copyright $\,$ Copyright 2022 by Joshua Lackman}
\date{}
\begin{document}
\maketitle
\thispagestyle{empty}
\setcounter{page}{2}
\chapter*{}
\begin{doublespace}
\begin{centering}
\textbf{\LARGE{Abstract}}
\vspace{1mm}\\The van Est Map on Geometric Stacks
\vspace{3mm}\\Joshua Lackman\footnote{\href{mailto:jlackman@math.toronto.edu}{jlackman@math.toronto.edu}}
\\Doctor of Philosophy
\\Graduate Department of Mathematics
\\University of Toronto
\\2022
\\
\end{centering}
\vspace{3mm}We generalize the van Est map and isomorphism theorem in three ways, and we propose a category of Lie algebroids and LA-groupoids (with equivalences). First, we generalize the van Est map from a comparison map between Lie groupoid cohomology and Lie algebroid cohomology to a (more conceptual) comparison map between the cohomology of a stack $\mathcal{G}$ and the foliated cohomology of a stack $\mathcal{H}\to\mathcal{G}$ mapping into it. At the level of Lie groupoids, this amounts to describing the van Est map as a map from Lie groupoid cohomology to the cohomology of a particular LA-groupoid. We do this by associating to any (nice enough) homomorphism of Lie groupoids $f:H\to G$ a natural foliation of the stack $[H^0/H]\,.$ In the case of a wide subgroupoid $H\xhookrightarrow{}G\,,$ this foliation can be thought of as equipping the normal bundle of $H$ with the structure of an LA-groupoid. In particular, this generalization allows us to derive classical results, including van Est's isomorphism theorem about the maximal compact subgroup, which we generalize to proper subgroupoids, as well as the Poincar\'{e} lemma; it also gives a new method of computing Lie groupoid cohomology.
\vspace{3mm}\\Secondly, we generalize the functions that we can take cohomology of in the context of the van Est map; instead of using functions valued in representations, we can use functions valued in modules — eg. $S^1$-valued functions and $\mathbb{Z}$-valued functions. This generalization allows us to obtain classical results about linearizing group actions, as well as results about lifting group actions to gerbes. Thirdly, everything we do works in both the smooth and holomorphic categories.
\vspace{3mm}\\At the end of this thesis we give a conjectural definition of Morita equivalences of Lie algebroids and generalized morphisms, and we prove a no-go theorem. We explore higher structures and conjecture a link between Lie \RNum{2} and the van Est theorem. This involves describing higher cohomology classes (eg. gerbes) as generalized morphisms in a higher category, similar to how principal bundles define generalized morphisms in a (2,1)-category. We conjecture the existence of a smooth version of Grothendieck's homotopy hypothesis, and we describe a category of LA-groupoids with equivalences.
\end{doublespace}
\chapter*{Acknowledgements}
I would like to begin by thanking my advisor, Marco Gualtieri. Marco spent many hours teaching me over the last several years, and he introduced me to the topic of this thesis (as well as a variety of other topics). In addition to mathematics, Marco gave me a lot of helpful advice with regards to the exposition of my paper and this thesis. Furthermore, with his support, I've had the opportunity to travel to places I may have never gone, with the most memorable being Tokyo (IPMU) and Banff (BIRS). There were times when I was absent for extended periods, without notice (particularly during 2020-2021), and I appreciate his understanding. Finally, I should mention that I'm writing this on a laptop Marco lent me so long ago that I can't even recall when.
\vspace{3mm}\\I would like to thank my committee members Lisa Jeffrey and Eckhard Meinrenken. Lisa, in particular, made many valuable comments with regards to the editing and exposition of this thesis, including the discovery of many errors that I would have been remiss to have left uncorrected. Eckhard's explanation of the van Est map is what initially allowed me to understand it. 
\vspace{3mm}\\I'd like to give a special thanks to Francis Bischoff for the many hours we've spent discussing math, including topics directly related to this thesis. I'd like to give a shout-out to Ahmed, Kasun and David, who were almost the only people I had contact with during the time surrounding the lockdowns. We had some fun times. I will also take the time to acknowledge the people who were there at the beginning: Krishan, who helped motivate me to take running more seriously all those years ago, Lennart, Sina, Adriano, Chris, Matt, Afiny, Zhara and the members of GLab. Last but not least, I'd like to acknowledge the wonderful time I had exploring Tokyo with Omar, Shinji and Marco (garbage bags and all). In particular, the time I spent at The Kavli Institute for the Physics and Mathematics of the Universe (IPMU) was terrific. I surely missed some individuals, and for that I apologize. 
\vspace{3mm}\\I am thankful to the administrative staff in the department of mathematics for all that they've done, including Jemima Merisca who helped set up the meetings and my defense, and for ensuring that I submit this thesis on time.
\vspace{3mm}\\I would like to acknowledge the funding I received from The Faculty of Arts and Science in the form of The Faculty of Arts and Science Top (FAST) Doctoral Fellowship.
\vspace{3mm}\\Finally, I'd like to thank my family for all of their support.
\tableofcontents
\chapter{Introduction and Applications}
\pagenumbering{arabic}
\section{Introduction to Parts 1 and 2}
\subsection{A Bit of History and Motivation}
In 1986, van Est (1921-2002) published a novel proof of Lie's third theorem, which he ascribed to Cartan (1869-1951) — the person who is credited with originally proving the theorem. Recall that Lie's third theorem states that every Lie algebra has an integration, and is considered to be the most difficult of Lie's theorems. The proof used the van Est map and the van Est isomorphism  theorem; more precisely, given that that every matrix Lie algebra integrates to a Lie group (which is much easier to prove) and that every Lie group has vanishing second homotopy group, the van Est isomorphism theorem completes the proof. 
\vspace{3mm}\\Let us give a brief synopsis of the van Est map and isomorphism theorem: let $G$ be a Lie group and let $E$ be a representation of $G\,.$ We can differentiate this structure to obtain the corresponding Lie algebra $\mathfrak{g}$ and the corresponding representation of $\mathfrak{g}$ on $E\,.$ From this data we get two cohomologies: the Lie group cohomology and Lie algebra cohomology with coefficients in $E\,,$ denoted $H^*(G,E)$ and $H^*(\mathfrak{g},E)\,,$ respectively. Originally, the van Est map $VE$ was a map
\begin{equation}
    VE:H^*(G,E)\to H^*(\mathfrak{g},E)\,.
\end{equation}
More generally, given a compact subgroup $K\xhookrightarrow{}G\,,$ with Lie algebra $\mathfrak{k}\,,$ the van Est map factors through the relative Lie algebra cohomology $H^*(\mathfrak{g},\mathfrak{k},E)\,.$ That is, there is a map 
\begin{equation}
VE_{G,K}:H^*(G,E)\to H^*(\mathfrak{g},\mathfrak{k},E)\,.
\end{equation}
Forms in the relative Lie algebra complex are forms in the Lie algebra complex of $\mathfrak{g}$ which evaluate to $0$ when contracted with any vector in $\mathfrak{k}\,,$ and which are invariant under the conjugation action of $H\,.$ Classically, these maps are what has been meant by van Est maps, and essentially what van Est proved amounts to the following theorems: 
\begin{theorem}\label{van est original}
Suppose $G\rightrightarrows *$ has vanishing homotopy groups up to degree $n\,.$ Then $VE$ is an isomorphism up to and including degree $n\,,$ and is injective in degree $n+1\,.$ 
\end{theorem}
\begin{theorem}\label{van est compact}
Let $K$ be the maximal compact subgroup of $G\rightrightarrows *\,.$ The map $VE_{G,K}$ is an isomorphism in all degrees.
\end{theorem}
Later on, the van Est map was extended to Lie groupoids by Weinstein, Xu and others: given a Lie groupoid $G\rightrightarrows G^0$ and a representation $E$ of $G\,,$ we obtain through differentiation a corresponding Lie algebroid $\mathfrak{g}\to G^0$ and a corresponding representation of $\mathfrak{g}$ on $E\,.$ There is a van Est map, still denoted $VE\,,$ from the Lie groupoid cohomology to the Lie algebroid cohomology:
\begin{equation}
    VE:H^*(G,E)\to H^*(\mathfrak{g},E)\,.
\end{equation}
Crainic proved the following result:
\begin{theorem}\label{crainic}
If the target fibers have vanishing homotopy groups up to and including degree $n\,,$ then $VE$ is an isomorphism up to and including degree n, and is injective in degree $n+1$.
\end{theorem} Crainic also described the image of the map in degree $n+1$ (and in fact proved a more general result involving a proper action). There are, in particular, applications of this result to the integration of Poisson manifolds, and more generally to the integration of Lie algebroids. 
\vspace{3mm}\\The van Est map is one of the main tools we have to compute Lie groupoid cohomology. Many others have worked on van Est maps and isomorphism theorems, some authors are: Arias Abad, Cabrera, Li-Bland, Meinrenken, Salazar, et al. — van Est maps have been proven to be very useful. However, in the author's opinion the van Est map as currently defined has three drawbacks, which will be addressed in this thesis (in no particular order):
\begin{enumerate}
    \item The van Est map is only defined for coefficients in a representation. However, we would like to consider more general coefficients so that we can use van Est theorems to prove a wider range of results. In particular, $S^1$-valued functions and our theorem are relevant to: computing characters and $S^1$-extensions of Lie groups; computing representations of Lie groupoids and the geometric quantization of Poisson manifolds and Courant Algebroids; the basic gerbe over a compact simple Lie group. In particular, the theorem we prove in this paper can be used to derive the following classical results (either immediately, or with a small amount of work):
    \begin{itemize}
        \item Let $P\to N$ be a principal torus-bundle with an action of a compact, simply connected Lie group $G$ on $N\,.$ Then the action of $G$ lifts to $P\,,$ and the lift is unique up to isomorphism.
        \item If $G\rightrightarrows G^0$ has $n$-connected target fibers, then $H^k(G,\mathbb{Z})\cong H^k(G^0,\mathbb{Z})$ for $0\le k\le n$ (this is a special case of the theorem)
        \item The Poincar\'{e} lemma (this is a special case of the theorem)
        \item van Est's original result, \Cref{van est compact} 
    \end{itemize}
The first two results are related to issues about coefficients, and the third and fourth result are related to point 3, which we will discuss in a moment (Weinstein and Xu in~\cite{weinstein}, Crainic and Zhu in~\cite{zhu} did consider a version of groupoid cohomology with coefficients in $S^1\,,$ and proved some isomorphism theorems in degrees one and two; van Est also wrote about more general coefficients in~\cite{van est 2}, as did Brylinski in~\cite{brylinski}).

   \item The second drawback is related to the first one: the van Est map is only defined in the smooth category, but it is desirable to have one in the holomorphic category as well (this is essentially changing coefficients from smooth to holomorphic functions).
    \item In the absence of a proper action, the van Est map for Lie groupoids doesn't give any information about the higher degree cohomology of the groupoid. We would like a more general theorem that contains \Cref{van est compact} as a special case, and allows us to compute higher degrees of cohomology at the infinitesimal level. In addition, van Est's result doesn't hold when you change coefficients (similarly, it doesn't hold in the holomorphic category), therefore we need a more general theorem to compute even the cohomology of Lie groups.
\end{enumerate}
Consider the following setting: let $\pi:Y\to X$ be a surjective submersion of smooth (complex) manifolds. There is a morphism $H^*(X,\mathcal{O})\to H_{\pi}^*(Y)\,,$ where $H_{\pi}^*(Y)$ denotes the foliated de Rham cohomology of $Y\,.$ Explicitly, the map is given by first applying the map  $H^*(X,\mathcal{O})\to H^*(Y,\pi^{-1}\mathcal{O})\,,$ and then taking a fiberwise de Rham resolution. Now, we have the following theorem (see criterion 1.9.4 in \cite{Bernstein}, \cite{buch}):
\begin{theorem}
Let $\pi:Y\to X$ be a surjective submersion of smooth (complex) manifolds, such that the fibers of $\pi$ are $n$-connected. Then the  morphism $H^*(X,\mathcal{O})\to H_{\pi}^*(Y)$ is an isomorphism up to degree $n$ and injective in degree $n+1\,.$
\end{theorem}
Of course, in the smooth setting $H^*(X,\mathcal{O})$ is zero in positive degrees, however this isn't true in the holomorphic category, and the point is that you can consider any sheaf of functions on $X$ valued in some abelian Lie group and an analogous result holds. The statement of this result is similar to the statement of the van Est theorem, with $Y$ playing the role of $G^0$ and $X$ playing the role of $G;$ a slight generalization of this result is used to prove the van Est theorem. That this result should, in addition, be a special case of a van Est-type isomorphism theorem was one of the author's main motivations for this direction of study.
\subsection{Generalizing the van Est Map}
In this thesis we are going to interpret the van Est map as a result about differentiable stacks. More precisely, a sufficiently nice map of stacks $[H^0/H]\to [G^0/G]$ determines a foliation of $[H^0/H]\,,$ and we can compute the foliated cohomology. The van Est map will then be, roughly, a map from the cohomology of $[G^0/G]$ to the foliated cohomology of $[H^0/H]\,.$ The case of the van Est map for Lie groupoids corresponds to the case that the map of stacks is the one represented by the inclusion $G^0\xhookrightarrow{}G\,.$ van Est's original result, \Cref{van est compact}, is obtained by taking the map of stacks to be the one represented by the inclusion of the maximal compact subgroup $K\xhookrightarrow{} G\,.$ The Poincar\'{e} lemma will be obtained by letting the map of stacks be the one represented by $X\to *\,,$ where $X$ is a contractible space and $*$ is a point.
\subsection{Rough Explanation of the van Est map}
A (nice enough) map $f:H\to G$ of Lie groupoids determines a ``foliation" of $H\,,$ which determines a Lie algebroid-groupoid over $H\,.$ There is a canonical map from the groupoid cohomology of $G$ to the foliated cohomology of $H\,,$ obtained by first applying the inverse image functor to cohomology classes, and then taking a resolution by foliated differential forms. The aformentioned notion of foliation is not always the usual one associated to Lie groupoids, but is one that is appropriate when working in the (2,1)-category of Lie groupoids. For example, consider a Lie group $G\,;$ there is a canonical map $*\to G\,,$ which is just the inclusion of the identity element. The foliation  of $*$ deteremined by this map would naively be the $0$ vector space, but with the notion of foliation we are using it is actually equivalent to the Lie algebra $\mathfrak{g}\,.$ 
\subsection{Addressing the Drawbacks}
To address drawbacks one and two, we first need a more general definition of Lie algebroid cohomology that allows us to use coefficients that are not in a represention\footnote{Weinstein and Xu allude to this possiblity in their paper on the quantization of symplectic groupoids}. For example, suppose we want to use $S^1$-coefficients. Then, if we let $\mathfrak{g}=TX$ for some manifold $X\,,$ changing coefficients from $\mathcal{O}$ (with the trivial action of $TX\,)$ to $\mathcal{O}^*$ would involve passing from de Rham cohomology \begin{equation}
\mathcal{O}_X\xrightarrow[]{\text{d}} \Omega^1_X\to \Omega^2_X\to\cdots
\end{equation}
to Deligne cohomology 
\begin{equation}
\mathcal{O}^*_X\xrightarrow[]{\text{dlog}} \Omega^1_X\to \Omega^2_X\to\cdots\,.
\end{equation}
More generally, given any Lie algebroid $\mathfrak{g}\to X$ and any abelian Lie group $A\,,$ the Lie algebroid forms we get are: 
\begin{enumerate}
    \item In degree $0\,,$ functions on $X$ taking values in $A\,,$
    \item In degree $n>0\,,$ Lie algebroid n-forms taking values in the Lie algebra $\mathfrak{a}$ of $A\,.$
\end{enumerate} 
In general, given a Lie groupoid $G\rightrightarrows G^0\,,$ the coefficients we consider are $G$-modules: these are essentially representations $\pi: M\to G^0$ of $G\,,$ except that, unlike a representation, the fibers of the map $\pi$ don't need to be vector spaces - they can be any abelian Lie group. Once this is done, drawbacks one and two are addressed by using a generalization of Crainic's proof of the van Est isomorphism theorem. 
\vspace{3mm}\\The third drawback is more subtle to resolve. In order to do this, it is best to think of the category of Lie groupoids as a (2,1)-category, where the 2-morphisms between maps of Lie groupoids are natural isomorphisms. In this (2,1)-category, there is a distinct notion of fibers of maps between Lie groupoids, as well as fibrations. Using this notion, and thinking of $G^0$ as a Lie groupoid with only identity morphisms, the fibers of the natural map $G^0\xhookrightarrow{} G$ are simply the target fibers of $G\rightrightarrows G^0\,.$ Therefore, thinking of Lie groupoids as objects in a (2,1)-category, we can restate the van Est isomorphism theorem for Lie groupoids as so:
\begin{theorem}
If the fibers of the natural map $G^0\xhookrightarrow{} G$ are n-connected (ie. have vanishing homotopy groups up to and including degree $n\,),$ then $VE$ is an isomorphism up to and including degree n, and is injective in degree $n+1\,.$
\end{theorem}
Now, we will show that to every nice enough homomorphism of $f:H\to G$ of Lie groupoids, one can associate a Lie algebroid-groupoid over $H\,,$ which we will denote $D\to H\,.$ It is then natural to ask: how do the cohomologies of $G$ and this Lie algebroid-groupoid $D\to H$ compare? First, we will define a van Est map 
\begin{equation*}
    VE: H^*(G,M)\to H^*(D\to H, f^*M)\,.
\end{equation*}We will then prove the following theorem:
\begin{theorem}\label{van est morphism}
Let $f:H\to G$ be a homomorphism of Lie groupoids, which is a surjective submersion at the level of objects. Suppose further that the fibers of $f$ are all n-connected. Then the van Est map is an isomorphism up to and including degree $n\,,$ and is injective in degree $n+1\,.$ 
\end{theorem}
We will also describe its image in degree $n+1\,.$ Letting $H=G^0\,,$ we recover the usual Lie algebroid of $G$ and the usual van Est map. 
\subsection{The Lie Algebroid-Groupoid Associated to the Normal Bundle of a Subgroup}
Before continuing the discussion of the van Est map, let's motivate one instance of associating a Lie algebroid-groupoid to a (nice enough) map $H\to G$ — it resolves the following conundrum: the normal bundle of the identity bisection inherits the structure of a Lie algebroid, so what structure does the normal bundle to a subgroupoid inherit? It should model a small neighborhood of the subgroupoid, in the same way that the Lie algebroid models a small neighborhood of the identity bisection.
\vspace{3mm}\\To illustrate how the normal bundle of $H^{(1)}\xhookrightarrow{}G^{(1)}$ inherits the structure of an LA-groupoid (short for Lie algebroid-groupoid), let's specialize to the case of Lie groups. Let $G$ be a Lie group and $H\xhookrightarrow{} G$ a subgroup. We claim to have the following Lie algebroid-groupoid:
\begin{equation}\label{LA group}
\begin{tikzcd}
H\ltimes_{\text{Ad}}\mathfrak{h}\ltimes\mathfrak{g} \arrow[r, shift left] \arrow[r, shift right] \arrow[d] & \mathfrak{g} \arrow[d] \\
H \arrow[r, shift left] \arrow[r, shift right]                                         & *                 
\end{tikzcd}
\end{equation}
Here, the right column is just the Lie algebra of $G\,,$ and the bottom row is just the Lie group $H\,.$ The Lie algebroid of the left column comes from the identification of $TH\cong H\times \mathfrak{h}$ (where $\mathfrak{h}$ is the Lie algebra of $H\rightrightarrows *\,).$ Then, this Lie algebroid is just the product of the Lie algebroids $TH\to H$ and $\mathfrak{g}\to *$ (it's really the trivial bundle of Lie algebras $\mathfrak{h}\times \mathfrak{g}$ over $H\,).$ Now for the top row: here $H\ltimes_{\text{Ad}} h$ is the semidirect product of $H$ and $\mathfrak{h}$ associated to the adjoint representation of $H$ on $\mathfrak{h}\,.$ There is a natural action of this group on $\mathfrak{g}:$ letting $(h,X_\mathfrak{h})\in H\ltimes \mathfrak{h}\,,\tilde{X}_\mathfrak{g}\in \mathfrak{g}\,,$ we have an action given by $(h,X_\mathfrak{h})\cdot \tilde{X}_\mathfrak{g}=Ad_h\tilde{X}_\mathfrak{g}+X_\mathfrak{h}\,.$ 
\vspace{3mm}\\Now to explain how the LA-groupoid in \ref{LA group} relates to the normal bundle of $H\xhookrightarrow{}G:$ applying the forgetful functor from LA-groupoids to VB-groupoids (ie. a vector bundle over a Lie groupoid), we obtain the following VB groupoid: 
\begin{equation}\label{VB group}
\begin{tikzcd}
H\ltimes_{\text{Ad}}\mathfrak{h}\ltimes\mathfrak{g} \arrow[r, shift left] \arrow[r, shift right] \arrow[d] & \mathfrak{g} \arrow[d] \\
H \arrow[r, shift left] \arrow[r, shift right]                                         & *                 
\end{tikzcd}
\end{equation}
ie. the diagram looks the same, we have just forgotten the Lie brackets. Now, the adjoint action of $H$ on $\mathfrak{g}$ descends to an action of $H$ on $\mathfrak{g}/\mathfrak{h}\,,$ and the groupoid $H\ltimes_{\text{Ad}}\mathfrak{h}\ltimes\mathfrak{g}\rightrightarrows\mathfrak{g}$ is Morita equivalent to $H\ltimes_{\text{Ad}}\mathfrak{g}/\mathfrak{h}\rightrightarrows \mathfrak{g}/\mathfrak{h}\,.$ As a result of this, the VB-groupoid \ref{VB group} is Mortia equivalent to the following VB-groupoid\footnote{If $\mathfrak{h}\subset Z(\mathfrak{g})$ (the center of $\mathfrak{g}$) and if $\mathfrak{g}\cong\mathfrak{h}\oplus\mathfrak{g}/\mathfrak{h}\,,$ then this is a Morita equivalence of LA-groupoids.}
\begin{equation}\label{VB group morita}
\begin{tikzcd}
H\ltimes_{\text{Ad}}\mathfrak{g}/\mathfrak{h} \arrow[r, shift left] \arrow[r, shift right] \arrow[d] & \mathfrak{g}/\mathfrak{h} \arrow[d] \\
H \arrow[r, shift left] \arrow[r, shift right]                                         & *                 
\end{tikzcd}
\end{equation}
Now applying the forgetful functor from VB-groupoids to vector bundles over manifolds, we get 
\begin{equation}\label{VB manifold group}
\begin{tikzcd}
H^{(1)}\times\mathfrak{g}/\mathfrak{h}\arrow[d]  \\
H^{(1)}      
\end{tikzcd}
\end{equation}
The vector bundle \ref{VB manifold group} is naturally identified with the normal bundle of $H^{(1)}\xhookrightarrow{} G^{(1)}\,.$ Hence, in this sense, the normal bundle inherits the structure of an LA-groupoid.
\subsection{A Sketch of the Proof of van Est's Original Result}
Now let's specialize the LA-groupoid \ref{LA group} to the case where $H=K$ is the maximal compact subgroup and $E$ is a representation of $G\,.$ The goal here will be to sketch the proof of \Cref{van est compact}. Two facts will be relevant:
\begin{enumerate}
    \item $K\xhookrightarrow{} G$ is a homotopy equivalence,
    \item The cohomology of a proper Lie groupoid, with coefficients in a representation, is trivial in all degrees higher than $0\,.$
\end{enumerate}
Fact 1 implies that the fiber, which is Morita equivalent to $G/K\,,$ is contractible. Then, using \Cref{van est morphism}, we get that the cohomology of \ref{LA group} with coefficients in $E$ is isomorphic to $H^*(G,E)\,.$ Now, both the top and bottom groupoids in \ref{LA group} are proper, therefore fact 2 implies that the cohomology of \ref{LA group} reduces to the invariant cohomology of the right column. To expound on this, the cohomology of the right column is just the Lie algebra cohomology of $\mathfrak{g}\,,$ ie. the cohomology of the complex
\begin{equation}\label{che complex}
E\xrightarrow[]{d} Hom(\mathfrak{g},E)\xrightarrow[]{d} Hom (\Lambda^2\mathfrak{g},E)\xrightarrow[]{d}\cdots\,.
\end{equation}
Now, the complex we get from \ref{LA group} is the subcomplex of \ref{che complex} consisting of those forms invariant under the action of $K\ltimes_{\text{Ad}}\mathfrak{k}\,,$ ie. forms which evaluate to $0$ upon contraction with any vector in $\mathfrak{k}\,,$ and which are invariant under the conjugation action of $K$ - this is exactly the aformentioned relative Lie algebra complex, therefore we have obtained van Est's result.
\subsection{LA-Groupoids of Homomorphisms \texorpdfstring{$H\to G$}{H to G}}
Now let's discuss an interpretation of the LA-groupoid associated to a map $H\to G\,.$ Recall that to every Lie groupoid $H$ we can associate an LA-groupoid $TH\to H$ by forming degreewise tangent bundles. A foliation of a Lie groupoid $H$ is a wide sub-LA-groupoid of the tangent LA-groupoid $TH\to H\,.$ The LA-groupoid cohomology of a foliation of $H$ can be thought of as the tangential de Rham cohomology, ie. the cohomology of differential forms which take an inputs only vectors in the foliation. We will explain how the LA-groupoid determined by a (nice enough) map between groupoids $H\to G$ can be thought of as a foliation of $H$ associated to the map $H\to G$ (in the (2,1)-category sense, ie. after replacing $H$ with a Morita equivalent groupoid). In particular, a Lie algebra $\mathfrak{g}\to *$ is a foliation of $*\,.$
\subsection{LA-Groupoid of \texorpdfstring{$Y\to X$}{Y to X}} We will first describe how this works in the extreme (but more intuitive) case that the morphism $H\to G$ is just a surjective submersion\footnote{Much of the author's intuition about Lie groupoids comes from surjective submersions between smooth manifolds} between smooth manifolds $\pi:Y\to X\,,$ and where the coefficients are in $\mathcal{O}\,.$ In this case, we can form the submersion groupoid $Y\times_X Y\rightrightarrows Y$ (whose formation can be thought of as replacing the map $Y\to X$ with the cofibration $Y\to Y\times_X Y\rightrightarrows Y\,),$ and we can take its Lie algebroid. This is the same Lie algebroid as the foliation determined by the fibers of $\pi\,,$
and the cohomology of this Lie algebroid is just the de Rham cohomology of differential forms which only take as inputs tangent vectors in the foliation. Thus, we have two methods of obtaining the same Lie algebroid (which can be thought of as an LA-groupoid by considering manifolds to be groupoids with only identity morphisms).
\subsection{LA-Groupoid of \texorpdfstring{$*\to G$}{* to G}} For a more involved example, we consider the extreme case on the opposite side of the spectrum: the case of a Lie group $G\rightrightarrows *$ and the mapping $*\xhookrightarrow{}G\,.$ We can also form a ``submersion groupoid", and the result is just the group $G\rightrightarrows *\,,$ and the Lie algebroid we get is $\mathfrak{g}\mapsto *\,.$ Now, the claim is that this Lie algebra can be thought of as a foliation of $*\,.$ Naively, the tangent bundle of $*$ is the zero vector space, therefore the foliation determined by this map seems like it should just be trivial. However, this isn't what we mean as this isn't the Morita invariant notion of foliation. What will will do is replace the map $*\xhookrightarrow{} G$ with a fibration, in the context of the (2,1)-category of Lie groupoids, which in this case will be the commutative diagram
\begin{equation}
\begin{tikzcd}
 G\rtimes G  \arrow[rd, "p_2"]                       &   \\
* \arrow[u, "{(*,*)}"] \arrow[r,] & G
\end{tikzcd}
\end{equation}
Here, $G\rtimes G$ is the action groupoid associated to the right action of $G$ on itself, and $p_2$ is the projection onto the second factor. Now, we can consider the LA-groupoid given by the tangent LA-groupoid $T(G\rtimes G)\to G\rtimes G\,,$ which has a natural map $p_{2*}$ to the tangent LA-groupoid $TG\to G\,,$ ie. we have a map
\begin{equation}
\begin{tikzcd}
T(G\rtimes G) \arrow[r, shift left] \arrow[r, shift right] \arrow[d] & TG \arrow[d] \arrow[rr, "p_{2*}", shift right=8] &  & TG \arrow[d] \arrow[r, shift right] \arrow[r, shift left] & * \arrow[d] \\
G\rtimes G \arrow[r, shift left] \arrow[r, shift right]              & G                                             &  & G \arrow[r, shift left] \arrow[r, shift right]            & *          
\end{tikzcd}
\end{equation}
Taking kernels of $p_{2*}$ as maps of vector bundles, we get a foliation of $G\rtimes G$ and a natural map to $\mathfrak{g}\,.$ Explicitly, the foliation and map are given by the following:
\begin{equation}
\begin{tikzcd}
TG\rtimes G \arrow[r, shift left] \arrow[r, shift right] \arrow[d] & TG \arrow[d] \arrow[rr, shift right=7] &  & \mathfrak{g} \arrow[d] \arrow[r, shift left] \arrow[r, shift right] & \mathfrak{g} \arrow[d] \\
G\rtimes G \arrow[r, shift left] \arrow[r, shift right]            & G                                      &  & * \arrow[r, shift left] \arrow[r, shift right]                  & *                 
\end{tikzcd}
\end{equation}
Here, the maps from the left and right column to $\mathfrak{g}$ are obtained by right translating vectors in $TG$ to the origin. Now, the map of groupoids on the top row is a Morita equivalence, which implies that this map of LA-groupoids is an equivalence. Therefore, the foliation is indeed Morita equivalent to $\mathfrak{g}\,,$ and closely corresponds to how $\mathfrak{g}$ is usually thought of as: the right invariant vector fields on $G\,.$ In addition, we have again obtained equivalent LA-groupoids using two different methods.
\subsection{Applications}\label{app}
Here we will state some applications of \Cref{van est morphism}. Before stating the first theorem, we will make some remarks. 
\vspace{4mm}\\Suppose we have a Lie group $G$ which acts on a manifold $N\,.$ Then given a subgroup $H$ of $G\,,$ there is an action of $H\ltimes_{\text{Ad}}\mathfrak{h}$ on $\mathfrak{g}\ltimes N\,,$ given by $(h, X_\mathfrak{h})\cdot (X_\mathfrak{g},n)=(\text{Ad}_h X_\mathfrak{g}+X_\mathfrak{h},h\cdot n)\,,$ where $X_\mathfrak{h}\in \mathfrak{h}\,, X_\mathfrak{g}\in\mathfrak{g}\,.$ From this, we get an action of $H\ltimes_{\text{Ad}}\mathfrak{h}$ on Lie algebroid forms on $\mathfrak{g}\ltimes N\,.$ Now we will state the first theorem, which generalizes \Cref{van est compact}, and is an application of \Cref{van est morphism} to the mapping $H\ltimes N\to G\ltimes N:$
\begin{theorem}
Let $G$ be a Lie group and $K$ its maximal compact subgroup. Let $N$ be a smooth manifold on which $G$ acts, and let $E\to N$ be a representation of $G\ltimes N\rightrightarrows N\,.$ Then we have that
\begin{equation}
    H^*(G\ltimes N, E)\cong H^*(\mathfrak{g}\ltimes N,\mathfrak{k}\ltimes N, E)\,,
\end{equation}
where the cohomology group on the right is the cohomology of the subcomplex of Lie algebroid forms on $\mathfrak{g}\ltimes N$ which are invariant under the action of $K\ltimes_{\text{Ad}}\mathfrak{k}\,.$
\end{theorem}
\vspace{4mm}The next result concerns lifting projective representations to representations, and is an application of \Cref{van est morphism} to the mapping $*\to G\,,$ with coefficients on $\mathcal{O}^*:$
\begin{theorem}
Let $G$ be a simply connected Lie group and let $V$ be a finite dimensional complex vector space. Let $\rho:G\to \text{PGL } (V)$ be a homomorphism. Then $G$ lifts to a homomorphism $\tilde{\rho}:G\to\text{GL } (V)\,.$ If $G$ is semisimple, this lift is unique.
\end{theorem} 
\vspace{4mm}The next theorem concerns lifting group actions to principal bundles, and is an application of \Cref{van est morphism} to the mapping of $N\to G\ltimes N\,,$ with coefficients in $T^n$ (the $n$-dimensional torus):
\begin{theorem}
Let $G$ be a compact, simply connected Lie group acting on a manifold $N\,,$ and let $P\to N$ be a principal bundle for the $n$-torus $T^n\,.$ Then, up to isomorphism, there is a unique lift of the action of $G$ to $P\,.$
\end{theorem}
\vspace{2mm}The next theorem generalizes a result proven by Crainic using different methods, and in particular gives a criterion for there to exist an integration of certain Lie algebroids. It is an application of \Cref{van est morphism} to the mapping $G^0\to G\rightrightarrows G^0\,,$ with coefficients in a $G$-module $M:$
\begin{theorem}\label{central extension1}
Consider the exponential sequence $0\to Z \to\mathfrak{m}\overset{\exp}{\to} M\,.$ Let
\begin{align}\label{LA extension1}
    0\to\mathfrak{m}\to \mathfrak{a} \to \mathfrak{g}\to 0
\end{align}
be the central extension of $\mathfrak{g}$ associated to $\omega\in H^2(\mathfrak{g},\mathfrak{m})\,.$ Suppose that $\mathfrak{g}$ has a simply connected integration $G\rightrightarrows X$ and that 
\begin{align}
    \int_{S^2_x}\omega\in Z
\end{align}
for all $x\in X$ and $S^2_x\,,$ where $S^2_x$ in a $2$-sphere contained in the source fiber over $x\,.$  Then $\mathfrak{a}$ integrates to a unique extension
\begin{align}
1\to M\to A \to G\to 1\,.
\end{align}
\end{theorem}
\vspace{4mm}\Cref{van est morphism} also ``knows" about some very classical results, including the Poincaré lemma, which concerns the mapping $\mathbb{R}^n\to *$ and coefficients in $\mathcal{O}$ (we do not claim it is a proof as it is surely circular, but the point is it does contain the result as a ``subtheorem"):
\begin{theorem}
Every closed differential form on $\mathbb{R}^n\,,$ in degree higher than $0\,,$ is exact.
\end{theorem}
\begin{proof}[Proof]
The LA-groupoid associated to the map $\mathbb{R}^n\to *$ is $T\mathbb{R}^n\,,$ whose cohomology is the de Rham cohomology of $\mathbb{R}^n$. By \Cref{van est morphism}, since $\mathbb{R}^n$ is contractible, this cohomology is the cohomology of the point $*\,,$ which is trivial in degrees higher than $0\,.$
\end{proof}
The next result, proved by Buchdahl in~\cite{buch}, is used when applying the Penrose transform in twistor theory (see section 4 of~\cite{bailey}, as well as~\cite{witten}). Again, this result is a special case of our theorem, however this result is essentially used in the proof.
\begin{theorem}
Let $Y\to X$ be a surjective submersion of complex manifolds with $n$-connected fibers. Then the natural map $H^*(X,\mathcal{O})\to H^*(Y,f^{-1}\mathcal{O})$ is an isomorphism up to degree $n$ and is injective in degree $n+1\,.$
\end{theorem}
More generally, in the previous theorem we can take local sections of any holomorphic vector bundle rather than $\mathcal{O}$ (in fact, our theorem shows that the we can use sheaves more general than local sections of vector bundles, and we also determine its image in degree $n+1)$.
\vspace{3mm}\\Let us remark that the van Est theorem we prove can be derived from a van Est theorem for double Lie groupoids, using the association of a double Lie groupoid to a (nice enough) map of Lie groupoids. In addition, a van Est theorem for relative cohomology (which will be defined in Part 2) should be derivable in this way.
\section{Introduction to Part 3}
Part 3 of this thesis is mostly disjoint from Parts 1 and 2, with the exception of the use of LA-groupoids towards the end, and some mention of the van Est map (but a deep understanding of this map is not needed). This part is largely conjectural — our aim is to understand what Morita equivalences and generalized morphisms of Lie algebroids are (and for completion, details need to be filled in). 
\vspace{3mm}\\Towards the end of Part 3 we define a category of LA-groupoids (with $n$-equivalences), which contains both Lie groupoids and Lie algebroids as objects and provides a ``wormhole" between them. In particular, this category unifies the category of differentiable stacks, the category of Lie algebroids and the homotopy category. This category has the following remarkable properties:
\begin{enumerate}
    \item Manifolds $X\,,Y$ are Morita equivalent in this category if and only if they are diffeomorphic.
    \item More generally, Lie groupoids $G\,,H$ are Morita equivalent in this category if and only if they are Morita equivalent in the category of Lie groupoids,
    \item Two tangent bundles $TX\,,TY$ are Morita equivalent in this category if and only if $X\,,Y$ are homotopy equivalent.
    \item If $G\rightrightarrows G^0$ is source $n$-connected, there is a canonical $n$-equivalence $\mathcal{I}:\mathfrak{g}\to G\,.$ This morphism can be interpreted as the integration functor. If $n=\infty$ this is a Morita equivalence (here $\mathfrak{g}$ is the Lie algebroid of $G$).
    \item With regards to the previous point, the van Est map is given by a pullback: $\mathcal{I}^*:H^*(G,M)\to H^*(\mathfrak{g},M)\,.$
    \item A Lie algebroid $\mathfrak{g}$ is integrable if and only if it is 1-equivalent to some Lie groupoid $G\,.$
    \item This category induces a natural notion of smooth homotopy equivalence of Lie groupoids.
    \item  A finite dimensional classifying space of a Lie groupoid $G$ is just a manifold $BG$ which is homotopy equivalent to $G$ (this may remind the reader of Grothendieck's homotopy hypothesis).
    \item If $EG\to BG$ is finite dimensional, then the Atiyah algebroid $\text{at}(EG)$ is Morita equivalent to $G\,.$ In particular, if $G$ is discrete then $\text{at}(EG)=T(BG)\,,$ therefore $G$ is Morita equivalent to $T(BG)\,.$
     \item Due to points 2, 3 and 4, we get the following result: suppose that $\mathcal{P}$ assigns to each LA-groupoid some property (eg.\ its cohomology) that is invariant under $n$-equivalence. Then if $X$ is homotopy equivalent to $Y\,;$ if $H\rightrightarrows H^0$ is Morita equivalent to $K\rightrightarrows K^0\,;$ if $G\rightrightarrows G^0$ is source $n$-connected, we get that $\mathcal{P}(TX)\cong \mathcal{P}(TY)\,;\mathcal{P}(H)\cong \mathcal{P}(K)\,;\mathcal{P}(G)\cong \mathcal{P}(\mathfrak{g})\,,$ respectively.
    \end{enumerate}
\vspace{3mm}To understand this we should start at the beginning. We make the following observation. The Lie algebroid cohomology of tangent bundles is invariant under homotopy equivalence of the underlying manifolds (for coefficients in $\mathcal{O}\,,$ this cohomology is the de Rham cohomology, for coefficients in $\mathbb{Z}$ it is the singular cohomology). Many interesting properties of tangent bundles are preserved under homotopy equivalences, and in addition we have the following result, due to Arias Abad, Quintero V\'{e}lez and  V\'{e}lez V\'{a}squez:
\begin{theorem}(see~\cite{Quintero}, Corollary 5.1)
If $f:X\to Y$ is a smooth homotopy equivalence, then the pullback functor 
\begin{equation}
    f^*:\textbf{Loc}_{\infty}(N)\to \textbf{Loc}_{\infty}(M)
\end{equation}
is a quasi-equivalence (ie. it is a quasi-equivalence between the dg categories of $\infty$-local systems).
\end{theorem}
We give a definition of Morita equivalence of Lie algebroids which says that two tangent bundles are Morita equivalent if and only if their underlying manifolds are homotopy equivalent. Before doing this, we motivate our definition. Once again, the definition we give is still conjectural. 
\vspace{3mm}\\One advantage of our formulation is that it offers an explanation for why certain cohomologies of manifolds are invariant under homotopy equivalences and others are not: a cohomology should be invariant under homotopy equivalences if it can be interpreted as Lie algebroid cohomology of the tangent bundle. In fact, we do a little more: similarly to how topological spaces have $n$-equivalences (which in nice topological categories give homotopy equivalences when $n=\infty$), we find that Lie algebroids also come with a natural notion of $n$-equivalences.
\vspace{3mm}\\Next we discuss how gerbes can be interpreted as generalized morphisms into a higher category, namely, $n$-gerbes are morphisms into the category of Lie $n$-groupoids (or $n$-fold groupoids; this may be related to Example 6.16 in~\cite{Nuiten}). For $n=1\,,$ this gives the usual interpretation of principal $G$-bundles as generalized morphisms into $G\,.$ This suggests an interpretation of the van Est map (for coefficients in a representation) as a map taking higher (generalized) morphisms of Lie groupoids into higher (generalized) morphisms of Lie algebroids.
\vspace{3mm}\\The fact that tangent bundles are Morita equivalent (under our definition) if and only if their underlying manifolds are homotopy equivalent, together with the fact that the van Est map can be interpreted as a map taking higher generalized morphisms to higher generalized morphisms, suggests the existence of a smooth version of the homotopy hypothesis (see~\cite{groth}), roughly (and up to equivalence)
\begin{equation}
    \text{Higher Lie groupoids}\cong\text{Higher Lie algebroids.}
\end{equation} 
After this discussion, we move on to discuss differentiablity of generalized morphisms. Already, the fact that Morita maps of Lie groupoids don't differentiate to (anything like) equivalences of Lie algebroids suggests that this isn't always possible: $\text{Pair}(X)$ is Morita equivalent to $*\,,$ however the number of points in the leaf space of its Lie algebroid $TX$ can be arbitrarily high, whereas the leaf space of $T*$ consists of a single point. Furthermore, letting $X=S^1\,,$ $TS^1$ has nontrivial representations, whereas $T*$ does not. We suggest the situation is even more dire. We prove the following no-go result:
\begin{theorem}\label{strong}
There is no notion of generalized moprhisms between Lie algebroids with the following properties:
\begin{enumerate}
    \item Associated to any generalized morphism $P:G\to H$ is a generalized morphism $dP:\mathfrak{g}\to\mathfrak{h}\,.$
    \item Generalized morphisms $\mathfrak{g}\to\mathfrak{h}$ induce pullback maps $H^{1}(\mathfrak{h})\to H^{1}(\mathfrak{g})\,,$ in such a way that the pullback of a trivial class is trivial.\footnote{One may assume that we are taking cohomology of cocycles valued in $\mathbb{R}\,.$}
    \item Pullback of cohomology commutes with the van Est map. That is, the following diagram is commutative:
    \begin{equation}
        \begin{tikzcd}
{H^{1}(H)} \arrow[r, "P^*"] \arrow[d, "VE"'] & {H^{1}(G)} \arrow[d, "VE"] \\
{H^{1}(\mathfrak{h})} \arrow[r, "dP^*"]      & {H^{1}(\mathfrak{g})}     
\end{tikzcd}
    \end{equation}
\end{enumerate}
\end{theorem}
Our resolution is that property 1 should not be required to hold, and we discuss, from the perspective of the smooth version of the homotopy hypothesis, why some generalized morphisms cannot be differentiated. We then move on to discuss in which sense certain generalized morphisms can be differentiated. What we find is that, a generalized morphism $G\to H$ can be differentiated ``up to degree $n$" if $H$ has source $n$-connected fibers (this is sufficient but not necessary, of course strict homomorphisms can always be differentiated).
\vspace{3mm}\\The smooth homotopy hypothesis suggests that there should be morphisms between Lie algebroids and Lie groupoids. Towards the end of chapter 2 (Part 3) we will make this precise: we define \textit{exotic} morphisms between Lie algebroids and Lie groupoids by using the embeddings of both Lie algebroids and Lie groupoids into LA-groupoids, and this will give us another interpretation of the van Est map. We prove the following result:
\begin{theorem}
Let $G\rightrightarrows G^0$ be a source $n$-connected Lie groupoid. Then $\mathfrak{g}\to G^0$ is $n$-equivalent to $G\rightrightarrows G^0\,,$ via a canonical $n$-generalized morphism (if $n=\infty$ they are Morita equivalent).
\end{theorem}
In fact, what we show is that invariance of LA-groupoid cohomology under $n$-equivalence (which we expect to be true up to degree $n$) implies the van Est isomorphism theorem and the homotopy invariance of de Rham cohomology and singular cohomology, and of course the Morita invariance of Lie groupoid cohomology.
\section{Structure of Thesis}
Part 1 of the thesis is due to a paper published by the author. It describes a generalization of the van Est map and isomorphism theorem to cohomology with more general coefficients than functions valued in a representation. This part does contain some important definitions and concepts which are important for part 2; the crucial material for understanding part 2 is contained in chapters 1 and 2 of part 1. In addition, this part describes a canonical module associated to a complex manifold with divisor, and a section called ``Integration by Prequantization" which describes an alternative way of integrating Lie algebroid cohomology classes that doesn't use the van Est map. Furthermore, section 4 contains a list of applications of the van Est map.
\vspace{3mm}\\Part 2 of this thesis contains the full generalization of the van Est map to stacks. This part is more conceptual in nature and contains a few more concepts from category theory than part 1, mostly in the context of our definitions of fibrations, cofibrations and foliations. If the intention is to understand the van Est map, sections 2.2, 3.5, 4.2 and 5 can be skipped, though it is not necessarily recommended; these sections, in particular, describe a higher category of double groupoids, which is necessary for our discussion of cofibrations, which leads us to showing that a nice enough map between Lie groupoids can be replaced with a cofibration; section 5, in particular, contains a description of the equivalence of two LA-groupoids associated to a groupoid homomorphism, as well as an explanation of how the normal bundle of a subgroupoid can be thought of as an LA-groupoid. Section 5 also describes how every wide subgroupoid of a Lie groupoid comes with a canonical representation. 
\vspace{3mm}\\Part 3 of this thesis is largely speculative and in this part we give conjectural definitions of Morita equivalences of Lie algebroids and generalized morphisms. It also describes a link between Lie's second theorem and the van Est isomorphism theorem, which involves viewing gerbes as generalized morphisms in a higher category. We discuss connections with the homotopy hypothesis and we prove a no-go theorem regarding generalized morphisms of Lie algebroids. We define generalized morphisms between Lie algebroids and Lie groupoids, and a category of LA-groupoids (with $n$-equivalences). We conclude this part with conjectures and future directions.
\vspace{3mm}\\We recommend reading the introduction first (at least the parts relevant to the reader's interest). Some of the material in this thesis will be known to experts (perhaps unknowingly to the author), and with regards to this material no originality is claimed. 
\part{Van Est Theory With Coefficients in a Module}
\chapter{Basics of Simplicial Manifolds, Stacks, Sheaves}
\subsection*{Outline of Part 1}
Part 1 of this thesis is based on a paper written by the author, see~\cite{Lackman}. This part is organized as follows: section $1$ is a brief review of simplicial manifolds, Lie groupoids and stacks, but it is important for setting up notation, results and constructions which will be used in the next sections. This section contains all of the results about stacks which are needed for this part. Section $2$ contains a review and a generalization of the Chevalley-Eilenberg complex. Section $3$ is where we define the van Est map and prove the main theorem of part 1. The next sections of part 1 concern applications of the main theorem, various new constructions of geometric structures involving Lie groupoids, and examples.
\\\\$\mathbf{Notation:}$ For the rest of the paper, we use the following notation: given a smooth (or holomorphic) surjective submersion $\pi:M\to X\,,$ we let $\mathcal{O}(M)$ denote the sheaf of smooth (or holomorphic) sections of $\pi\,.$
\section{Simplicial Manifolds}
In this section we briefly review simplicial manifolds, sheaves on simplicial manifolds and their cohomology. 
\begin{definition}Let $Z^\bullet$ be a (semi) simplicial manifold, ie. a contravariant functor from the (semi) simplex category to the category of manifolds. A sheaf $\mathcal{S}_\bullet$ on $Z^\bullet$ is a sheaf $\mathcal{S}_n$ on $Z^n\,,$ for all $n\ge 0\,,$ such that for each morphism $f:[n]\to[m]$ we have a morphism $\mathcal{S}(f):Z(f)^{-1}\mathcal{S}_n\to  \mathcal{S}_m\,,$ and such that $\mathcal{S}(f\circ g)=\mathcal{S}(f)\circ \mathcal{S}(g)\,.$ A morphism between sheaves $\mathcal{S}_\bullet$ and $\mathcal{G}_\bullet$ on $Z^\bullet$ is a morphism of sheaves $u^n:\mathcal{S}_n\to \mathcal{G}_n$ for each $n\ge 0$ such that for $f:[n]\to [m]$ we have that
$u^m\circ \mathcal{S}(f)=\mathcal{G}(f)\circ u^n\,.$  We let $\textrm{Sh}(Z^\bullet)$ denote the category of sheaves on $Z^\bullet\,.$ 
$\blacksquare$\end{definition}
\begin{definition}
Given a sheaf $\mathcal{S}_\bullet$ on a (semi) simplicial manifold $Z^\bullet\,,$ we define $Z^n:=\text{Ker}\big[\Gamma(\mathcal{S}_n)\xrightarrow{\delta^*} \Gamma(\mathcal{S}_{n+1})\big]\,,\,B^n:=\text{Im}\big[\Gamma(\mathcal{S}_{n-1})\xrightarrow{\delta^*}\Gamma(\mathcal{S}_n)\big]\,,$
where $\delta^*$ is the alternating sum of the face maps, ie.
\begin{align*}
    \delta^*=\sum_{i=0}^n(-1)^id_{n,i}^{-1}\,,
    \end{align*}
 where $d_{n,i}:Z^n\to Z^{n-1}$ is the $i^{\textrm{th}}$ face map. We then define the naive cohomology (see \cite{nlab}) \begin{equation*}H^n_{\text{naive}}(Z^\bullet\,,\mathcal{S}_\bullet):=Z^n/B^n\,.
\end{equation*}
$\blacksquare$\end{definition}
\begin{definition}[see~\cite{Deligne}]\label{cohomology simplicial} Given a (semi) simplicial manifold $Z^\bullet\,,$ $\textrm{Sh}(Z^\bullet)$ has enough injectives, and we define \begin{equation*}
H^n({Z^\bullet\,,\mathcal{S}_\bullet}):=R^n\Gamma_\text{inv}(\mathcal{S}_\bullet)\,,
\end{equation*}
where $\Gamma_\text{inv}:\textrm{Sh}(Z^\bullet)\to \mathbf{Ab}$ is given by $\mathcal{S}_\bullet \mapsto\text{Ker}\big[\Gamma(\mathcal{S}_0)\xrightarrow{\delta^*} \Gamma(\mathcal{S}_1)\big]\,.$
$\blacksquare$\end{definition}
\theoremstyle{definition}\begin{rmk}\label{acyclic resolution}
As usual, in addition to injective resolutions one can use acylic resolutions to compute cohomology.
\end{rmk}
\theoremstyle{definition}\begin{rmk}[see \cite{Deligne}]\label{double complex}
A convenient way to compute $H^*({Z^\bullet\,,\mathcal{S}_\bullet})$ is to choose a resolution 
\begin{align*}
    0\xrightarrow{}\mathcal{S}_\bullet\xrightarrow{}\mathcal{A}^0_\bullet\xrightarrow{\partial_\bullet^0}\mathcal{A}^1_\bullet\xrightarrow{\partial_\bullet^1}\cdots
\end{align*}
such that 
\begin{align*}
    0\xrightarrow{}\mathcal{S}_n\xrightarrow{}\mathcal{A}^0_n\xrightarrow{\partial_n^0}\mathcal{A}^1_n\xrightarrow{\partial_n^1}\cdots
\end{align*}
is an acyclic resolution of $\mathcal{S}_n\,,$ for all $n\ge 0\,,$ and then take the cohomology of the total complex of the double complex $C^p_q=\Gamma(A^p_q)\,,$  with differentials $\delta^*$ and $\partial^p_q\,.$
\end{rmk}
$\mathbf{}$
\\The following theorem is a well-known consequence of the Grothendieck spectral sequence: 
\begin{theorem}[Leray Spectral Sequence]\label{spectral}
Let $f:X^\bullet\to Y^\bullet$ be a morphism of simplicial topological spaces, and let $\mathcal{S}_\bullet$ be a sheaf on $X^\bullet\,.$ Then there is a spectral sequence $E^{pq}_*\,,$ called the Leray spectral sequence, such that $E^{pq}_2=H^p(Y^\bullet,R^qf_*(\mathcal{S}_\bullet))$ and such that 
\begin{align*}
 E^{pq}_2\Rightarrow H^{p+q}(X^\bullet,\mathcal{S}_\bullet)\,.
\end{align*}
\end{theorem}
\subsection{Stacks}
Here we briefly review the theory of differentiable stacks. A differentiable stack is in particular a category, and first we will define the objects of the category, and then the morphisms. All manifolds and maps can be taken to be in the smooth or holomorphic categories. The following definitions can be found in~\cite{Kai}.
\begin{definition}
Let $G\rightrightarrows G^0$ be a Lie groupoid. A $G$-principal bundle (or principal $G$-bundle) is a manifold $P$ together with a surjective submersion $P\overset{\pi}{\to}M$ and a map $P\overset{\rho}{\to} G^0\,,$ called the moment map, such that there is a right $G$-action on $P\,,$ ie.
a map $P\sideset{_{\rho}}{_{t}}{\mathop{\times}} G\to P\,,$ denoted $(p,g)\mapsto p\cdot g\,,$ such that
\begin{itemize}
    \item $\pi(p\cdot g)=\pi(p)$
    \item $\rho(p\cdot g)=s(g)$
    \item $(p\cdot g_1)\cdot g_2=p\cdot(g_1g_2)$
\end{itemize}
and such that 
\begin{align*}
    P\sideset{_{\rho}}{_{t}}{\mathop{\times}} G\to P\sideset{_{\pi}}{_{\pi}}{\mathop{\times}} P\,,\;(p,g)\mapsto (p,p\cdot g)
\end{align*}
is a diffeomorphism.
$\blacksquare$\end{definition}
\begin{definition}
A morphism between $G$-principal bundles $P\to M$ and $Q\to N$ is given by a commutative diagram of smooth maps
\[
\begin{tikzcd}
P\arrow{r}{\phi}\arrow{d} & Q\arrow{d}
\\M\arrow{r} & N
\end{tikzcd}
\]
such that $\phi(p\cdot g)=\phi(p)\cdot g\,.$ In particular this implies that $\rho\circ\phi(p)=\rho(p)\,.$
$\blacksquare$\end{definition}
\begin{definition}
Let $G\rightrightarrows G^0$ be a Lie groupoid. Then we define $[G^0/G]$ to be the category of $G$-principal bundles, together with its natural functor to the category of manifolds (which takes a $G$-principal bundle to its base manifold). We call $[G^0/G]$
a (differentiable or holomorphic) stack.
$\blacksquare$\end{definition}
\theoremstyle{definition}\begin{rmk}
Given a Lie groupoid $G\rightrightarrows G^0\,,$ there is a canonical Grothendieck topology on $[G^0/G]\,,$ hence we can talk about sheaves on stacks and their cohomology. What is most important to know for the next sections is that a sheaf on $[G^0/G]$ is in particular a contravariant functor 
\[F:[G^0/G]\to \mathbf{Ab}\,.
\]
See Section~\ref{appendix:Cohomology of Sheaves on Stacks} for details.
\end{rmk}
\subsection{Groupoid Modules}
We now define Lie groupoid modules. Their importance is due to the fact that these are the structures which differentiate to representations; they will be one of the main objects we study in this paper.
\begin{definition}
Let $X$ be a manifold. A family of groups over $X$ is a Lie groupoid $M\rightrightarrows X$ such that the source and target maps are equal. A family of groups will be called a family of abelian groups if the multiplication on $M$ induces the structure of an abelian group on its source fibers, ie. if
$s(a)=s(b)$ then $m(a,b)=m(b,a)\,.$
$\blacksquare$\end{definition}
\theoremstyle{definition}\begin{exmp}\label{trivial}
Let $A$ be an abelian Lie group and let $X$ be a manifold. Then $X{\mathop{\times}} A$ is naturally a family of abelian groups, with the source
and target maps being the projection onto the first factor $p_1:X{\mathop{\times}} A\to X\,.$ This will be called a trivial family of abelian groups, and will be denoted $A_X\,.$
\end{exmp}
\theoremstyle{definition}\begin{exmp}\label{foag}
One way of constructing families of abelian groups is as follows: Let $A$ be an abelian group, and let $\text{Aut}(A)$ be its 
automorphism group. Then to any principal $\text{Aut}(A)$-bundle $P$ we have a canonical family of abelian groups - it is given by the fiber bundle
whose fibers are $A$ and which is associated to $P\,.$ Families of abelian groups constructed in this way are locally trivial in the sense
that locally they are isomorphic to the trivial family of abelian groups given by $A_{\mathbb{R}^n}\,,$ for some $n$ (compare this with vector bundles).
\end{exmp}
\begin{definition}(see \cite{Tu}$\,$): Let $G\rightrightarrows G^0$ be a Lie groupoid. A $G$-module $M$ is a family of abelian groups together with an action of $G$ on $M$ such that for $a\,,b\in G^0\,,$ $G(a,b):M_a\to M_b$ acts by homomorphisms $($here $G(a,b)$ is the set
of morphisms with source $a$ and target $b)\,.$ If $M$ is a vector bundle\footnote{Here we are implicitly using the fact that the forgeful functor from the category of finite dimensional vector spaces to the category of simply connected abelian Lie groups
is an equivalence of categories.}, $M$ will be called a representation of $G\,.$
$\blacksquare$\end{definition}
\theoremstyle{definition}\begin{exmp}
Let $G\rightrightarrows G^0$ be a groupoid and let $A$ be an abelian group. Then $A_{G^0}$ is a family of abelian groups (see Example~\ref{trivial}), and it is a $G$-module with the trivial action, that is $g\in G(x,y)$ acts by $g\cdot (x,a)=(y,a)\,.$ We will call this a trivial
$G$-module.
\end{exmp}
\theoremstyle{definition}\begin{exmp}\label{SL2}
Let $G=SL(2,\mathbb{Z})\,,$ which is the mapping class group of the torus. Every $T^2$-bundle over $S^1$ is isomorphic to one with transition functions in $SL(2,\mathbb{Z})\,,$ with the standard open cover of $S^1$ using two open sets. All of these are naturally $\Pi_1(S^1)$-modules since $SL(2,\mathbb{Z})$ is discrete. In particular, the Heisenberg manifold is a $\Pi_1(S^1)$-module. Explicitly,
consider the matrix 
\begin{align*}
    \begin{pmatrix}
    1 & 1 \\
    0 & 1 
    \end{pmatrix}\in SL(2,\mathbb{Z})\,.
\end{align*}
This matrix defines a map from $T^2\to T^2\,,$ and it corresponds to a Dehn twist. The total space of the corresponding $T^2$-bundle is diffeomorphic to the Heisenberg manifold $H_M$, which is the quotient of the Heisenberg group by the right action of the integral Heisenberg subgroup on itself, ie. we make the identification
\begin{align*}
    \begin{pmatrix}
    1 & a & c\\
    0 & 1 & b \\
    0 & 0 & 1
    \end{pmatrix}\sim \begin{pmatrix}
    1 & a+n & c+k+am\\
    0 & 1 & b+m \\
    0 & 0 & 1
    \end{pmatrix}\,,
\end{align*}
where $a\,,b\,,c\in\mathbb{R}$ and $n\,,m\,,k\in\mathbb{Z}\,.$ The projection onto $S^1$ is given by mapping to $b\,.$
\\\\The fiberwise product associated to the bundle $H_M\to S^1$ is given by
\begin{align*}
      \begin{pmatrix}
    1 & a & c\\
    0 & 1 & b \\
    0 & 0 & 1
    \end{pmatrix}\cdot  \begin{pmatrix}
    1 & a' & c'\\
    0 & 1 & b \\
    0 & 0 & 1
    \end{pmatrix}=  \begin{pmatrix}
    1 & a+a' & c+c'\\
    0 & 1 & b \\
    0 & 0 & 1
    \end{pmatrix}\,.
\end{align*}
See Example~\ref{module example} for more.
\end{exmp}
\begin{definition}
Let $M\,,N$ be $G$-modules. A morphism $f:M\to N$ is a morphism of the underlying groupoids such that if $s(g)=s(m)\,,$ then $f(g\cdot m)=g\cdot f(m)\,.$
$\blacksquare$\end{definition}
\begin{proposition}
Let $M\to X$ be a family of abelian groups. Then $H^1(X,\mathcal{O}(M))$ classifies principal $M$-bundles over $X$ for which $\rho=\pi\,.$
\end{proposition}
Before concluding this section we will make a remark on notation:
\theoremstyle{definition}\begin{rmk}Given a family of abelian groups $E\overset{\pi}{\to} Y\,,$ we can form its sheaf of sections, which as previously stated we denote by $\mathcal{O}(E)\,.$ In addition, given a map $f:X\to Y$ we get a family of abelian groups on $X\,,$ given by $f^*E=X{\mathop{\times}}_Y E\,.$
\end{rmk}
\subsection{Sheaves on Lie Groupoids and Stacks}
In this section we discuss the relationship between sheaves on $[G^0/G]\,,$ sheaves on $\mathbf{B}^\bullet G$ and $G$-modules ($\mathbf{B}^\bullet G$ is the nerve of $G\,,$ see appendix~\ref{appendix:Cohomology of Sheaves on Stacks} for more).
\subsubsection{Sheaves: Lie Groupoids to Stacks}\label{sheaves on stacks}
Here we discuss how to obtain a sheaf on the stack $[G^0/G]$ from a $G$-module.
\\\\Let $M$ be a $G$-module for $G\rightrightarrows G^0\,.$ We obtain a sheaf on $[G^0/G]$ as follows: consider the object of $[G^0/G]$ given by 
\[
\begin{tikzcd}
P \arrow{r}{\rho}\arrow{d}{\pi} & G^0 \\
X
\end{tikzcd}
\]
We can form the action groupoid $G\ltimes P$ and consider the $(G\ltimes P)$-module given by $\rho^*M\,.$ To $P$ we assign the abelian group $\Gamma_\text{inv}(\rho^*\mathcal{O}(M))$ (ie. the sections invariant under the $G\ltimes P$ action). To a morphism between objects of $[G^0/G]$ the functor just assigns the set-theoretic pullback. This defines a sheaf on $[G^0/G]\,,$ denoted $\mathcal{O}(M)_{[G^0/G]}\,.$

\subsubsection{Sheaves: Stacks to Lie Groupoids}\label{stacks to lie groupoids}
Here we discuss how to obtain a sheaf on $\mathbf{B}^\bullet G$ from a sheaf on $[G^0/G]\,,$ and we define the cohomology of a groupoid with coefficients taking values in a module.
\\\\Let $G\rightrightarrows G^0$ be a Lie groupoid and let $\mathcal{S}$ be a sheaf on $[G^0/G]\,.$ Consider the object of $[G^0/G]$ given by 
\[
\begin{tikzcd}
P \arrow{r}{\rho}\arrow{d}{\pi} & G^0 \\
X
\end{tikzcd}
\]
We can associate to each open set $U\subset X$ the object of $[G^0/G]$ given by
\[
\begin{tikzcd}
P\vert_U \arrow{r}{\rho}\arrow{d}{\pi} & G^0 \\
U
\end{tikzcd}
\]
We get a sheaf on $X$ by assigning to $U\subset X$ the abelian group $\mathcal{S}(P\vert_U)\,.$ 
\\\\Now for all $n\ge 0\,,$ the spaces $\mathbf{B}^{n} G$ are canonically identified with $G$-principal bundles, by identifying $\mathbf{B}^{n} G$ with the object of $[G^0/G]$ given by

\begin{equation}\label{canonical object}
\begin{tikzcd}[row sep=large, column sep=large]
 \mathbf{B}^{n+1} G \arrow{r}{d_{1,1}\circ p_{n+1}}\arrow{d}{d_{n+1,n+1}} & G^0 \\
\mathbf{B}^{n} G
\end{tikzcd}
\end{equation}
where $p_{n+1}$ is the projection onto the $(n+1)^{\text{th}}$ factor. Hence given a sheaf $\mathcal{S}$ on $[G^0/G]$ we obtain a sheaf on $\mathbf{B}^{n} G\,,$ for all $n\ge 0\,,$ denoted $\mathcal{S}(\mathbf{B}^{n} G)\,,$ and together these form a sheaf on $B^\bullet G\,.$ Furthermore, given a $G$-module $M$ we have that \begin{align*}
    \mathcal{O}(M)_{[G^0/G]}({G^0})\cong \mathcal{O}(M)\,.
\end{align*}
Moreover, we have the following lemma:
\begin{lemma}
Let $M$ be a $G$-module. Then the sheaf on $\mathbf{B}^\bullet G$ given by $\mathcal{O}(M)_{[G^0/G]}(\mathbf{B}^{\bullet}G)\,,$ is isomorphic to the sheaf of sections of the simplicial family of abelian groups given by
\begin{align*}
    \mathbf{B}^\bullet(G\ltimes M)\to \mathbf{B}^\bullet G\,.
\end{align*}
\end{lemma}
\theoremstyle{definition}\begin{definition}
Let $G\rightrightarrows G^0$ be a Lie groupoid and let $M$ be a $G$-module. We define 
\[
H^*(G,M):=H^*(\mathbf{B}^{\bullet}G,\mathcal{O}(M)_{[G^0/G]}(\mathbf{B}^{\bullet}G))\,.
\]
$\blacksquare$\end{definition}
\theoremstyle{definition}\begin{rmk}[See~\cite{Kai}]\label{morita inv}
Let $G\rightrightarrows G^0$ and $K\rightrightarrows K^0$ be Lie groupoids and let $\phi:G\to K$ be a Morita morphism (see \Cref{moritamapjeff}). Then the pullback $\phi^*$ induces an equivalence of categories \[
\phi^*:[K^0/K]\to[G^0/G]\,.
\]
Furthermore, let $\mathcal{S}$ be a sheaf on $[G^0/G]\,.$ Then the pushforward sheaf $\phi_*\mathcal{S}:=\mathcal{S}\circ\phi^*$ is a sheaf on $[K^0/K]$ and we have a natural isomorphism 
\[
H^*(\mathbf{B}^{\bullet}G,\mathcal{S}(\mathbf{B}^{\bullet}G))\cong H^*(\mathbf{B}^{\bullet}K,\phi_*\mathcal{S}(\mathbf{B}^{\bullet}K))\,.
\]
\end{rmk}
\subsection{Godement Construction for Sheaves on Stacks}
Here we discuss a version of the Godement resolution for sheaves on stacks, and we show how it can be used to compute cohomology.
\begin{definition}
Let $G\rightrightarrows G^0$ be a Lie groupoid and let $\mathcal{S}$ be a sheaf on $[G^0/G]\,.$ We define the Godement resolution of $\mathcal{S}$ as follows:
Consider the object of $[G^0/G]$ given by
\[
\begin{tikzcd}
P \arrow{r}{\rho}\arrow{d}{\pi} & G^0 \\
X
\end{tikzcd}
\]
and consider the corresponding sheaf on $X$ (see~\Cref{stacks to lie groupoids}), denoted by $\mathcal{S}(X)\,.$
We can then consider, for each $n\ge 0\,,$ the $n^\text{th}$ sheaf in the Godement resolution of $\mathcal{S}(X)\,,$ denoted $\mathbb{G}^n(\mathcal{S}(X))\,,$ and to $P$ we assign the abelian group $\Gamma(\mathbb{G}^n(\mathcal{S}(X)))\,.$ These define sheaves on $[G^0/G]$ which we denote by $\mathbb{G}^n(\mathcal{S})\,.$ 
$\blacksquare$\end{definition}
For a sheaf $\mathcal{S}$ on $[G^0/G]$ we obtain a resolution by using $\mathbb{G}^\bullet(\mathcal{S})$ in the following way: 
\begin{align*}
    \mathcal{S}\xhookrightarrow{}\mathbb{G}^1(\mathcal{S})\to \mathbb{G}^2(\mathcal{S})\to\cdots\,.
    \end{align*}
The sheaves $ \mathbb{G}^n(\mathcal{S})$ are not in general acyclic on stacks, however the sheaves $\mathbb{G}^n(\mathcal{S})(\mathbf{B}^m G)$ are acyclic on $\mathbf{B}^m G$ and hence can be used to compute cohomology (see Theorem~\ref{stack groupoid cohomology} and Remark~\ref{double complex}).
\subsection{Examples}\label{important examples}
The constructions in the previous sections will be important in Section~\ref{van Est map} when defining the van Est map; it is crucial that modules define sheaves on stacks in order to use the Morita invariance of cohomology. Here we exhibit examples of the constructions from the previous sections which will be used in Section~\ref{van Est map}. 
\begin{proposition}\label{example surj sub}Let $f:Y\to X$ be a surjective submersion, and consider the submersion groupoid $Y{\mathop{\times}}_f Y\rightrightarrows Y\,.$ This groupoid is Morita equivalent (see \Cref{moritaequivjeff}) to the trivial $X\rightrightarrows X$ groupoid, hence their associated stacks, $[Y/(Y{\mathop{\times}}_f Y)]$ and $[X/X]\,,$ are categorically equivalent. 
\end{proposition}
We now describe the functor $f^*:[X/X]\to[Y/(Y{\mathop{\times}}_f Y)]$ which gives this equivalence:
\\\\An $X\rightrightarrows X$ principal bundle is given by a manifold $N$ together with a map $\rho:N\to X$ (the $\pi$ map here is the identity map $N\to N)\,.$ To such an object, we let $f^*(N,\rho)=N\sideset{_\rho}{_{f}}{\mathop{\times}} Y\,.$ This is a $Y{\mathop{\times}}_f Y$ principal bundle in the following way:
\[\begin{tikzcd}
N\sideset{_\rho}{_{f}}{\mathop{\times}} Y \arrow{r}{\rho = p_2}\arrow{d}{\pi = p_1} & Y \\
N
\end{tikzcd}
\]
The functor $f^*$ is an equivalence of stacks. 
\\\\Now suppose we have a sheaf $\mathcal{S}$ on the stack $[Y/(Y{\mathop{\times}}_f Y)]\,,$ then we obtain a sheaf on $[X/X]$ by using the pushforward of $f\,,$ ie. to an object $(N,\rho)\in [X/X]$ we associate the abelian group $f_*\mathcal{S}(N,\rho):=\mathcal{S}(f^*(N,\rho))\,.$ We then obtain a sheaf on the simplicial space $\mathbf{B}^\bullet(X\rightrightarrows X)$ as follows: First note that $\mathbf{B}^n(X\rightrightarrows X)=X$ for all $n\ge 0\,,$ so the sheaves are the same on all levels. Now let $U\xhookrightarrow{\iota} X$ be open. Then $(U,\iota)\in [X/X]\,,$ so to this object we assign the abelian group $f_*\mathcal{S}(U,\iota)\,.$
\begin{proposition}\label{key example}Suppose $M$ is a $Y{\mathop{\times}}_f Y$-module and $f$ has a section $\sigma:X\to Y\,.$ We then obtain a sheaf (and its associated Godement sheaves) on $[X/X]\,,$ and in particular we obtain a sheaf (and its associated Godement sheaves) on $X\in [X/X]\,,$ which we describe as follows:
\end{proposition}
We use the notation in Proposition~\ref{example surj sub}. We have that
\[f^*(U,\iota)=\begin{tikzcd}
U_{\iota}{\mathop{\times}}_f Y \arrow{r}{\rho= p_2}\arrow{d}{\pi= p_1} & Y \\
U
\end{tikzcd}
=\begin{tikzcd}
Y\vert_U \arrow{r}{}\arrow{d}{f} & Y \\
U
\end{tikzcd}
\]
We then see that $\,\Gamma_\text{inv}(\rho^*\mathcal{O}(M))\cong \Gamma(\sigma\vert_U^*\mathcal{O}(M))\,,$ 
hence the sheaf we get on $X$ is simply $\sigma^*\mathcal{O}(M)\,.$ Furthermore, the sheaves we get on $X$ by applying the Godement construction to $\mathcal{O}(M)_{[Y/(Y{\mathop{\times}}_f Y)]}$ are simply $\mathbb{G}^\bullet(\sigma^*\mathcal{O}(M))\,.$
\begin{lemma}\label{acyclic edit}
Suppose we have a sheaf $\mathcal{S}$ on the stack $[Y/(Y{\mathop{\times}}_f Y)]\,,$ then the associated Godement sheaves $\mathbf{G}^\bullet(\mathcal{S})$ are acyclic.
\end{lemma}
\begin{proof}
This follows from the fact that $[Y/(Y{\mathop{\times}}_f Y)]\,,$ is Morita equivalent to $[X/X]\,,$ since cohomology is invariant under Morita equivalence of stacks, and the fact that the Godement sheaves on a manifold are acyclic. 
\end{proof}
\theoremstyle{definition}\begin{rmk}
Let $X$ be a manifold and let $X\rightrightarrows X$ be the trivial Lie groupoid. Let $\mathcal{S}$ be a sheaf on $[X/X]\,.$ Then we recover the usual cohomology: 
\[
H^*(\mathbf{B}^\bullet X,\mathcal{S}(\mathbf{B}^\bullet X))
=H^*(X,\mathcal{S}(X))\,.
\]
This will be important in computing the cohomology of submersion groupoids, since they are Morita equivalent to trivial groupoids.
\end{rmk}
\chapter{Chevalley-Eilenberg Complex for Modules}\label{Chevalley}
In this section we review the Chevalley-Eilenberg complex associated to a representation of a Lie algebroid. Then we generalize Lie
algebroid representations to Lie algebroid modules and define their Chevalley-Eilenberg complex. These will be used in Section~\ref{van Est map}.
\section{Lie Algebroid Representations}
\begin{definition}Let $\mathfrak{g}\overset{\pi}{\to} Y$ be a Lie algebroid, with anchor map $\alpha:\mathfrak{g}\to TY\,,$ and recall that $\mathcal{O}(\mathfrak{g})$ denotes the sheaf of sections of $\mathfrak{g}\overset{\pi}{\to} Y\,.$ A representation of $\mathfrak{g}$ is a vector bundle
 $E\to Y$ together with a map\begin{align*}
     \mathcal{O}(\mathfrak{g})\otimes\mathcal{O}(E)\to\mathcal{O}(E)\,,\,X\otimes s\mapsto L_X(s)
 \end{align*}
 such that for all open sets  $\,U\subset Y$ and for all $f\in \mathcal{O}_Y(U)\,,X\in \mathcal{O}(\mathfrak{g})(U)\,,
 s\in \mathcal{O}(E)(U)\,,$ we have that
 \begin{enumerate}
     \item $L_{fX}(s)=fL_X(s)\,,$
     \item $L_X(fs)=fL_X(s)+(\alpha(X)f)s\,,$
     \item $L_{[X,Y]}(s)=[L_X,L_Y](s)\,.$
 \end{enumerate}
 $\blacksquare$\end{definition}
\begin{definition}
Let $E$ be a representation of $\mathfrak{g}\,.$ Let $\mathcal{C}^n(\mathfrak{g},E)$ denote the sheaf of $E$-valued
 $n$-forms on 
 $\mathfrak{g}\,,$ ie. the sheaf of sections of $\Lambda^n \mathfrak{g}^*\otimes E\,.$ There is a canonical differential\footnote{Meaning in particular that $d_\text{CE}^2=0\,.$}
 \begin{align*}
     d_\text{CE}:\mathcal{C}^n(\mathfrak{g},E)\to \mathcal{C}^{n+1}(\mathfrak{g},E)\,,\,n\ge0
 \end{align*}
 defined as follows: let $\omega\in \mathcal{C}^n(\mathfrak{g},E)(V) $ for some open set $V\,.$ Then for $X_1\,,\ldots \,, X_{n+1}\in \pi^{-1}(m)\,,\,m\in V\,,$ choose local extensions $\mathbf{X}_1\,,\ldots\,,\mathbf{X}_{n+1}$ of these vectors, ie. choose
 \begin{align*}
     p\mapsto\mathbf{X}_1(p)\,,\ldots\,,p\mapsto\mathbf{X}_{n+1}(p)\in \mathcal{O}(\mathfrak{g})(U)\,,
     \end{align*}
     for some open set $U$ such that $m\in U\subset V\,,$ and such that $\mathbf{X}_i(m)=X_i$ for all $1\le i\le n+1)\,.$ Then let
 \begin{align*}
     d_{\text{CE}}\omega(X_1\,,\ldots\,,X_{n+1})&=\sum_{i<j}(-1)^{i+j-1}\omega([\mathbf{X}_i,\mathbf{X}_j],\mathbf{X}_1,\ldots ,\hat{\mathbf{X}}_i,\ldots
     , \hat{\mathbf{X}}_j,\ldots , \mathbf{X}_{n+1})
\vert_{p=m}     \\& +\sum_{i=1}^{n+1}(-1)^i L_{\mathbf{X}_i}(\omega(\mathbf{X}_i,\ldots , \hat{\mathbf{X}}_i ,\ldots , \mathbf{X}_{n+1}))\vert_{p=m}\,.
 \end{align*}
This is well-defined and independent of the chosen extensions.
 $\blacksquare$\end{definition}
\subsection{Lie Algebroid Modules}
We will now define Lie algebroid modules and define their Chevalley-Eilenberg complexes; these will look like the Chevalley-Eilenberg complexes associated to representations, except for possibly in degree zero (though representations will be seen to be special 
cases of Lie algebroid modules).
\begin{definition}
Let $\mathfrak{g}\to Y$ be a Lie algebroid, and let $M$ be a family of  abelian groups, with Lie algebroid $\mathfrak{m}$ and exponential map $\exp:\mathfrak{m}\to M\,.$\footnote{Note that $\mathfrak{m}$ is just a vector bundle and the exponential map is given by the fiberwise exponential map taking a Lie algebra to its corresponding Lie group.} Then a $\mathfrak{g}$-module structure on $M$ is given by the following: a $\mathfrak{g}$-representation structure on $\mathfrak{m}$ $($ie. a morphism  $\mathcal{O}(\mathfrak{g})\otimes\mathcal{O}(\mathfrak{m})\to\mathcal{O}(E)\,,\,X\otimes s\mapsto L_X(s))\,,$ together with a morphism of sheaves \begin{align*}
     \mathcal{O}(\mathfrak{g})\otimes_{\mathbb{Z}}\mathcal{O}(M)\to\mathcal{O}(\mathfrak{m})\,,\,X\otimes_{\mathbb{Z}} s\mapsto \tilde{L}_X(s)
 \end{align*}
  such that for all open sets  $\,U\subset Y$ and for all $f\in \mathcal{O}_Y(U)\,,X\in \mathcal{O}(\mathfrak{g})(U)\,,
 s\in \mathcal{O}(M)(U)\,,\sigma\in \mathcal{O}(\mathfrak{m})(U)\,, $ we have that
 \begin{enumerate}
     \item $\tilde{L}_{fX}(s)=f\tilde{L}_X(s)\,,$
     \item $\tilde{L}_{[X,Y]}(s)=(L_X\tilde{L}_Y-L_Y\tilde{L}_X)(s)\,,$
     \item $\tilde{L}_X(\exp{\sigma})=L_X(\sigma)\,.$
 \end{enumerate}
If $M$ is endowed with such a structure we call it a $\mathfrak{g}$-module.
$\blacksquare$\end{definition}
\begin{definition}\label{forms}Let $\mathfrak{g}\to X$ be a Lie algebroid and let $M$ be a $\mathfrak{g}$-module. We then define sheaves on $X\,,$ called ``sheaves of $M$-valued forms", as follows: let
\begin{align*}
&\mathcal{C}^0(\mathfrak{g},M)=\mathcal{O}(M)\,,
\\&   \mathcal{C}^n(\mathfrak{g},M)=\mathcal{O}(\Lambda^n \mathfrak{g}^*\otimes \mathfrak{m})\,,\;n> 0\,.
\end{align*}
Furthermore, for $s\in\mathcal{O}(M)(U)\,,$ we define $d_\text{CE}\log f$ by $d_\text{CE}\log f(X):=\tilde{L}_X(s)\,.$ We then have a cochain complex of sheaves given by
\begin{equation}\label{CE}
    \mathcal{C}^0(\mathfrak{g},M)\xrightarrow{d_{\text{CE}}\log}\mathcal{C}^1(\mathfrak{g},M)\xrightarrow{d_\text{CE}}\mathcal{C}^2(\mathfrak{g},M)\xrightarrow{d_\text{CE}}\cdots\,.
\end{equation}
 $\blacksquare$\end{definition}
\begin{definition}
The sheaf cohomology of the above complex of sheaves is denoted by $H^*(\mathfrak{g},M)\,.$
$\blacksquare$\end{definition}
\begin{definition}
Let $M\,,N$ be $\mathfrak{g}$-modules. A morphism $f:M\to N$ is a morphism of the underlying families of abelian groups such that the induced map $df:\mathfrak{m}\to\mathfrak{n}$ satisfies $\tilde{L}_X(f\circ s)=df\circ \tilde{L}_X(s)\,,$ for all local sections $X$ of $\mathfrak{g}$ and $s$ of $M\,.$
$\blacksquare$\end{definition}

 \theoremstyle{definition}\begin{exmp}
Here we will show that the notion of $\mathfrak{g}$-modules naturally extends the notion of $\mathfrak{g}$-representations. Let $E$ be a representation of $\mathfrak{g}\,.$ By thinking of the fibers of $E$ as abelian groups it defines a family of  abelian groups. The exponential map $E\overset{\exp}{\to} E$ is the identity, hence its kernel is the zero section and $E$ naturally defines a $\mathfrak{g}$-module where $d_\text{CE}\log =d_\text{CE}\,.$ So the definition of a $\mathfrak{g}$-module and its Chevalley-Eilenberg complex recovers the definition of a $\mathfrak{g}$-representation and its Chevalley-Eilenberg complex given by Crainic in~\cite{Crainic}.
\end{exmp}
 \theoremstyle{definition}\begin{exmp}\label{representation}
The group of isomorphism classes of $\mathfrak{g}$-representations on complex line bundles is isomorphic to $H^1(\mathfrak{g},\mathbb{C}^*_M)\,,$ where $\mathbb{C}^*_M$ is the $\mathfrak{g}$-module for which $\tilde{L}_X s=\mathrm{dlog}\,s(\alpha(X))\,,$ for a local section $s$ of $\mathbb{C}^*_M\,.$ The corresponding statement holds for real line bundles, with $\mathbb{C}^*_M$ replaced by $\mathbb{R}^*_M\,.$
\end{exmp}
\theoremstyle{definition}\begin{exmp}(Deligne Complex)
Let $X$ be a manifold and  $\mathfrak{g}=TX\,.$ Then letting $M=\mathbb{C}^*_X\,,$ we have that $\mathfrak{m}=\mathbb{C}_X$ naturally carries a representation of $TX\,,$ ie. where the differentials are the de Rham differentials. Letting $\exp:\mathfrak{m}\to M$ be the usual exponential map, it follows that $M$ is a $\mathfrak{g}$-module, and in fact the complex~\eqref{CE} in this case is known as the Deligne complex.
\end{exmp}
For a less familiar example we have the following:
\theoremstyle{definition}\begin{exmp}\label{module example}
Consider the space $S^1$ and the group $\mathbb{Z}/2\mathbb{Z}=\{-1,1\}\,.$ This group is contained in the automorphism  groups of $\mathbb{Z}\,,\mathbb{R}$ and $\mathbb{R}/\mathbb{Z}\,,$ hence we get nontrivial families of abelian groups over $S^1$  as follows (compare with Example~\ref{foag}): Let $A$ be any of the groups $\mathbb{Z}\,,\mathbb{R}\,,\mathbb{R}/\mathbb{Z}\,.$ Now cover $S^1$ in the standard way  using two open sets $U_0\,,U_1\,,$ and glue together the bundles $U_0{\mathop{\times}} A\,,U_1{\mathop{\times}} A$ with the transition functions
$-1\,,1$ on the two connected components of $U_0\cap U_1\,.$ Denote these families of abelian groups by $\tilde{\mathbb{Z}}\,,\tilde{\mathbb{R}}\,,\widetilde{\mathbb{R}/\mathbb{Z}}$ respectively. The space $\tilde{\mathbb{R}}$ is toplogically
the M\"{o}bius strip, and $\widetilde{\mathbb{R}/\mathbb{Z}}$ is topologically the Klein bottle.
\\\\Next, there is a canonical flat connection on these bundles of groups which is compatible with the fiberwise group structures, hence these families of abelian groups are modules for $\Pi_1(S^1)\,,$ the fundamental groupoid of $S^1\,.$ 
\\\\Furthermore, the $TS^1$-representation associated to the $TS^1$-module of $\tilde{\mathbb{Z}}$ is the rank $0$ vector bundle over $S^1\,,$ and the $TS^1$-representations associated to the $TS^1$-modules of $\tilde{\mathbb{R}}\,,\,\widetilde{\mathbb{R}/\mathbb{Z}}$ are isomorphic to the Mobius strip, ie. the line bundle obtained by gluing together $U_0{\mathop{\times}}\mathbb{R}\,,\,U_ 1{\mathop{\times}} \mathbb{R}$ using the same transition functions as discussed above. The Chevalley-Eilenberg differential, on each local trivialization $U_0{\mathop{\times}}\mathbb{R}\,,\,U_1{\mathop{\times}}\mathbb{R}\,,$ is just the de Rham differential.
\\\\The cohomology groups are $H^i(TS^1,\tilde{\mathbb{R}})=0$ in all degrees, and 
\[
H^i(TS^1,\widetilde{\mathbb{R}/\mathbb{Z}})=\begin{cases}
      \mathbb{Z}/2\mathbb{Z}, & \text{if}\ i= 0 \\
      0, & \text{if}\ i>0\,.
    \end{cases}
\]
\end{exmp}
\begin{theorem}\label{g-module}
Suppose $G\rightrightarrows G^0$ is a Lie groupoid. There is a natural functor
\begin{align*}
    F:G\text{-modules}\to \mathfrak{g}\text{-modules}\,.
\end{align*}
Furthermore, if $G$ is source simply connected then this functor restricts to an equivalence of categories on the subcategories of $G$-modules and $\mathfrak{g}$-modules for which $\exp:\mathfrak{m}\to M$ is a surjective submersion.
\end{theorem}
\begin{proof}$\mathbf{}$
\\\\$\mathbf{1.}$ For the first part, let $M$ be a $G$-module and for $x\in G^0\,$ let $\gamma:(-1,1)\to G(x,\cdot)$ be a curve in the source fiber such that $\gamma(0)=\text{Id}(x)\,.$ We define 
    \begin{align*}
    \tilde{L}_{\dot{\gamma}(0)}r:=\frac{d}{d\epsilon}\Big\vert_{\epsilon=0}r(x)^{-1}[\gamma(\epsilon)^{-1}\cdot r(t(\gamma(\epsilon)))]
    \end{align*}
for a local section $r$ of $\mathcal{O}(M)\,.$ One can check that this is well-defined and that property 1 is satisfied. Now note that the action of $G$ on $M$ induces a linear action of $G$ on $\mathfrak{m}\,,$ and we get a $\mathfrak{g}$-representation on $\mathfrak{m}$ by defining
\begin{align*}
    L_{\dot{\gamma}(0)}\sigma:=\frac{d}{d\epsilon}\Big\vert_{\epsilon=0}\sigma(x)^{-1}[\gamma(\epsilon)^{-1}\cdot \sigma(t(\gamma(\epsilon)))]
    \end{align*}
for a local section $\sigma$ of $\mathfrak{m}\,.$ With these definitions property 2 is satisfied. 
\\\\Now note that this action of $G$ on $\mathfrak{m}$ preserves the kernel of $\exp:\mathfrak{m}\to M\,.$ Let $\sigma$ be a local section of $\mathfrak{m}$
around $x$ such that
$\exp{\sigma}=e\,.$ Then 
\begin{align*}
    L_{\dot{\gamma}(0)}\sigma=\frac{d}{d\epsilon}\Big\vert_{\epsilon=0}\sigma(x)^{-1}[\gamma(\epsilon)^{-1}\cdot \sigma(t(\gamma(\epsilon)))]\,,
    \end{align*}
and since the $G$-action preserves the kernel of 
$\exp\,,$ which is discrete, we have that 
\begin{align*}
    \gamma(\epsilon)^{-1}\cdot \sigma(t(\gamma(\epsilon)))=\sigma(x)\,,
    \end{align*}
hence $L_{\dot{\gamma}(0)}(\sigma)=0\,,$ therefore $L(\sigma)=\tilde{L}(\exp{\sigma})=0\,,$ from which property 3 follows. Since it can be seen that morphisms of $G$-modules induce morphisms of $\mathfrak{g}$-modules, this completes the proof.
\\\\$\mathbf{2.}$ For the second part, let $M$ be a $\mathfrak{g}$-module for which $\exp:\mathfrak{m}\to M$ is a surjective submersion, and suppose $G$ is source simply connected. Then in particular $\mathfrak{m}$ is a $\mathfrak{g}$-representation, and it is known that for source simply connected groupoids Rep$(G)\cong\text{Rep}(\mathfrak{g})$ (by Lie's second theorem for Lie groupoids), hence $\mathfrak{m}$ 
integrates to a $G$-representation. Property 3 implies that the $G$-action preserves the kernel of $\exp\,,$
hence the action of $G$ on $\mathfrak{m}$ descends to $M\,.$ More explicitly: let $g\in G(x,y)$ and let $m\in M_x\,,$ ie. the source fiber of $M$ over $x\,.$ Let 
$\tilde{m}\in\mathfrak{m}_x$ be such that $\exp{\tilde{m}}=m$ and define
\begin{align*}
    g\cdot m=\exp{(g\cdot{\tilde{m}})}\,.
\end{align*}
This is well-defined since the action of $G$ preserves the kernel of $\exp\,.$ Hence the functor is essentially surjective. Now again using the fact that for source simply connected groupoids
Rep$(G)\cong\text{Rep}(\mathfrak{g})\,,$ it follows that the functor is fully faithful, and since it is also essentially surjective, this completes the proof.
 \end{proof}
Note that by composing with $\exp$ we obtain a natural map 
\begin{align*} 
V_{\mathfrak{g}}^*\overset{\exp}{\to} V_{G}^*\,.
\end{align*}
Claim: this map is an isomorphism. To see this, let $f\in V^*_{\mathfrak{g}}$ and suppose $\exp{f}=0\,.$ Let $v\in V$ and $t\in \mathbb{R}\,.$ Then $\exp{f(tv)}=0\implies \exp{tf(v)}=0.$ But taking $t\ne 0$ small enough such that $tf(v)\in U$ we get that $tf(v)=0\,,$ hence $f(v)=0\,.$
So $\exp$ is injective. 
\\\\Now let $f\in V^*_G\,.$ Let $v\in V\,.$ We must have $f(0)=0\,,$ so the map $[0,1]\to G$ defined by $t\mapsto f(tv)$ defines a path in $G$ starting at $0\,,$ therefore since $\mathfrak{g}$ is a covering space of $G\,,$ by the unique path lifting property there is a unique path $X_v:[0,1]\to\mathfrak{g}$ starting at $0$ which lifts $G\,.$ So define $X\in V^*_{\mathfrak{g}}$ by $X(v)=X_v(1)\,.$ Then $X(tv)=X_{tv}(1)$ and this defines a path lifting $f(tv)$ and starting at $0\,,$ hence $X(tv)=tX(v)\,.$ $X(a+b)=X_{a+b}(1)$

\\\\so by continuity for small enough $t\ne 0\in\mathbb{R}$ we have that $f(tv)\in \exp(U)\,.$ Hence there exists a unique $X_{tv}\in\mathfral{g}$ such that $\exp{X_{tv}}=f(tv)\,.$ Claim: $\exp{\frac{X_{tv}}{t}}=f(v).$
Now define $X\in V^*_{\mathfrak{g}}$ by $X(v)=\frac{X_{tv}}{t}\,.$

Let $\Lambda^0 V^*:=G\,.$ Let $\omega\in \Lambda^n V^*\,.$
\end{comment}

\chapter{Van Est Map}\label{van Est map}
\section{Definition}
In this section we will discuss a generalization of the van Est map that appears in~\cite{Crainic}. It will be a map
$H^*(G,M)\to H^*(\mathfrak{g},M)\,,$ for a $G$-module $M\,,$ which will be an isomorphism up to a certain degree which depends on the connectivity of the source fibers of $G\,.$ Let us remark that one doesn't need to know the details of the map to understand the main theorem of this paper, \Cref{van Est image}, and if the reader wishes they may skip ahead to~\Cref{main theorem section}.
\\\\Given a groupoid $G\rightrightarrows G^0\,,$ $G$ naturally defines a principal $G$-bundle with the moment map given by $t\,,$ ie. the action is given by the left multiplication of $G$ on itself. Being consistent with the previous notation, we denote the resulting action groupoid by $G\ltimes G$ and note that it is isomorphic to $G_s{\mathop{\times}}_{s} G\rightrightarrows G\,,$ hence it is Morita equivalent to the trivial $G^0\rightrightarrows G^0$ groupoid. 
\begin{definition}
We let $\mathbf{E}^\bullet G:=\mathbf{B}^\bullet(G\ltimes G)\,.$
The simplicial map $\kappa:\mathbf{E}^\bullet G\to\mathbf{B}^\bullet G$ induced by the groupoid morphism $\pi_1:G\ltimes G\to G$ makes $\mathbf{E}^\bullet G$ into a simplicial principal $G$-bundle, and the fiber above $(g^1\,,\ldots\,,g^n)\in \mathbf{B}^n G$ is $t^{-1}(s(g_n))\,.$ 
$\blacksquare$\end{definition}
\theoremstyle{definition}\begin{rmk}Note that $G\ltimes G$ is a groupoid object in $[G^0/G]\,,$ and as a principal $G$-bundle it is the canonical object associated to $G$ via diagram~\ref{canonical object} 
\end{rmk}
\begin{definition}
Let $\Omega_{\kappa\;q}^p(\kappa^*M)$ denote the sheaf of sections of $\Lambda^p T^*_\kappa\, \mathbf{E}^qG\,(\kappa^*M)\,,$ the $\kappa$-foliated covectors taking values in $\kappa^*M\,.$ Succinctly, from $M$ we get a family of abelian groups on $\mathbf{B}^q G\,,$ given by \[
\mathbf{B}^q(G\ltimes M)\to \mathbf{B}^qG\,,
\]
which we denote by $M_{\mathbf{B}^qG}\,;$ we then have that $\kappa^*M_{\mathbf{B}^q G}$ is a module for the submersion groupoid 
\[\mathbf{E}^qG\times_{\mathbf{B}^qG} \mathbf{E}^qG\rightrightarrows \mathbf{E}^qG\,,
\]
and $\Omega_{\kappa\;q}^p(\kappa^*M)$ is the sheaf of $\kappa^*M_{\mathbf{B}^q G}$-valued p-forms associated to the corresponding Lie algebroid module (see~\Cref{forms}). Explicitly, $\Omega_{\kappa\;q}^0(\kappa^*M)$ is the sheaf of sections of~\footnote{Really, we should write $\kappa^*M_{\mathbf{B}^q G}\to\mathbf{E}^q G\,,$ but for notational simplicitly we suppress the subscript.}
\[
\kappa^*M\to \mathbf{E}^q G\,,
\]
and for $p\ge 1\,,$ $\Omega_{\kappa\;q}^p(\kappa^*M)$ is the sheaf of $\kappa$-foliated $p$-forms taking values in $\kappa^*\mathfrak{m}\,.$ There is a differential 
\begin{equation}
\Omega_{\kappa\;q}^0(\kappa^*M)\xrightarrow{\text{dlog}} \Omega_{\kappa\;q}^1(\kappa^*M)
\end{equation}
which is defined as follows: let $U$ be an open set in $\mathbf{E}^q G$ and let $X_g$ be a vector tangent to a $\kappa$-fiber at a point $g\in U\,.$ Let $f\in \Omega_{\kappa\;q}^0(\kappa^*M)(U)\,.$ Define $\text{dlog}\,f\in \Omega_{\kappa\;q}^1(\kappa^*M)(U)$ by 
\[
\text{dlog}\,f\,(X_g)=f(g)^{-1}f_*(X_g)\,,
\]
where in order to identify this with a point in $\kappa^*\mathfrak{m}_g$ we are implicity using the canonical identification of $\kappa^*M_g$ with $\kappa^*M_{g'}$ for any two points $g\,,g'$ in the same $\kappa$-fiber (here $\kappa^*M_g$ is the fiber of $\kappa^*M$ over $g)\,.$ We also use the canonical identification of $\kappa^*\mathfrak{m}_g$ with $\kappa^*\mathfrak{m}_{g'}$ for any two points $g\,,g'$ in the same $\kappa$-fiber to define the differentials for $p>0:$
\begin{equation}
\Omega_{\kappa\;q}^p(\kappa^*M)\overset{\text{d}}{\to}\Omega_{\kappa\;q}^{p+1}(\kappa^*M)\,.
\end{equation}
$\blacksquare$\end{definition}
$\mathbf{}$
\begin{theorem}
There is an isomorphism
\begin{equation*}
Q:H^*(\mathbf{E}^\bullet G\,,\kappa^{-1}\mathcal{O}(M))\to H^*(\mathfrak{g}\,,M)\,.
\end{equation*}
\end{theorem}
\begin{proof}Form the sheaf $\kappa^{-1}\mathcal{O}(M)$ on $\mathbf{E}^\bullet G\,.$ This sheaf is not in general a sheaf on the stack $[G/(G\ltimes G)]\,,$ but it is resolved by sheaves on stacks
in the following way:\footnote{These are sheaves on stacks because \begin{equation*}
\Lambda^n T^*_\kappa G(\kappa^*M)\cong\Lambda^n T^*_tG(t^*M)
\end{equation*}
(where $t$ is the target map),
and the latter are $(G\ltimes G)$-modules.}
\begin{equation}\label{deligne1}
    \kappa^{-1}\mathcal{O}(M)_\bullet\xhookrightarrow{} \mathcal{O}(\kappa^*M)_\bullet\overset{\text{dlog}}{\to}\Omega_{\kappa\,\bullet}^1(\kappa^*M)
    \overset{\text{d}}{\to}\Omega_{\kappa\,\bullet}^2(\kappa^*M)\to\cdots\,.
\end{equation}
We let, for all $q\ge 0\,,$
\begin{align}\label{C}C^\bullet_q:=\mathcal{O}(\kappa^*M)_q\to\Omega_{\kappa\,q}^1(\kappa^*M)
    \to\Omega_{\kappa\,q}^2(\kappa^*M)\to\cdots\,.
    \end{align}
We can then take the Godement resolution of $C^\bullet_q$ and get a double complex for each $q\ge 0:$ \begin{align*}
C^\bullet_q\xhookrightarrow{} \mathbb{G}^0(C^\bullet_q)\to \mathbb{G}^1(C^\bullet_q)\to\cdots\,.
\end{align*}
\\All of the sheaves $\mathbb{G}^p(C^r_\bullet)$ are sheaves on stacks, and it follows that these sheaves are acyclic (as sheaves on stacks) since $G\ltimes G\rightrightarrows G$ is Morita equivalent to a submersion groupoid, and \Cref{acyclic edit}. Hence $\mathbb{G}^p(C^r_\bullet)$ can be used to compute cohomology (see Remark~\ref{acyclic resolution}) and we have that
\begin{align*}
    H^*(\mathbf{E}^\bullet G, \kappa^{-1}\mathcal{O}(M))\cong  H^*(\text{Tot}(\Gamma_\text{inv}(\mathbb{G}^\bullet (C^\bullet_0))))\,.
\end{align*}
Now we have that 
\begin{align*}
    \Gamma_\text{inv}(\mathbb{G}^p(C^q_0)))=\Gamma(\mathbb{G}^p(i^*C^q_0))\,, 
\end{align*}
where $i:G^0\to G$ is the identity bisection. Since all of the differentials in~\eqref{C} preserve invariant sections, they descend to differentials on 
$\Gamma(\mathbb{G}^\bullet(i^*C^\bullet_0))\,,$ hence 
\begin{align*}
    H^*(\text{Tot}(\Gamma_\text{inv}(\mathbb{G}^\bullet (C^\bullet_0))))\cong H^*(\mathfrak{g},M)\,.
\end{align*}
\end{proof}
\theoremstyle{definition}\begin{definition}\label{van Est}
Let $M$ be a $G$-module. We define a map 
\begin{equation*}
H^*(G\,,M)\to H^*(\mathfrak{g}\,,M)
\end{equation*}
given by the composition
\begin{equation*}
H^*(G\,,M)\overset{\kappa^{-1}}{\to} H^*(\mathbf{E}^\bullet G\,,\kappa^{-1}\mathcal{O}(M)_\bullet)\overset{Q}{\to}H^*(\mathfrak{g}\,,M)\,.
\end{equation*}
This is the van Est map; we denote it by $VE\,.$
\end{definition}
\theoremstyle{definition}\begin{rmk}
Taking $M=\mathbb{C}_{G^0}$ as a smooth abelian groups with the trivial $G$-action, the sheaves in the resolution of $\kappa^{-1}\mathcal{O}(M)$ in~\eqref{deligne1} are already acyclic (as sheaves on stacks). Hence our map coincides with the map \begin{align*}
    H^*(G\,,M)&\to H^*(\mathbf{E}^\bullet G\,,\kappa^{-1}\mathcal{O}(M)_\bullet)
    \\&=H^*(\Gamma_\text{inv}(\mathcal{O}(\kappa^*M)_0)\to\Gamma_\text{inv}(\Omega_{\kappa\,0}^1(M))
  \to\cdots)
  \to H^*(\mathfrak{g}\,,M)\,,
\end{align*}
which is the van Est map as described in \cite{Meinrenken}.
\end{rmk} 
\subsection{Van Est for Truncated Cohomology}
In order to emphasize geometry on the space of morphisms rather than on the space of objects we perform a truncation. That is, we truncate the contribution of $G^0$ to $H^*(G,M)$ by considering instead the cohomology 
\[
H^*(\mathbf{B}^\bullet G,\mathcal{O}(M)^0_\bullet)\,,
\]
where $\mathcal{O}(M)^{0}_n=\mathcal{O}(M)_n$ for all $n\ge 1\,,$ and where $\mathcal{O}(M)^{0}_0$ is the trivial sheaf on $G^0\,,$ ie. the sheaf that assigns to every open set the group containing only the identity.
\\\\We define 
\[
H^*_0(G,M):=H^{*+1}(\mathbf{B}^\bullet G,\mathcal{O}(M)^{0}_\bullet)\,.
\]
There is a canonical map 
\[
H^*_0(G,M)\to H^{*+1}(G,M)\
\]
induced by the morphism of sheaves on $\mathbf{B}^\bullet G$ given by $\mathcal{O}(M)^{0}_\bullet\xhookrightarrow{}\mathcal{O}(M)_\bullet\,.$
Similarly, we can truncate $M$ from $H^*(\mathfrak{g},M)$ by considering instead 
\[
H^*_0(\mathfrak{g},M):=H^{*+1}(0\to\mathcal{C}^1(\mathfrak{g},M)\to\mathcal{C}^2(\mathfrak{g},M)\to\cdots)\,.
\]
Then in like manner there is a canonical map 
\[
H^*_0(\mathfrak{g},M)\mapsto H^{*+1}(\mathfrak{g},M)
\]
induced by the inclusion of the truncated complex into the full one.
\begin{theorem}\label{canonical edit}
There is a canonical map $VE_0$ lifting $VE\,,$ ie. such that the following diagram commutes:
\begin{equation}\label{diagram}
\begin{tikzcd}
H^*_0(G,M) \arrow{d}{}\arrow{r}{VE_0} & H^*_0(\mathfrak{g},M)\arrow{d} \\
H^{*+1}(G,M)\arrow{r}{VE} & H^{*+1}(\mathfrak{g},M) 
\end{tikzcd}
\end{equation}
where $H_0^*$ denotes truncated cohomology, as define above.
\end{theorem}
\begin{proof}
Consider the ``normalized" sheaf on $\mathbf{E}^\bullet G$ given by $\widehat{\mathcal{O}(\kappa^*M)}_\bullet\,,$ where $\widehat{\mathcal{O}(\kappa^*M)}_n=\mathcal{O}(\kappa^*M)_n$ for $n\ge 1\,,$ and where $\widehat{\mathcal{O}(\kappa^*M)}_0$ is the subsheaf of $\mathcal{O}(\kappa^*M)_0$ consisting of local sections which are the identity on $G^0\,.$ Then \[
H^*(\mathbf{E}^\bullet G,\mathbb{G}^n(\widehat{\mathcal{O}(\kappa^*M)}_\bullet))=0\,,
\] and in particular, $\mathbb{G}^n(\widehat{\mathcal{O}(\kappa^*M)}_\bullet)$ is acyclic.
$\mathbf{}$
\\\\Now consider the sheaf $\widehat{\kappa^{-1}\mathcal{O}(M)}_\bullet$ on $\mathbf{E}^\bullet G$ given by $\widehat{\kappa^{-1}\mathcal{O}(M)}_n=\kappa^{-1}\mathcal{O}(M)_n$ for $n\ge 1\,,$ and such that $\widehat{\kappa^{-1}\mathcal{O}(M)}_0$ is the subsheaf of $\kappa^{-1}\mathcal{O}(M)_0$ consisting of local sections which are the identity on $G^0\,.$
Then there is a canonical embedding $\kappa^{-1}\mathcal{O}(M)^0_\bullet\xhookrightarrow{}\widehat{\kappa^{-1}\mathcal{O}(M)}_\bullet\,,$ hence we get a map
\begin{align*}
    H^*(\mathbf{B}^\bullet G,\mathcal{O}(M)^0_\bullet)\to H^*(\mathbf{E}^\bullet G,\widehat{\kappa^{-1}\mathcal{O}(M)}_\bullet)\,.
\end{align*}
Now we have that the following inclusion is a resolution:
\begin{align*}
    \widehat{\kappa^{-1}\mathcal{O}(M)}_\bullet\xhookrightarrow{} \widehat{\mathcal{O}(\kappa^*M)}_\bullet\to\Omega_{\kappa\,\bullet}^1(M)
    \to\Omega_{\kappa\,\bullet}^2(M)\to\cdots\,.
\end{align*}
Then one can show that 
\begin{align*}
   & H^*(\mathbf{E}^\bullet G,\widehat{\mathcal{O}(\kappa^*M)}_\bullet\to\Omega_{\kappa\,\bullet}^1(M)
    \to\Omega_{\kappa\,\bullet}^2(M)\to\cdots)
    \\&\cong H^{*}(\mathbf{E}^\bullet G,0\to\Omega_{\kappa\,\bullet}^1(M)
    \to\Omega_{\kappa\,\bullet}^2(M)\to\cdots)\,,
\end{align*}
and since $\Omega_{\kappa\,\bullet}^1(M)\to\Omega_{\kappa\,\bullet}^2(M)\to\cdots$ is a complex of sheaves on stacks, by a similar argument made when defining the van Est map in the previous section we get that 
\begin{align*}
    H^{*+1}(\mathbf{E}^\bullet G,0\to\Omega_{\kappa\,\bullet}^1(M)\to\Omega_{\kappa\,\bullet}^2(M)\to\cdots)\cong H^*_0(\mathfrak{g},M)\,.
\end{align*}
Then $VE_0$ is the map 
\begin{align*}
   & H^*_0(G,M)=H^{*+1}(\mathbf{B}^\bullet G,\mathcal{O}(M)^0_\bullet)\xrightarrow{\kappa^{-1}} H^{*+1}(\mathbf{E}^\bullet G,\kappa^{-1}\mathcal{O}(M)^0_\bullet)
   \\&\to H^{*+1}(\mathbf{E}^\bullet G,\widehat{\kappa^{-1}\mathcal{O}(M)}_\bullet)
   \,{\cong}\, H^{*+1}(\mathbf{E}^\bullet G,\widehat{\mathcal{O}(\kappa^*M)}_\bullet\to\Omega_{\kappa\,\bullet}^1(M)
    \to\cdots)
    \\&{\cong}\,H^{*+1}(\mathbf{E}^\bullet G,0\to\Omega_{\kappa\,\bullet}^1(M)
    \to\cdots)
    \xrightarrow{\cong} H^{*}_0(\mathfrak{g},M)\,.
\end{align*}
\end{proof}
\theoremstyle{definition}\begin{rmk}
The van Est map (including the truncated version) factors through a local van Est map defined on the cohomology of the local groupoid [see~\cite{Meinrenken}], ie. to compute the van Est map one can first localize the cohomology classes to a neighborhood of the identity bisection.
\end{rmk}
\subsection{Properties of the van Est Map}
In this section we discuss some properties of the van Est map; the main results pertain to its kernel and image.
\\\\Recall that given a sheaf $\mathcal{S}_\bullet$ on a (semi) simplicial space $X^\bullet\,,$ we calculate its cohomology by taking an injective resolution $0\to\mathcal{S}_\bullet\to \mathcal{I}^0_\bullet\to\mathcal{I}^1_\bullet\to\cdots$
and computing 
\begin{align*}
    H^*(\,\Gamma_\text{inv}(\mathcal{I}^0_0)\to\Gamma_\text{inv}(\mathcal{I}^1_0)\to\cdots)\,.
\end{align*}
By considering the natural injection $\Gamma_\text{inv}(\mathcal{I}_0^n)\hookrightarrow\Gamma(\mathcal{I}^n_0)$ we get a map
\begin{equation}\label{restriction}
    r:H^*(X^\bullet,\mathcal{S}_\bullet)\to H^*(X^0,\mathcal{S}_0)\,.
    \end{equation}
Similarly, for a cochain complex of abelian groups $\mathcal{A}^0\to \mathcal{A}^1\to\cdots$ there is a map 
\begin{align*}
H^*(\mathcal{A}^0\to \mathcal{A}^1\to\cdots)\overset{r}{\to} H^*(\mathcal{A}^0)\,.
\end{align*}
Using this, we have the following result, which gives an enlargement of diagram \eqref{diagram}:
\begin{lemma}\label{commute}
The following diagram is commutative:
\[
\begin{tikzcd}
H^{*}(G^0,\mathcal{O}(M))\arrow{r}{\delta^*}\arrow{d}{\parallel} &H^*_0(G,M)\arrow{r}{}\arrow{d}{VE_0} &H^{*+1}(G,M) \arrow{r}{r} \arrow{d}{VE} & H^{*+1}(G^0,\mathcal{O}(M))\arrow{d}{\parallel} \\
H^{*}(G^0,\mathcal{O}(M))\arrow{r}{d_{CE}\log}& H^*_0(\mathfrak{g},M)\arrow{r}{}& H^{*+1}(\mathfrak{g},M) \arrow{r}{r}& H^{*+1}(G^0,\mathcal{O}(M))
\end{tikzcd}
\]
\end{lemma}
\begin{lemma}\label{kernel}
Suppose that $X\overset{\pi}{\to} Y$ is a surjective submersion with $(n-1)$-connected fibers, for some $n>0\,,$ and with a section $\sigma\,.$  Consider an exact sequence of of families of abelian groups on $Y$ given by
\begin{align*}
    0\to Z\xrightarrow{}\mathfrak{m}\xrightarrow{\exp{}}M\,.
\end{align*} Let $\omega\in H^0(X,\Omega^n_{\pi}(\pi^*\mathfrak{m}))$ be closed (ie. $\omega$ is a closed, foliated $n$-form on $X$) and suppose that
\begin{equation}\label{in kernel}
\int_{S^n(\pi^{-1}(y))}\omega \in Z\;\;\text{     for all } y \in Y \text{ and all } S^n(\pi^{-1}(y))\,,
\end{equation}
where $S^n(\pi^{-1}(y))$ is an $n$-sphere contained in the source fiber over $y\,.$ 
Let $[\omega]$ denote the class $\omega$ defines in $H^n(X,\pi^{-1}\mathcal{O}(\mathfrak{m}))\,.$ Then $\exp{[\omega]}=0\,.$ 
\begin{proof}
From Equation~\ref{in kernel} we know that $\exp{\omega}\vert_{\pi^{-1}(y)}=0$ for each $y\in Y\,,$ therefore since the source fibers of $X$ are $(n-1)$-connected, by Theorem~\ref{spectral theorem} we have that 
\begin{align*}
    \exp{[\omega]}=\pi^{-1}\beta
    \end{align*}
for some $\beta\in H^n(Y,\mathcal{O}(M))\,.$ Since $\pi\circ \sigma:Y\to Y$ is the identity this implies that
\begin{align*}
    \exp{\sigma^{-1}[\omega]}=\beta\,,
    \end{align*}
    but $\sigma^{-1}[\omega]=0$ since $\omega$ is a global foliated form. Hence $\beta=0\,,$ hence $\exp{[\omega]}=0\,.$
\end{proof}
\end{lemma}
\begin{corollary}\label{groupoid kernel}
Suppose that $G\rightrightarrows X$ is source $(n-1)$-connected for some $n>0\,.$ Consider an exact sequence of families of abelian groups  on $X$ given by
\begin{align*}
    0\xrightarrow{} Z\xrightarrow{}\mathfrak{m}\xrightarrow{\exp{}} M\,.
\end{align*} Let $\omega\in H^0(\mathcal{C}^0(\mathfrak{g},\mathfrak{m}))$ be closed (ie. it is a closed $n$-form in the Chevalley-Eilenberg complex) and suppose that
\begin{equation}\label{in kernel 2}
\int_{S^n(s^{-1}(x))}\omega \in Z\;\;\text{     for all } x \in X \text{ and all } S^n(s^{-1}(x))\,,
\end{equation}
where in the above we have left translated $\omega$ to a source-foliated $n$-form, and where $S^n(s^{-1}(x))$ is an $n$-sphere contained in the source fiber over $x\,.$
Let $[\omega]$ denote the class $\omega$ defines in $H^n(\mathbf{E}^\bullet G,\kappa^{-1}\mathfrak{m})\,.$ Then $r(\exp{([\omega])})=0\,,$ where $r$ is as in Equation~\ref{restriction}. 
\end{corollary}
\begin{proof}
This follows directly from Lemma~\ref{kernel}.
\end{proof}
\subsection{Main Theorem}\label{main theorem section}
Before proving the main theoem of the paper, we will discuss translation of Lie algebroid objects: similarly to how one can translate Lie algebroid forms to differential forms along the source fibers, one can translate all Lie algebroid cohomology classes (eg. 1-dimensional Lie algebroid representations) to cohomology classes along the source fibers (in the case of a Lie algebroid representation, translation will result in a principal bundle with flat connection along the source fibers). 
We will describe it in degree 1, the translations in higher degrees works analogously. 
\theoremstyle{definition}\begin{definition}\label{translation}

Let $G\rightrightarrows G^0$ be a Lie groupoid and let $M$ be a $G$-module. Let $\{U_i\}_i$ be an open cover of $G^0$ and let $\{(h_{ij},\alpha_i\}_{ij}$ represent a class in $H^1(\mathfrak{g},M)$ (here the $h_{ij}$  are sections of $M$ over $U_i\cap U_j\,,$ and the $\alpha_i$ are Lie algebroid 1-forms taking values in $\mathfrak{m}$). Then on $\mathbf{B}^1G$ we get a class in foliated cohomology (foliated with respect to the source map), ie. a class in
\begin{align}
    H^1(\mathcal{O}(s^*M)\overset{\text{dlog}}{\to}\Omega_{s}^1(s^*M)
    \overset{\text{d}}{\to}\Omega_{s}^2(s^*M)\to\cdots)\,,
\end{align}
defined as follows:
we have an open cover of $\mathbf{B}^1G$ given by $\{t^{-1}(U_i)\}_i\,.$ We then get a principal $s^*M$-bundle over $\mathbf{B}^1G$ given by transition functions $t^*h_{ij}$ defined as follows: for $g\in t^{-1}(U_{ij})$ let \begin{align}
    t^*h_{ij}(g):=g^{-1}\cdot h_{ij}(t(g))\,.
\end{align}Similarly, we define foliated 1-forms $t^*\alpha_i$ on $t^{-1}(U_i)$ as follows: for $g\in t^{-1}(U_i)$ with $s(g)=x\,,$ and for $V_g\in T_g(s^{-1}(x))\,,$ let
\begin{align*}
t^*\alpha_i(V_g):=g^{-1}\cdot\alpha_{i}(R_{g^{-1}}V_g)\,,
\end{align*}
where $R_{g^{-1}}$ denotes translation by $g^{-1}\,.$ Then the desired class is given by the cocycle $\{(t^*h_{ij},t^*{\alpha_i})\}_i$ on the open cover $\{t^{-1}(U_i)\}\,.$ For each $x\in G^0\,,$ by restricting the cocycle to $s^{-1}(x)$ we also get a class in \begin{align*}
H^1(s^{-1}(x),\mathcal{O}(M_x)\xrightarrow{\text{dlog}}\Omega^1\otimes\mathfrak{m}_x\to\Omega^2\otimes\mathfrak{m}_x\to\cdots)\,.
\end{align*}
\\Similarly, we can translate any class $\alpha\in H^{\bullet}(\mathfrak{g},M)$ to a class in
\begin{align}\label{class}
    H^{\bullet}(\mathcal{O}(s^*M)\overset{\text{dlog}}{\to}\Omega_{s}^1(s^*M)
    \overset{\text{d}}{\to}\Omega_{s}^2(s^*M)\to\cdots)\,,
\end{align}
and we denote 
this class by $t^*\alpha\,.$ Furthermore, for each $x\in G^0$ we obtain as class in
\begin{align*}
H^\bullet(s^{-1}(x),\mathcal{O}(M_x)\xrightarrow{\text{dlog}}\Omega^1\otimes\mathfrak{m}_x\to\Omega^2\otimes\mathfrak{m}_x\to\cdots)\,,
\end{align*} and we denote this class by $t_x^*\alpha\,.$ 
\\\\Alternatively, given a class $\alpha\in H_0^{\bullet}(\mathfrak{g},M)\,,$ we can translate this to a class in 
\begin{align}
    H^\bullet(0\to\Omega_{s}^1(s^*M)
    \overset{\text{d}}{\to}\Omega_{s}^2(s^*M)\to\cdots)\,,
\end{align}
and we denote this class by $t_0^*\alpha\,.$ In this case the notation $t^*\alpha$ will be used to mean the class obtained is in~\ref{class} by first viewing $\alpha$ as a class in $H^\bullet(\mathfrak{g},M)\,.$
$\blacksquare$
\end{definition}
\begin{proposition}
With the previous definition, we have the following commutative diagram:
\[
\begin{tikzcd}
H^\bullet_0(\mathfrak{g},M) \arrow{d}{}\arrow{r}{t_0^*} & H^{\bullet+1}(0\to\Omega_{s}^1(s^*M)
    \overset{\text{d}}{\to}\cdots)
\arrow{d} \\
H^{\bullet+1}(\mathfrak{g},M)\arrow{r}{t^*} 
&   H^{\bullet+1}(\mathcal{O}(s^*M)\to\Omega_{s}^1(s^*M)
    \overset{\text{d}}{\to}\cdots)
\end{tikzcd}
\]
$\blacksquare$
\end{proposition} 
The importance of the previous definition is due to the fact that given a class in $\alpha\in H^{\bullet}(\mathfrak{g},M)\,,$ the class $t^*\alpha$ defines a class in $H^{\bullet}(E^\bullet G\,,\kappa^{-1}\mathcal{O}(M))$ (or if $\alpha\in H^\bullet_0(\mathfrak{g},M)\,,$ then $t_0^*\alpha$ defines a class in $H^{\bullet}(E^\bullet G\,,\widehat{\kappa^{-1}\mathcal{O}(M)})\,,$ see~\Cref{van Est map}). 
We are now ready to state and prove the main theorem of the paper:
\begin{theorem}[Main Theorem]\label{van Est image}
Suppose $G\rightrightarrows G^0$ is source $n$-connected and that $M$ is a $G$-module fitting into the exact sequence
\[
0\to Z\to \mathfrak{m}\overset{exp}{\to} M\,,
\] where $\mathfrak{m}$ is the Lie algebroid of $M\,.$ Then the van Est map  $\,VE:H^*(G,M)\to H^*(\mathfrak{g},M)$ is an isomorphism in degrees $\le n$ and injective in degree $(n+1)\,.$ The image of $VE$ in degree $(n+1)$ are the classes $\alpha\in H^{n+1}(\mathfrak{g},M)$ such that for all $x\in G^0\,,$ the translated class $t_x^*\alpha$ (see Definition~\ref{translation}) is trivial in \begin{align*}
H^{n+1}(s^{-1}(x),\mathcal{O}(M_x)\xrightarrow{\text{dlog}}\Omega^1\otimes\mathfrak{m}_x\to\Omega^2\otimes\mathfrak{m}_x\to\cdots)\,.
\end{align*} The same statement holds for $VE_0:H_0^*(G,M)\to H_0^*(\mathfrak{g},M)$ with a degree shift, that is: the truncated van Est map  $VE_0:H_0^*(G,M)\to H_0^*(\mathfrak{g},M)$ is an isomorphism in degrees $\le n-1$ and injective in degree $n\,.$ The image of $VE_0$ in degree $n$ are the classes $\alpha\in H_0^{n}(\mathfrak{g},M)$ such that for all $x\in G^0\,,$ the translated class $t_x^*\alpha$ is trivial in \begin{align*}
H^{n+1}(s^{-1}(x),\mathcal{O}(M_x)\xrightarrow{\text{dlog}}\Omega^1\otimes\mathfrak{m}_x\to\Omega^2\otimes\mathfrak{m}_x\to\cdots)\,.
\end{align*}
In particular, let $\omega$ be a closed Lie algebroid $(n+1)$-form, ie.
\[
\omega \in \ker\big[\Gamma(\mathcal{C}^{n+1}(\mathfrak{g},M))\xrightarrow{d_\text{CE}}\Gamma(\mathcal{C}^{n+2}(\mathfrak{g},M))\big]\,.
\]
Then $[\omega]\in H^{n}_0(\mathfrak{g},M)$ is in the image of $VE_0$ if and only if \begin{equation}\label{van Est condition}
\int_{S^{n+1}_x}\omega \in Z\;\;\text{     for all } x \in G^0 \text{ and all } S^{n+1}_x\,,
\end{equation}
 where $S^{n+1}_x$ in an $(n+1)$-sphere contained in the source fiber over $x\,.$\footnote{For the case of smooth Lie groups, this seems to be shown in~\cite{wockel1}, although our proof is still different.}
 \end{theorem}
\begin{proof}
The statement regarding $VE$ follows from the fact that 
\begin{equation}\label{iso of coho}
 \  H^*(\mathfrak{g},M)= H^*(\mathbf{E}^\bullet G,\kappa^{-1}\mathcal{O}(M)_\bullet)
\end{equation}
and Theorem~\ref{spectral theorem}. For the statement regarding $VE_0$ we use the fact that
\begin{align*}
    H^*_0(\mathfrak{g},M)\cong  H^{*+1}(\mathbf{E}^\bullet G,\widehat{\kappa^{-1}\mathcal{O}(M)}_\bullet)\,, \end{align*}
and the fact that the map
\begin{align*}
    H^{*+1}(\mathbf{E}^\bullet G,\kappa^{-1}\mathcal{O}(M)^0_\bullet)\xrightarrow{}H^{*+1}(\mathbf{E}^\bullet G,\widehat{\kappa^{-1}\mathcal{O}(M)}_\bullet) 
\end{align*}
is an isomorphism in degrees $\le n-1$ and is injective in degree $n\,.$
Furthermore, Theorem~\ref{spectral theorem} implies that \begin{align*}
    H^*(\mathbf{B}^{\bullet}G,\mathcal{O}(M)^0_\bullet)\xrightarrow{\kappa^{-1}}H^*(E\mathbf{G}^{\bullet},\kappa^{-1}\mathcal{O}(M)^0_\bullet)
    \end{align*}
is an isomorphism in degrees $\le n-1$ and is injective in degree $n\,,$ hence we get that the map
$H^*_0(G,M)\to H^*_0(\mathfrak{g},M)$ is an isomorphism in degrees $\le n-1$ and injective in degree $n\,.$ The statement regarding its image in degree $n$ follows from Corollary~\ref{groupoid kernel}.
\end{proof}
\theoremstyle{definition}\begin{exmp}This is a continuation of Example~\ref{module example}.
The source fibers of $\Pi_1(S^1)\rightrightarrows S^1$ are contractible, hence Theorem~\ref{van Est image} shows that
the cohomology groups are $H^i(\Pi_1(S^1),\tilde{\mathbb{R}})=0$ in all degrees, and 
\[
H^i(\Pi_1(S^1),\widetilde{\mathbb{R}/\mathbb{Z}})=H^{i+1}(\Pi_1(S^1),\tilde{\mathbb{Z}})=\begin{cases}
      \mathbb{Z}/2\mathbb{Z}, & \text{if}\ i= 0 \\
      0, & \text{if}\ i\ne 0\,,
    \end{cases}
\] and this result agrees with the computation done in Example~\ref{module example}.
We also have that $\Pi_1(S^1)$ is Morita equivalent to the fundamental group $\pi_1(S^1)\cong\mathbb{Z}\,,$ and the associated $\mathbb{Z}$-modules are the abelian groups $\mathbb{Z}\,,\,\mathbb{R}\,,\,\mathbb{R}/\mathbb{Z}\,,$ where even integers act trivially and odd integers act by inversion. One can also use this information to compute $H^{i}(\Pi_1(S^1),\tilde{\mathbb{Z}})$
and indeed find that \[
H^{i}(\Pi_1(S^1),\tilde{\mathbb{Z}})=\begin{cases}
      \mathbb{Z}/2\mathbb{Z}, & \text{if}\ i= 1 \\
      0, & \text{if}\ i\ne 1\,.
    \end{cases}
\]\end{exmp} 

\subsection{Groupoid Extensions and the van Est Map}
To every extension 
\begin{equation}\label{ab ext}
     1\to A\to E\to G\to 1
    \end{equation}
    of a Lie groupoid $G$ by an abelian group $A$ (see~\ref{abelian extensions}) one can associate a class in $H^1_0(\mathfrak{g},A)$ (where $\mathfrak{g}$ is the Lie algebroid of $G$) in two ways: one is given by the extension class of the short exact sequence $0\to \mathfrak{a}\to\mathfrak{e}\to\mathfrak{g}\to 0$ determined by~\ref{ab ext}, and the other is given by applying the van Est map to the class in $H^1_0(G,A)$ determined by~\ref{ab ext}. Here we will show that these two classes are the same.
\begin{theorem}\label{groupoid extension}
Let $M$ be a $G$-module and consider an extension of the form
\begin{align*}
    1\to M\to E\to G\to 1
\end{align*}
and let $\alpha\in H^1_0(G,M)$ be its isomorphism class. Then the isomorphism class of the Lie algebroid associated to $VE(\alpha)\in H^1_0(\mathfrak{g},M)$ is equal to the isomorphism class of the Lie algebroid $\mathfrak{e}$ of $E\,.$
\end{theorem}
\begin{proof}
Let $\{U_i\}_i$ be an open cover of $G^0\xhookrightarrow{} G$ on which there are local sections $\sigma_i:U_i\to E$ such that $\sigma$ takes $G^0\xhookrightarrow{}G$ to $G^0\xhookrightarrow{}E\,.$ These define a class $\alpha\in H^1_0(G,M)$ by taking $g_{ij}=\sigma^{-1}_i\cdot \sigma_j$ on $U_i\cap U_j\,,$ and where $h_{ijk}=\sigma_k^{-1}\cdot \sigma_i\cdot \sigma_j$ on $p_1^{-1}(U_i)\cap p_2^{-1}(U_j)\cap m^{-1}(U_k)\subset \mathbf{B}^{2}G\,.$ The sections $\sigma_{ii}$ induce a splitting of 
\begin{align*}
    0\to\mathfrak{m}\vert_{U_i}\to\mathfrak{e}\vert_{U_i}\to\mathfrak{g}\vert_{U_i}\to 0\,,
\end{align*}
which in turn gives a canonical closed 2-form $\omega\in C^2(\mathfrak{g}\vert_{U_i},M)\,,$ and the isomorphism given by $g_{ij}:E\vert_{U_i\cap U_j}\to E\vert_{U_i\cap U_j}$ induces an isomorphism  $\mathfrak{e}\vert_{U_i\cap U_j}\to \mathfrak{e}\vert_{U_i\cap U_j}$ given by $g_{ij*}$ (ie. the pushforward). Now the argument in Theorem 5 in~\cite{Crainic} implies that $VE(h_{iii})=[\omega_i]\,,$ and then one can check that $VE(\alpha)$ is the class given by $\{(\omega_i,g_{ij*})\}_{ij}\,.$
\end{proof}
\chapter{Applications}
\section{Groupoid Extensions and Multiplicative Gerbes}
Here we describe applications of the main theorem (\Cref{van Est image}) to the integration of Lie algebroid extensions, to representations, and to multiplicative gerbes. 
\\\\If we take $M=E$ to be a representation in Theorem~\ref{van Est image}, then $Z=\{0\}$ and we obtain the following result, due to Crainic (see~\cite{Crainic}, section 2.3).
\begin{theorem}[\cite{Crainic}]
Suppose $G\rightrightarrows X$ is source $(n-1)$-connected and that $E$ is a $G$-representation. Then the van Est map  $\,VE:H^*(G,E)\to H^*(\mathfrak{g},E)$ is an isomorphism in degrees $\le n-1$ and is injective in degree $n\,.$  Furthermore, $\omega\in H^n(\mathfrak{g},E)$ is in the image of $VE$ if and only if \begin{equation}
\int_{S^n_x}\omega=0\;\;\text{     for all } x \in X \text{ and all } S^n_x\,,
\end{equation}
 where $S^n_x$ in an $n$-sphere contained in the source fiber over $x\,.$ 
\end{theorem}
Now we will prove a result about the integration of Lie algebroid extensions, which generalizes the above result in the $n=2$ case. At least in the case where $M=S^1$ this is due to Crainic and Zhu (see~\cite{zhu}), but their proof is different.
\begin{theorem}\label{central extension}
Consider the exponential sequence $0\to Z \to\mathfrak{m}\overset{\exp}{\to} M\,.$ Let
\begin{align}\label{LA extension}
    0\to\mathfrak{m}\to \mathfrak{a} \to \mathfrak{g}\to 0
\end{align}
be the central extension of $\mathfrak{g}$ associated to $\omega\in H^2(\mathfrak{g},\mathfrak{m})\,.$ Suppose that $\mathfrak{g}$ has a simply connected integration $G\rightrightarrows X$ and that 
\begin{align}\label{condition Z}
    \int_{S^2_x}\omega\in Z
\end{align}
for all $x\in X$ and $S^2_x\,,$ where $S^2_x$ in a $2$-sphere contained in the source fiber over $x\,.$  Then $\mathfrak{a}$ integrates to a unique extension
\begin{align}\label{LG extension}
1\to M\to A \to G\to 1\,.
\end{align}
\end{theorem}
In particular, if $G$ and $M$ are Hausdorff then $\mathfrak{a}$ admits a Hausdorff integration.\footnote{This generalizes Theorem 5 in~\cite{Crainic}, with a different proof.} 
\begin{proof}
By Theorem~\ref{van Est image} $H^1_0(G,M)$ is isomorphic to the subgroup of $H^1_0(\mathfrak{g},M)$ which have periods in $Z$ along the source fibers. Hence by Theorem~\ref{groupoid extension} the Lie algebroid extension in~\ref{LA extension} integrates to an extension of the form~\ref{LG extension}. Since in particular $\mathbf{B}^1A$ is a principal $M$-bundle over $G\,,$ it must be Hausdorff if $G$ and $M$ are.
\end{proof}
\begin{remark}Note that in fact a stronger result than the above theorem can be made. Suppose we have an extension 
\begin{align}\label{nonabelian}
0\to\mathfrak{m}\xrightarrow{\iota}\mathfrak{a}\xrightarrow{\pi} \mathfrak{g}\to 0\,,
\end{align}where now $\mathfrak{m}$ isn't assumed to be abelian, so that $M$ is a nonabelian module. However, suppose there is a splitting of~\eqref{nonabelian} such that the curvature $\omega$ takes values in the center of $\mathfrak{m}\,,$ denoted $Z(\mathfrak{m})$ (which we assume is a vector bundle). Then two things occur: First, let $\sigma:\mathfrak{g}\to\mathfrak{a}$ denote the splitting. Then we get an action of $\mathfrak{g}$ on $\mathfrak{m}$ defined by $
\iota(L_{X}W):=[\sigma(X),\iota(W)]\,,$ for $X\in \mathcal{O}(\mathfrak{g})\,,W\in \mathcal{O}(\mathfrak{m})$ (here we are defining $L_X W\,.$ One can check that this is in the image of $\iota$ and so defines a local section of $\mathcal{O}(\mathfrak{m})\,,$ and that this action is compatible with Lie brackets). Assume that this action integrates to an action of $G$ on $M\,,$ making $M$ into a (nonabelian) G-module. 
\\\\The second thing that occurs is that we get a central extension given by
\begin{align}\label{nonabelian2}
0\to Z(\mathfrak{m})\to (Z(\mathfrak{m})\oplus\mathfrak{g},\omega)\to \mathfrak{g}\to 0\,,
\end{align}
where $\omega$ is the curvature of $\sigma\,,$ and $\mathfrak{g}$ acts on $Z(\mathfrak{m})$ as above. The extension~\eqref{nonabelian2} is a reduction of~\eqref{nonabelian} in the following sense: we can form the Lie algebroid $Z(\mathfrak{m})\oplus \mathfrak{m}$ and this Lie algeboid has a natural action of $Z(\mathfrak{m})\,,$ and the quotient is isomorphic to $\mathfrak{m}\,.$ Similarly, we can form the Lie algebroid $(Z(\mathfrak{m})\oplus\mathfrak{g},\omega)\oplus \mathfrak{m}\,,$ and this Lie algebroid also has a natural action of $Z(\mathfrak{m})\,,$ and the quotient is isomorphic to $\mathfrak{a}\,.$ Therefore, the extension~\eqref{nonabelian} is associated to the extension~\eqref{nonabelian2} in a way that is analogous to the reduction of the structure group of a principal bundle.
\\\\Assume now that the extension~\eqref{nonabelian2} integrates to an extension 
\begin{align}\label{abelian ext}
1\to Z(M)\xrightarrow{\iota} E \xrightarrow{\pi}  G\to 1\,,
\end{align}
where $G$ is the source simpy connected groupoid integrating $\mathfrak{g}\,.$ Then we can form the product Lie groupoid $E _{s}{\times}_{s} M:$ the multiplication is given by \begin{align*}
(e,m)(e',m')=(ee',m(\pi(e)^{-1}\cdot m'))\,,
\end{align*} where $t(e)=s(e')\,.$ Similarly to the Lie algebroid extension case, the family of abelian groups $Z(M)$ acts on the family of groups $Z(M)_{s}{\mathop{\times}}_s M\,,$ as well as on the Lie groupoid $E _{s}{\mathop{\times}}_s M\,,$ and the quotient of the former is isomorphic to $M\,,$ and the quotient of the latter integrates $\mathfrak{a}$ in~\eqref{nonabelian}. This gives us an extension
\begin{align}
    1\to M\to A\to G\to 1
\end{align}
integrating~\eqref{nonabelian}. Therefore, if  we can integrate~\eqref{nonabelian2} we can also integrate~\eqref{nonabelian}.
\\\\One should notice the similarity between the construction we've just described and the construction described in Lemma 3.6 in~\cite{rui}, in the special case of a regular Lie algebroid (ie. where the anchor map has constant rank), and where the extension is given by
\begin{align}
    0\to \text{ker}(\alpha)\to \mathfrak{a}\xrightarrow{\alpha}\text{im}(\alpha)\to 0\,,
\end{align}
where $\alpha$ is the anchor map of $\mathfrak{a}\,.$ The obstruction to integration described there coincides with the obstruction given by~\Cref{van Est image} for the integration of~\eqref{nonabelian2}, and we've shown that the vanishing of this obstruction is sufficient for the integration of~\eqref{nonabelian}, and hence of $\mathfrak{a}\,,$ to exist. 
\end{remark}

The above results concerned the degree $1$ case in truncated cohomology. We will now apply the main theorem to the integration of rank one representations, which concerns degree $1$ in nontruncated cohomology. First we make use of the following result:
\begin{proposition}\label{G representations}
The group of isomorphism classes of representations of $G\rightrightarrows G^0$ on complex line bundles is isomorphic to $H^1(G\,,\mathbb{C}^*_{G^0})\,.$ The corresponding statement for real line bundles holds, with $\mathbb{C}^*_{G^0}$ replaced by $\mathbb{R}^*_{G^0}\,.$ See Example~\ref{representation}.
\end{proposition}
The following statement is already known, we are just giving a cohomological proof.
\begin{theorem}
Let $G\rightrightarrows G^0$ be a source simply connected Lie groupoid. Then $\text{Rep}(G,1)\cong \text{Rep}(\mathfrak{g},1)\,,$ where
$\text{Rep}(G,1)\,,\,\text{Rep}(\mathfrak{g},1)$ are the categories of 1-dimensional representations, ie. representations on line bundles.
\end{theorem}
\begin{proof}
This follows directly from Theorem~\ref{van Est image}, Example~\ref{representation} and Proposition~\ref{G representations}.
\end{proof}
Now for the degree $2$ case in truncated cohomology: we use the main theorem to give a proof of an integration result concerning the multiplicative gerbe on compact, simple and simply connected Lie groups (see~\cite{konrad}). For the purposes of this thesis it is enough to think of a gerbe as the data given by a degree 2 \^{C}ech cocycle.
\begin{theorem}
Let $G$ be a simply connected Lie group. Then for each $\alpha\in H^2_0(\mathfrak{g},\mathbb{R})$ which is integral on $G\,,$ there is a class in $H^2_0(G,S^1)$ integrating it.
\end{theorem}
\begin{proof}
It is well known that simply connected Lie groups are 2-connected (since $\pi_2(G)=0$ for Lie groups), so Theorem~\ref{van Est image} immediately gives the result.
\end{proof}
\subsection{Group Actions and Lifting Problems}
In this section we apply~\Cref{van Est image} to study the problems of lifting projective representations to representations, and to lifting Lie group actions to principal torus bundles.
\subsubsection{Lifting Projective Representations}
\begin{theorem}
Let $G$ be a simply connected Lie group and let $V$ be a finite dimensional complex vector space. Let $\rho:G\to \text{PGL } (V)$ be a homomorphism. Then $G$ lifts to a homomorphism $\tilde{\rho}:G\to\text{GL } (V)\,.$ If $G$ is semisimple, this lift is unique.
\end{theorem} 
\begin{proof}
We have a central extension
\begin{align}\label{GL extension}
    1\to \mathbb{C}^*\to \text{GL } (V)\to \text{PGL } (V)\to 1\,,
\end{align}
and the corresponding Lie algebra extension splits: the Lie algebra of $\text{PGL } (V)$ is isomorphic to $\mathfrak{g}\mathfrak{l}(V)/\mathbb{C}\,,$ where $\lambda\in\mathbb{C}$ acts on $X\in\mathfrak{g}\mathfrak{l}(V)$ by taking $X\mapsto X+\lambda\,\mathbf{I}\,.$ The map \begin{align*}
    \mathfrak{g}\mathfrak{l}(V)/\mathbb{C}\to\mathfrak{g}\mathfrak{l}(V)\,, X\mapsto X-\frac{\text{tr}(X)}{\text{dim}(V)}\mathbf{I}
    \end{align*}
    is a Lie algebra homomorphism. Therefore, since $G$ is simply connected,~\Cref{van Est image} implies that the extension of $G$ that we get by pulling back the extension given by \eqref{GL extension} via $\rho$ is trivial (since the pullback of a trivial Lie algebra extension is trivial). However, a trivialization of the pullback extension is the same thing as a lifting of the homomorphism $\rho$ to a homomorphism $\tilde{\rho}:G\to\text{GL } (V)\,,$ hence such a lifting exists. 
    \\\\Now for uniqueness: it is easy to see that the liftings of $\rho$ are a torsor for $\text{Hom}(G,\mathbb{C}^*)\,,$ but again by~\Cref{van Est image} we have that $\text{Hom}(G,\mathbb{C}^*)\cong \text{Hom}(\mathfrak{g},\mathbb{C})\,,$ and the right side is $0$ if $G$ is semisimple. Hence if $G$ is semisimple there is a unique lift.
    
\end{proof}
\begin{remark}
One can also use the above method to give a proof of Bargmann's theorem (\cite{bargmann}), that is, if $H^2(\mathfrak{g},\mathbb{R})=0\,,$ then every projective representation of a (infinite dimensional) Hilbert space lifts to a representation.
\end{remark}
\subsubsection{Lifting Group Actions to Principal Bundles and Quantizations}
Now we will look at a different lifting  problem, one involving compact, semisimple Lie groups. First let us remark the following well-known result:
\begin{lemma}
A compact Lie group is semisimple if and only if its fundamental group is finite.
\end{lemma}
Now the aim of the rest of this section is to prove the following result:
\begin{theorem}\label{compact group actions}
Let $G$ be a compact, semisimple Lie group acting on a manifold $X\,.$ Suppose $P\to X$ is a principal bundle for the $n$-torus $T^n\,.$ Then the action of $G$ on $X$ lifts to an action of $G$ on $P^{|\pi_1(G)|}$ (here $P^{|\pi_1(G)|}$ is the principal $T^n$-bundle whose torsor over $x\in X$ is the product of the torsor over $x$ in $P$ with itself $|\pi_1(G)|$ times), and the lift is unique up to isomorphism (ie. any two lifts differ by a principal bundle automorphism). 
\end{theorem}
In particular, if $G$ is compact and simply connected, then actions of $G$ on a manifold $X$ lift to all principal $T^n$-bundles over $X\,.$
\theoremstyle{definition}\begin{exmp}
Consider the standard action of $SO(3)$ on $S^2\,.$ We have that $\pi_1(SO(3))=\mathbb{Z}/2\mathbb{Z}\,,$ hence $|\pi_1(SO(3))|=2\,.$ Therefore,~\Cref{compact group actions} implies that the action of $SO(3)$ on $S^2$ lifts to an action on all even degree principal $S^1$-bundles over $S^2\,,$ in a unique way up to isomorphism. On the other hand, since $SU(2)$ is simply connected, its standard action on $S^2$ lifts to an action on all principal $S^1$-bundles over $S^2\,,$ again in a unique way up to isomorphism.
\end{exmp}
Before proving~\Cref{compact group actions}, we will prove the following result, which is interesting in its own right and is related to the Riemann-Hilbert correspondence\footnote{In particular, this result determines exactly when a flat connection on a gerbe integrates to an action of the fundamental groupoid on the gerbe.}
\begin{lemma}\label{fundamental group}
Let $X$ be a connected manifold with universal cover $\tilde{X}$ and suppose that $\pi_k(X)=0$ for all $2\le k\le m\,.$ Let $T^n\xrightarrow{\text{dlog}} \mathbf{\Omega}^\bullet$ be the Deligne complex. Then $H^k(\pi_1(X),T^n)\cong H^k(X,T^n\xrightarrow{\text{dlog}} \mathbf{\Omega}^\bullet)$ for all $2\le k\le m\,,$ and the following sequence is exact: \begin{align}
    0\to H^{m+1}(\pi_1(X),T^n)\to H^{m+1}(X,T^n\to \mathbf{\Omega}^\bullet)\to H^{m+1}(\tilde{X},T^n\to \mathbf{\Omega}^\bullet)\,.
\end{align}
\end{lemma}
\begin{proof}
We have that $\pi_1(X)$ is Morita equivalent to $\Pi_1(X)$ (the fundamental group and fundamental groupoid of $X\,,$ respectively), so by Morita invariance \begin{align*}
H^{\bullet}(\pi_1(X),T^n)\cong H^{\bullet}(\Pi_1(X),T_X^n)\,.\end{align*}
The result then follows from~\Cref{van Est image}.
\end{proof}
\begin{corollary}\label{power is zero}
Let $G$ be a connected Lie group and as usual let $\mathbf{B}^1G$ be the underlying manifold. Then for every class $\alpha \in H^1(\mathbf{B}^1G,T^n\to \mathbf{\Omega}^\bullet)\,,$ we have that $\alpha^{|\pi_1(G)|}=1\,.$ 
\end{corollary}
\begin{proof}
From~\Cref{fundamental group} we have the well-known result that $H^1(\mathbf{B}^1G,T^n\to \mathbf{\Omega}^\bullet)\cong H^1(\pi_1(G),T^n)\,.$ The latter is equal to $\text{Hom}(\pi_1(G),T^n)\,,$ however every $f\in \text{Hom}(\pi_1(G),T^n)$ satisfies $f^{|\pi_1(G)|}=1\,,$ completing the proof.
\end{proof}
We now state a proposition that will be needed for the proof of~\Cref{compact group actions} (for a proof of this proposition, see~\cite{Crainic}).
\begin{proposition}\label{vanishing cohomology prop}
Let $G\rightrightarrows X$ be a proper Lie groupoid (ie. the map $(s,t):G\to X\times X$ is a proper map). Let $E\to X$ be a representation of $G\,.$ Then $H^k(G,E)=0$ for all $k\ge 1\,.$ 
\end{proposition}
The key to proving~\Cref{compact group actions} is the following lemma:
\begin{lemma}\label{vanishing cohomology}
Let $G$ be a compact, simply connected Lie group acting on a manifold $X\,.$ Then $H^1_0(G\ltimes X,T^n_X)=0\,.$
\end{lemma}
\begin{proof}
Since $G$ is compact the action is proper, hence $G\ltimes X$ is a proper groupoid, hence from~\Cref{vanishing cohomology prop} we see that $H^k(G\ltimes X,\mathbb{R}^n_X)=0$ for all $k\ge 1\,.$ This implies that $H^k_0(G\ltimes X,\mathbb{R}^n_X)=0$ for all $k\ge 2\,.$  Since simply connected Lie groups are $2$-connected,~\Cref{van Est image} implies that $H^2_0(G\ltimes X,\mathbb{Z}^n_X)=0\,.$ Hence, from the short exact sequence
$0\to\mathbb{Z}^n\to\mathbb{R}^n\to T^n\to 0\,,$ we get that $H^1_0(G\ltimes X,T^n_X)=0\,.$ 
\end{proof}
We are now ready to prove~\Cref{compact group actions} for simply connected groups.
\begin{lemma}\label{compact simply connected}
Let $G$ be a compact, simply connected Lie group acting on a manifold $X\,.$ Suppose $P\to X$ is a principal bundle for the $n$-torus $T^n\,.$ Then the action of $G$ on $X$ lifts to an action of $G$ on $P\,,$ and the lift is unique up to isomorphism.
\end{lemma}
\begin{proof}
Consider the gauge groupoid of $P$ given by $\text{At}(P):=P\times P/T^n\rightrightarrows X\,,$ where the action of $T^n$ is the diagonal action (here the source and target maps are the projections onto the first and second factors, respectively, and a morphism with source $x$ and target $y$ is a $T^n$-equivariant morphism between the fibers of $P$ lying over $x$ and $y\,,$ respectively). The gauge groupoid fits into a central extension of $\text{Pair}(X)\,,$ ie.
\begin{align}\label{Atiyah sequence}
    1\to T^n_X\to \text{At}(P)\to\text{Pair}(X)\to 1\,.
\end{align}
A lift of the $G$-action to $P\to X$ is equivalent to a lift of the canonical homomorphism $G\ltimes X\xrightarrow{(s,t)}\text{Pair}(X)$ to $\text{At}(P)\,,$ which is equivalent to a trivialization of the central extension of $G\ltimes X$ given by pulling back, via $(s,t)\,,$ the central extension given by \eqref{Atiyah sequence}. From~\Cref{vanishing cohomology} we know that such a trivialization exists, hence the $G$-action lifts to $P\,.$
\\\\Uniqueness up to isomorphism follows from the fact that the isomorphism classes of different lifts are a torsor for the image of $H^0_0(G\ltimes X,T^n)$ in $H^1(G\ltimes X,T^n)\,,$ and that the image is trivial follows from the exponential sequence $1\to \mathbb{Z}^n\to\mathbb{R}^n\to T^n\to 1\,,$ since both $H^1(G\ltimes X,\mathbb{R}^n_X)$ and $H^1_0(G,\mathbb{Z}^n_X)$ are trivial (the former follows from~\Cref{vanishing cohomology prop}, the latter follows from~\Cref{van Est image}).
\end{proof}
Now we can prove~\Cref{compact group actions}. One way of doing this is to look at the action of $\pi_1(G)$ on its universal cover, another way is  the following:
\begin{proof}[Proof of Theorem \ref{compact group actions}]
Let $\tilde{G}$ be the universal cover of $G\,.$ From~\Cref{compact simply connected} we know that the corresponding action of $\tilde{G}$ on $X$ lifts to an action on $P\,,$ giving us a class $\alpha\in H^1(\tilde{G}\ltimes X,T^n)$ whose underlying principal bundle on $X$ is $P\,.$ Hence after applying the van Est map we get a class $\text{VE}(\alpha)\in H^1(\mathfrak{g}\ltimes X,T^n)\,,$ whose underlying principal bundle on $X$ is also $P\,.$ 
\\\\After translating $\text{VE}(\alpha)$, we get a flat $T^n$-bundle on each source fiber of $G\ltimes X\,,$ ie. for each $x\in X$ we get a flat $T^n$-bundle on $G$, which we denote by $P_{(G,x)}\,.$ Then by~\Cref{power is zero} we have that $P_{(G,x)}^{|\pi_1(G)|}$ is trivial. However, $P_{(G,x)}^{|\pi_1(G)|}$ is the right translation of $|\pi_1(G)|\cdot\text{VE}(\alpha)$ (where the $\mathbb{Z}$-action is the natural one on cohomology classes),
hence by~\Cref{van Est image} we get the existence of a lift of the $G$-action to $P^{|\pi_1(G)|}\,.$ 
\\\\Uniqueness follows from the same argument as in~\Cref{compact simply connected}.
 \end{proof}
In the same vein, given a Hamiltonian group action on a symplectic manifold $G\, \rotatebox[origin=c]{-90}{$\circlearrowright$}\,(M,\omega)\,,$ one can show that, if $G$ is simply connected, this action lifts to a quantization. The reason is that if one considers the action Lie algebroid $\mathfrak{g}\ltimes M\,,$ with anchor $\alpha\,,$ the moment map $\mu$ trivializes $\alpha^*\omega$ in the truncated Lie algebroid complex. The rest follows from the van Est isomorphism theorem.
\subsection{Quantization of Courant Algebroids}
In this section we will discuss applications of our main theorem to the quantization of Courant algebroids, as discussed in~\cite{gen kahler}. 
\\\\Let $C$ be a smooth Courant algebroid over $X$ (see~\cite{gen kahler} for a definition)  associated to a $3$-form $\omega\,,$ and suppose that it is prequantizable, that is $\omega$ has integral periods. Let $g$ denote an $S^1$-gerbe prequantizing $\omega\,.$ Let $D\subset C$ be a Dirac structure. Then in particular, $D$ is a Lie algebroid, and as explained in~\cite{gen kahler} $g$ can be equipped with a flat $D$-connection, denoted $A\,.$ This determines a class $[(g,A)]\in H^2(D,S^1_X)\,.$ Suppose $D$ integrates to a Lie groupoid. We can then ask about the integrability of $[(g,A)]\,,$ or in other words: does the action of $D$ on $g$ integrate to an action of the corresponding source simpy connected groupoid on $g\,?$
Here we give a class of examples that does integrate, and it relates to the basic gerbe on a compact, simple Lie group. $($see~\cite{erbe},~\cite{Severa}$)\,.$
\theoremstyle{definition}\begin{exmp}
Let $G$ be a compact, simple Lie group with universal cover $\tilde{G}\,.$ 
and let $\langle \cdot,\cdot\rangle$ be the unique bi-invariant $2$-form which at the identity is equal to the Killing form. Associated to $\langle \cdot,\cdot\rangle$ is a bi-invariant and integral $3$-form $\omega$, called the Cartan $3$-form, given at the identity by
\begin{align*}\omega\vert_e=\frac{\langle \,[\,\cdot\,,\,\cdot\,]\,,\cdot\rangle\vert_e}{2}\,.
\end{align*}
The Dirac structure in this case, called the Cartan-Dirac structure, is the action Lie algebroid $\mathfrak{g}\ltimes G\,,$ where the action is the adjoint action of $\mathfrak{g}$ on $G\,.$
From this there is a canonical class $\alpha\in H^2(\mathfrak{g}\ltimes G,S^1_G)\,,$ whose underlying gerbe on $G$ is called the basic gerbe. The source simply connected integation of $\mathfrak{g}\ltimes G$ is $\tilde{G}\ltimes G\,,$ where the action of $\tilde{G}$ on $G$ is the one lifting the action of $G$ on itself by conjugation. Since the source fibers of $\tilde{G}\ltimes G$ are diffeomorphic to $\tilde{G}\,,$ which is necessarily $2$-connected, by Theorem~\ref{van Est image} we have that 
\begin{align*}
    H^2(\tilde{G}\ltimes G,S^1_G)\overset{VE}{\cong}H^2(\mathfrak{g}\ltimes G,S^1_G)\,.
\end{align*} Hence $\alpha$ integrates to a class in $H^2(\tilde{G}\ltimes G,S^1_G)\,.$
\end{exmp}To summarize, we have proven the following [see~\cite{erbe},~\cite{krepski}]:
\begin{theorem}[Integration of Cartan-Dirac structures]
Let $G$ be a compact, simple Lie group with universal cover $\tilde{G}\,.$ Then the adjoint action of $\mathfrak{g}$ on the basic gerbe (where the action is given by the Cartan-Dirac structure) integrates to an action of $\tilde{G}$ on the basic gerbe.
\end{theorem} 
\subsection{Integration of Lie \texorpdfstring{$\infty$-Algebroids}{infinity algebrois}}
In this section we will discuss the integration and quantization of Lie $\infty$-algebroids. See~\cite{pym2} for more details. We consider Lie $\infty$-algebroids of the following form: 
\\\\Let $\mathfrak{g}$ be a Lie algebroid and let $M\to X$ be a $\mathfrak{g}$-module. Let $\omega\in C^n(\mathfrak{g},M)$ be closed, $n>2\,.$ We can define a two term Lie $(n-1)$-algebroid as follows: Let $\mathcal{L}=\mathfrak{m}\oplus\mathfrak{g}$ where $\mathfrak{m}$ has degree $2-n$ and $\mathfrak{g}$ has degree $0\,.$ Let all differentials be zero except for the degree $0$ and degree $-n$ differentials. Define the degree 
$0$ differential as follows: for $U$ an open set in $X$ and for $m_1\,,m_2\in\mathcal{O}(\mathfrak{m})(U)\,,g_1\,,g_2\in\mathcal{O}(\mathfrak{g})(U)\,,$ let
\begin{align*}
    [m_1+g_1,m_2+g_2]_0=[g_1,g_2]+d_{CE}m_2(g_1)-d_{CE}m_1(g_2)\,,
    \end{align*}
    where $[g_1,g_2]$ is the Lie bracket of $g_1\,,g_2$ in $\mathfrak{g}\,.$
    Define the degree $2-n$ bracket by as follows: for $g_1\,,\ldots\,,g_n\in\mathcal{O}(\mathfrak{g})(U)\,,$ let 
    \begin{align*}
        [g_1\,,\ldots\,,g_n]_n=\omega(g_1\,,\ldots\,,g_n)\,,
    \end{align*}
    otherwise if any of inputs is in $\mathcal{O}(\mathfrak{m})(U)$ let the bracket be zero. This defines a Lie $(n-1)$-algebroid. 
\\\\Since the universal cover of a $k$-dimensional torus $($for $k\ge 1)$ is contractible,~\Cref{van Est image} gives us the following result, at the level of cohomology:
\begin{corollary}
All Lie $(n-1)$-algebroids associated to closed $n$-forms on the $k$-dimensional torus $T^k$ integrate to multiplicative $(n-2)$-gerbes.
\end{corollary}
We now apply the previous results to Lie $2$-algebras (an L$-\infty$ algebra concentrated in the two lowest degrees). As proved in~\cite{baez2}, all Lie $2$-algebras are equivalent to ones of the form
\begin{align}\label{lie 2}
V\to \mathfrak{g}\,,
\end{align}
where the only nonzero brackets are the degree $0$ and $-1$ brackets, and where the degree $-1$ bracket is given by a closed $3$-form. Furthermore, if $\omega\,,\omega'$ define equivalent Lie $2$-algebras, then $[\omega]=[\omega']$ in $H^3_0(\mathfrak{g},V)\,,$ implying that the map from Lie $2$-algebras to $H^3_0(\mathfrak{g},V)$ is canonical. Since simply connected Lie groups are $2$-connected, Theorem~\ref{van Est image} can help us determine when a Lie $2$-algebra integrates.
\begin{theorem}\label{2 algebra}
Let $\mathcal{L}$ be a Lie $2$-algebra represented by the $3$-form $\omega\,.$ Let $G$ be the simply connected integration of $\mathfrak{g}\,.$ Then if the periods $P(\omega)$ of $\omega$ form a discrete subgroup of $V\,,$ then $\mathcal{L}$ integrates to a class in $H_0^2(G,V/P(\omega))\,.$ 
\end{theorem}
\theoremstyle{definition}\begin{rmk}
Note that in~\cite{andre} it is shown that the obstruction to integrating a Lie $2$-algebra to a Lie $2$-group is that the periods of $\omega$ form a discrete subgroup of $V\,,$ ie. the obstruction is the same as the one in the above theorem. To explain this, we note the following: it is shown in~\cite{schommer} that to every class in $H^2_0(G,S^1)$ there corresponds an equivalence class of Lie $2$-groups. We expect that under this correspondence,~\Cref{2 algebra} shows that the Lie $2$-algebras which satisfy the hypotheses of this theorem integrate to Lie $2$-groups.
\end{rmk}
\subsection{Applications to Gerbes}
Let us recall that isomorphism classes of gerbes with a flat connection on a manifold $X$ are classified by $H^2(TX,X\times\mathbb{C}^*)\,.$ If the underlying gerbe of a gerbe with a flat connection is trivial then there is a lift of this
gerbe to some $\omega\in H^2(TX,X\times\mathbb{C})\,.$ Furthermore, actions of $\Pi_1(X)$ on gerbes are classified by $H^2(\Pi_1(X),\mathbb{C}^*)\,.$ We will call a gerbe with a flat connection integrable if its isomorphism class is in the image of
the van Est map 
\[ 
H^2(\Pi_1(X),X\times\mathbb{C}^*)\to H^2(TX,X\times\mathbb{C}^*)\,.
\]
\Cref{van Est image} then implies the following:
\begin{theorem}\label{gerbe}
If $\pi_2(X)=0$, then all gerbes with a flat connection are integrable, and the correspondence is one-to-one. Otherwise, suppose the underlying gerbe of the gerbe with a flat connection is trivial and let $\pi:\tilde{X}\to X$ be the universal cover of $X\,.$ Then it is integrable if and only if $\pi^*\omega$ has integral periods, where $\omega\in H^2(TX,X\times\mathbb{C})$ is the lift of the class of the gerbe with a flat connection. Moreover if this is the case then the integration is unique.
\end{theorem}
\begin{remark}
In fact since 
\begin{align}\label{TX iso}
H^2(TX,X\times\mathbb{C}^*)\cong H^2(X,\mathbb{C}^*)\,,
\end{align}
one can show that
\[
H^2(\Pi_1(X),X\times\mathbb{C}^*)\overset{\text{VE}}{\cong} \ker[\pi^*:H^2(TX,X\times\mathbb{C}^*)\to H^2(T\tilde{X},\tilde{X}\times\mathbb{C}^*)]\,,
\]
where $\pi^*$ is the pullback induced by~\eqref{TX iso}.
This is consistent with the result from exercise 159 in~\cite{davis} which states that there is an exact sequence
\[
\pi_2(X)\to H_2(X,\mathbb{Z})\to H_2(\pi_1(X),\mathbb{Z})\to 0\,,
\]
from which one can deduce that 
\[
H^2(\pi_1(X),\mathbb{C}^*)\cong \ker[\pi^*:H^2(X,\mathbb{C}^*)\to H^2(\tilde{X},\mathbb{C}^*)]\,.
\]
\end{remark}
$\mathbf{}$
\\Morita invariance of cohomology gives us the following:
\begin{theorem}\label{gerbe fundamental group}
Integrable gerbes with a flat connection on $X$ are in one-to-one correspondence with isomorphism classes of central extensions of the form
\[
0\to\mathbb{C}^*\to E\to \pi_1(X)\to 0\,,
\]
ie. $H^2(\pi_1(X),\mathbb{C}^*)\,.$
\end{theorem}
$\mathbf{}$
\\Combining~\Cref{gerbe} and~\Cref{gerbe fundamental group} we get the following:
\begin{corollary}
There is a canonical embedding 
\[
H^2(\pi_1(X),\mathbb{C}^*)\xhookrightarrow{} H^2(TX,X\times\mathbb{C}^*)\,.
\]
If $\pi_2(X)=0$ then this embedding
is an isomorphism.
\end{corollary}
\begin{remark}
One should compare the above results with the well-known theorem which states that line bundles with a flat connection are in one-to-one correpsondence with $\text{Hom}(\pi_1(X),\mathbb{C}^*)\cong H^1(\pi_1(X),\mathbb{C}^*)\,.$ One can also prove this using the Mortia invariance of cohomology and the fact that line bundles with flat connections always integrate uniquely to a class in $H^1(\Pi_1(X),X\times\mathbb{C}^*)\,.$ Note that $H^2(\pi_1(X),\mathbb{C}^*)$ is known as the Schur multiplier of $\pi_1(X)$ (or its dual, depending on conventions; see~\cite{greg}).
\end{remark}
\subsection{Van Est Map: Heisenberg Action Groupoids}\label{heisenberg action}
In this section we will apply the tools developed in the previous sections to integrate a particular Lie algebroid extension and show that we get a Heisenberg action groupoid.
\\\\Consider the space $\mathbb{C}^2$ with divisor $D=\{xy=0\}\,.$ Then the $2$-form \begin{align*}
    \omega=\frac{dx\wedge dy}{xy}
    \end{align*}
    is a closed form in $C_0^2(T_{\mathbb{C}^2}(-\log{D}),\mathbb{C}_{\mathbb{C}^2})\begin{footnote}{On $X\backslash D$ the $2$-form $\omega/2\pi i$ is the curvature of the Deligne line bundle associated to the holomorphic functions $x$ and $y\,.$}\end{footnote}\,.$ The source simply connected integration of $T_{\mathbb{C}^2}(-\log{D})$ is 
$\mathbb{C}^2\ltimes \mathbb{C}^2\,,$ where the action of $\mathbb{C}^2$ on itself is given by
\begin{align*}
    (a,b)\cdot (x,y)=(e^ax,e^by)\,.
\end{align*}
Since the source fibers are contractible Theorem~\ref{van Est image} tells us that the central extension of $T_{\mathbb{C}^2}(-\log{D})$ defined by $\omega$ integrates to an $\mathbb{C}_{\mathbb{C}^2}$ central extension of $\mathbb{C}^2\ltimes \mathbb{C}^2\,.$ We will describe the central extension here. First we will compute the integration of $\omega:$ we define coordinates on $\mathbf{B}^{\bullet\le 2} (\mathbb{C}{\mathop{\times}}\mathbb{C}\ltimes \mathbb{C}{\mathop{\times}}\mathbb{C})$ as follows:
\begin{align*}
&(x,y)\in \mathbb{C}^2= \mathbf{B}^0 (\mathbb{C}{\mathop{\times}}\mathbb{C}\ltimes \mathbb{C}{\mathop{\times}}\mathbb{C})\,,
\\&(a,b,x,y)\in \mathbb{C}^2{\mathop{\times}}\mathbb{C}^2=\mathbf{B}^1 (\mathbb{C}{\mathop{\times}}\mathbb{C}\ltimes \mathbb{C}{\mathop{\times}}\mathbb{C})\,,
\\& (a',b',a,b,x,y)\in \mathbf{B}^2 (\mathbb{C}{\mathop{\times}}\mathbb{C}\ltimes \mathbb{C}{\mathop{\times}}\mathbb{C})\,.
\end{align*}
On $\mathbf{E}^{\bullet\le 2} (\mathbb{C}{\mathop{\times}}\mathbb{C}\ltimes \mathbb{C}{\mathop{\times}}\mathbb{C})$ we have coordinates
\begin{align*}
&(a,b,x,y)\in \mathbf{E}^0 (\mathbb{C}{\mathop{\times}}\mathbb{C}\ltimes \mathbb{C}{\mathop{\times}}\mathbb{C})\,,
\\&(a',b',a,b,x,y)\in \mathbf{E}^1 (\mathbb{C}{\mathop{\times}}\mathbb{C}\ltimes \mathbb{C}{\mathop{\times}}\mathbb{C})\,,
\\& (a'',b'',a',b',a,b,x,y)\in \mathbf{E}^2 (\mathbb{C}{\mathop{\times}}\mathbb{C}\ltimes \mathbb{C}{\mathop{\times}}\mathbb{C})\,,
\end{align*}
where the map $\kappa:\mathbf{E}^{\bullet\le 2} (\mathbb{C}{\mathop{\times}}\mathbb{C}\ltimes \mathbb{C}{\mathop{\times}}\mathbb{C})\to\mathbf{B}^{\bullet\le 2} (\mathbb{C}{\mathop{\times}}\mathbb{C}\ltimes \mathbb{C}{\mathop{\times}}\mathbb{C})$
is given by
\begin{align*}
    &(a,b,x,y)\mapsto (e^a x,e^b y)\,,
    \\& (a',b',a,b,x,y)\mapsto (a',b',e^a x, e^b y)\,,
    \\& (a'',b'',a',b',a,b,x,y)\mapsto (a'',b'',a',b',e^a x, e^b y)\,.
\end{align*}
When we right translate $\omega$ to $\mathbf{E}^0(\mathbb{C}\ltimes\mathbb{C})$ we get the fiberwise form $da\wedge db\,.$ This is exact, with primitive $a\,db\,.$ When we pullback $a\,db$ to $\mathbf{E}^1(\mathbb{C}\ltimes\mathbb{C})$ we get the fiberwise form $a'\,db\,,$ and this is exact, with primitive $a'b\,.$ When we pullback $a'b$ to $\mathbf{E}^2(\mathbb{C}\ltimes\mathbb{C})$ we get the function $a''b'\,,$ and this is $\kappa^*a'b\,.$ So the cocycle integrating $\omega$ is $f(a',b',a,b,x,y)=a'b\,.$ 
\\\\One can show that the central extension associated to this cocycle is an action groupoid of the complex Heisenberg group acting on $\mathbb{C}{\mathop{\times}}\mathbb{C}\,,$ ie. we have the following proposition:
\begin{proposition}
The logarithmic $2$-form $\frac{dx\wedge dy}{xy}$ on $\mathbb{C}^2$ with divisor $xy=0$ defines a Lie algebroid extension of $T_{\mathbb{C}^2}(-\log{\{xy=0\}})\,.$ This Lie algebroid extension integrates to an extension of $\mathbb{C}^2\ltimes \mathbb{C}^2$ given by a Heisenberg action groupoid. More precisely,
the extension is of the form
\begin{align}\label{heisenberg extension}
    0\to\mathbb{C}_{\mathbb{C}^2}\to H\ltimes\mathbb{C}^2 \to\mathbb{C}^2\ltimes \mathbb{C}^2\to 0\,,
\end{align}
where $H$ is
the subgroup of matrices of the form
\begin{align*}
    \begin{pmatrix}
    1 & a & c\\
    0 & 1 & b \\
    0 & 0 & 1
    \end{pmatrix}
\end{align*}
for $a\,,b\,,c\in\mathbb{C},$ and the action on $\mathbb{C}{\mathop{\times}}\mathbb{C}$ is given by $(a,b,c)\cdot (x,y)=(e^a x,e^b y)\,,$ where $(a,b,c)$ represents the above matrix.
\end{proposition}
\subsection{Van Est Map: \texorpdfstring{$\mathbb{C}\ltimes\mathbb{P}^1$}{C to P1}}
In this section we will classify two different geometric strucutres: rank one holomorphic representations of $\Pi_1(\mathbb{P}^1,\{0,\infty\})\cong \mathbb{C}\ltimes\mathbb{P}^1\,,$ which are classified by 
\[
H^1(\Pi_1(\mathbb{P}^1,\{0,\infty\}),\mathbb{C}^*_{\mathbb{P}^1})\,,
\]
and rank one holomorphic representations of its Lie algebroid, denoted $T_{\mathbb{P}^1}(-\log{\{0,\infty\}})\,,$ which are classified by 
\[H^1(T_{\mathbb{P}^1}(-\log{\{0,\infty\}}),\mathbb{C}^*_{\mathbb{P}^1})\,.
\]
We then compute the van Est map between them and explicitly show that it is an isomorphism.
\vspace{3mm}\\Let's begin: consider the action of $\mathbb{C}$ on $\mathbb{P}^1$ given by $a\cdot [z:w]=[e^{a}z:w]\,,$ and form the action groupoid given by $\mathbb{C}\ltimes\mathbb{P}^1\,.$ Then representations of $\mathbb{C}\ltimes\mathbb{P}^1$ on holomorphic line bundles are classified by $H^1(\mathbb{C}\ltimes\mathbb{P}^1,\mathbb{C}^*_{\mathbb{P}^1})\,,$ and these are the global versions of flat logarithmic connections on holomorphic line bundles, with poles at $[0:1]\,,$ also known as holomorphic representations of the Lie algebroid of $\mathbb{C}\ltimes\mathbb{P}^1\,.$ The sheaf of sections of the Lie algebroid of $\mathbb{C}\ltimes\mathbb{P}^1$ is isomorphic to the sheaf of sections of $T_{\mathbb{C}}$ which vanish at the origin and $\infty\,.$
\begin{steps}
\item Let $U_0\,,U_1$ be the standard open covering of $\mathbb{P}^1\,.$ Then we get an open covering of $\mathbf{B}^1(\mathbb{C}\ltimes\mathbb{P}^1)$ by using the open cover $\{s^{-1}U_i\cap t^{-1}U_j\}_{i,j\in \{0,1\}}\,.$ A standard Mayer-Vietoris argument shows that this is a good cover in degree one, ie. it can be used to compute cohomology in degree one. Let $U_{ij}=s^{-1}U_i\cap t^{-1}U_j\,.$ The inequivalent degree one cocycles are given by the following:
\begin{align}\label{cocycle}
    &\sigma_{00}(a,z)=e^{(k+\lambda)a}\,,\,\sigma_{01}(a,z)=e^{\lambda a}z^{-k}\,,
       \\ &\sigma_{10}(a,z)=e^{(k+\lambda) a}z^k\,,\,\sigma_{11}(a,z)=e^{\lambda a}\,,\nonumber\\& g_{01}(z)=z^k\,,\nonumber
\end{align}
where $\sigma_{ij}$ are functions on $U_{ij}$ and $g_{01}$ is a function on $U_0\cap U_1$ representing the principal bundle, and where $k\in\mathbb{Z}\,,\lambda\in\mathbb{C}\,.$ Hence 
\begin{align*}
    H^1(\mathbb{C}\ltimes\mathbb{P}^1,\mathbb{C}^*_{\mathbb{P}^1})\cong \mathbb{C}{\mathop{\times}}\mathbb{Z}\,.
\end{align*}
Now we compute the van Est map on these classes. First recall that the van Est map factors through the cohomology of the local groupoid, so we only need to be concerned with a neighborhood of the identity bisection.
$\mathbf{}$

\item  Pull back the cocycle via $\kappa^{-1}$ to a neighborhood of $G$ in
$G\ltimes G\,,$ and get the cocycle given by the functions
\begin{align*}
    \kappa^{-1}\sigma_{00}\,,\,\kappa^{-1}\sigma_{11}\,,
         t^{-1}g_{01}
\end{align*}
defined on the open sets $\kappa^{-1}U_{00}\,,\kappa^{-1}U_{11}\,,t^{-1}U_{10}\,,$ respectively. 
\item Now by the condition that~\eqref{cocycle} is a cocycle, it follows that 
\begin{align*}
    &\kappa^{-1}\sigma_{00}=\delta^*\sigma_{00}\;\text{   on  }\;\kappa^{-1}U_{00}\cap s^{-1}U_{00}\cap t^{-1}U_{00}\,,
    \\&\kappa^{-1}\sigma_{11}=\delta^*\sigma_{11}\;\text{   on  }\;\kappa^{-1}U_{11}\cap s^{-1}U_{11}\cap t^{-1}U_{11}\,,
    \end{align*}
and these two open sets cover a neighbourhood of $G^0$ in $G\ltimes G\,.$ Explicitly,
\begin{align*}
    \kappa^{-1}U_{ii}\cap s^{-1}U_{ii}\cap t^{-1}U_{ii}=\{(g_1,g_2)\in G{\mathop{\times}} G: g_1\,,g_2\,,g_1g_2\in U_{ii}\}\,,
\end{align*}
for $i=0\,,1\,.$
So for the second step we get the functions
\begin{align*}
    \sigma_{00}\,,\,\sigma_{11}\,,
         s^{-1}g_{01}\,,
\end{align*}
defined on the open sets $U_{00}\,,U_{11}\,,U_{00}\cap U_{11}\,,$ respectively.
\item Now apply $\mathrm{dlog}+\partial$ to $\sigma_{00}\,,\sigma_{11}\,,$ and we get the elements
\begin{align*}
    \mathrm{dlog}\sigma_{00}\,,\,\mathrm{dlog}\sigma_{11}\,.
         \frac{\sigma_{00}}{\sigma_{11}}s^{-1}g_{01}\,,
\end{align*}
defined on the open sets $U_{00}\,,U_{11}\,,U_{00}\cap U_{11}\,,$ respectively. Explicitly, these are given, respectively, by
\begin{align*}
    (k+\lambda)\frac{da}{2\pi i}\,,\,
    \lambda\frac{da}{2\pi i}\,,\,
    e^{ka}z^k\,.
\end{align*}
\item Now this cocycle is pulled back from the following cocycle in the Chevalley-Eilenberg complex via $t\,:$
\begin{align}\label{lie cocycle}
    \alpha_0(z)=(k+\lambda)\frac{da}{2\pi i}\,,
    \alpha_1(z)=\lambda\frac{da}{2\pi i}\,,
    g_{01}(z)=z^k\,,
\end{align}
where these maps are defined on $U_0\,,U_1\,, U_0\cap U_1\,,$ respectively.
\end{steps}
$\mathbf{}$
\\The anchor map in this case is $\alpha(\partial_a\vert_{(0,z')})=z'\,\partial_z\vert_{z'}\,,$
which is an embedding of sheaves, and hence the sheaf of sections of the Lie algebroid is isomorphic to the sheaf on $\mathbb{P}^1$ generated by $z\,\partial_z$ on $U_0$ and $\tilde{z}\,\partial_{\tilde{z}}$ on $U_1\,.$ Under this isomorphism of sheaves, $da$ gets sent to $dz/z\,.$
We can use the isomorphism to identify the Lie algebroid cocycle in~\eqref{lie cocycle} with the cocycle in the sheaf of logarithmic differential forms given by
\begin{align}\label{flat log connections}
    \alpha_0(z)=\frac{(k+\lambda)}{2\pi i}\,\frac{dz}{z}\,,
    \alpha_1(z)=\frac{\lambda}{2\pi i}\,\frac{dz}{z}\,,
    g_{01}(z)=z^k\,.
\end{align}
$\mathbf{}$
\\To summarize, we have the following:
\begin{proposition}
The cocycles in~\ref{cocycle} give an isomorphism
$H^1(\mathbb{C}\ltimes\mathbb{P}^1,\mathbb{C}^*_{\mathbb{P}^1})\cong\mathbb{C}{\mathop{\times}}\mathbb{Z}\,;
$ the cocycles in~\ref{flat log connections} give an isomorphism
$H^1(T_{\mathbb{P}^1}(-\log{\{0,\infty\}}),\mathbb{C}^*_{\mathbb{P}^1})\cong \mathbb{C}{\mathop{\times}}\mathbb{Z}\,.$
Under these isomorphisms the van Est map
\[
VE:H^1(\mathbb{C}\ltimes\mathbb{P}^1,\mathbb{C}^*_{\mathbb{P}^1})\to H^1(T_{\mathbb{P}^1}(-\log{\{0,\infty\}}),\mathbb{C}^*_{\mathbb{P}^1})
\]
is given by $(\lambda,k)\mapsto(\lambda,k)\,.$
\end{proposition}
\section{Heisenberg Manifold as a Higher Structure}
In this section we will show that the Heisenberg manifold has several compatible geometric structures on it; in particular, it is a principal bundle in the category of groupoids, or a groupoid in the category of principal bundles (in fact, one can enhance the construction we make to obtain a $\Pi_1(S^1)$-module in the category of principal bundles with a connection, since the Heisenberg manifold is naturally a principal bundle with connection over $T^2$).
\vspace{3mm}\\The Heisenberg manifold, denoted $H_M\,,$  is the quotient of the Heisenberg group by the right action of the integral Heisenberg subgroup on itself, ie. we make the identification
\begin{align*}
    \begin{pmatrix}
    1 & a & c\\
    0 & 1 & b \\
    0 & 0 & 1
    \end{pmatrix}\sim \begin{pmatrix}
    1 & a+n & c+k+am\\
    0 & 1 & b+m \\
    0 & 0 & 1
    \end{pmatrix}\,,
\end{align*}
where $a\,,b\,,c\in\mathbb{R}$ and $n\,,m\,,k\in\mathbb{Z}\,.$ 
\\\\$H_M$ is a principal $S^1$-bundle over $T^2$ by projecting onto $(a,b)\,.$  Furthermore, we get a $T^2$-bundle over $S^1$ by projecting onto $b\,,$ making $H_M$ into a family of abelian groups over $S^1\,.$
\\\\More explcitly, the product associated to the bundle $H_M\to S^1$ is given by
\begin{align*}
      \begin{pmatrix}
    1 & a & c\\
    0 & 1 & b \\
    0 & 0 & 1
    \end{pmatrix}\cdot  \begin{pmatrix}
    1 & a' & c'\\
    0 & 1 & b \\
    0 & 0 & 1
    \end{pmatrix}=  \begin{pmatrix}
    1 & a+a' & c+c'\\
    0 & 1 & b \\
    0 & 0 & 1
    \end{pmatrix}\,,
\end{align*}
Putting this together, we have the following diagram:
\begin{equation}\label{hei}
    \begin{tikzcd}
H_M \arrow[r, shift right] \arrow[r, shift left] \arrow[d] & S^1 \arrow[d] \\
T^2 \arrow[r, shift right] \arrow[r, shift left]           & S^1          
\end{tikzcd}
\end{equation}
Here, the principal bundle on the right is the trivial principal bundle (the map is the identity), and the groupoid on the bottom is the trivial $S^1$ family of abelian groups over $S^1\,.$
\vspace{3mm}\\As mentioned earlier, we can enhance \ref{hei} with connections to obtain a diagram of the following form:
\begin{equation}
\begin{tikzcd}
H_M \arrow[d, "{(S^1\text{-principal bundle},\nabla_2)}"'] \arrow[rrr, "{(T^2\text{-family of groups},\nabla_1)}", shift left] \arrow[rrr, shift right] &  &  & S^1 \arrow[d, "\text{Trivial principal bundle}"] \\
T^2 \arrow[rrr, shift left] \arrow[rrr, "\text{Trivial } S^1\text{-family of groups}"', shift right]                                                    &  &  & S^1                                             
\end{tikzcd}
\end{equation}
In the bottom row and right column, the connections are the trivial ones on the trivial bundles. The connection on the top row is flat, making $H_M$ into a $\Pi_1(S^1)$-module, and the connection on the left is the one associated with the quantization of $T^2\,.$ One might say that the quantization of $T^2$ is a $\Pi_1(S^1)$-module.

\section{The Canonical Module Associated to a Complex Manifold and Divisor}
Given a complex manifold $X$ and a (simple normal crossings) divisor $D\,,$ we construct a natural module for the Lie groupoid $\text{Pair}(X,D)$ (which is the terminal integration of $T_X(-\log{D})\,,$ the Lie algebroid whose sheaf of sections is the sheaf of sections of $T_X$ which are tangent to $D)\,.$ These are modules for which the underlying surjective submersion does not define a fiber bundle, and in particular the underlying family of abelian groups is not locally trivial. Generically the fiber will be $\mathbb{C}^*\,,$ but over $D$ the fibers will degenerate to $\mathbb{C}^*{\mathop{\times}} \mathbb{Z}^k\,,$ for some $k$ depending on the point $D\,.$
\subsection{The Module \texorpdfstring{$\mathbb{C}^*_{\mathbb{C}}(*\{0\})$}{C}}
Here we will do a warm up example for the general case to come in the next section. More precisely, we will construct
a family of abelian groups whose sheaf of sections is isomorphic to the sheaf of nonvanishing meromorphic functions with a possible pole or zero
only at the origin, and we will show that it is naturally a module for the terminal groupoid integrating $T_\mathbb{C}(-\log{\{0\}})\,,$ the Lie algebroid whose sheaf of sections is
isomorphic to the sheaf of sections of $T\mathbb{C}$ vanishing at the origin. This space was defined in~\cite{luk}.
\\\\Consider the action groupoid $\mathbb{C}^*\ltimes\mathbb{C}\rightrightarrows\mathbb{C}\,,$ where the action of $\mathbb{C}^*$ on $\mathbb{C}$ is given by
\begin{align*}
    a\cdot x=ax\,.
\end{align*}This is the terminal groupoid integrating $T_\mathbb{C}(-\log{\{0\}})\,.$ We will construct a module for this groupoid as follows:
consider the family of abelian groups given by
\begin{align*}
{\mathbb{C}}{\mathop{\times}}\mathbb{C}^*{\mathop{\times}}\mathbb{Z}\overset{p_1}{\to}\mathbb{C}\,.
\end{align*}
This family of abelian groups is a $\mathbb{C}^*\ltimes\mathbb{C}$-module with action given by
\begin{align}
(a,x)\cdot(x,y,i)=(ax,a^{-i}y,i)\,.
\end{align}
There is a submodule given by
\begin{align*}
\mathbb{C}{\mathop{\times}}\mathbb{Z}\backslash \{(0,j):j\ne 0\}\overset{p_1}{\to}\mathbb{C}\,,
\end{align*}
where the embedding into ${\mathbb{C}}{\mathop{\times}}\mathbb{C}^*{\mathop{\times}}\mathbb{Z}$ is given by $(x,j)\mapsto (x,x^{-j},j)\,,$ for $x\ne 0\,,$ and $(0,0)\mapsto (0,1,0)\,.$
We can then form the quotient to get another module, denoted $\mathbb{C}^*_{\mathbb{C}}(*\{0\})\,.$ Formally, we have the following:
\theoremstyle{definition}\begin{definition}
 We define the space $\mathbb{C}^*_{\mathbb{C}}(*\{0\})$ as
\begin{align*}
   \mathbb{C}^*_{\mathbb{C}}(*\{0\}):= {\mathbb{C}}{\mathop{\times}}\mathbb{C}^*{\mathop{\times}}\mathbb{Z}/\sim\,,\,(x,y,i)\sim (x,x^{-j}y,i+j)\,,\;x\ne 0\,.
\end{align*}
$\blacksquare$\end{definition}
\begin{proposition}The space $\mathbb{C}^*_{\mathbb{C}}(*\{0\})$ is a complex manifold and there is a holomorphic surjective submersion $\pi:M\to\mathbb{C}$ given by $\pi(x,y,i)=x\,$ The space $\mathbb{C}^*_{\mathbb{C}}(*\{0\})$ is a family of abelian groups with product defined by a
\begin{align*}
(x,y,i)\cdot(x,y',j)=(x,yy',i+j)\,.
\end{align*}
It is a $\mathbb{C}^*\ltimes\mathbb{C}$-module with action given by
\begin{align*}
    (a,x)\cdot(x,y,i)=(ax,a^{-i}y,i)\,,
\end{align*}
and there is a short exact sequence of modules given by
\begin{align*}
    0\to \mathbb{C}{\mathop{\times}}\mathbb{Z}\backslash \{(0,j):j\ne 0\}\to {\mathbb{C}}{\mathop{\times}}\mathbb{C}^*{\mathop{\times}}\mathbb{Z}\to \mathbb{C}^*_{\mathbb{C}}(*\{0\})\to 0\,. 
\end{align*}
The fiber of $\mathbb{C}^*_{\mathbb{C}}(*\{0\})$ over a point $x\ne 0$ is isomorphic to $\mathbb{C}^*\,,$ and the fiber over $x=0$ is isomorphic to $\mathbb{C}^*{\mathop{\times}}\mathbb{Z}\,.$ 
\end{proposition}
\begin{proof}
We prove that it is a complex manifold. First we show that we can cover the space with charts whose transition functions are holomorphic. For each $i\in\mathbb{Z}\,,$ we get a chart given by $\mathbb{C}\times\mathbb{C}^*\,,$ taking $(x,y,i)\mapsto (x,y)\,.$ On the intersection between the $i$ and $j$ coordinate systems, the transition function is given by $(x,y)\mapsto (x,x^{-j}y)\,,$ which is holomorphic. 
\newline\newline To prove the space is Hausdorff, we observe that away from $x=0\,,$ the space is just $\mathbb{C}^*\times\mathbb{C}^*\,.$ Now take two points $(0,y,i)\,,(x,y',j)\,,$ $x\ne 0\,.$ We get disjoint neighborhoods of these points by choosing small enough neighborhoods $U_i\,, U_j\,,$ such that the projections onto the $x$-coordinate are disjoint. Now given two distinct points $(0,y,i)\,,(0,y',j)\,,$ with $j>i$ we obtain two disjoint neighborhoods by choosing $x\in \mathbb{C}$ such that $|x^{i-j}y|>|y'|\,,$ and then choosing small enough disks around $y\,,y'\,.$ Now suppose we take two distinct points $(0,y,i)\,,(0,y',i)\,.$ We get two disjoint neighborhoods by choosing disjoint neighborhoods of $y\,,y'\in\mathbb{C}^*\,,$ and taking all $x\in\mathbb{C}^*\,.$
\end{proof}
\begin{proposition}\label{sheaf identification}The sheaf $\mathcal{O}(\mathbb{C}^*_{\mathbb{C}}(*\{0\}))$ (where sections here are taken to be holomorphic) is isomorphic to the sheaf of meromorphic functions on $\mathbb{C}$ with poles or zeroes only at $x=0\,,$ denoted $\mathcal{O}^*(*\{0\})\,.$
\end{proposition}
\begin{proof}
Consider the morphism of sheaves defined as follows: for an open set $U\subset\mathbb{C}$ and a holomorphic section $s(x)=(x,f(x),i)$ of $\mathbb{C}^*_{\mathbb{C}}(*\{0\})$ over $U\,,$ define a meromorphic function on $U\,,$ with a possible pole/zero only at $x=0\,,$ by $x^if(x)\,,\,x\in U\,.$ This map is an isomorphism of sheaves.
\end{proof}
Now to any $G$-module there is an associated $G$-representation, and the representation associated to $\mathbb{C}^*_{\mathbb{C}}(*\{0\})$ is the trivial one, ie. $\mathfrak{m}\cong\mathbb{C}{\mathop{\times}}\mathbb{C}$ with the projection map being the projection onto the first factor, and the action of $\mathbb{C}^*\ltimes\mathbb{C}$ is given by 
\begin{align*}
(a,x)\cdot(x,y)=(ax,y)\,.
\end{align*}
We identify $\mathfrak{m}$ with points $(x,y,0)\in\mathbb{C}{\mathop{\times}}\mathbb{C}{\mathop{\times}}\mathbb{Z}\,,$ where the second $\mathbb{C}$ is identified with the Lie algebra of $\mathbb{C}^*\,.$ The sheaf of sections of $\mathfrak{m}$ is naturally isomorphic to the sheaf of $\mathbb{C}$-valued functions on $\mathbb{C}\,.$
\begin{proposition}The Chevalley-Eilenberg complex associated to $\mathbb{C}^*_{\mathbb{C}}(*\{0\})$ is isomorphic to the complex
\begin{align*}
    \mathcal{O}_\mathbb{C}^*(*\{0\})\overset{\mathrm{dlog}}{\to}\Omega^1_\mathbb{C}(\log D)\,.
\end{align*}
\end{proposition}
\begin{proof}We will compute $\mathrm{d}_\text{CE}\,\mathrm{log}:$ consider the meromorphic function $x^nf(x)\,,$ $x\in U\,,$ where $f$ is holomorphic and nonvanishing. We identify it with the local section of $\mathbb{C}^*_{\mathbb{C}}(*\{0\})$ given by $s(x)=(x,f(x),n)\,.$ Now the anchor map is given by \begin{align*}
    \alpha:\textit{Lie}(\mathbb{C}^*\ltimes\mathbb{C})\to T\mathbb{C}\,,\,\alpha(\partial_x,x)=x\partial_x\,.
    \end{align*}
Then we can compute that
\begin{align*}
&\tilde{L}_{(\partial_x,x)}s(x)=\frac{d}{d\varepsilon}\Big\vert_{\varepsilon=0}\,(x,e^{n\varepsilon}f(e^{\varepsilon}x)f(x)^{-1},0)
=(x,n+xf'(x)f(x)^{-1},0)
\\&=(x,\mathrm{dlog} (x^nf)\,(x\partial_x),0)=(x,\mathrm{dlog} (x^nf)\,\alpha(\partial_x,x),0)\,,
\end{align*}
so $f$ differentiates to $\mathrm{dlog} (x^nf)\,,$ so that $\mathrm{d}_\text{CE}\,\mathrm{log}$ corresponds to $\mathrm{dlog}$ under the identification of sheaves used in Proposition~\ref{sheaf identification}. This completes the proof.
\end{proof}
\subsection{The Module \texorpdfstring{$\mathbb{C}^*_X(*D)$ and Pair$(X,D)$}{the module}}
Here we will generalize the construction in the previous section to arbitrary complex manifolds and smooth divisors.
\begin{proposition}
Let $X$ be a complex manifold of complex dimension $n\,,$ and let $D$ be a smooth divisor. Then there is a canonical family of abelian groups $\mathbb{C}^*_X(*D)\to X$ such that $\mathcal{O}(\mathbb{C}^*_X(*D))$ (where sections here are taken to be holomorphic) is isomorphic to $\mathcal{O}^*(*D)\,,$ the sheaf of nonvanishing meromorphic  functions with poles or zeros only on $D\,.$
\end{proposition} 
\begin{proof}We can construct a family of abelian groups as follows: choose an open cover $\{\mathbb{D}^n_i\}_i$ of $X$ by polydiscs $($ie.  $\mathbb{D}_i=\{z\in\mathbb{C}:|z|<1\})\,,$ with coordinates $(x_{i,1},\mathbf{x}_i)=(x_{i,1}\,,x_{i,2}\,,\ldots\,,x_{i,n})$ on $\mathbb{D}^n_i\,,$ in such a way that 
\begin{align*}D\cap \mathbb{D}^n_i=\{x_{i,1}=0\}\,.
\end{align*}
Then on $\mathbb{D}^n_i$ form the family of abelian groups $\mathbb{D}^n_i{\mathop{\times}}\mathbb{C}^*{\mathop{\times}}\mathbb{Z}/\sim\,,$ where \begin{align*}
(x_{i,1},\mathbf{x}_i,y,k)\sim (x_{i,1},\mathbf{x}_i,x_{i,1}^{-l}y,k+l) \textit{  for } x_{i,1}\ne 0\,,
\end{align*}
where the surjective submersion is given by the projection onto $(x_{i,1},\mathbf{x}_i)\,,$ and where the product is given by 
\begin{align*}
    (x_{i,1},\mathbf{x}_i,y,k)\cdot (x_{i,1},\mathbf{x}_i,y',l)=(\mathbf{x}_i,yy',k+l)\,.
\end{align*}
We can glue these families of abelian groups together in the following way: on $\mathbb{D}^n_i\cap\mathbb{D}^n_j$ we have a nonvanishing holomorphic function $g_{ij}$ satisfying $x_{j,1}=g_{ij}x_{i,1}\,.$ Now let 
\begin{align*}
(x_{i,1},\mathbf{x}_i,y,k)\sim(x_{j,1},\mathbf{x}_j,g_{ij}^{-k}y,k)\,.
\end{align*}
This gluing preserves the fiberwise group structure, hence we obtain a family of abelian groups, denoted 
\begin{align*}
    \mathbb{C}^*_X(*D)\overset{\pi}{\to}X\,.
\end{align*}
As in the previous section, where this was done for $(X,D)=(\mathbb{C},\{0\})\,,$ the sheaf $\mathcal{O}(\mathbb{C}^*_X(*D))$ is isomorphic to $\mathcal{O}^*(*D)\,.$
\end{proof}
\begin{proposition}[see~\cite{pym}] There is a terminal integration of $T_X(-\log{} D)$ (denoted by $\text{Pair}(X,D))\,,$ the Lie algebroid whose sheaf of
sections is isomorphic to the sheaf of sections of $T_X$ which are tangent to $D\,.$ 
\end{proposition}
\begin{proof}The terminal integration, $\text{Pair}(X,D)\,,$ can be described locally as follows (here the notation is as in the previous proposition):
the set of morphisms $\mathbb{D}^n_i\to\mathbb{D}^n_j$ is given by all
\begin{align*}
  &(a,\mathbf{x}_j,x_{i,1},\mathbf{x}_i)\in\mathbb{C}^*{\mathop{\times}}\mathbb{D}_j^{n-1}{\mathop{\times}}\mathbb{D}_i{\mathop{\times}}\mathbb{D}_i^{n-1}
  \\&\text{such that } (ax_{i,1},\mathbf{x}_j)\in \mathbb{D}^n_j\,.
\end{align*}
The source, target and multiplication maps are:
\begin{align*}
    &s(a,\mathbf{x}_j,x_{i,1}\,,\mathbf{x}_i)=(x_{i,1}\,,\mathbf{x}_i)\in \mathbb{D}_i^{n}\,,
    \\&t(a,\mathbf{x}_j,x_{i,1}\,,\mathbf{x}_i)=(ax_{i,1},\mathbf{x}_j)\in\mathbb{D}_j^{n}\,,
    \\&(a',\mathbf{x}_k,ax_{i,1},\mathbf{x}_j)\cdot(a,\mathbf{x}_j,x_{i,1}\,,\mathbf{x}_i)
    \\&=(a'a,\mathbf{x}_k,x_{i,1}\,,\mathbf{x}_i)\in\mathbb{C}^*{\mathop{\times}}\mathbb{D}_k^{n-1}{\mathop{\times}}\mathbb{D}_i{\mathop{\times}}\mathbb{D}_i^{n-1}\,.
\end{align*}
The gluing maps on the groupoid are induced by the gluing maps on $X\,,$ that is, 
\begin{align*}
   &(a,\mathbf{x}_j,x_{i,1}\,,\mathbf{x}_i)\sim \Big(a\frac{g_{jl}(ax_{i,1})}{g_{ik}(x_{i,1})},\mathbf{x}_l,x_{k,1}\,,\mathbf{x}_k\Big) 
   \end{align*}
  if
\begin{align*}
&(x_{i,1},\mathbf{x}_i)\in \mathbb{D}_i^{n}\sim (x_{k,1},\mathbf{x}_k)\in \mathbb{D}_k^{n}\,, 
\\&(ax_{i,1},\mathbf{x}_j)\in \mathbb{D}_j^{n}\sim (x_{l,1},\mathbf{x}_l)\in \mathbb{D}_l^{n}\,.
\end{align*}
\end{proof}
\begin{proposition}The morphism 
\begin{align}\label{TX-module}
\mathrm{dlog}:\mathcal{O}^*(*D)\to \Omega_X^1(\log{} D)
\end{align}
endows $\mathbb{C}^*_X(*D)$ with the structure of a $T_X(-\log D)$-module, and this structure integrates to give $\mathbb{C}^*_X(*D)$ the structure of a $\mathrm{Pair}(X,D)$-module.
\end{proposition}
\begin{proof}Define an action of $\mathrm{Pair}(X,D)$ on $\mathbb{C}^*_X(*D)$ as follows (the notation is as in the previous two propositions):
\begin{align*}
    &(a,\mathbf{x}_j,x_{i,1}\,,\mathbf{x}_i)\cdot(x_{i,1},\mathbf{x}_i,y,k)
    \\&=(ax_{i,1},\mathbf{x}_j,a^{-k}y,k)\,.
\end{align*}
This is a well-defined action by fiberwise isomorphisms, and it indeed differentiates to the $T_X(-\log{}D)$-module defined by
\eqref{TX-module}.
\end{proof}
Essentially the same construction can be done in the case that $D$ is a simple normal crossing divisor. In a neighborhood $U$ of a simple crossing divisor which is biholomorphic to a polydisk, we can choose coordinates $\mathbf{x}=(x_1,\ldots,x_n)$ on $\mathbb{D}^n$ such that
the simple normal crossing divisor is given by $x_1\cdots x_k=0\,.$ Then
\begin{align*}
   & \mathbb{C}^*_X(*D)\vert_{\mathbb{D}^n}=\mathbb{D}^n{\mathop{\times}}\mathbb{C}^*{\mathop{\times}}\mathbb{Z}^k/\sim\,,
    \,(\mathbf{x},y,\mathbf{i})\sim (\mathbf{x},x_{j_1}^{-m_{j_1}}\cdots x_{j_l}^{-m_{j_l}}y,\mathbf{i}+\mathbf{m})
   \\&\text{away from } x_{j_1}\cdots x_{j_l}=0\,, \text{ where } j_1,\ldots,j_l\in \{1,\ldots,k\}
   \\&\text{and where } m_{j_1},\ldots, m_{j_l}\text{  are the nonzero components of }\mathbf{m}\in\mathbb{Z}^k\,.
\end{align*}
Alternatively, it can locally be described as \begin{align*}
    \mathbb{C}^*_X(*{\{x_1=0\}})\otimes_{\mathbb{C}^*}\cdots\otimes_{\mathbb{C}^*}\mathbb{C}^*_X(*{\{x_k=0\}})\,,
    \end{align*}
where the $\mathbb{C}^*$-action is the one induced by the action of $\mathbb{C}^*_U$ on $\mathbb{C}^*_X(*D)\,,$ which comes from the embedding $\mathbb{C}^*_U\xhookrightarrow{}\mathbb{C}^*_X(*D)\,.$
\newline\newline To summarize this section, we have proven the following:
\begin{theorem}
Let $X$ be a complex manifold and let $D$ be a simple normal crossing divisor. There is a family of abelian groups 
\begin{align*}
    \mathbb{C}^*_X(*D)\overset{\pi}{\to}X
\end{align*}
whose sheaf of holomorphic sections is isomorphic to $\mathcal{O}_X^*(*D)\,.$
Furthermore, there is a canonical action of $\mathrm{Pair}(X,D)$ on $M$ making it into a $\mathrm{Pair}(X,D)$-module, and this module structure integrates the canonical $T_X(-\log{}D)$-module structure on $M$
induced by the morphism
\begin{align*}
\mathrm{dlog}:\mathcal{O}_X^*(*D)\to \Omega_X^1(\log{} D)\,.
\end{align*}
\end{theorem}

\section{Integration of Cohomology Classes by Prequantization}
In this section we describe an alternative approach to integration of classes in Lie algebroid cohomology that may sometimes be used, and which doesn't directly involve the van Est map (more accurately, this method could be combined with the previous method). We call it integration by prequantization because in the case that the Lie algebroid is the tangent bundle and one is trying to integrate a $2$-form $\omega$, this method uses the line bundle whose first Chern class is the cohomology class of $\omega\,.$ 
We will first describe this method and then give some examples. 
\\\\Suppose we have a $G$-module $N$ and we are interested in integrating a class in the cohomology of the truncated complex, $\alpha\in H^*_0(\mathfrak{g},N)\,.$ 
Now suppose we have a $G$-module $M$ such that $\mathfrak{m}=\mathfrak{n}\,,$ and such that there is a map $N\to M$ of $G$-modules which differentiates to the identity map on $\mathfrak{n}\,.$ 
In this case
the morphism 
\[
\mathcal{O}(M)\xrightarrow{\text{d}_{\text{CE}}\text{log}}\mathcal{C}^1(\mathfrak{g},M)\cong \mathcal{C}^1(\mathfrak{g},N)
\]
induces a morphism
\begin{align*}
H^*(G^0,\mathcal{O}(M))\to H_0^*(\mathfrak{g},N)\,.    
\end{align*}
Then one can try lift $\alpha$ to a class $\tilde{\alpha}\in H^*(G^0,\mathcal{O}(M))\,.$ If a lift can be found, then one can attempt to integrate $\alpha$ to a class in $H^*_0(G,N)$ by showing that $\delta^*\tilde{\alpha}$ is in the image of the map $H^*_0(G,N)\to H^*_0(G,M)\,.$ If this succeeds then this class in $H^*_0(G,N)$ integrates $\alpha\,.$ We can summarize this method with the following proposition:
\begin{proposition} Let $G$ be a Lie groupoid, and let $N,M$ be $G$-modules with the same underlying Lie algebroids $\mathfrak{n}\,.$ Suppose further that there is a map of $G$-modules $f:N\to M$ which differentiates to the identity on $\mathfrak{n}$ (in particular this means that the Lie algebroids of $N$ and $M$ are the same as $G$-representations). The following diagram is commutative:
\[
    \begin{tikzcd}[row sep=3em, column sep = 3em]
 & H^*_0(G,M) \arrow[bend left=60,"VE_0"]{dd}
 \\ H^*_0(G,N)\arrow{ur}{f}\arrow{d}{VE_0}& H^*(G^0,\mathcal{O}(M))\arrow{u}{\delta^*}\arrow{d}{\mathrm{d}_{\mathrm{CE}}\mathrm{log}} 
 \\H^*_0(\mathfrak{g},N)\arrow{r}{\cong}& H^*_0(\mathfrak{g},M)
\end{tikzcd}
\]
\end{proposition}
\theoremstyle{definition}\begin{exmp}
Let $X$ be a manifold and let $\omega$ be a closed $2$-form which has integral periods. Then there is a class $g\in H^1(X,\mathcal{O}^*)$ which lifts $\omega\,,$ ie. a principal $\mathbb{C}^*$-bundle. We then have that $\delta^*g\in H^1_0(\textrm{Pair}(X),\mathbb{C}^*_X)$ integrates $\omega\,.$ 
\end{exmp}
\theoremstyle{definition}\begin{exmp}
Consider the trivial $(\mathbb{C}^*\ltimes\mathbb{C}\rightrightarrows \mathbb{C})$-module $\mathbb{C}^*_{\mathbb{C}}\,,$ and let $\mathfrak{g}$ be its Lie algebroid. Consider the class in $H^0_0(\mathfrak{g},\mathbb{C}^*_{\mathbb{C}})$ given by $\frac{dz}{z}\,.$ This class is not in the image of 
\[
\mathcal{O}^*_{\mathbb{C}}\xrightarrow{\mathrm{d}\mathrm{log}}\mathcal{C}^1(\mathfrak{g},\mathbb{C}^*_{\mathbb{C}})\,.
\]
However, $\mathbb{C}^*_{\mathbb{C}}\xhookrightarrow{} \mathbb{C}^*_{\mathbb{C}}(*\{0\})$ $($where $\mathbb{C}^*_{\mathbb{C}}(*\{0\})$ is as in the previous section$)\,,$ and they have the same Lie algebroids, and in addition the class $\frac{dz}{z}$ is in the image of 
\[
\mathcal{O}^*_{\mathbb{C}}(*D)\xrightarrow{\mathrm{d}_{CE}\mathrm{log}}\mathcal{C}^1(\mathfrak{g},\mathbb{C}^*_{\mathbb{C}})\,,
\]
namely $\mathrm{d}\mathrm{log}\,z=\frac{dz}{z}\,.$ We then have that $\delta^*z\,(a,z)=a\,,$ which is $\mathbb{C}^*$-valued. Hence the morphism $(a,z)\mapsto a$ integrates $\frac{dz}{z}\,.$
\end{exmp}
To get examples involving the integration of extensions, we have the following proposition:
\begin{proposition}
Let $X$ be a complex manifold with smooth divisor $D\,,$ and let $\Pi_1(X,D)\rightrightarrows X$ be the source simply connected integration of $T_X(-\log{D})\,.$ Then the subgroup of classes in $H^1_0(T_X(-\log{D}),\mathbb{C}_{X})$ which are integral on $X\backslash D$ embeds into $H^1_0(\Pi_1(X,D),\mathbb{C}^*_{X})\,.$
\end{proposition}
\begin{proof}
Let $\omega\in H^1_0(T_X(-\log{D}),\mathbb{C}_X)$ be a class which is prequantizable, which means that it is in the image of the map 
\[
H^1(X,\mathcal{O}^*_X(*D))\to H^1_0(T_X(-\log{D}),\mathbb{C}_X)\,.
\]
It is proved in~\cite{luk} that this is equivalent to $\omega$ being integral on $X\backslash D\,.$ 
\\\\There is a short exact sequence of $\Pi_1(X,D)$-modules
\begin{align*}
    0\to \mathbb{C}^*_X\overset{\iota}{\to} \mathbb{C}^*_X(*D)\overset{\pi}{\to} \mathrm{\acute{e}t}(\iota_*\mathcal{O}(\mathbb{Z}_D))\to 0
\end{align*}
$($where $\iota:D\to X$ is the inclusion and \'{e}t means the \'{e}tal\'{e} space, which may be non-Hausdorff, but this is fine $)\,.$ From this we get the long exact sequence 
\begin{align*}
&H^0_0(\Pi_1(X,D),\mathrm{\acute{e}t}(\iota_*\mathcal{O}(\mathbb{Z}_D)))\to H^1_0(\Pi_1(X,D),\mathbb{C}^*_X)\to H^1_0(\Pi_1(X,D),\mathbb{C}^*_X(*D))
\\&\to H^1_0(\Pi_1(X,D),\mathrm{\acute{e}t}(\iota_*\mathcal{O}(\mathbb{Z}_D)))\,.    
\end{align*}
Now $H^0_0(\Pi_1(X,D),\mathrm{\acute{e}t}(\iota_*\mathcal{O}(\mathbb{Z}_D)))=0$ since a morphism of groupoids must be $0$ on the identity bisection, so since the fibers of $\mathrm{\acute{e}t}(\iota_*\mathcal{O}(\mathbb{Z}_D))$ are discrete and the source fibers $\Pi_1(X,D)$ are connected, any such morphism must be identically $0\,.$ So we get the long exact sequence
\begin{align*}
  0  \to H^1_0(\Pi_1(X,D),\mathbb{C}^*_X)\to H^1_0(\Pi_1(X,D),\mathbb{C}^*_X(*D))\to H^1_0(\Pi_1(X,D),\mathrm{\acute{e}t}(\iota_*\mathcal{O}(\mathbb{Z}_D)))\,.
\end{align*}
If we let $\alpha\in H^1(X,\mathbb{C}_X^*(*D))\,,$ then $t^*\alpha-s^*\alpha\in H^1_0(\Pi_1(X,D),\mathbb{C}_X^*(*D))\,,$ and 
\begin{align*}
\pi(t^*\alpha-s^*\alpha)=t^*\pi(\alpha)-s^*\pi(\alpha)=0\,,
\end{align*}
where the latter equality follows from the fact that $\pi(\alpha)$ is a module for the full subgroupoid over $D\,,$ which follows from the following: there is a morphism from the full subgroupoid over $D$ to $\Pi_1(D)\,,$ and $\pi(\alpha)$ is a module for $\Pi_1(D)$ since $\pi(\alpha)$ is a local system. 
\\\\Hence there is a unique lift of $\alpha$ to $H^1_0(\Pi_1(X,D),\mathbb{C}^*_X)\,.$ Hence all of the prequantizable classes in $H^2(T_X(-\log{D}),\mathbb{C}_X)$ integrate to classes in $H^1_0(\Pi_1(X,D),\mathbb{C}^*_X)\,,$
\end{proof}
What this proposition means is that any closed logarithmic $2$-form on a complex manifold $X$ with smooth divisor $D\,,$ which has integral periods on $X\backslash D\,,$ defines a $\mathbb{C}^*$-groupoid extension of $\Pi_1(X,D)\,.$
\theoremstyle{definition}\begin{exmp}We can specialize to the case $X=\mathbb{P}^2$ and where $D$ is a smooth projective curve of degree $\ge 3$ and genus $g$ in $\mathbb{P}^2\,.$ Then as proved in~\cite{luk}, the prequantizable subgroup of $H^1_0(T_{\mathbb{P}^2}(-\log{D}),\mathbb{C}^*_X)$ is isomorphic to $\mathbb{Z}^{2g}\,.$ Hence $\mathbb{Z}^{2g}\xhookrightarrow{}H^1_0(\Pi_1(\mathbb{P}^2,D),\mathbb{C}^*_{\mathbb{P}^2})\,.$
\end{exmp}
\part{Van Est Theory on Geometric Stacks}
\chapter{The (2,1)-Category of Lie Groupoids and Stacks}
Here we will briefly describe the (2,1)-category of Lie groupoids and the (2,1)-category of geoemtric stacks. 
\section{(2,1)-Category of Lie groupoids} Let's start with the (2,1)-category of Lie groupoids:
\begin{itemize}
    \item The objects are Lie groupoids
    \item the morphisms are homomorphisms of Lie groupoids
    \item The 2-morphisms are smooth (holomorphic) natural transformations
    \item The weak equivalences are Morita maps
\end{itemize}
There are two notions of fiber products: one coming from the categroy of Lie groupoids, which we will call the strong fiber product, and the other coming from the (2,1)-category of Lie groupoids, which we will call the fiber product. They are defined as so:
\begin{definition}
Given two homomorphisms of Lie groupoids $f_1:H\to G\,,f_2:K\to G\,,$ we get a third groupoid by taking the fiber product at the level of objects and morphisms: 
\begin{equation}
H^{(1)}\times_{G^{(1)}}K^{(1)}\rightrightarrows H^0\times_{G^0}K^0\,.
\end{equation}
If the resulting groupoid is a Lie groupoid, we call it the strong fiber product (or strong pullback), denoted $H\times_{G!}K$ or $f_2^!H\,.$ This is in particular the case if the maps at the level of objects and arrows are transversal.
\end{definition}
Now given a morphism of Lie groupoids $f:H\to G$ and an object $g^0\in G^0\,,$ we call $f^{-1}(g^0)$ the kernel of $f$ over $g^0$ (assuming this kernel exists). Thinking of $\{g^0\}\rightrightarrows \{g^0\}$ as the trivial Lie groupoid, it comes with a natural map into $G\,,$ and the kernel of $f$ over $g^0$ is equivalently given by the strong fiber product $H\times_{G!}\{g^0\}\,.$ We make the following definition:
\begin{definition}
Given a map $f:H\to G$ of Lie groupoids and an object $g^0\in G^0\,,$ the kernel of $f$ over $g^0\,,$ if it exists, is given by $f^{-1}(g^0)\,,$ or equivalently it is given by $H\times_{G!}\{g^0\}\,.$ 
\end{definition}
Now the second definition of fiber product, which is the main one we will be using and is the one that has the right universal property in the (2,1)-category of Lie groupoids, is given by the following:
\begin{definition}
Given two homomorphisms of Lie groupoids $f_1:H\to G\,,f_2:K\to G\,,$ we get a third groupoid as so (see \cite{Moerdijk}):
\begin{itemize}
    \item The objects are triples $(h^0,g,k^0)\in H^0\times G^{(1)}\times K^0$ where $g$ is an arrow $f_1(h^0)\to f_2(k^0)\,.$
\item An arrow between the objects $(h^0,g,k^0)\to (h'^{ 0},g',k'^{ 0})$ is given by a pair $(h,k)\in H^{(1)}\times K^{(1)}$ such that $h\,,k$ are arrows from $h^0\to h'^0\,,k^0\to k'^0\,,$ respectively, such that $g'\,f_1(h)=f_2(k)\,g\,.$
\end{itemize}
If this groupoid is a Lie groupoid, it will be called the fiber product (or pullback), and denoted $H\times_G K$ or $f_1^*K\,.$ This will be a Lie groupoid as long as the space of objects is a manifold. This is in particular the case if $t\circ p_2:H^0\times_{G^0}G^{(1)}\to G^0$ is submersion, where $p_2:H^0\times_{G^0}G^{(1)}\to G^{(1)}$ is the projection onto the second factor.
\end{definition}
Now given a morphism $f:H\to G$ of Lie groupoids, we can ask what the fibers of the map are. Fibers exist only over objects in $G^0\,,$ and the fiber over an object $g^0\xhookrightarrow{} G\,.$
\begin{definition}
Let $f:H\to G$ be a map of Lie groupoids and let $g^0\xhookrightarrow{} G$ be an object in $G^0\,.$ We can consider the trivial Lie groupoid $\{g^0\}\rightrightarrows \{g^0\}\,,$ and this comes with a morphism into $G\,.$ Assuming the fiber product $H\times_G \{g^0\}$ exists, we call it the fiber of $f$ over $g^0\,.$
\end{definition}
\subsection{Computing Fiber Products}
Here we will collect some basic results about fibers and fiber products:
\begin{exmp}
If $f:H\to G$ is a homomorphism of Lie groups, then there is only one object, hence only one fiber, and it is given by $H\ltimes G\rightrightarrows G\,.$ If $H\xhookrightarrow{}G\,,$ then the fiber is Morita equivalent to $G/H\,.$ In particular, if $H\xhookrightarrow{}G$ is the maximal compact subgroup, then the fiber is contractible. 
\end{exmp}
\begin{exmp}
If $f:Y\to X$ is a map of smooth manifolds, thought of as groupoids, then the fibers are just the fibers as maps between manifolds (if it exists as a smooth manifold).
\end{exmp}
\begin{proposition}
If $f:H\to G$ is a Morita map, then the fibers are all pair groupoids, hence the fibers are all Morita equivalent to a point.
\end{proposition}
\begin{exmp}
If $f:G\to X$ is a map from a groupoid to a manifold, then the fiber (if it exists) over a point $x\in X$ is just the kernel $f^{-1}(x)$ 
 \end{exmp}
\begin{exmp}
If $f:X\to G$ is a map from a manifold to a groupoid, then the fiber over a point $g^0\in G^0$ is the manifold $X\sideset{_{f}}{_{s}}{\mathop{\times}} t^{-1}(g^0)\,.$ In particular, if $X=G^0\,,$ then the fibers are just the target fibers.
\end{exmp}
\begin{proposition}\label{morita of fibers}
Given morphisms of Lie groupoids $f_1:H\to G\,,f_2:K\to G\,,$ there is a canonical morphism of Lie groupoids $H\times_{G!}K\xhookrightarrow{}H\times_G K\,,$ assuming they exist. In particular, given a morphism $f:H\to G$ and an object $g^0\in G^0\,,$ there is a natural inclusion $f^{-1}(g^0)\xhookrightarrow{}H\times_G\{g^0\}\,.$ In \Cref{equiv1}, \Cref{equiv2}, we give conditions on which these inclusions are Morita equivalences.
\end{proposition}
\begin{proposition}
Suppose $F:H\to G$ is a Lie groupoid homomorphism and $g$ is an arrow $g^0\to g'^0\,.$ Then if the fiber over $g^0$ exists, then so does the fiber over $g'^0\,,$ and they are isomorphic.
\begin{proof}
At the level of objects, the isomorphism is given by $(h^0,g')\to (h^0,gg')\,.$ The morphisms are already naturally identified.
\end{proof}
\end{proposition}
\begin{corollary}
Suppose $F:H\to G$ is a Lie groupoid homomorphism such that $G$ is a transitive groupoid. Then if one fiber of $F$ exists, they all exist and are all isomorphic.
\end{corollary}
\begin{exmp}
If $f:H\to G$ is a homomorphism, then, $G^0\times_{G} H=H\ltimes P\rightrightarrows P\,,$ where $P$ is the bibundle associated to the morphism $f\,.$ We can also consider the map $f\vert_{H^0}:H^0\to G^0\xhookrightarrow{}G\,,$ where we consider $H^0\rightrightarrows H^0\,, G^0\rightrightarrows G^0$ to be the trivial Lie groupoids, and the fiber product $G^0\times_{G}H^0=P\,.$
\end{exmp}
\begin{exmp}
Suppose we have generalized morphisms $P_1:G\to H\,,\,P_2:H\to K\,.$ We then have two action groupoids given by $P_1\rtimes H\,,\,H\ltimes P_2$ with correponding morphisms into $H\,.$ We can form the fiber product 
\begin{equation}
 H\ltimes P_2\times_H P_1\rtimes H\,,
\end{equation}
and this groupoid is Morita equivalent to $(P_2\times_{H^0} P_1)/H\,,$ hence can be identified with the composition $P_2\circ P_1\,.$
\end{exmp}

\section{(2,1)-Category of Stacks}
Now, to get the (2,1)-category of geometric stacks we localize the (2,1)-category of Lie groupoids at the weak equivalences. We can describe the category as follows:
\begin{itemize}
  \item The objects are Lie groupoids
    \item the morphisms are anafunctors
    \item The 2-morphisms are ananatural transformations
\end{itemize}
In the above definition, an anafunctor $H\to G$ is a generalized morphism given by a roof (see \cite{Li}):
\begin{equation}
    \begin{tikzcd}
  & K \arrow[ld, "\cong"'] \arrow[rd] &   \\
H &                                   & G
\end{tikzcd}
\end{equation}
Here the left leg is a Morita equivalence and the map $K^0\to H^0$ is a surjective submersion (in fact, we can even take the map $K\to H$ to be a fibration).
We can compose anafunctors via the strong fiber product:
\begin{equation}
    \begin{tikzcd}
  &                                   & K\times_{H!} K' \arrow[ld, "\cong"'] \arrow[rd] &                                    &   \\
  & K \arrow[ld, "\cong"'] \arrow[rd] &                                              & K' \arrow[ld, "\cong"'] \arrow[rd] &   \\
I &                                   & H                                            &                                    & G
\end{tikzcd}
\end{equation}
\vspace{3mm}\\ An ananatural  transformation between two anafunctors $H\xleftarrow[]{\cong}K\to G\Rightarrow H\xleftarrow[]{\cong}K'\to G$ is given by a natural transformation between the composite functors $K\times _G K'\to K\to G$ and $K\times_G K'\to K'\,.$ 
One can also compose ananatural transformations, but we won't define it here (see \cite{Li} for more).
 \chapter{Double Lie Groupoids and LA-Groupoids}
 In this chapter we will present the information about double Lie groupoids and LA-groupoids which is relevant to the theorem we wish to prove. Essentially, a double Lie groupoid is a Lie groupoid internal to the category of Lie groupoids, ie. the space of arrows and the space of objects are both Lie groupoids. An LA-groupoid is essentially a Lie algebroid internal to the category of Lie groupoids, ie. the total space and the base space of the Lie algebroid are Lie groupoids. There is a differentiation functor from double Lie groupoids to LA-groupoids.
 \begin{definition}(see page 5 of~\cite{mehtatang} for complete details)
 A double groupoid is a groupoid internal to the category of groupoids. We will denote it as so 
 \begin{equation}\label{double groupoid}
     \begin{tikzcd}
G^{01} \arrow[r, shift left] \arrow[r, shift right] \arrow[d, shift left] \arrow[d, shift right] & G^{11} \arrow[d, shift left] \arrow[d, shift right] \\
G^{00} \arrow[r, shift left] \arrow[r, shift right]                                              & G^{10}                                             
\end{tikzcd}
 \end{equation}
 We will denote the source and target maps from $G^{ij}\to G^{kl}$ by $s_{ij,kl},\,t_{ij,kl}\,,$ respectively.
 \end{definition}
 \begin{definition}
 A double Lie groupoid is a double groupoid such that all rows and columns are Lie groupoids, and such that the double source map 
 \begin{equation}
     (s_{01,00},\,s_{01,11}):G^{01}\to G^{00}\sideset{_{s_{00,10}}}{_{s_{11,10}}}{\mathop{\times}} G^{11}
     \end{equation}
 is a surjective submersion.
 \end{definition}
 \vspace{3mm}Now we can describe the infinitesimal analogue of a double Lie groupoid, in the vertical direction.
 \begin{definition}
 An LA-groupoid (short for Lie algebroid-groupoid), denoted as \ref{LA-groupoid def} is a Lie algebroid internal to the category of Lie groupoids. That is, the top and bottom rows are Lie groupoids, the left and right columns are Lie algebroids, all structure maps are compatible, and such that the map (to be defined below)
 \begin{equation}
     (p_1\,,s_A):A^1\to M^1\sideset{_{s_M}}{_{p_0}}{\mathop{\times}}M^0
 \end{equation}
 is a surjective submersion. Here, $p_1\,,p_2$ are the projection maps  $A^1\to M^1\,,A^2\to M^2\,,$ respectively, and $s_A\,,s_M$ are the source maps $A^1\to A^0\,,M^1\to M^0\,,$ respectively. 
 \begin{equation}\label{LA-groupoid def}
     \begin{tikzcd}
A^1 \arrow[r, shift left] \arrow[r, shift right] \arrow[d] & A^0 \arrow[d] \\
M^1 \arrow[r, shift left] \arrow[r, shift right]           & M^0          
\end{tikzcd}
 \end{equation}
 \end{definition}
 \section{Category of Double Lie groupoids and LA-Groupoids}\label{defLAM}
 Here we define morphisms of double Lie groupoids, LA-groupoids, and Morita equivalences (compare with~\cite{del hoyo}).
 \begin{definition}\label{mor double}
 A morphism $f$ between double Lie groupoids consists of four functions, \begin{equation}
 f^{00}\,,f^{10}\,,f^{01}\,,f^{11}\,,
 \end{equation}
 where $f^{ij}$ maps the $ij$ corner to the $ij$ corner, for which the corresponding maps of Lie groupoids are all morphisms.
 \end{definition}
 \begin{definition}\label{mordo}
 A Morita map of double Lie groupoids is a morphism of double Lie groupoids for which the morphism between the top rows (or left columns) is a Morita equivalence.
 \end{definition}
 \begin{definition}
 A morphism of LA-groupoids consists of four functions, $f^{00}\,,f^{10}\,,f^{01}\,,f^{11}\,,$ where $f^{ij}$ maps the $ij$ corner to the $ij$ corner, for which the corresponding maps of Lie groupoids and Lie algebroids are all morphisms. 
 \end{definition}
 \begin{definition}
  A Morita map of LA-groupoids is a morphism of LA-groupoids for which the morphism between the top rows is a Morita equivalence.
 \end{definition}
 \begin{remark}\label{generate}
The definition of Morita map of double Lie groupoids we've given is not quite the definition we should give. This category comes with two notions of weak equivalence, one in the horizontal direction and one in the vertical direction. Really, we should take the smallest subcategory containing all of these weak equivalences to get a category with weak equivalences, or better, a homotopical category. Alternatively, there may be a nicer definition.
 \end{remark}
 \section{Higher Category of Double Lie Groupoids}
 Now since cofibrations will be mentioned several times in this thesis, we wish to define 2-morphisms between morphisms of double Lie groupoids. We will do this with respect to the top rows, however, by taking the transpose of the diagram we get the definition with respect to the left columns.
\begin{definition}
Consider morphisms $f_1\,,f_2$ between double Lie groupoids
\begin{equation}
   \begin{tikzcd}
H^{01} \arrow[d, shift left] \arrow[d, shift right] \arrow[r, shift left] \arrow[r, shift right] & H^{11} \arrow[rr, "f_2", shift right=5] \arrow[d, shift left] \arrow[d, shift right] &  & G^{01} \arrow[d, shift right] \arrow[r, shift left] \arrow[d, shift left] \arrow[r, shift right] & G^{11} \arrow[d, shift right] \arrow[d, shift left] \\
H^{00} \arrow[r, shift left] \arrow[r, shift right]                                              & H^{10} \arrow[rr, "f_1"', shift left=6]                                              &  & G^{00} \arrow[r, shift right] \arrow[r, shift left]                                              & G^{10}                                             
\end{tikzcd}
\end{equation}
A 2-morphism $f_1\Rightarrow f_2$ is given by a functor 
\begin{equation}
    \begin{tikzcd}
                                                    &                   & G^{01} \arrow[r, shift right] \arrow[r, shift left] & G^{11} \\
H^{00} \arrow[r, shift left] \arrow[r, shift right] & H^{10} \arrow[ru] &                                                     &       
\end{tikzcd}
\end{equation}
for which the induced map $H^{00}\to G^{01}$ defines a natural transformation between $f_1\,,f_2$ when restricted to the groupoids in the left column.
\end{definition}
Now we have a (2,1)-category of double Lie groupoids, and we may now discuss cofibration in the category, which we will do later. In addition, one can invert weak equivalences to obtain a new category, analogous to what is done for groupoids, but we won't be needing that.
\begin{remark}Note that because 2-morphisms are functors, we also have morphisms of 2-morphisms, ie. 3-morphisms. Therefore we really have a (3,1)-category.\end{remark}
\chapter{Replacing a Map With a Fibration/Cofibration}
In the introdution, we mentioned obtaining equivalent LA-groupoids associated to a map using two different methods. The two methods of obtaining LA-groupoids given a (nice enough) map $H\to G$ correspond to the two methods of constructing ``groupoids" out of such a map, which we will call the fibration and cofibration replacements. By analogy one should think back to homotopy theory, where one can replace a map $Y\to X$ with a fibration, namely the mapping path space, or a cofibration, namely the mapping cylinder. One should keep analogies with homotopy theory in mind when reading this chapter (if one thinks of homotopy theory as really being about $\infty$-groupoids, these are more than analogies).
\section{Fibration and Cofibrations}
The context here is that we are thinking about the (2,1)-category of Lie groupoids (the one where no localization has been performed). Given any 2-category there is a notion of fibration (see \cite{str}, \cite{riehl}), and the first thing we will do here is define a fibration of Lie groupoids. There are several notions of ``fibrations" of Lie groupoids in the literature, but as far as the author can tell they are distinct from the one we are about to define, which should be thought of as analogous to Hurewicz fibrations (whereas, for example, Kan fibrations are analogous to Serre fibrations). Some fibrations currently defined in the literature are, in particular, what we call quasifibrations (ie. see \cite{Mackenzie}).
\vspace{3mm}\\Let us expound on fibrations for a moment. A morphism of $\infty$-groupoids (modelled by simplicial spaces) is a Kan fibration if it has the ``homotopy lifting property" with respect to the standard $n$-simplices. For a map of groupoids $f:H\to G$ (ie. the groupoids have no higher morphisms) this is equivalent to $f$ having the ``homotopy lifting property" with respect to the standard $0$-simplex, meaning that if $h^0\in H^0\,,g\in G$ are such that $s(g)=f(h^0)\,,$ then there must exists an $h\in H$ such that $s(h)=h^0$ and $f(h)=g\,.$  Mackenzie (see~\cite{Mackenzie}, definition 2.4.3) requires the following stronger condition to hold: $f_0:H^0\to G^0$ and $H\to f_0^!G$ must be surjective submersions. 
\vspace{3mm}\\The fibrations we discuss will have the homotopy lifting property with respect to all Lie groupoids. In particular, they will be Kan fibrations, and in the important examples we will consder they will be fibrations in the sense of Mackenzie. The two things that could prevent a fibration in the sense of this paper from being a fibration in the sense of Mackenzie are: 
\begin{itemize}
    \item the technical condition of certain maps being submersions,
    \item the definition of a Hurewicz fibration $X\to Y$ from topology doesn't imply that the map is surjective (which could happen if $Y$ isn't path connected), and analogously our definition of fibrations doesn't imply that $f_0$ is surjective.
\end{itemize}
\begin{remark}
A remark on conventions and notation: in this thesis, when we speak of a 2-morphism (or natural transformation) of a morphism $f\,,$ we mean a 2-morphism/natural transformatio $f\Rightarrow f'\,,$ where $f'$ is some morphism. We may not explicitly write $f'\,.$ This is justified by the following: a natural transformation between morphisms of groupoids $f,f':H\to G$ is determined by a map $g_H:H^0\to G^{(1)}$ which satisfies $s(g_H(h^0))=f(h^0)\,,t(g_H(h^0))=f'(h^0)$ and which satisfies the desired commuation relations. Conversely, since arrows in a Lie groupoid are invertible, a morphism $f:H\to G$ and a map $g_H:H^0\to G^{(1)}$ satisfying $s(g_H(h^0))=f(h^0)$ determines a morphism $f':H\to G$ and a 2-morphism $f\Rightarrow f'\,.$ As for notation, we may denote a natural transforamtion of a morphism $H\to G$ by using the notation $g_H:H^0\to G^{(1)}$ (here the $g$ in $g_H$ references the arrows in codomain $G\,,$ and the $H$ references the domain). In addition, objects will be denoted with a superscript $0\,,$ ie. an object of $G$ will be denoted by $g^0$ (arrows will be denoted by $g\,,$ including identity arrows).
\end{remark}
\begin{definition}
A morphism $F:E\to G$ of Lie groupoids is a fibration if it has the lifitng property with respect to 2-morphisms. That is, $F$ is a fibration if the following condition holds: let $f:H\to G$ be a morphism of Lie groupoids and let $g_H:H^0\to G^{(1)}$ define a natural transformation of $f\,.$ Suppose there exists a lift $\tilde{f}:H\to E\,,$ ie. $f=F\circ\tilde{f}\,,$ then there must exist a lift of the natural transformation $g\,,$ ie. a map $e_H:H^0\to E^{(1)}$ satisfying $g_H=F\circ e_H$ which defines a natural transformation of $f\,.$
\end{definition}
\begin{definition}
A morphism $\iota:A\to G$ is a cofibration if it has the the extension property with respect to 2-morphisms, ie. suppose we have a morphism $f:A\to H$ together with a natural transformation of $f\,,$ given by $h_A:A^0\to H^{(1)}\,.$ If there exists a map $\tilde{f}:G\to H$ satisfying $f=\tilde{f}\iota\,,$ there must be a natural transformation of $\tilde{f}\,,$ given by a map $h_G:G^0\to H^{(1)}\,,$ satisfying $h_A=h_G\iota\,.$
\end{definition}
The previous definitions raise the following question: given a morphism of Lie groupoids, when can one replace it with an equivalent fibration/cofibration? The answer to the former is always. On the other hand, the author believes that a morphism which isn't already a cofibration seldom has a cofibration replacement (though we will see later that if we use double groupoids they exist far more often).
\section{The Canonical Fibration Replacement}
Now to explain how to replace a morphism with a fibration (see~\cite{hoyor} section 2.2,~\cite{fernandes} page 6 for discussions about the same canonical factorization). Given a map $H\to G$ there is a canonical bibundle $P=G^{(1)}\sideset{_t}{_f}{\mathop{\times}} H^0$ for $H$ and $G\,,$ where $H$ acts via a left action and $G$ acts via a right action, forming the groupoid $H\ltimes P\rtimes G\rightrightarrows P\,.$ From this, we get a canonical fibration replacement for the map $H\to G\,,$ given by the following commutative diagram:
\begin{equation}\label{fibration rep}
\begin{tikzcd}
H\ltimes P\rtimes G \arrow[rd, "p_3"] &   \\
H \arrow[u, "\iota"] \arrow[r, "A"']  & G
\end{tikzcd}                                        
\end{equation}
The map $p_3$ (projection onto the third factor) is a fibration and $\iota$ is a cofibration. In addition, letting 
\begin{equation*}
p_1:H\ltimes P \rtimes G\to H
\end{equation*}
be the projection onto the first factor, we have that $p_1$ is a retraction of $\iota\,,$ and there is a 2-morphism $c:\iota p_1\Rightarrow \mathbbm{1}_{H\ltimes P\rtimes G}$ such that for $h^0\in H^0\,,$ $c\iota(h^0)=\mathbbm{1}_{\iota(h^0)}\,.$  In particular, $\iota$ is a Morita map. The groupoid $H\ltimes P\rtimes G$ is isomorphic to $G\times_G H$ (and as we will see, $H\ltimes P\rtimes G$ naturally has the structure of a double Lie groupoid).
\begin{definition}
Let $f:H\to G$ be a morphism of Lie groupoids. We will say a fibration $F:E\to G$ is a fibration replacement for $f$ if the following conditions hold: there are maps $\iota:H\to E\,,$ $p:E\to H$ such that there exists 2-morphisms $p\iota\Rightarrow \mathbbm{1}_H\,,\iota p\Rightarrow \mathbbm{1}_E\,,$ and such that $f=F\iota\,.$
\begin{remark}
Rather than requiring that there are Morita maps both ways in the definition of fibration, one might require instead that there is a map just one way, as in the definition of fibration in a model category — but we won't be needing to do this. 
\end{remark}
\end{definition}
The discussion above proves the following:
\begin{proposition}
Given any morphism $f:H\to G$ of Lie groupoids, there is a fibration replacement for $f\,.$\footnote{If we allow the space of arrows to be non-Hausdorff, then we must allow the base to be non-Hausdorff as well, otherwise the fibration replacement may not exist in the category. We will assume everything is Hausdorff.}
\end{proposition}
Now given morphisms $H\to G\,, H'\to G'\,,$ we will call them equivalent if there is a diagram of one of the following forms, where the vertical arrows are Morita maps:
\begin{equation}\label{morita maps}
   \begin{tikzcd}
H' \arrow[r]                                                            & G'           & H' \arrow[r] \arrow[d] \arrow[r, Rightarrow, bend right=74, shift right=2] & G'           \\
H \arrow[r] \arrow[u] \arrow[ru, Rightarrow]                            & G \arrow[u]  & H \arrow[r]                                                                & G \arrow[u]  \\
H' \arrow[r]                                                            & G' \arrow[d] & H' \arrow[d] \arrow[r] \arrow[rd, Rightarrow]                              & G' \arrow[d] \\
H \arrow[r] \arrow[u] \arrow[r, Rightarrow, bend left=74, shift left=2] & G            & H \arrow[r]                                                                & G           
\end{tikzcd}
\end{equation}
Now the following proposition shows that fibration replacements are essentially unique:
\begin{proposition}\label{canonical map}
Given any pair of fibration replacements for $H\to G\,,H'\to G'$ (fitting into the diagram in the top left) given by $E\to G\,,E'\to G'\,,$ respectively, there is a commutative diagram of the following form:
\begin{equation}
\begin{tikzcd}
E \arrow[d] \arrow[r, dashed] & E' \arrow[d] \\
G \arrow[r, dashed]           & G'          
\end{tikzcd}
\end{equation}
\end{proposition}
\begin{proof}
What we want to do is show that the map $E\twoheadrightarrow G'$ given by the composition $E\twoheadrightarrow G\twoheadrightarrow G'$ lifts to a map $E\to E'$ (note that we are using different arrows for ease of exposition, they do not carry any connotation). We will do this by showing that there is another map $E\rightharpoonup G'$ which factors through $E'$ and which is equivalent to the first map via a $2$-morphism. The map in question is the one given by the following commutative diagram:
\begin{equation}
    \begin{tikzcd}
E \arrow[d,harpoon]  &                                 & E' \arrow[d] \\
H \arrow[r,harpoon]                    & H' \arrow[r,harpoon] \arrow[ru, dashed] & G'          
\end{tikzcd}
\end{equation}
Now the proof follows from the following diagram:
\begin{equation}
    \begin{tikzcd}
                                                                                                    &  & G \arrow[r,two heads]                                  & G'           \\
E \arrow[rr, harpoon,shift right=0.5] \arrow[rru, two heads, shift left=1] \arrow[rru, Rightarrow, bend right, shift left] &  & H \arrow[r,harpoon] \arrow[u] \arrow[ru, Rightarrow] & H' \arrow[u,harpoon]
\end{tikzcd}
\end{equation}
\end{proof}
Now we have the following result, which already gives us one application of fibrations (see \Cref{morita of fibers} for more information):
\begin{proposition}\label{equiv1}
Suppose $F:E\to G$ is a fibration which is a surjective submersion at the level of objects, and let $f:H\to G$ be a morphism of Lie groupoids. Then the strong fiber product and the fiber product both exist, and the canonical morphism $E\times_{G!} H\xhookrightarrow{} E\times_G H$ is a Morita equivalence.
\end{proposition}
Given the previous result, it is in particular true that if $F:E\to G$ is a fibration which is a surjective submersion at the level of objects, then the the canonical inclusion $F^{-1}(g^0)\xhookrightarrow{} E\times_G \{g^0\}$ is a Morita equivalence. In light of this, we make the following definition:
\begin{definition}\label{quasifibration}
We will call a map $f:H\to G$ a quasifibration if for each $g^0\in G^0$ the kernel over $g^0$ exists and if the canonical inclusion $f^{-1}(g^0)\xhookrightarrow{}H\times_G \{g^0\}$ is a Morita equivalence.
\end{definition}
Now we will define what it means for a map of Lie groupoids $f:H\to G$ to be a surjective submersion. This is what Mackenzie calls a fibration in \cite{Mackenzie}. They are the correct maps for defining simple foliations of Lie groupoids.
\begin{definition}\label{surj sub}
We call a map $f:H\to G$ a surjective submersion of Lie groupoids if both the map on the space of objects and the map $H\to G\sideset{_s}{_{f^0}}{\mathop{\times}}H^0\,, h\mapsto (f(h),s(h))$ are surjective submersions.
\end{definition}
\begin{proposition}\label{equiv2}
A surjective submersion of Lie groupoids is a quasifibration.
\end{proposition}
\begin{exmp}
A map of Lie groups $H\to G$ is a quasifibration if and only if it is a surjective submersion at the level of arrows. It is also a surjective submersion if and only if it is a surjective submersion at the level of arrows.
\end{exmp}
\section{Properties of Fibrations}
Now in homotopy theory, a fiber bundle is in particular a fibration, but this is not true for Lie groupoids. One might wonder, if the map $H\to G$ is a fiber bundle in the category of Lie groupoids, would there be an advantage to replacing this with a fibration? The answer is yes. 
 \vspace{3mm}\\
Consider for example the homomorphism $\mathbb{R}\to S^1\,.$ This is a fiber bundle in the category of Lie groupoids. However, the fibration replacement $\mathbb{R}\ltimes S^1\rtimes S^1\to S^1$ has at least one interesting property that the map $\mathbb{R}\to S^1$ doesn't have (aside from the 2-morphism lifting property): the fibers of the map $\mathbb{R}\to S^1$ (as a map of spaces) are all diffeomorphic to the fiber over the identity, but not canonically. However, the fibers of the fibration replacement $\mathbb{R}\ltimes S^1\rtimes S^1\to S^1$ are all canonically identified with the fiber over the identity, $\mathbb{R}\ltimes S^1\,.$ We have the following result (a similar observation is made in~\cite{fernandes}):
 \begin{proposition}\label{fibration property}
 Let $F:E\to G$ be a morphism of Lie groupoids. Suppose there is an action of $G$ on $E$ with respect to the moment map $t\circ F\,,$ which is compatible with $F$ and the action of $G$ on itself with respect to the target map, in the sense that, if $F(e)=g$ and s(g')=t(g), then $F(g'\cdot e)=g'g$ and $s(g'\cdot e)=s(e)$  (in particular, $F$ is a morphism of Lie groupoids as well as $G$-spaces). Then $F$ is a fibration.
 \end{proposition}
Such a $G$-action will in particular identify the fibers $F^{-1}(g)$ and $F^{-1}(g')\,$ if $s(g)=s(g')\,.$ Let us remark that, in particular, given such a $G$-action, we get a canonical section of \begin{equation}
E\xrightarrow[]{(F,s)} G\sideset{_s}{_F}{\mathop{\times}}E^0\,.
\end{equation}
In the special case of Lie groups, one can ask what the fibrations are. As the following proposition shows, maps of Lie groups are almost never fibrations (however in the case of discrete groups, the condition is equivalent to the map being surjective).
\begin{proposition}
Let $F:H\to G$ be a map of Lie groups; $F$ is a fibration if and only if, as manifolds, $H$ is a trivial fiber bundle with respect to $F\,,$ ie. $H\cong \text{ker }F\times G\,,$ with the map to $G$ being the projection onto the second factor. 
\end{proposition}
\begin{proof}
First, if $H\cong \text{ker }F\times G$ as manifolds, then $F$ admits a section, given by $g\mapsto (e,g)\,.$ This section allows us to lift 2-morphisms. Conversely, suppose $F:H\to G$ is a fibration. Consider the map $f:G\rightrightarrows G\to G\rightrightarrows *\,,$ which sends everything to the identity element. We have a 2-morphism given by the identity map $G\to G\,.$ The map $f$ factors through $F\,,$ therefore the 2-morphism must lift, ie. there must be a section of $F\,.$ Since the kernel of $F$ acts on the fibers of $F\,,$ this section gives us the desired identification $H\cong \text{ker }F\times G\,.$
\end{proof}
 \section{Split Fibrations}
We define a splitting of a Lie groupoid $A\rightrightarrows A^0$ to be an embedding of $A$ as the diagonal of a double Lie groupoid, ie.
\begin{equation}
    \begin{tikzcd}
A \arrow[r, shift right] \arrow[r, shift left] \arrow[d, shift left] \arrow[d, shift right] & C \arrow[d, shift right] \arrow[d, shift left] \\
B \arrow[r, shift left] \arrow[r, shift right]                                              & A^0                                           
\end{tikzcd}
\end{equation}
where the source and target maps of $A\rightrightarrows A^0$ are equal to the double source and double target maps of the above double groupoid (for more on obtaining a simplicial manifold from the diagonal of a bisimplicial manifold, see \cite{mehtatang}).
\vspace{3mm}\\Now given a map $F:E\to G$ equipped with a compatible $G$-action as in \Cref{fibration property}, we get a splitting of $E\rightrightarrows E^0\,.$ The double groupoid is essentially an action groupoid, which we will describe below:
\begin{equation}
\begin{tikzcd}
E \arrow[r, shift right] \arrow[r, shift left] \arrow[d, shift right] \arrow[d, shift left] & E^0\rtimes G \arrow[d, shift right] \arrow[d, shift left] \\
\text{Ker }F \arrow[r, shift right] \arrow[r, shift left]                                   & E^0                                                      
\end{tikzcd}
\end{equation}
\begin{itemize}
    \item The groupoid on the bottom row is just the subgroupoid $\text{Ker }F$ of $E\,.$
\item Now for the groupoid in the left column: we have an identifiation of $E$ with $\text{Ker }F\sideset{_F}{_s}{\mathop{\times}}G\,,$ and associated to this identification is an action groupoid of $G$ on $\text{Ker }{F}\,.$ The action is defined as follows: let $e\in \text{Ker }F\,,$ and let $g\in G$ be such that $s(g)=F(e)\,.$ Then $(e,F(e))\cdot g=(e', g)\,,$ and we define $e\cdot {g}=e'\,.$ 
\item Now for the groupoid in the right column: let $e^0\in E^0\,,$ and $g$ be such that $s(g)=F(e^0)\,.$ We can identify $e^0$ with the identity morphism in $E\,,$ denoted $\iota(e^0)\,,$ and we define $e^0\cdot g:=t(\iota(e^0)\cdot g)$
\item Finally for the groupoid in the top row. There is an action of $\text{Ker }F$ on $E^0\rtimes G\,,$ defined as follows: suppose $s(e)=e^0\,,$ we then define $e\cdot (e^0,g)=(t(e),g)\,.$
\end{itemize}
We can think of this groupoid as an action groupoid of a groupoid on another groupoid, ie. $G$ acts on $\text{ker } F\rightrightarrows E^0\,.$ Due to this discussion, we make the following definition:
\begin{definition}(see~\cite{Mackenzie}, definition 2.5.2)
We call a fibration $E\to G$ with a choice of $G$-action, as in \Cref{fibration property}, a split fibration.
\end{definition}
\begin{remark}
Note that our use of the term ``splitting" refers exclusively to the fact that the groupoid ``splits" into a double groupoid. Mackenzie (in~\cite{Mackenzie}, definition 2.5.2) coincidentally uses the terminology ``split fibration" in a way that seems to agree with how we use the term split fibration.
\end{remark}
\begin{exmp}
Let $A\,,B$ be Lie groups, and consider the fibration $A\times B\to A\,.$ We have an action of $A$ on $A\times B\,,$ given by $a'\cdot (a,b)=(aa'^{-1}\,,b)\,,$ and so in particular the action on the kernel of this map is trivial. We then have the following splitting of $A\times B$:
\begin{equation}
    \begin{tikzcd}
A\times B \arrow[r, shift left] \arrow[r, shift right] \arrow[d, shift left] \arrow[d, shift right] & A \arrow[d, shift left] \arrow[d, shift right] \\
B \arrow[r, shift left] \arrow[r, shift right]                                                      & *                                             
\end{tikzcd}
\end{equation}
\end{exmp}
Previously we discussed what the fibrations of Lie groups are, and similarly one can ask what the split fibrations of Lie groups are. We have the following result: (see~\cite{Mackenzie})
\begin{proposition}
Let $f:H\to G$ be a map of Lie groups. Then $H$ is a split fibration if and only if it is a semidirect product of $\text{ker }f$ and $G\,.$ 
\end{proposition}
\begin{exmp}
Consider a semidirect product $N\rtimes H\,.$ There is a natural morphism $H\to N\rtimes H\,,$ and this defines the action of $H$ on $N\rtimes H\,.$ The splitting of this group is then given by the following double groupoid:
\begin{equation}
    \begin{tikzcd}
N\rtimes H \arrow[r, shift left] \arrow[r, shift right] \arrow[d, shift left] \arrow[d, shift right] & H \arrow[d, shift left] \arrow[d, shift right] \\
N \arrow[r, shift left] \arrow[r, shift right]                                                      & *                                             
\end{tikzcd}
\end{equation}
Here the groupoid in the left column is just the action groupoid of $H$ acting on $N$ as a space, and the groupoid in the top row is just the action groupoid of $N$ acting trivially on $H$ as a space, ie. it is a bundle of Lie groups over $H$ (so the source and target maps are just the projection onto $H)\,.$
\end{exmp}
\subsection{The Canonical Split Fibration}
In the case that the fibration is of the form $G\times_G H\to G\,,$ there is a canonical $G$-action as in \Cref{fibration property}, and the associated splitting is given by
 \begin{equation}\label{replacement}
\begin{tikzcd}
H\ltimes P\rtimes G \arrow[r, shift left] \arrow[r, shift right] \arrow[d, shift right] \arrow[d, shift left] & P\rtimes G \arrow[d, shift right] \arrow[d, shift left] \\
H\ltimes P \arrow[r, shift right] \arrow[r, shift left]                                              & P                                           
\end{tikzcd}
\end{equation}
There is an advantage to thinking of $H\ltimes P\rtimes G$ as a double groupoid, since the fibers of the map
\begin{equation}
\begin{tikzcd}\label{fol}
H\ltimes P\rtimes G \arrow[d, shift left] \arrow[d, shift right] \arrow[r, shift left] \arrow[r, shift right] & P\rtimes G \arrow[d, shift left] \arrow[d, shift right] \arrow[r, shift right=6] & G \arrow[r, shift right] \arrow[r, shift left] \arrow[d, shift left] \arrow[d, shift right] & G \arrow[d, shift left] \arrow[d, shift right] \\
H\ltimes P \arrow[r, shift left] \arrow[r, shift right]                                                       & P                                                                                & G^0 \arrow[r, shift left] \arrow[r, shift right]                                            & G^0                                           
\end{tikzcd}    
\end{equation}
at each vertical level of the corresponding bisimplicial space are the fibers of the map $H\to G\,.$ This is not true for the map 
\begin{equation}
    \begin{tikzcd}
H\ltimes P\rtimes G \arrow[d, shift left] \arrow[d, shift right] \arrow[r, shift right=6] & G \arrow[d, shift left] \arrow[d, shift right] \\
P                                                                                         & G^0                                           
\end{tikzcd}
\end{equation}
where the fibers only appear as the kernel over an object in $G^0$ (and this will be true for any split fibration). Even worse, typically for morphisms $H\to G$ the fibers aren't embedded in $H$ in any way. In the context of this thesis, this is the main reasons for replacing a map $H\to G$ with a fibration; this will allow us to use results about simplicial/bisimplicial manifolds to study Lie groupoids.
\section{The Canonical Cofibration}\label{cofibration replacement}
The second construction one can make from a (nice enough) map $f:H\to G$ can be interpreted as replacing the map $f:H\to G$ with a cofibration; it may also be interpreted as presenting the stack $[G^0/G]$ by a double groupoid with base $H\rightrightarrows H^0\,.$  First we will motivate the construction.
\vspace{3mm}\\Suppose $Y\to X$ is a surjective submerison. In the category of manifolds isomorphisms are diffeomorphisms, therefore unless this map is also injective (making it a cofibration) there can be no cofibration replacement. However, in the category of Lie groupoids we have the submersion groupoid $Y\times_X Y\rightrightarrows Y\,,$ which is Morita equivalent to $X\,,$ and $Y\xhookrightarrow{} Y\times_X Y$ is an injection; in addition it is a cofibration. Therefore, in the category of Lie groupoids, we can replace a surjective submersion (or any submersion) between manifolds with a cofibration, and $Y\times_X Y\rightrightarrows Y$ may be called a cofibration replacement of $Y\to X\,.$ Of course, it also gives a presentation of the stack $[X/X]\,.$
\vspace{3mm}\\For maps of Lie groupoids, we will generalize the construction of the submersion groupoid. The object we will be replacing $G$ with is $H\times_G H\,,$ which though a priori is only a Lie groupoid, actually has the structure of a double Lie groupoid; this is analogous to how, given two submersions of manifolds $Y,Z\to X\,,$ the fiber product $Y\times_X Z$ is just a manifold - however, in the special case that $Y=Z\,,$ the fiber product inherits the structure of a groupoid. Explicitly, the double groupoid is given by
\begin{equation}\label{HH}
\begin{tikzcd}
H^{(1)}\sideset{_{f\circ s}}{_{t}}{\mathop{\times}} G^{(1)}\sideset{_s}{_{f\circ s}}{\mathop{\times}} H^{(1)} \arrow[d, shift right] \arrow[d, shift left] \arrow[r, shift right] \arrow[r, shift left] & H^{(0)}\sideset{_f}{_{t}}{\mathop{\times}} G^{(1)}\sideset{_s}{_{f}}{\mathop{\times}} H^{(0)} \arrow[d, shift right] \arrow[d, shift left] \\
H^{(1)} \arrow[r, shift right] \arrow[r, shift left]                                                                                                                                                    & H^0                                                                                                                                       
\end{tikzcd}
\end{equation}
Here, the bottom groupoid is simply $H$ and the top groupoid is $H\times_G H\,.$
Equivalently, we can write \ref{HH} in a condensed form as
\begin{equation}
 \begin{tikzcd}
H^{(1)}\times_G H^{(1)} \arrow[d, shift right] \arrow[d, shift left] \arrow[r, shift right] \arrow[r, shift left] & H^0\times_G H^0 \arrow[d, shift right] \arrow[d, shift left] \\
H^{(1)} \arrow[r, shift right] \arrow[r, shift left]                                                              & H^0                                                         
\end{tikzcd}
\end{equation}
The fiber product in the bottom row is with respect to the map $f^0:H^0\to G\,,$ and in the top row the fiber product is with respect to the map $f\circ s:H^{(1)}\to G\,.$ Using the strong pullback, we can write this double groupoid as
\begin{equation}
\begin{tikzcd}
(f^0\circ s)^!G \arrow[d, shift right] \arrow[d, shift left] \arrow[r, shift right] \arrow[r, shift left] & f^{0!}G \arrow[d, shift right] \arrow[d, shift left] \\
H^{(1)} \arrow[r, shift right] \arrow[r, shift left]                                                    & H^0                                              
\end{tikzcd}
\end{equation}
We will now make the following definitions:
\begin{definition}\label{canonical cofi}
Given a (nice enough) homomorphism $f:H\to G\,,$ we define $H\times_G H\rightrightarrows H$ to be the double Lie groupoid in \ref{HH}. We may also denote it by $f^!G\,.$ We will sometimes call this the canonical cofibration (associated to $f$).
\end{definition}
We will now explain what we mean by ``caononical cofibration". First we make a definition:
\begin{definition}
Let $f:H\to G$ be a morphism (here $H\,,G$ may be Lie groupoids or double Lie groupoids). A cofibration replacement for $f$ is given by a pair of maps $\iota:H\to K\,,F:K\to G$ (where $K$ is a Lie groupoid or double Lie groupoid), such that $\iota$ is a cofibration, $F$ is a fibration which is also a Morita map, and such that $f=F\iota\,.$
\end{definition}
We will now show that if $f:H\to G$ is essentially surjective then $H\times_G \rightrightarrows H$ is Morita equivalent to $G\rightrightarrows G^0\,;$ while there is no strict morphism between them, there is a natural roof. While we do this, we will also discuss the sense in which the map $H\xhookrightarrow{}H\times_G H\rightrightarrows{}H$ is a cofibration replacement for $H\to G\,.$
\vspace{3mm}\\Consider the map $H\ltimes P\rtimes G\to G\,.$ We can form the fiber product with respect to the objects and arrows to get the following double groupoid (with the appropriate fiber products in the top row):
\begin{equation}\label{cofi}
    \begin{tikzcd}
H\times H\ltimes(P\times P)\rtimes G \arrow[d, shift right] \arrow[d, shift left] \arrow[r, shift right] \arrow[r, shift left] & P\times_{G^0} P \arrow[d, shift right] \arrow[d, shift left] \\
H\ltimes P\rtimes G \arrow[r, shift right] \arrow[r, shift left]                                                               & P                                                           
\end{tikzcd}
\end{equation}
Now there is a natural morphism from \ref{cofi} to $G\rightrightarrows G^0\,,$ which we think of as a double groupoid in the following way:
\begin{equation}\label{G}
    \begin{tikzcd}
G \arrow[d, shift right] \arrow[d, shift left] \arrow[r, shift right] \arrow[r, shift left] & G^0 \arrow[d, shift right] \arrow[d, shift left] \\
G \arrow[r, shift right] \arrow[r, shift left]                                              & G^0                                             
\end{tikzcd}
\end{equation}
The map from \ref{cofi} to \ref{G} is given by projection onto $G\,.$ This map is a Morita equivalence (where we view the double groupoid as the groupoid in the left column over the groupoid in the right column).
\vspace{3mm}\\Now there is a natural inclusion from $H\rightrightarrows H^0$ to the double groupoid in \ref{cofi}, and composing this map with the map from \ref{cofi} to $G\rightrightarrows G^0$ gives us our original map $H\to G\,.$ Moreover, this inclusion is a cofibration. Therefore, the map from $H$ to \ref{cofi} is a cofibration replacement for $H\to G\,.$ In general, if the map $H\to G$ isn't essentially surjective but $H^0\to G^0$ is a submersion, we can form the disjoint union Lie groupoid $H\sqcup G^0$ which will map into $G\,,$ and will be essentially surjective and a submersion at the level of objects. We will now summarize this result :
\begin{proposition}
Let $f:H\to G$ be such that the induced map $H^0\to G^0$ is a submersion. Then a cofibration replacement for $f$ exists in the category of double Lie groupoids.
\end{proposition}
\vspace{3mm}Now on the other hand, we also have a Morita equivalence from \ref{cofi} to $f^!G\,.$ Therefore $f^!G\,,$ by definition, is Morita equivalent to $G\,.$ In addition, the natural inclusion $H\xhookrightarrow{}f^!G$ is a cofibration. Now there isn't a morphism $f^!G\to G\,,$ so $H\xhookrightarrow{}f^!G$ isn't exactly a cofibration replacement for $H\xhookrightarrow{} G\,,$ it is almost just as good, so we will call it the canonical cofibration.
\vspace{3mm}\\Now we will state a sufficient condition for a map of Lie groupoids $H\to G$ to be a cofibration:
\begin{proposition}
Suppose $H\to G$ is a map of Lie groupoids such that the map on the space of objects is a diffeomorphism, then $f$ is a cofibration. In particular, all maps of Lie groups are cofibrations.
\end{proposition}
Homomorphisms which aren't diffeomorphisms at the level of objects are often not cofibrations (in general, a condition for a map $H\to G$ to be a cofibration is probably that the map $H^0\to G^0$ is a closed embedding). Here we will give an example:
\begin{exmp}
Let $G\rightrightarrows *$ be a Lie group, and consider the identity morphism $G\to G\,.$ We can consider the trivial groupoid $G\rightrightarrows G\,,$ and there is a unique homomorphism $f$ mapping into $G\rightrightarrows *\,,$ which sends everything to $*\,.$ We get a natural transformation $f\Rightarrow f$ by sending the space of objects of $G\rightrightarrows G$ to the space of arrows of $G\rightrightarrows *$ using the identity map. Now $f$ factors through the identity morphism, therefore the identity morphism extends this map, however there can be no extension of this natural transformation since $G\rightrightarrows *$ has only one object, so the identity map $G\to G$ can't factor through it.
\end{exmp}
\begin{remark}
Due to these results about fibrations and cofibrations, one can probably put a model structure on the category of $\infty$-fold Lie groupoids (that is, the category consisting of Lie groupoids, double groupoids, triple groupoids, etc.), with respect to a certain nice class of maps (ie. submersions at the level of objects).
\end{remark}
\subsection{``Relative" Lie Groupoid Cohomology}
There is a global version of relative Lie algebra cohomology, and there is also a Lie groupoid analogue of this which we will now discuss. We put relative in quotations as there doesn't appear to be a long exact sequence associated with this cohomology which involves a groupoid and a subgroupoid. However, in \cref{analogies} we will exhibit a cohomology which does fit into such a long exact sequence. 
\vspace{3mm}\\Once again, one can think about the canonical cofibration associated to a map $H\to G$ as presenting the stack $[G^0/G]$ as a double groupoid over $H\rightrightarrows H^0\,.$ This is a useful construction to make when comparing the cohomology of two groupoids as it assembles both groupoids into a single object.
\begin{exmp}
Let's specialize \ref{HH} to the case where $H\xhookrightarrow{}G$ is a wide subgroupoid. In this case, the double groupoid is 
\begin{equation}\label{properr}
    \begin{tikzcd}
H^{(1)}\sideset{_s}{_t}{\mathop{\times}} G^{(1)}\sideset{_s}{_s}{\mathop{\times}} H^{(1)} \arrow[d, shift right] \arrow[d, shift left] \arrow[r, shift right] \arrow[r, shift left] & G^{(1)} \arrow[d, shift right] \arrow[d, shift left] \\
H^{(1)} \arrow[r, shift right] \arrow[r, shift left]                                                                          & G^0                                                 
\end{tikzcd}
\end{equation}
Notice that, if $H$ is proper, then the groupoids in all of the rows of the corresponding bisimplicial manifold are proper. Since the cohomology of a proper groupoid with values in a represenation vanishes in positive degree, the cohomology of \ref{properr} reduces to the cohomology of the right column, and thus one can work with cocycles for which the  pullback by $\delta_h^*$ is trivial. Therefore, $H^*(G,E)$ is isomorphic to the cohomology of the subcomplex of $Z(G,E)$ consisting of functions $f:G^{(n)}\to E$ such that $\delta_h^*f=0\,.$ These are functions such that, 
\begin{equation}\label{naive}
f(g_1,g_2,\ldots,g_n)=f(h_1g_1h_2^{-1},h_2g_2h_3^{-1},\ldots,h_ng_nh_{n+1}^{-1})\,,
\end{equation}
whenever the expression on the right makes sense. In degree $0$ we get functions invariant under the action of $H\,,$ ie. $f(s(h))=f(t(h))\,.$ 
\vspace{3mm}\\Notice that this double groupoid relates the cohomology of $H,G$ and the ``$H$-invariant" (or ``relative") cohomology of $G\,,$ given by the kernel of $\delta_h^*\,.$
It also relates the cohomology of $H,G$ and the cohomology of the mapping cone (to be defined in \cref{analogies}), which in this case may be interpreted as the relative cohomology (of $G$ relative to $H\,$).
\end{exmp}
Let's rephrase what was previously said about reducing the cohomology of the double groupoid to that of the right column. With any double complex there is an associated spectral sequence; actually, there are two, but we will focus on the one where we compute the first page using the horizontal differentials, and the second page is then computed using  the vertical differentials. This spectral sequence converges to the cohomology of the total complex, which in this case is the cohomology of $G$ (assuming Morita invariance of cohomology of double groupoids). In the case that $H$ is proper, this spectral sequence collapses on the second page to the first column. Summarizing this:
\begin{proposition}
Let $G$ be a Lie groupoid with a representation $E\,,$ and let $K$ be a wide and proper Lie subgroupoid. Then $H^*(G,E)$ is isomorphic to the cohomology of the subcomplex of functions $f:G^{(n)}\to E$ consisting of those functions which satisfy 
\begin{equation}\label{k invariant}
f(g_1,g_2,\ldots,g_n)=f(k_1g_1k_2^{-1},k_2g_2k_3^{-1},\ldots,k_ng_nk_{n+1}^{-1})\,,
\end{equation}
whenever the expression on the right makes sense. In degree $0$ we get functions invariant under the action of $K\,,$ ie. $f(s(k))=f(t(k))$ (here $g_i\in G^{(1)}\,, k_i\in K^{(1)})\,.$
\end{proposition}
\begin{remark}
The subcomplex of functions which satisfy \Cref{naive} seems to be the complex of functions on the ``naive" quotient of $G$ by by the double groupoid $\text{Pair}(H)\,.$
\end{remark}
\subsection{LA-Groupoid Associated to the Canonical Cofibration}
In the previous section we discussed replacing a nice enough map $f:H\to G$ with a cofibration. Now one can ask: what is the LA-groupoid associated to the canonical cofibration $H\times_G H\rightrightarrows H\,?$ It is given by the following:
\begin{equation}\label{LAg}
    \begin{tikzcd}
(f\circ s)^!\mathfrak{g} \arrow[d] \arrow[r, shift right] \arrow[r, shift left] & f^!\mathfrak{g} \arrow[d] \\
H^{(1)} \arrow[r, shift right] \arrow[r, shift left]                            & H^0                      
\end{tikzcd}
\end{equation}
We will discuss this LA-groupoid more in \Cref{equiv la group}. We may denote it $f^!\mathfrak{g}\,.$ In light of this, we see that it can be useful to replace a map with the canonical cofibration even if the map is already a cofibration.
\vspace{3mm}\\Now let's specialize \ref{LAg} to the case that $f:H\to G$ is an  inclusion of Lie groups. The resulting LA-groupoid is the following:
\begin{equation}
\begin{tikzcd}\label{LAgroup}
H\ltimes_{\text{Ad}}\mathfrak{h}\ltimes\mathfrak{g} \arrow[r, shift left] \arrow[r, shift right] \arrow[d] & \mathfrak{g} \arrow[d] \\
H \arrow[r, shift left] \arrow[r, shift right]                                         & *                 
\end{tikzcd}
\end{equation}
\vspace{3mm}\\Now in the case that $H\subset Z(G)\,,$ something special happens if the sequence $0\to\mathfrak{h}\to\mathfrak{g}\to\mathfrak{g}/\mathfrak{h}\to 0$ splits as a direct sum. This can be useful for computing cohomology as it gives a simpler model of the LA-groupoid:
\begin{proposition}\label{semidirect}
Suppose that $\mathfrak{h}\subset Z(\mathfrak{g})$ (the center of $\mathfrak{g}$) and that $\mathfrak{g}\cong\mathfrak{h}\oplus\mathfrak{g}/\mathfrak{h}\,.$
Then the map $(h,[X])\mapsto (h,0,[X])$ induces a Morita map of LA-groupoids between 
\begin{equation}
\begin{tikzcd}
H\times\mathfrak{g}/\mathfrak{h} \arrow[r, shift left] \arrow[r, shift right] \arrow[d] & \mathfrak{g}/\mathfrak{h} \arrow[d] \\
H \arrow[r, shift left] \arrow[r, shift right]                                         & *                 
\end{tikzcd}
\end{equation}
and \ref{LAgroup}. Here, the Lie algebroid on the left is just a trivial bundle of Lie algebras, ie. the pullback of $\mathfrak{g}/\mathfrak{h}\to *\,.$
\end{proposition}
For completion, we will state the global analogue of the previous proposition:
\begin{proposition}
Suppose $N\subset Z(G)$ and that $G\cong N\times G/N\,.$ Then letting $\iota:N\to G$ be the inclusion map, we have that $\iota^!G$ is Morita equivalent to the following double groupoid:
\begin{equation}
\begin{tikzcd}
N\times G/N \arrow[d, shift right] \arrow[d, shift left] \arrow[r, shift right] \arrow[r, shift left] & G/N \arrow[d, shift right] \arrow[d, shift left] \\
N \arrow[r, shift right] \arrow[r, shift left]                                                        & *                                               
\end{tikzcd}
\end{equation}
Here the groupoids in the top row and left column are trivial bundles of groups. Note that this is a splitting of $N\times G/N\rightrightarrows *\,.$
\end{proposition}
\section{Analogies Between Lie Groupoids and Homotoy Theory}\label{analogies}
Here we will just collect some observations the author believes to display analogies between (double) Lie groupoids and homotopy theory (which are more than analogies when thinking about topological spaces as being equivalent to their fundamental $\infty$-groupoid). Already to make some of these constructions we've had to exit the category of Lie groupoids and enter the category of double Lie groupoids; in this section we will, in a sense, have to leave the category of double Lie groupoids (depending on how you interpret the constructions). We will in particular discuss mapping cones, relative cohomology and suspension.
\vspace{3mm}\\We have already discussed fibrations and cofibrations, and the canonical fibration and cofibration replacements, which are analogous to the mapping path space and the mapping cylinder. Another construction we could make is of the mapping cone: given a map $f:H\to G$ of Lie groupoids which is a \hyperref[surj stacks]{surjective submersions of stacks} (ie. $f$ is essentially surjective and is a submersion at the level of objects), we can form the canonical cofibration and then collapse the base to a point. We get the following (semi-) bisimplicial space\footnote{To be precise, it is a semi-simplicial manifold in the category of simplicial manifolds.}:
\begin{equation}\label{cone}
\begin{tikzcd}
H^{(1)}\sideset{_{f\circ s}}{_{t}}{\mathop{\times}} G^{(1)}\sideset{_s}{_{f\circ s}}{\mathop{\times}} H^{(1)} \arrow[d, shift right] \arrow[d, shift left] \arrow[r, shift right] \arrow[r, shift left] & H^{(0)}\sideset{_f}{_{t}}{\mathop{\times}} G^{(1)}\sideset{_s}{_{f}}{\mathop{\times}} H^{(0)} \arrow[d, shift right] \arrow[d, shift left] \\
* \arrow[r, shift right] \arrow[r, shift left]                                                                                                                                                    & *                                                                                                                                       
\end{tikzcd}
\end{equation}
Now in the case that $H\to G$ isn't essentially surjective but is still a submersion at the level of objects, one can form the canonical cofibration  by forming the canonical cofibration of $H\sqcup G^0\to G$ instead, and then one can collapse $H$ in the base to a point. We will denote \ref{cone} by $C(f)\,.$  We can compute the cohomology of $C(f)\,.$ If $H\xhookrightarrow{} G$ is a subgroupoid, one might call this the relative cohomology of $G$ (relative to $H$).
\vspace{3mm}\\Now we can compute the cohomology of \ref{cone}. Since we've collapsed the bottom row to a point, we are going to shift the degree of cohomology by one, so that what we would naturally call degree $1$ is now degree $0\,.$ If the canonical cofibration is analogous to the mapping cylinder, then the above construction should be analogous to the mapping cone. To support this idea, we have the following proposition:
\begin{proposition}
Let $f:H\to G$ be a morphism which is a submersion at the level of objects. We get the following long exact sequence (where the coefficients are associated to $M$):
\begin{equation}
\cdots \to H^{n}(G)\to H^{n}(H)\to H^{n}(C(f))\to  H^{n+1}(G)\to H^{n+1}(H)\to\cdots
\end{equation}
\end{proposition}
Here, the map $H^{n}(C(f))\to  H^{n+1}(G)$ is the one associated to the inclusion of $C(f)\xhookrightarrow{}f^!G;$ the map $H^{n}(G)\to H^{n}(H)$ is given by restricting the cohomology classes of $f^!G$ to the bottom row, ie. $H\rightrightarrows H^0\,;$ the morphism $H^{n}(H)\to H^{n}(C(f))$ is given by pulling back cohomology classes from $H$ to $C(f)$ by using the embedding of $H\rightrightarrows H^0$ into the bottom row of $f^!G$ and pulling back cohomology classes to the second row of $C(f)$ via $\delta^*_v\,.$
\vspace{3mm}\\Note that we in particular get a long exact sequence by taking a Lie groupoid $G\rightrightarrows G^0$ and leting $H=G^0\,.$ In this case, we get a long exact sequence relating the cohomologies of $G^0,[G^0/G]$ and the cohomology classes on $G$ corresponding to multiplicative objects (what was called ``truncated cohomology" in Part 1). That is, the truncated cohomology is the cohomology of the mapping cone $G^0\xhookrightarrow{}G\,.$ In this case the corresponding long exact sequence was first communicated to the author by Francis Bischoff.
\vspace{3mm}\\Finally, if we take the mapping cone of $H\to *$ we get the suspension of $H\,.$ Explicitly, this is given by collapsing the base $H$ of the double groupoid $\text{Pair}(H)\rightrightarrows H$ to a point.
\chapter{Foliations of Lie Groupoids and Stacks}
\section{The (2,1)-Category of Foliations}
As opposed to groupoids internal to the category of smooth manifolds, one can consider groupoids internal to the category of foliations, ie. a Lie groupoid $G\rightrightarrows G^0$ such that $G^0,\,G^{(1)}$ are foliated manifolds, and such that all structure maps are maps of foliations. There is a forgetful functor from groupoids internal to the category of foliations to the category of Lie groupoids. A foliation of a Lie groupoid $G\rightrightarrows G^0$ is essentially a lift of $G$ to a groupoid internal to foliated manifolds.
\begin{definition}
A foliation of a Lie groupoid $G\rightrightarrows G^0$ is a foliation of $G^0,\,G^{(1)}$ such that all structure maps are maps of foliations.
\end{definition}
Foliated Lie groupoids naturally form a (2,1)-category. A morphism of foliated groupoids 
\begin{equation*}
f:H\to G
\end{equation*}
is a morphism of groupoids which is also a degreewise map of foliations. 
A 2-morphism between $f_1\,,f_2:H\to G$ is a natural transformation $g_h:H^0\to G^{(1)}\,,$ $f_1\Rightarrow f_2$ such that $g_h$ is a map of foliated manifolds. 
\vspace{3mm}\\Of course, one can talk about Morita equivalences of foliations, this is a little bit more subtle. Before doing so, we will briefly go over another way of thinking about foliations:  Associated to a foliation of a Lie groupoid $G\rightrightarrows G^0$ is a Lie algebroid subbundle of the tangent bundle to $G\,,$
\begin{equation}
    \begin{tikzcd}
TG \arrow[r] \arrow[d, shift left] \arrow[d, shift right] & G \arrow[d, shift left] \arrow[d, shift right] \\
TG^0 \arrow[r]                                            & G^0                                           
\end{tikzcd}
\end{equation}
where the subbundle of $TG^0\,,TG^{(1)}$ consists of vectors tangent to the leaves of the foliation. This in particular gives a Lie algebroid groupoid. We will identify this LA-groupoid with the corresponding foliation of the Lie groupoid.
\begin{definition}
A Morita map of foliated Lie groupoids $H\to G$ is a map of foliated Lie groupoids for which the induced map 
\begin{equation}
    \begin{tikzcd}
TH \arrow[d, shift right] \arrow[d, shift left] \arrow[rr, shift right=7] &  & TG \arrow[d, shift right] \arrow[d, shift left] \\
TH^0                                                                       &  & TG^0                                            
\end{tikzcd}
\end{equation}
is a Morita map.
\end{definition}
\begin{proposition}\label{morita foliations}
Let $F:H\to G\,,$ $f:G\to H$ be morphisms of foliated groupoids, such that $f\circ F\Rightarrow \mathbbm{1}_H\,, F\circ f\Rightarrow \mathbbm{1}_G$ (the 2-morphisms are required to be compatible with the foliations). Then $F$ and $f$ induce Morita maps of foliations. 
\end{proposition}
The previous discussion implies the following proposition:
\begin{proposition}
Associated to a foliation of a Lie groupoid $H$ is a sub LA-groupoid of the tangent LA-groupoid $TH\to H\,.$ We will denote this sub LA-groupoid by $D\to H\,,$ or $D_H\to H$ if there is risk of any confusion.
\end{proposition}
We can now slightly rephrase the definition of Morita map: a map beteween foliatied Lie groupoids $H\to G$ is a Morita map if the induced map of LA-groupoids from $D_H\to H$ to $D_G\to G$ is a Morita equivalence. By analogy with the relative tangent bundle of a submersion $Y\to X\,,$ we make the following definition
\begin{definition}
A foliation $D\to H$ associated to a (nice enough) map $H\to G$ will be called the relative tangent bundle (of $TH$ relative to $TG\,.$)\footnote{See \Cref{Normal bundles} for a discussion on why we can think of this as a normal bundle}
\end{definition}
We can now invert weak equivalences (ie. Morita maps) to obtain a new (2,1)-category. This category is just the (2,1) category of anafunctors (ie. roofs), but where the objects are foliated Lie groupoids, and the morphisms and two morphisms are compatible with the foliations. This leads us into the next section, but first we will define simple foliations.
\begin{definition}\label{simple}
We will call a foliation of a Lie groupoid $H$ simple if the foliation, at the level of objects and arrows, is given by a \hyperref[surj sub]{surjective submersion} $H\to G$ (see \cite{eli}).
\end{definition}
\begin{proposition}\label{cohomology of fol}
Given a (nice enough) morphism $H\to G\,,$ the fibration replacement $G\times_G H\to G$ is a simple foliation.
\end{proposition}
Now give a nice enough map $f:H\to G\,,$ we have discussed a way of obtaining a foliation of a fibration replacement of $f\,.$ However, if $f$ satisfies the conditions in \Cref{simple}, we have a foliation of $H$ itself. These two foliations are Morita equivalent.
\begin{proposition}
Suppose $f:H\to G$ satisfies the conditions in \Cref{simple}, so that we get a simple foliation of $H\,.$ Then the canonical map $H\to G\times_G H$ is a Morita equivalence of foliations.
\end{proposition}
Furthermore, given a simple foliation $f:H\to G\,,$ there is a canonical double groupoid integrating this LA-groupoid, and it is given by the following:
\begin{equation}
    \begin{tikzcd}
H^{(1)}\times_{G^{(1)}}H^{(1)} \arrow[d, shift right] \arrow[d, shift left] \arrow[r, shift right] \arrow[r, shift left] & H^{0}\times_{G^{0}}H^{0} \arrow[d, shift right] \arrow[d, shift left] \\
H^{(1)} \arrow[r, shift right] \arrow[r, shift left]                                                    & H^0                                           
\end{tikzcd}
\end{equation}
This is a double Lie groupoid over $H\rightrightarrows H^0\,.$ Of course, we have already described another double groupoid over $H\rightrightarrows H^0\,,$ given by the canonical cofibration associated to $f;$ these two double Lie groupoids are canonically Morita equivalent. Therefore, the constructions we've been making agree with the usual constructions in the case that the map $f:H\to G$ defines a simple foliation.
\begin{lemma}\label{iso of coh}
Let $f:H\to G$ be a simple foliation, and let $M$ be a $G$-module. Then we have an isomorphism of cohomology $H^*(\mathbf{B}^{\bullet}H,f^{-1}\mathcal{O}(M))\cong H^*(D\to H,f^*M)\,.$ 
\end{lemma}
\begin{proof}
The proof goes by forming the nerve of $D\to H$ and taking a resolution by fiberwise differential forms of $f^{-1}\mathcal{O}(M))_{\mathbf{B}^n H}\,,$ for each $n\ge 0\,.$
\end{proof}
Now let us summarize what  we will call the (2,1)-category of foliations of Lie groupoids.
\begin{itemize}
    \item The objects are foliations of Lie groupoids, which can equivalently be thought of as a LA-groupoid subbundle of $TG\to G\,,$ equivalently, a VB-subbbundle for which sections are closed under the Lie bracket. 
    \item The morphisms between foliations of $H$ and $G$ are morphisms $f:H\to G$ for which $f_*$ is a morphism of LA-groupoids (equivalently, a morphism of VB-groupoids).
    \item Given morphisms $f,g:H\to G$ of foliated Lie groupoids, a 2-morphism $f\Rightarrow g$ is given by a 2-morphism $f\Rightarrow g$ of maps between Lie groupoids for which the derivative maps into the subbundle, ie. a 2-morphism is given by a map $h:H^0\to G^{(1)}$ satisfying the standard conditions, such that $h_*$ maps vectors in the foliation of $H^0$ to vectors in the foliation of $G^{(1)}\,.$
\end{itemize}
This category has weak equivalences, which are given by morphisms which induce a Morita equivalence of LA-groupoids.
\subsection{(2,1)-Category of Foliations of Stacks}
The (2,1)-category of foliations of stacks is essentially the (2,1)-category of stacks, but where all Lie groupoids are foliated and all morphisms and 2-morphisms are compatible with the foliations:
\begin{itemize}
    \item The objects are foliated Lie groupoids
    \item The morphisms are anafunctors for which the maps are maps of foliated Lie groupoids, and for which the left leg is a Morita map of foliated Lie groupoids.
    \item The 2-morphisms are 2-morphism in the (2,1)-category of Lie groupoids which are maps of foliated manifolds.
\end{itemize}
\begin{definition}
A foliation of a stack $\mathcal{G}=[G^0/G]$ is a foliation of $G$ up to Morita equivalence of foliations. 
\end{definition}
Now we will define simple foliations of stacks; the foliations relevant to the van Est map are all simple:
\begin{definition}
A foliation of a stack is simple if it can be presented by a simple foliation of Lie groupoids (see \Cref{simple} for the definition of simple foliation of Lie groupoids).
\end{definition}
Given a map of stacks $\mathcal{F}:\mathcal{H}\to\mathcal{G}\,,$ we should have a criterion for determining when it can be presented by a \hyperref[surj sub]{surjective submersion} of Lie groupoids, so that it defines a simple foliation. Before doing this, we make a definition.
\begin{definition}\label{surj stacks}
A map of stacks $\mathcal{F}:\mathcal{H}\to\mathcal{G}$ is called a surjective submersion if it can be presented by a \hyperref[surj sub]{surjective submersion of Lie groupoids} $F:H\to G\,.$ Similarly, we may call a map of Lie groupoids $F:H\to G$ a surjective submersion of stacks if the induced map of stacks is a surjective submersion. 
\end{definition}
\begin{lemma}
If a map of stacks $\mathcal{F}:\mathcal{H}\to\mathcal{G}$ is a surjective submersion, then given any presentation of the map $f:H\to G\,,$ the map $G\sideset{_s}{_{f^0}}{\mathop{\times}}H^0\to G^0\,,(g,h^0)\mapsto t(g)$ will be as surjective submersion of Lie groupoids. 
\end{lemma}
Given the previous result, in order to determine if a map of stacks is a surjective submersion we only need to check it on one presentation.
\vspace{3mm}\\Now \Cref{canonical map}, \Cref{morita foliations} imply the following result:
\begin{proposition}
A surjective submersion of stacks $\mathcal{F}:\mathcal{H}\to\mathcal{G}$ determines a simple foliation of $\mathcal{H}\,.$ 
\end{proposition}
\begin{remark}
One way to understand these definitions of surjective submersions for Lie groupoids and stacks is to require the following property: a map $f:H\to G$ (thought of as Lie groupoids or stacks) should be a surjective submersion if and only if the double structure $H\times_G H$ associated to it exists and is Morita equivalent to $G\,.$ If we take the fiber product to be the strong one, we get the definition of surjective submersion for Lie groupoids, but if we take the fiber product to be the one appropriate for stacks, we get the definition of surjective submersion for stacks.
\end{remark}
We have the following simply but useful criterion for determining determining if a map of stacks is a surjective submersion.
\begin{proposition}
 Suppose a map of stacks $\mathcal{F}:\mathcal{H}\to\mathcal{G}$ can be presented by a map $f:H\to G$ which is essentially surjective and which is a submersion at the level of objects. Then $\mathcal{F}$ is a surjective submersion.
\end{proposition}
Given the definition of foliation of a stack, every foliation of a manifold defines a foliation of the associated stack, however, it is also possible for a singular foliation of a manifold $X$ to ``lift" to a foliation of the associated stack. This will happen whenever the singular foliation is induced by an integrable Lie algebroid (due to the fact that every integrable Lie algebroid defines a foliation of the stack associated to the base). In particular, Lie algebroids with almost injective anchor maps (ie. the anchor map is injective on a dense open set in the base) are integrable and thus define simple foliations of the stack associated to the space of objects.
\begin{exmp}\label{singular foliation}
Consider the singular foliation of $S^2$ induced by the action of $S^1\,.$ This  foliation is singular at the north and south poles. We can form the groupoid $G=S^2\rtimes S^1\,,$ and then the canonical fibration replacement of the map $S^2\xhookrightarrow{}  S^2\rtimes S^1\,,$ given by  $G\rtimes G\,,$ is Morita equivalent to $S^2\,.$ The foliation of $G\rtimes G$ induced by the mapping $G\rtimes G\to G$ is Morita equivalent, as LA-groupoids, to the Lie algebroid $\mathfrak{g}\to S^2\,;$ away from the singular points of the foliation on $S^2\,,$ this Morita equivalence of LA-groupoids is a Morita equivalence of foliations. In this sense this singular foliation lifts to a (regular) foliation of the stack. 
\begin{remark}
In the context of example \ref{singular foliation}, one can pull back the standard symplectic form on $S^2$ to $G\rtimes G$ via the target map. This will define a $0$-shifted sympletcic form on $G\rtimes G\,,$ and one might be tempted to do a geometric quantization of $S^2$ by geometrically quantizing the stack (using $G\rtimes G$ as a representative); the motivation is due to the fact that the singular foliation lifts to a regular foliation on the stack. The foliation on the stack is generically Lagrangian (using the definition of Lagrangian for a $0$-shifted symplectic structure), however the two leaves corresponding to the singular leaves of the north and south pole are not Lagrangian, but the symplectic form still vanishes on those leaves (this corresponds to the fact that the north and south poles are isotropic, but not coisotropic). This is however a maximally isotropic foliation in the sense that, even locally, there is no foliation by isotropic submanifolds whose leaves contain the leaves of this foliation as proper subsets. One can still compute the ``Bohr-Sommerfeld" leaves, and the dimension of the space of sections obtained agrees with the quantization of $S^2$ via the Kahler polarization.
\end{remark}
\end{exmp}
Given a Lie algebra $\mathfrak{g}\,,$ we will denote its canonical integration by $\tilde{G}\,.$ The category of Lie algebras can be naturally upgraded to a (2,1)-category, where a 2-morphism $f_1\Rightarrow f_2$ between $f_1\,,f_2:\mathfrak{h}\to\mathfrak{g}$ is given by a $\tilde{g}\in\tilde{G}$ such that $f_1=\text{Ad}_{\tilde{g}}^*\,f_2\,.$
\vspace{3mm}\\Now given a Lie algebra $\mathfrak{g}\,,$ we get a simple foliation of $[*/*]$ given by $\tilde{G}\rtimes\tilde{G}\to \tilde {G}\,.$ Conversely, a foliation of $[*/*]$ is given, in particular, by a mapping $F:\text{Pair}(X)\to H\,,$ for some manifold $X$ and some Lie groupoid $H\rightrightarrows H^0\,.$ Letting $*$ be a point in $X\,,$ we get a Lie algebra $\mathfrak{g}$ by taking the Lie algebra of the isotropy group over $F(*)\,.$ Given any other point $*'\in X\,,$ we have a canonical isomorphism between the isotropy groups of $F(*)$ and $F(*')\,,$ so in this sense the Lie algebra $\mathfrak{g}$ is well-defined. We have the following result:
\begin{proposition}
The full (2,1)-subcategory of simple foliations of $[*/*]$ is equivalent to the (2,1)-category of Lie algebras.
\end{proposition}
\begin{remark}
From this point of view, the Lie algebroid cohomology of $\mathfrak{g}\mapsto G^0$ is the same as the foliated cohomology of the stack $[G^0/G^0]\,,$ with respect to the canonical foliation $G\rtimes G\to G\,.$ In particular, $H^1(\mathfrak{g},\mathbb{C}^*_{G^0})$ classifies line bundles with foliated flat connection on the stack $[G^0/G^0]\,.$
\end{remark}
\begin{proposition}
The foliation determined by an integrable Lie algebroid $\mathfrak{g}\to M$ is independent of the source-connected Lie groupoid integrating it.
\begin{proof}
The foliation associated to any source-connected Lie groupoid $G\rightrightarrows M$ integrating $\mathfrak{g}\to M$ is equivalent to the foliation associated to the source simply connected integration.
\end{proof}
\end{proposition}
The following conjecture is a converse to the result that integrable Lie algebroids determine foliations of the base:
\begin{conjecture}
A Lie algebroid $\mathfrak{g}\to X$ is integrable if and only if it is equivalent, as an LA-groupoid, to a simple foliation.
\end{conjecture}
\section{Leaves of a Foliation}
Given a foliation of a stack, we can present it by a foliation of a Lie groupoid, and the leaves in the space of arrows passing through the identity bisection are subgroupoids. We would like to say that the union of these leaves are the stack. We won't be able to quite say this, but something similar will hold. First we will describe categorical unions:
\subsection{Categorical Union}
In category with a given object $\mathcal{C}\,,$ a subobject  $\mathcal{A}\xhookrightarrow{} \mathcal{C}$ is defined to be an equivalence class of monomorphisms. In addition, one can form the category of subobjects of $ \mathcal{C}\,,$ where a morphism between subobjects is essentially an inclusion. Now given two subobjects $\mathcal{A},\mathcal{B}\xhookrightarrow{} \mathcal{C}\,,$ their union is defined to be the coproduct of $\mathcal{A},\mathcal{B}$ in the category of subobjects of $\mathcal{C}\,.$ 
\vspace{3mm}\\Lets consider the category of sets. Consider a set $X$ and let $A,B\subset X\,.$ A morphism between subsets $A\to C$ is an inclusion $A\subset C\,.$ The coproduct of $A$ and $B$ is in particular a subset of $X$ receiving morphisms from $A\,,B\,;$ the coproduct is $A\cup B\,.$
Now given a third subset $C\subset X\,,$ one can form the union $(A\cup B)\cup C\,.$  In particular, given any collection of subsets $\{A_i\}_{i\in I}\,,$ one can form all finite unions, obtaining a new collection of sets $\{A_j\}_{j\in J}\,,$ where $J$ a directed set, given by 
\begin{equation}
J=\coprod_{n=1}^{\infty} I^n\,.
\end{equation}
For example, if $j=(i_1,i_2,i_3)\,,$ then $A_j=A_{i_1}\cup A_{i_2}\cup A_{i_3}\,.$ One can now form the union $\bigcup_{i\in I} A_i$ as a direct limit in this way. 
\subsection{Union of Leaves}
Now in the category of groupoids, a subgroupoid is a subobject, however in the (2,1)-category of groupoids, only full subgroupoids behave as subobjects. Given a Lie groupoid $G$ and two full subgroupoids $H,K\,,$ their union is the full subgroupoid over $H^0\cup K^0$ (meaning that it contains all morphisms between objects in $H^0\cup K^0$).  Given a collection of full subgroupoids $\{G_i\}_{i\in I}$ such that $\cup_i G^0_i=G^0\,,$ there union is $\cup_i G_i=G\,.$ Therefore, if the foliation is by full subgroupoids we can say their union is the groupoid.
\vspace{3mm}\\Now suppose we have a foliation of a stack, given by a foliation of a Lie groupoid $G\rightrightarrows G^0\,.$ Let $\{L_i\}_{i\in I}$ be the set of leaves in $G^{(1)}$ intersecting the space of objects. Typically these subgroupoids will not be full, however the categorical image of $L_i$ in $G$ is just the full subgroupoid over $L_i^0\,,$ therefore we can say that the union of the images of the leaves of the stack is the stack itself (the image of a morphism in a category is essentially the smallest monomorphism which the morphism factors though). 
\vspace{3mm}\\
Another observation is the following: A Lie groupoid-principal bundle determines a foliation of the total space of the principal bundle, and a foliation of the Lie groupoid refines this foliation of the total space. Therefore, given a foliation of a Lie groupoid, we get a foliation of the objects of the associated stack.
\subsection{Foliations of Lie Groups}
Here we will show that foliations of Lie groupoids, in some sense, generalize normals subgroups. The context is foliations of a Lie group $G\to *$ (ie. the space of arrows is foliated, and the foliation is compatible with the structure maps). We have the following observation (made by Francis Bischoff, and a similar observation made by Eli Hawkins in \cite{eli}): 
\begin{proposition}
Foliations of a Lie group are in bijective correspondence with normal subgroups.
\end{proposition}
\begin{proof}
To see this, note that since the foliation is compatible with the composition, the composition must takes two leaves to a third leaf. Hence, the set of leaves has a multiplication on it, and we can form the quotient to get a group. This implies that the leaf intersecting the origin is a normal subgroup, and the leaves are the cosets.
\end{proof}
\begin{remark}
Note that the above result implies, in particular, that all foliations of Lie groups are simple, since the quotient of a Lie group by a normal Lie subgroup always exists.
\end{remark}
\chapter{LA-Groupoids Associated to a Map}
Here we will further discuss the LA-groupoids associated to a (nice enough) map $:H\to G\,.$
So far, given a (nice enough) map $f:H\to G\,,$ we have two ways of associating an LA-groupoid (see \Cref{cofibration replacement}): the first way is by forming the fibration replacement and taking the associated foliation, and the second way is by forming the canonical cofibration and taking the associated LA-groupoid. 
\section{Equivalence of the Two LA-groupoids}\label{equiv la group}
Here we will show here that the resulting LA-groupoids are Morita equivalent. 
First, we will give the construction of the two LA-groupoids: Let $f:H\to G$ be a (nice enough) map of Lie groupoids. First we form the fibration replacement $H\ltimes P\rtimes G\to G\,,$ and from this we get an LA-groupoid by taking the kernels of the left and right columns as a map of vector bundles:
\begin{equation}
\begin{tikzcd}
TH\ltimes TP\rtimes TG \arrow[r, shift left] \arrow[d] \arrow[r, shift right] & TP \arrow[d] \arrow[r, "p_{3*}", shift right=8] & TG \arrow[d] \arrow[r, shift right] \arrow[r, shift left] & TG^0 \arrow[d] \\
H\ltimes P\rtimes G \arrow[r, shift left] \arrow[r, shift right]              & P                                               & G \arrow[r, shift left] \arrow[r, shift right]            & G^0           
\end{tikzcd}
\end{equation}
Explicitly, it is given by 
\begin{equation}\label{exx}
 \begin{tikzcd}
TH\ltimes T_sP\rtimes G \arrow[r, shift left] \arrow[d] \arrow[r, shift right] & T_sP \arrow[d] \\
H\ltimes P\rtimes G \arrow[r, shift left] \arrow[r, shift right]             & P                        
\end{tikzcd}
\end{equation}
\vspace{3mm}Now the second way of obtaining an LA-groupoid (see \Cref{cofibration replacement}) is given by 
\begin{equation}
\begin{tikzcd}
(f\circ s)^!\mathfrak{g} \arrow[r, shift left] \arrow[r, shift right] \arrow[d] & f^!\mathfrak{g} \arrow[d] \\
H^{(1)} \arrow[r, shift right] \arrow[r, shift left]                            & H^0                      
\end{tikzcd}
\end{equation}
We can rewrite this as
\begin{equation}
\begin{tikzcd}
TH\ltimes TH^0\times_{TG^0}\mathfrak{g} \arrow[d] \arrow[r, shift left] \arrow[r, shift right] & TH^0\times_{TG^0}\mathfrak{g} \arrow[d] \\
H \arrow[r, shift left] \arrow[r, shift right]                                                 & H^0                                    
\end{tikzcd}
\end{equation}
Now, there is a natural map
\begin{equation}\label{morita LA}
     \begin{tikzcd}
TH\ltimes T_sP\rtimes G \arrow[r, shift left] \arrow[d] \arrow[r, shift right] & T_sP \arrow[d] \arrow[r, "", shift right=9] & TH\ltimes TH^0\times_{TG^0}\mathfrak{g}  \arrow[r, shift left] \arrow[d] \arrow[r, shift right] & TH^0\times_{TG^0}\mathfrak{g} \arrow[d] \\
H\ltimes P\rtimes G \arrow[r, shift left] \arrow[r, shift right]             & P                                                      & H \arrow[r, shift left] \arrow[r, shift right]                                    & H^0                      
\end{tikzcd}
\end{equation}
where the map on bottom row is given by the projection on the first factor, and the map on the top row is given by right translation (via $G$) of vectors in $T_sP=TH^0\times_{TG^0}T_sG$ to vectors in $TH^0\ltimes TH^0\times_{TG^0}\mathfrak{g}\,.$ This map is a Morita equivalence of the groupoids in the top row, therefore these LA-groupoids are Morita equivalent.
\begin{remark}
Specializing this result to the case of $\iota:G^0\xhookrightarrow{} G\,,$ we can interpret $\mathfrak{g}$ in the following way: we can equip $G\rightrightarrows G^0$ with the trivial LA-groupoid structure $0_G\to G$ (ie. the zero vector bundle). We can then interpret $\mathfrak{g}$ as being $\iota^!0_G$ (this pullback should be understood as the pullback appropriate to the (2,1)-category of groupoids, ie. we pull back $0_G\to G$ to the fibration replacement).
\end{remark}
\begin{remark}This gives one kind of duality between fibrations and cofibrations, ie. given a (nice enough) map $f:H\to G\,,$ we have two methods of obtaining LA-groupoids, one using fibrations and one using cofibrations, and they both agree. Related to this duality is another: given a (nice enough) map $f:H\to G\,,$ we can compute the fibers of the map by computing the kernel of the fibration replacement $G\times_G H\to G$ over an object in $G^0\,.$ Similarly, we can compute the fibers as the kernel of the target map from the top groupoid to the bottom groupoid in \ref{HH}, over an object in $H^0\,.$ 
\end{remark}
\section{The Normal Bundle is an LA-Groupoid}
Here we will show how the normal bundle of a wide Lie subgroupoid can be interpreted as an LA-groupoid. First we will specialize the previous construction to the case that $H=G$ (the one associated to the canonical cofibration). We have the following LA-groupoid:
\begin{equation}
\begin{tikzcd}\label{LA-groupoid2}
TG\ltimes_{TG^0}\mathfrak{g} \arrow[d, shift left] \arrow[r] \arrow[d, shift right] & G^{(1)} \arrow[d, shift left] \arrow[d, shift right] \\
\mathfrak{g} \arrow[r]                                                       & G^0                                           
\end{tikzcd}
\end{equation}
ie. there is a natural action of $TG$ on $\mathfrak{g}\,.$
To expand on the groupoid in the left column of \ref{LA-groupoid2}, note that given a groupoid $G\rightrightarrows G^0\,,$ we can form the tangent groupoid $TG\rightrightarrows TG^0\,.$ Now  a groupoid naturally acts on itself, with moment map being the target. Therefore, we can form the groupoid $TG\rtimes TG\rightrightarrows TG\,,$ and we can form the subgroupoid 
\begin{equation}\label{TG action T_sG}
T_sG\rtimes TG\rightrightarrows T_sG\,,
\end{equation}which consists of the action of $TG$ on vectors in $TG$ which are tangent to the source fibers. Using right translation, we have an action of $TG$ on $\mathfrak{g}\,,$ 
\begin{equation}
    \mathfrak{g} \sideset{}{_{TG^0}}{\mathop{\rtimes}} TG \rightrightarrows \mathfrak{g}\,,
\end{equation}
here, the moment map for $\mathfrak{g}$ is just the anchor map. One way of describing this action is to choose an adjoint representation up to homotopy, ie. a splitting of the sequence
\begin{equation}
  \begin{tikzcd}
t^*\mathfrak{g} \arrow[r, "r"] & TG^{(1)} \arrow[r, "s_*"] \arrow[l, "\omega", dotted, bend left, shift left=2] & s^*TG^0
\end{tikzcd}
\end{equation}
such that the splitting is the canonical one when restricted to $G^0\,.$
Then, one obtains an adjoint action up to homotopy, given by 
\begin{equation}
    Ad_g(X_{s(g)}):=\omega_g(g(X_{s(g)}-\alpha(X_{s(g)})))g^{-1}\,.
\end{equation}
Now we may define an action of $TG$ on $\mathfrak{g}\,,$ given by \begin{equation}
\tilde{X}_g\cdot X_{s(g)}=Ad_g(X_{s(g)})+\omega_g(\tilde{X_g})g^{-1}\,.
\end{equation}
Let us emphasize that the above action of $TG$ is a bonafide action, and that it doesn't depend in any way on the choice of splitting. 
\subsection{The Normal Bundle}Using the action of $TG$ on $\mathfrak{g}$ given in the previous section, we deduce that, given a (nice enough) homomorphism $H\to G\,,$ we get a natural action of $TH$ on $\mathfrak{g}\,,$ and $TH\ltimes_{TG^0}\mathfrak{g}$ is an LA-groupoid over $H\,.$ In the case that $H\xhookrightarrow{}G$ is a wide Lie subgroupoid specializes to 
\begin{equation}
\begin{tikzcd}
TH\ltimes\mathfrak{g} \arrow[d, shift left] \arrow[r] \arrow[d, shift right] & H^{(1)} \arrow[d, shift left] \arrow[d, shift right] \\
\mathfrak{g} \arrow[r]                                                       & G^0                                           
\end{tikzcd}
\end{equation}
where the action of $TH$ on $\mathfrak{g}$ is given by \begin{equation}
\tilde{X}_g\cdot X_{s(g)}=Ad_g(X_{s(g)})+\omega_g(\tilde{X_g})g^{-1}\,.
\end{equation}
We will now show how this LA-groupoid can be thought of as equipping the normal bundle of $H^{(1)}\xhookrightarrow{}G^{(1)}$ with additional structure.
\vspace{3mm}\\First, apply the forgetful functor to VB-groupoids, so that we obtain the same diagram but have forgotten the Lie brackets. Now, consider the following VB-groupoid
\begin{equation}\label{LA-groupoidss}
\begin{tikzcd}
H\ltimes(\mathfrak{g}/\mathfrak{h}) \arrow[d, shift left] \arrow[r] \arrow[d, shift right] & H \arrow[d, shift left] \arrow[d, shift right] \\
\mathfrak{g}/\mathfrak{h} \arrow[r]                                                       & G^0                                           
\end{tikzcd}
\end{equation}
Here, 
\begin{equation}\label{action}
h\cdot X=Ad_h(X_{s(h)})\end{equation}
for $X\in\mathfrak{g}/\mathfrak{h}\,.$ This is well-defined, as we can see as follows: consider $X_{s(h)}+Y_{s(h)}\,,$ where $Y_{s(h)}\in \mathfrak{h}_{s(h)}\,.$ Let $W_h\in TH_h$ be such that $s_*W_h=\alpha(X_{s(h)}+Y_{s(h)})$ (this is possible, since $s$ is a submersion). Then 
\begin{equation}\label{normal3}
W_h\cdot (X_{s(h)}+Y_{s(h)})=Ad_hX_{s(h)}+Ad_hY_{s(h)}+\omega_h(W_h)h^{-1}\,.
\end{equation}Now recall, that action is independent of $\omega\,,$ and in addition this action at a point $g\in G$ only depends on $\omega_g\,,$ thus we may choose $\omega_h$ so that $\omega_h(TH)\in t^*\mathfrak{h}\,.$ Then, $Ad_hY_{s(h)}+\omega_h(W_h)h^{-1}\in \mathfrak{h}_{s(h)}\,,$ Then we see that the action given in \ref{normal3} is independent of $Y_{s(h)}$ and $W_h\,,$ and thus the action given by \cref{action} is well-defined. Now, the natural homomorphism $TH\ltimes \mathfrak{g}\to H\ltimes\mathfrak{g}/\mathfrak{h}$ is a Morita equivalence\footnote{Note that, if $\mathfrak{g}=\mathfrak{h}\oplus\mathfrak{g}/\mathfrak{h}\,,$ then this is a Morita equivalence of LA-groupoids.}, and from this we get that our VB-groupoid is Morita equivalent to \ref{LA-groupoidss}. Now, applying the forgetful functor from VB-groupoids to vector bundles over manifolds, we get 
\begin{equation}
\begin{tikzcd}\label{LA-groupoid}
H^{(1)}\times_{G^0}\mathfrak{g}/\mathfrak{h} \arrow[r]  & H^{(1)}\,,                                      \end{tikzcd}
\end{equation}
which is naturally identified, via translation, with the normal bundle of $H^{(1)}\xhookrightarrow{} G^{(1)}\,.$
\section{Geometric Construction of the Natural Representation of Wide Subgroupoids}
In the previous section, we showed that given a wide subgroupoid (ie. a subgroupoid which contains all objects) $H\xhookrightarrow{} G\,,$ there is a natural action of $H$ on $\mathfrak{g}/\mathfrak{h}\,.$ We will now construct this action geometrically. \vspace{3mm}\\First, let $X_{s(g)}\in \mathfrak{g}/\mathfrak{h}\vert_{s(g)}\,.$ Choose a lift of $X_{s(g)}$ to $\tilde{X}_{s(g)}\in\mathfrak{g}\,.$ Now, since $H$ is a wide subgroupoid and $s$ is a submersion, there exists a vector $X'_{s(h)}\in TH_{s(h)}$ such that $t_* X'_{s(h)}=t_*(\tilde{X}_{s(h)})\,.$ Therefore, $t_*(\tilde{X}_{s(h)}-X'_{s(h)})=0\,,$ and using the groupoid $TG\rightrightarrows TG^0\,,$ we have that $(h,0)\in TH\vert_h$ is composable with $\tilde{X}_{s(h)}-X'_{s(h)}\,,$ and we get a vector $(h,0)\cdot(\tilde{X}_{s(h)}-X'_{s(h)})\in TH\vert_h\,.$ Now once again, since $H$ is a wide subgroupoid and $s$ is a submersion, there is a vector $Y_h\in TH\vert_h$ such that $s_*Y_h=s_*(\tilde{X}_{s(h)}-X'_{s(h)})\,,$ therefore $s_*((h,0)\cdot(\tilde{X}_{s(h)}-X'_{s(h)})-Y_h)=0\,.$ Hence, we can right translate to get 
\begin{equation}
((h,0)\cdot(\tilde{X}_{s(h)})-X'_{s(h)}-Y_h)\cdot h^{-1}\in \mathfrak{g}\vert_{t(h)}\,,.
\end{equation}
After passing to the quotient, we get a well-defined action of $H$ on $\mathfrak{g}/\mathfrak{h}\,.$ 
\vspace{3mm}\\Furthermore, the analogous argument shows that, if $f:H\to G$ is a homomorphism which is a surjective submersion on the base, one gets a representation of $H$ on $f^*\mathfrak{g}/\mathfrak{h}\,.$
\begin{remark}
One can study the cohomology and the truncated cohomology 
\begin{equation}
H^*(H,\mathfrak{g}/\mathfrak{h})\,, H_0^*(H,\mathfrak{g}/\mathfrak{h})\,,
\end{equation}
respectively. Roughly, these should classify deformations in the normal direction; note that, $H^{(1)}$ is a bitorsor for $H\rightrightarrows H^0\,.$ The authors interpretation of these cohomologies are as follows: the degree $0$ cohomology should classify deformations of $H^{(1)}\xhookrightarrow{}G^{(1)}$ in the normal direction, as a bitorsor for $H\rightrightarrows H^0\,.$ In the second case, the degree $0$ cohomology should classify deformations of $H\xhookrightarrow{} G$ in the normal direction, as a Lie groupoid.
\vspace{3mm}\\Furthermore, associated to any Lie groupoid representation $E\to G^0$ of $G\rightrightarrows G^0$ is a cohomology class $H^1(G,\mathbb{C}^*_X)$ (or $H^1(G,\mathbb{R}^*_X)$ if the representation is a real vector bundle); this class is obtained by taking the induced representation on the determinant bundle $\Lambda^{\text{top}} E\,,$ and since one-dimensional representations are classified by $H^1(G,\mathbb{C}^*_X)$ we get a natural cohomology class. Therefore, associated to any (nice enough) homomorphism $f:H\to G\,,$ there is a natural cohomology class in degree $1$ — this generalizes the fact that to any submanifold $Y\xhookrightarrow{}X$ there is a natural degree $1$ cohomology class associated to the determinant of the normal bundle.

\end{remark}
\section{Normal Bundles of Stacks in General}\label{Normal bundles}
Here we will show how some of the constructions we've been making generalize the construction of normal bundles. 
\vspace{3mm}\\First, let us record an observation: consider an embedding $\iota:N\xhookrightarrow{}M\,;$ the normal bundle is defined as $\iota^*TM/TN\,.$ Now we can think of this in the following way: consider the following VB-groupoid
\begin{equation}
    \begin{tikzcd}
TN\ltimes\iota^*TM \arrow[d] \arrow[r, shift right] \arrow[r, shift left] & \iota^*TM \arrow[d] \\
N \arrow[r, shift right] \arrow[r, shift left]                            & N                  
\end{tikzcd}
\end{equation}
There is a natural Morita map to the normal bundle, so they are equivalent descriptions. However, this description works for any smooth map $\pi:Y\to X$ between manifolds, ie. we can consider the following to be the normal bundle
\begin{equation}\label{normal}
    \begin{tikzcd}
TY\ltimes\pi^*TX \arrow[d] \arrow[r, shift right] \arrow[r, shift left] & \pi^*TX \arrow[d] \\
Y \arrow[r, shift right] \arrow[r, shift left]                            & Y              
\end{tikzcd}
\end{equation}
In the case that $\pi$ is a surjective submersion, $\pi^*TX/TY$ doesn't give the right vector bundle, however \ref{normal} is Morita equivalent to the relative tangent bundle, which is what we want. One should be able to make such a construction for any smooth morphism $H\to G$ of Lie groupoids, however it will probably be a Lie algebroid in the category of double Lie groupoids.
In addition, in this case the relative tangent bundle really is a normal bundle, since the map $Y\to X$ is equivalent to the map $Y\xhookrightarrow{} Y\times_X Y\rightrightarrows Y$ and the relative tangent bundle is the normal bundle to this map.  Since injectivity of objects isn't invariant under Morita equivalence, it makes sense to think of normal bundles in this way.
\vspace{3mm}\\We have already generalized the relative tangent bundle for a (nice enough) map $f:H\to G$ between Lie groupoids, so we feel justified in calling it the normal bundle. One can probably make sense of construction \ref{normal2} in general for any smooth map $f:H\to G$ of Lie groupoids, and if $f$ is a submersion we should in addition get a Lie algebroid structure; but that's not the point of this thesis.
In addition, the relative tangent bundle really is a normal bundle, since the map $Y\to X$ is equivalent to the map $Y\xhookrightarrow{} Y\times_X Y\rightrightarrows Y\,.$ Since injectivity of objects isn't invariant under Morita equivalence, it makes sense to think of normal bundles in this way.
\vspace{3mm}\\We have already generalized the relative tangent bundle for a (nice enough) map $f:H\to G$ between Lie groupoids, so we feel justified in calling it the normal bundle, at least in the context of double Lie groupoids where we can always replace $f$ with an embedding. One can probably make sense of construction \ref{normal} in general for any smooth map $f:H\to G$ of Lie groupoids, and if $f$ is a submersion we should in addition get a Lie algebroid structure; but that's not the point of this thesis.
\chapter{The van Est Map}
 The van Est map with respect to a (nice enough) homomorphism $f:H\to G$ is a map from the Lie groupoid cohomology of $G$ to the associated foliated cohomology of $H\,.$ Formally, it is given by the following composition:
 \begin{equation}
\begin{tikzcd}
{H^*(G,M)} \arrow[r] \arrow[rr, "VE", bend right, shift right=2] & {H^*(H,f^{-1}\mathcal{O}(M))} \arrow[r] & {H^*_{\text{dR}}(f:H\to G,M)}
\end{tikzcd}
 \end{equation}
 Here, the second map is obtained by taking a fiberwise de Rham resolution, and therefore is an isomorphism. The first map is not in general an isomorphism. However, if the fibers $f$ are $n$-connected (which is equivalent to the classifying space of its fibers being $n$-connected) it will be an isomorphism up to degree $n\,,$ and injective in degree $n+1\,.$ Its image in degree $n+1$ will consists of classes which pull back to a trivial cohomology class on each fiber.
 \vspace{3mm}\\Heuristically, this is true because any class in $H^*(H,f^{-1}\mathcal{O}(M))$ which vanishes on the fibers ``should" be pulled back from the base. However, there are obstructions to doing this, but the obstructions lie in lower degree cohomology (of a locally constant sheaf) of the fibers, which is zero in the degrees we are considering due to the connectivity assumption.
 \vspace{3mm}\\More precisely, consider the case of a surjective submersion between spaces $\pi:Y\to X\,,$ where we take the cohomology of $X$ with respect to functions valued in some abelian Lie group. If this were a fiber bundle, then we could locally write $\pi^{-1}(U)=U\times F\,,$ where $F$ is the fiber of $\pi\,.$ Then one can make a Leray spectral sequence argument to derive the result, since the local product formula would give us a good handle on the derived functors of $\pi\,.$ However, in the case that $\pi$ isn't a fiber bundle, the spectral sequence doesn't offer much help because the derived functors can be very complicated. To illustrate this, consider the following example (from the paper ``the relative de rham sequence"):
 \begin{exmp}Let $Y=\mathbb{R}^2-\{(0,0)\}\,,X=\mathbb{R}\,.$ Let $\pi$ be the projection onto the first factor — this is a surjective submersion (but not a fiber bundle). The sheaf we will put on $X$ is $\mathcal{O}_X\,,$ so that the sheaf we get on $Y$ is $\pi^{-1}\mathcal{O}_X\,.$ Consider the following foliated form, which defines a class in $H^1(Y,\pi^{-1}\mathcal{O}_X):$ 
 \begin{equation}
     \frac{x\,dy}{x^2+y^2}\,.
 \end{equation}
 Notice that, when restricted to each fiber, this form is trivial. Away from $\{(x,y):x=0\}\,,$ all the primitives are of the form $g(x)=\arctan{(y/x)}+f(x)\,.$ However, due to the limiting behavior of $\arctan{(y/x)}$ as $x\to 0\,,$ there is no function $f(x)$ which will make $g(x)$ continuous on all of $Y\,.$ Therefore, this one form is not trivial over any neighborhood of $x=0$ — this displays the complexity of the derived functors of $\pi\,.$ 
 \begin{remark}Of course, this example doesn't satisfy the connectivity assumptions since the fibers are not all connected. However, such a phenomenon could not happen for a fiber bundle — locally a fiber bundle is of the form $U\times F\,,$ and a primitive for any foliated one form which is trivial along each fiber can be found through integration. 
 \end{remark}
 \end{exmp}
 \section{Definition of the van Est map}
Here we will state and prove the main theorem of this paper (there are some applications stated in~\Cref{app}). Let us first remark that some of what we do depends on the Morita invariance of LA-groupoid cohomology, which has been shown in \cite{Waldron} (see section 5.4.5) for coefficients in a representation. We are using more general coefficients, but we expect the Morita invariance to hold. Though we don't need Morita invariance to state a theorem which is essentially equivalent.
\vspace{3mm}\\Now the first thing we must do is define the van Est map. Let us first explain what it is in the case of a surjecive submersion $Y\to X$ and where the module is just $X\times S^1\,:$
\vspace{3mm}\\Consider a surjective submersion $\pi:Y\to X$ and a module for $X$ given by $X\times S^1$ (this is automatically a module since $X$ has only identity morphisms). We will denote the sheaf of function of $X\times S^1$ by $\mathcal{O}^*\,.$ The map $\pi$ induces a morphism 
\begin{equation}\label{1}
    \pi^{-1}:H^*(X,\mathcal{O}^*)\to H^*(Y,\pi^{-1}\mathcal{O}^*)\,.
\end{equation}
Now we can take a leafwise resolution of $\pi^{-1}\mathcal{O}^*$ by leafwise differential forms, given by 
\begin{equation}\label{2}
    \pi^{-1}\mathcal{O}^*\to\mathcal{O}^*\to \Omega_{\pi}^1(Y)\to \Omega_{\pi}^2(Y)\to\cdots\,.
\end{equation}
From this, we get a map 
\begin{equation}
   H^*(Y,\pi^{-1}\mathcal{O}^*)\to H^*_{\pi}(Y,\mathcal{O}^*\to \Omega_{\pi}^1(Y)\to\cdots)\,,
\end{equation}
which is an isomorphism.
Composing \Cref{2} with \Cref{1}, we get a map
\begin{equation}
    H^*(X,\mathcal{O}*)\to H^*_{\pi}(Y,\mathcal{O}^*\to \Omega_{\pi}^1(Y)\to\cdots)\,;
\end{equation}
this is the van Est map in this special case. Now if the fibers of $\pi$ are n-connected, the van Est map is an isomorphism up to degree $n\,,$ injective in degree $n+1\,,$ and its image in degree $n+1$ consists of cohomology classes which pull back to zero along each fiber. This follows from \Cref{bisimpliciall}.
\vspace{3mm}\\The definition of the van Est map for Lie groupoids will be defined for a \hyperref{surj sub}[surjective submersion of Lie groupoids], and will proceed directly analogously. For stacks, the van Est map will be defined for a \hyperref[surj stacks]{surjective submersions of stacks}, and it will proceed by presenting the surjective submersion of stacks by a surjective submersion of Lie groupoids, and then using the van Est map there.
\vspace{3mm}\\In the following, $D\to H$ refers to a foliation of $H$ (ie. $D$ is a subbundle of $TH\rightrightarrows TH^0)\,,$ $H^*(D\to H,f^*M)$ means the LA-groupoid cohomology of $D\to H$ with coefficients in $f^*M\,.$ Some good examples to keep in mind: when $M=G^0\times \mathbb{R}$ we are taking cohomology with respect to the sheaf of $\mathbb{R}$-valued functions on the nerbve, and when $M=G^0\times S^1$ we are taking cohomology with respect to the sheaf of $S^1$-valued functions on the nerve (for more on LA-groupoid cohomology, see \cite{mehta}).
\begin{definition}
Let $f:H\to G$ be a simple foliation of Lie groupoids and let $M$ be a $G$-module. The van est map is a map
\begin{equation}
  VE:H^*(G,M)\to H^*(D\to H,f^*M)\,,  
\end{equation}
given by pulling back cohomology classes
\begin{equation}
    H^*(\mathbf{B}^\bullet G,\mathcal{O}(M))\xrightarrow[]{f^{-1}}H^*(\mathbf{B}^\bullet H, f^{-1}\mathcal{O}(M))\,,
\end{equation}
and then taking a resolution via leafwise differential forms.
\end{definition}
\begin{remark}
By Morita invariance of LA-groupoid cohomology, given a (nice enough) map $f:H\to G\,,$ we have an isomorphism between the cohomologies $H^*(P\to G\times_G H, \tilde{f}^*M)$ and $H^*(f^!\mathfrak{g}\to H, f^*M)\,.$ Therefore, we get a map $H^*(G,M)\to H^*(f^!\mathfrak{g}\to H, f^*M)$ as well. This is gives the usual van Est map in the case that the mapping is $G^0\xhookrightarrow{} G\,.$ One can show this explicitly however, similarly to what was done in part 1.
\end{remark}
We will now define the van Est map on stacks. In the following, we define a family of abelian groups $\mathcal{M}$ over a stack $\mathcal{G}$ to be a family of abelian groups associated to $G$-module $M\,,$ where $G$ is a presentation of $\mathcal{G}\,$ :
\begin{definition}
Let $\mathcal{F}:\mathcal{H}\to\mathcal{G}$ be a \hyperref[surj stacks]{surjective submersion} of differentiable (holomorphic) stacks, and let $\mathcal{M}$ be a family of abelian groups over $\mathcal{G}\,.$ The van Est map,
\begin{equation}
    \mathcal{VE}:H^*(\mathcal{G},\mathcal{M})\to H^*(\mathcal{D}\to\mathcal{H},\mathcal{F}^*\mathcal{M})
\end{equation}
is defined by choosing a Lie groupoid presentation of both the simple foliation associated to $\mathcal{F}:\mathcal{H}\to\mathcal{G}$ and $\mathcal{M}\,,$ and applying the van Est map associated to Lie groupoids. This is independent of any choices made, since any two choices of presentations of the simple foliation will be Morita equivalent).
\end{definition}
\begin{remark}
In a similar way to what is done in Part 1, one could define a ``truncated" version of this more general van Est map (so that the sheaf on $G^0$ is the sheaf which assigns to an open set the one point group), and prove an isomorphism theorem analogous to the one we are about to prove.
\end{remark}
\subsection{The van Est Isomorphism Theorem}
Before continuing, we need to make a remark about what it means for a Lie groupoid to be $n$-connected. There are two ways of doing this: one is by using a definition of homotopy groups for Lie groupoids (as in \cite{Xutu}, \cite{noohi}), in which case $n$-connected would mean that the first $n$-homotopy groups are trivial. Equivalently, this means that the classifying space is $n$-connected.
\vspace{3mm}\\Before stating and proving the isomorphism theorems, we will state two lemmas:
\begin{lemma}\label{bar}
Consider the double Lie groupoid:
\begin{equation}\label{dd}
    \begin{tikzcd}
H\ltimes P\rtimes G \arrow[d, shift right] \arrow[d, shift left] \arrow[r, shift right] \arrow[r, shift left] & P\rtimes G \arrow[d, shift right] \arrow[d, shift left] \\
H\ltimes P \arrow[r, shift right] \arrow[r, shift left]                                                       & P                                                      
\end{tikzcd}
\end{equation}
Associated to \ref{dd} is a bisimplicial manifold. Applying the bar functor gives us the total simplicial manifold; the total simplicial manifold is given by a fiber product of the antidiagonal components (see \cite{mehtatang}, \cite{cegarra}). In this special case, the total simplicial manifold is just the simplicial manifold associated to the diagonal, which is the nerve of $H\times _G G\,.$ Since the cohomology (with respect to some sheaf) of the total simplicial manifold is the cohomology of \ref{dd}, this implies that the cohomology of \ref{dd} is the cohomology of $G\times_G H\,.$
\end{lemma}
\begin{proof}
That the total simplicial manifold is the nerve of the diagonal follows from a computation of the total simplicial manifold. There is a canonical projection map from the components of the total simplicial manifold to the components of the bisimplicial manifold, and there is a canonical inclusion map from the components of the bisimplicial manifold to the components of the total simplicial manifold. These are chain maps and are mutual inverses at the level of cocylces.
\end{proof}
\begin{lemma}\label{bisimpliciall} Suppose $Y^{\bullet,\bullet}$ is a bisimplicial topological sapce, and $X^\bullet$ is a simplicial topological sapce (considered as a bisimplicial topological space $X^{\bullet,\bullet}$ that is constant in the first $\bullet\,,$ ie. $X^{i,\,j}=X^{0,\,j}$ for all $i$). Suppose $f:Y^{\bullet,\bullet}\to X^{\bullet}$ is such that the restriction $f:Y^{i,\,j}\to X^{j}$ is a locally fibered map for all $i\,,j\,,$ and that the fibers $F^{\bullet,j}$of the map $Y^{\bullet,j}\to X^j$ are $n$-connected as simplicial spaces. Let $\mathcal{A}$ be a sheaf on $X^{\bullet}\,.$ Then the map $H^*(X^\bullet,\mathcal{A})\to H^*(Y^{\bullet,\bullet},f^{-1}\mathcal{A}) $ is an isomorphism up to degree $n\,,$ and is injective in degree $n+1\,.$ The image in degree $n+1$ consists of cohomology classes which vanish along each fiber.
\end{lemma}
\begin{proof}
This follows from a generalization of \Cref{spectral theorem}, (equivalently, this follows from a generalization of criterion 1.9.4 in \cite{Bernstein} to a mapping $X^{\bullet}\to Y$ and the Leray spectral sequence).
\end{proof}
\begin{proposition}\label{ve groupoid}
Let $f:H\to G$ be a \hyperref[surj sub]{surjective submersion} of Lie groupoids such that the fibers of $f$ are all $n$-connected. Then the van Est map, $VE\,,$ is an isomorphism up to and including degree $n\,,$ it is injective in degree $n+1\,,$ and its image in degree $n+1$ consists of those cohomology classes which are trivial along the fibers.
\end{proposition}
\begin{proof}
First we replace $H$ with its canonical fibration replacement, and then we split it using the canonical splitting. We get a map
\begin{equation}\label{other la}
   \begin{tikzcd}
H\ltimes P\rtimes G \arrow[d, shift right] \arrow[d, shift left] \arrow[r, shift right] \arrow[r, shift left] & P\rtimes G \arrow[d, shift right] \arrow[d, shift left] \arrow[r, shift right=7] & G \arrow[d, shift right] \arrow[d, shift left] \\
H\ltimes P \arrow[r, shift right] \arrow[r, shift left]                                                       & P                                                                                & G^0                                           
\end{tikzcd}
\end{equation}
and using this map we can take the inverse image of cohomology of $G\,.$ That this map has the desired isomorphism properties follows from \Cref{bisimpliciall}\,, and that the same is true for the mapping $G\times_G H\to G$ follows from \Cref{bar}\,. Finally, that the same is true for the mapping $H\to G$  follows from factoring this map as $H\to G\times_G H\to G$ and using Morita invariance of LA-groupoid cohomology.
\end{proof}
\begin{theorem}\label{ve stacks}
Let $\mathcal{F}:\mathcal{H}\to\mathcal{G}$ be a \hyperref[surj stacks]{surjective submersion} of differentiable (holomorphic) stacks, and let $\mathcal{M}$ be a family of abelian groups over $\mathcal{G}\,.$ Suppose that the fibers of $\mathcal{F}$ are all $n$-connected. Then the van Est map, $\mathcal{V}\mathcal{E}\,,$ is an isomorphism up to and including degree $n\,,$ is injective in degree $n+1\,,$ and its image in degree $n+1$ consists of cohomology classes which vanish along the leaves of the associated foliation.
\end{theorem}
\begin{proof}
After choosing a presentation of $\mathcal{F}$ and $\mathcal{M}\,,$ this follows from \Cref{ve groupoid}.
\end{proof}
\begin{remark}
In light of \Cref{ve groupoid}, one way of showing that a class $\alpha\in H^n(\mathfrak{g},M)$ integrates to $H^n(G,M)$ is to show that there is a wide subgroupoid $\iota:H\xhookrightarrow{}G\,,$ with $n$-connected fibers, such that $\alpha$ integrates to $\iota^!\mathfrak{g}\,.$ Given this, a natural question is the following: if we pull back $\alpha$ to $\mathfrak{h}$ and this class integrates to $H\,,$ under what circumstances will it integrate to $\iota^!\mathfrak{g}\,?$ In degree one the answer seems to be always.
\end{remark}
\section{Computations Using van Est}
In this section we will give some example computations. We have several different models that we can use to compute the foliated cohomology of a surjectice submersion of stacks $f:H\to G\,.$ One is given by the LA-groupoid associated to the fibration replacement, a second one is given by the LA-groupoid associated to the canonical cofibration, and a third is obtained from the proof \Cref{ve groupoid}: we can compute the LA-groupoid cohomology as the ``$G$-invariant" cohomology of the foliation 
\begin{equation}\label{inv}
    \begin{tikzcd}
H\rtimes P \arrow[d, shift right] \arrow[d, shift left] \arrow[r, shift right=7] & G^0 \arrow[d, shift right] \arrow[d, shift left] \\
P                                                                                & G^0                                             
\end{tikzcd}
\end{equation}
This map is the map in the bottom row of \ref{other la}. Note that since $P\rtimes G$ is Morita equivalent to a manifold, $H^0\,,$ the leafwise cohomology of \ref{other la} is the $G$-invariant cohomology of \ref{inv} (compare this with the computation of the van Est map in part 1). Furthermore, the $G$-invariant forms on \ref{inv} can be identified with forms on $f^!\mathfrak{g}\,.$ Therefore, \ref{other la} gives us a method computing the ``van Est map" from $H^*(G,M)$ to $H^*(f^!\mathfrak{g},M)\,.$ We put van Est in quotations here, because $VE$ really maps into the foliated cohomology, but using Morita invariance we may identify it with a map into the cohomology of $f^!\mathfrak{g}\,.$ We will do this from now on.
\vspace{3mm}\\While the simplest model to use when computing the van Est map is the LA-groupoid associated to the fibration replacement, the simplest model to use when computing cohomology is the LA-groupoid associated to the canonical cofibration. Therefore, when computing cohomology, we will use this model.  
\vspace{3mm}\\Now for the strategy. Suppose one wants to compute Lie groupoid cohomogy of some groupoid $G$ up to degre $n\,.$ The idea is to choose some wide subgroupoid $H$ whose cohomology is easier to compute (though it doesn't necessarily have to be a subgroupoid) and for which the fibers of the map $H\to G$ are at least $(n-1)$-connected (but preferably $n$-connected). If the coefficients take values in a representation, then a good choice is a proper subgroupoid since the cohomology a priori vanishes in positive degree. From a cohomological perspective, the generalization of maximal compact subgroups from Lie theory are wide and proper subgroupoids for which the inclusion map has fibers which are connected and have vanishing homotopy groups in positive degree. These do not necessarily exist.
\subsection{Examples}
Now in all of the following examples, the cocycle described on the Lie groupoid is just specified by a function on one of the levels of the nerve, so we will not be needing a resolution. In this case then, for the double groupoid \ref{dd} there are three differentials, a vertical differential $\delta^*_v\,,$ a horizontal differential $\delta^*_h\,,$ and the foliated de Rham differential $d\,.$ These differentials all commute, and the associated bisimplicial complex is bigraded with respect to the vertical and horizontal directions. The differential on the total complex will be 
\begin{equation}\label{differential}
    \delta_v^*+(-1)^{\text{deg v}}\delta_h^*+(-1)^{\text{deg v+deg h}}\,d\,,
\end{equation}
where deg v, deg h mean the degrees in the vertical and horizontal directions, respectively. Note that $-1$ here denotes inversion in the module we are taking cohomology with respect to, eg. if $f$ is an $S^1$-valued function, then $(-f)(x)$ means $1/f(x)\,;$ if $f$ is $\mathbb{R}$-valued, $(-f)(x)$ means $-f(x)\,.$ 
\vspace{3mm}\\Similarly, for the LA-groupoids in the following examples (corresponding to $f^!\mathfrak{g}\,,$ where $f:H\to G$ is a homomorphism), the total differential will be 
\begin{equation}
    d+(-1)^{\bullet}\delta_h^*\,,
    \end{equation}
where $d$ is the Chevalley-Eilenberg differential and $\bullet$ is the form degree. 
\vspace{3mm}\\In general, there will be an additional differential due to the fact that if we use sheaves that are not acyclic on manifolds then we must take a resolution in order to describe general cocycles. Therefore, in the context of the preceding paragraphs there would be four and three differentials comprising the total differential, respectively.
\begin{exmp}
Conisder the groupoid $\mathbb{R}^*\,,$ We wish to compute the cohomology $H^*(\mathbb{R}^*,S^1)\,.$ The van Est from part 1 isn't useful here because $\mathbb{R}^*$ isn't connected. Furthermore, since $S^1$ isn't a vector space, van Est's original result doesn't apply. 
\vspace{3mm}\\
We will make use of the maximal compact subgroup $\mathbb{Z}/2\xhookrightarrow{}\mathbb{R}^*\,.$ Since the fibers of this map are contractible, \Cref{ve stacks} tells us that we can compute cohomology as the foliated cohomology of this map. The model we will use to do this is the one associated to the cofibration replacement, the LA-groupoid is the following:
\begin{equation}
 \begin{tikzcd}
\mathbb{Z}/2\times\mathbb{R} \arrow[r, shift right] \arrow[r, shift left] \arrow[d] & \mathbb{R} \arrow[d] \\
\mathbb{Z}/2 \arrow[r, shift right] \arrow[r, shift left]                           & *                   
\end{tikzcd}   
\end{equation}
The Lie algebroid differentials here are trivial since $\mathbb{R}$ is abelian, and the groupoid in the top row is just the trivial bundle of $\mathbb{Z}/2$ groups over $\mathbb{R}\,.$ Now the degree $0$ cohomology of this LA-groupoid is just $S^1\,.$ 
\vspace{3mm}\\A cohomology class in degree $1$ is given by a closed Lie algebra $1$-form on $\mathbb{R}$ (and again, any form here is automatically closed since $\mathbb{R}$ is abelian) and a homomorphism $\mathbb{Z}/2\to S^1\,;$ the compatibility condition in this example is trivial. There is only one nontrivial homomorphim $\mathbb{Z}/2\to S^1\,,$ therefore, the cohomology in degree 1 is just $\mathbb{R}\times \mathbb{Z}/2\,.$  
\vspace{3mm}\\Note the Lie algebra $\mathbb{R}$ has no forms in degree higher than $1\,,$ therefore cohomology classes in degree $n>1$ are given by a function $f_1:(\mathbb{Z}/2)^{n-1}\times\mathbb{R}\to\mathbb{R}\,,$ linear in $\mathbb{R}$ (representing a $1$-form), and a group cocycle for $f_2:(\mathbb{Z}/2)^n\to S^1$ such that the pair $f_1,f_2$ satisfies the compatibility condition. The only nontrivial compatibility conidtion in this example is that $\delta^*f_1=0\,,$ and this can be true if and only if $n$ is odd, in which case the cocycle is necessarily trivial since the $\mathbb{Z}/2$ action on $\mathbb{R}$ is trivial. Therefore, the cohomology in degree $n>1$ is just $H^{n}(\mathbb{Z}/2,S^1)\,.$
\vspace{3mm}\\By the exponential sequence $0\to \mathbb{Z}\to\mathbb{R}\to S^1\to 0$ and the fact that the cohomology valued in a vector space of a compact group is trivial, we find that the cohomology in degree $n>1$ is just $H^{n+1}(\mathbb{Z}/2,\mathbb{Z})\,.$ One way of computing this is by using the fact that $B(\mathbb{Z}/2)\cong \mathbb{R}P^{\infty}\,;$ using this, we see that $H^{n+1}(\mathbb{Z}/2,\mathbb{Z})=\mathbb{Z}/2$ if $n>1$ is odd, and is $0$ otherwise.
\vspace{3mm}\\Summarizing, we get that 
\begin{equation}
    H^n(\mathbb{R}^*,S^1)= \begin{cases}
        S^1 & n=0 \\
        \mathbb{R}\times\mathbb{Z}/2 & n=1 \\
        \mathbb{Z}/2 & n>1\text{ is odd} \\
        0 & n>1\text{ is even}\\
    \end{cases}
\end{equation}
We can explicitly write down generators. In degree $1\,,$ the cohomology classes are generated by the following cocycles: $f_1(x)=e^{ia\log{|x|}}\,,$ where $a\in\mathbb{R}\,,$ and $f_2(x)= -1$ if $x<0$ and is $1$ otherwise. In degree $n>1$ where $n$ is odd, the generating cocycle is given by $f:\mathbb{R}^{*n}\to S^1\,,f(x_1,\ldots,x_n)=-1$ if $x_1,\ldots,x_n<0\,,$ and is equal to $1$ otherwise. 
\end{exmp}
In the next example we will compute the van Est map in degree 1, and we will then compute the cohomology in all degrees.
\begin{exmp}
Consider the smooth Lie groups $S^1\xhookrightarrow{}\mathbb{C}^*\,.$  Since the fibers of the map $S^1\xhookrightarrow{}\mathbb{C}^*$ are contractible, \Cref{ve stacks} tells us that we can compute the cohomology at the level of LA-groupoids. Let's compute the van Est map in degree one, with coefficients in $\mathbb{C}^*\,.$ Degree one cohomology classes with coefficients in $\mathbb{C}^*$ are homomorphisms $\mathbb{C}^*\to\mathbb{C}^*\,,$ which are generated by $f(z)=z\,,g(z)=|z|^{\gamma}\,,\gamma\in\mathbb{C}\,,$ so $H^1(\mathbb{C}^*,\mathbb{C}^*)=\mathbb{Z}\times\mathbb{C}\,.$ 
\begin{itemize}
\item On the canonical fibration replacement, the van Est map is just given by the pullback with respect to the projection onto the third fact::
\begin{equation}
 \begin{tikzcd}
S^1\ltimes\mathbb{C}^*\rtimes\mathbb{C}^* \arrow[d, shift right] \arrow[d, shift left] \arrow[rr, shift right=7] &  & \mathbb{C}^* \arrow[d, shift right] \arrow[d, shift left] \\
\mathbb{C}^*                                                                                                     &  & *                                                        
\end{tikzcd}
\end{equation}
Therefore, \begin{equation}\label{ve ex}
\mathcal{VE}(f)(e^{i\theta},\lambda,z)=z\,,\mathcal{VE}(g)(e^{i\theta},\lambda,z)=|z|^{\gamma}\,.
\end{equation}
\item Now the computation of the van Est map, when mapping into the LA-groupoid associated to the canonical cofibration, is more involved. The double groupoid morphism we get is 
\begin{equation}\label{s1}
\begin{tikzcd}
S^1\ltimes \mathbb{C}^*\rtimes \mathbb{C}^* \arrow[r, shift left] \arrow[r, shift right] \arrow[d, shift right] \arrow[d, shift left] & \mathbb{C}^*\rtimes \mathbb{C}^* \arrow[d, shift right] \arrow[d, shift left] \\
S^1\ltimes \mathbb{C}^* \arrow[r, shift right] \arrow[r, shift left]                                              & \mathbb{C}^*  
\end{tikzcd}
\longrightarrow
\begin{tikzcd}
\mathbb{C}^* \arrow[r, shift left] \arrow[r, shift right] \arrow[d, shift right] \arrow[d, shift left] & \mathbb{C}^* \arrow[d, shift right] \arrow[d, shift left] \\
\bullet \arrow[r, shift right] \arrow[r, shift left]                                              & \bullet
\end{tikzcd}
\end{equation}
Let's first apply the van Est map to $f\,,$ which lives on the top right corner of the diagram on the right (see \ref{differential} for the differential). First we will pull back $f$ to $\mathbb{C}^*\ltimes \mathbb{C}^*$ via the projection $p$ onto the second factor. The map we get is $p^*f(\lambda\,,z)=z\,.$
\vspace{3mm}\\Now $p^*f=\delta_v^*w\,,$ where $w:\mathbb{C}^*\to\mathbb{C}^*\,,$ $w(\lambda)=\lambda\,.$ We have that
\begin{equation*}
    -\text{dlog}\,1/w=\frac{d\lambda}{\lambda}\,,(\delta^*_h w)(e^{i\theta},\lambda)=e^{i\theta}\,.
    \end{equation*}
thus the pair 
\begin{equation*}
\bigg(e^{i\theta}\,,\frac{d\lambda}{\lambda}\bigg)
\end{equation*} 
form a cocycle. Similarly, we can apply the van Est map to $g\,,$ and we'll get
\begin{equation*}
\bigg(1\,,\frac{\gamma}{2}\Big(\frac{d\lambda}{\lambda}+\frac{d\bar{\lambda}}{\bar{\lambda}}\Big)\bigg)
\end{equation*}
Therefore, we arrive at the following\footnote{Note that, this isn't quite where the van Est map is supposed to map to. We will discuss this further after completing this computation.}:
\begin{equation}\label{ve f}
    \mathcal{VE}(f)=\bigg(e^{i\theta}\,,\frac{d\lambda}{\lambda}\bigg)\,, \mathcal{VE}(g)=\bigg(1\,,\frac{\gamma}{2}\Big(\frac{d\lambda}{\lambda}+\frac{d\bar{\lambda}}{\bar{\lambda}}\Big)\bigg)\,.
\end{equation}
\vspace{3mm}\\Now the fibers of the map $S^1\xhookrightarrow{}\mathbb{C}^*$ are contractible, which means $\mathcal{VE}$ is an isomorphism in all degrees. We will now show that the cocycles we've chosen do indeed generate the degree one cohomology by computing them at the level of the LA-groupoid.
\vspace{3mm}\\One forms on $\mathbb{C}^*$ which pull back via $\delta^*_v$ to give $0$ are invariant one forms, which are of the form 
\begin{equation*}
 \alpha\frac{d\lambda}{\lambda}\,,  \beta\frac{d\bar{\lambda}}{\bar{\lambda}}\,,
\end{equation*}
for $\alpha\,,\beta\in\mathbb{C}\,.$
Now when we pull back 
\begin{equation}
\alpha\frac{d\lambda}{\lambda}+ \beta\frac{d\bar{\lambda}}{\bar{\lambda}}
\end{equation}
via $\delta_h^*\,,$ we get $(\alpha-\beta) \,d\theta\,,$ which is only dlog exact if $\alpha-\beta\in\mathbb{Z}\,;$ in this case, the corresponding cocycle is given by
\begin{equation}
    \Big(e^{i(\alpha-\beta)\theta}, \alpha\frac{d\lambda}{\lambda}+ \beta\frac{d\bar{\lambda}}{\bar{\lambda}}\Big)
\end{equation}
Now we have the following equality:
\begin{equation}
    \alpha\frac{d\lambda}{\lambda}=(\alpha-\beta)\frac{d\lambda}{\lambda}+\beta\frac{d\lambda}{\lambda}\,.
\end{equation}Using this equality, we see that the degree one cohomology is generated by the following cocycles:
\begin{equation}
     \Big(e^{i\alpha\theta}, \alpha\frac{d\lambda}{\lambda}\Big)\,,\;\; \bigg(1\,,\beta\Big(\frac{d\lambda}{\lambda}+\frac{d\bar{\lambda}}{\bar{\lambda}}\Big)\bigg)\,,\;\alpha\in\mathbb{Z}\,,\beta\in\mathbb{C}\,.
\end{equation}
Therefore, we see that $H^1(\mathbb{C}^*,\mathbb{C}^*)=\mathbb{C}\times\mathbb{Z}\,,$ agreeing with what we said earlier.
\item Actually, looking at \ref{ve f}, we see that we haven't quite described the cohomology classes on an LA-groupoid. We will now show what $\mathcal{VE}(f)\,,\mathcal{VE}(g)$ look like on the two LA-groupoids we've been discussing. We will start with the LA-groupoid associated to the canonical cofibration, followed by the LA-groupoid associated to the foliation. The LA-groupoid associated to the canonical cofibration is given by
\begin{equation}
    \begin{tikzcd}
S^1\times \mathbb{R}\ltimes\mathbb{C} \arrow[r, shift right] \arrow[r, shift left] \arrow[d] & \mathbb{C} \arrow[d] \\
S^1 \arrow[r, shift right] \arrow[r, shift left]                                             & *                   
\end{tikzcd}
\end{equation}
Here, $S^1$ is acting trivially. A degree 1 cocycle is given by a one form for the Lie algebra $\mathbb{C}\to *$ together with a homomorphism $S^1\to \mathbb{C}^*$ satisfying the compatibility condition. In this context, \ref{ve f} takes the form 
\begin{equation}
     \mathcal{VE}(f)=\big(e^{i\theta}\,,d\lambda\big)\,, \mathcal{VE}(g)=\big(1\,,\frac{\gamma}{2}\big(d\lambda+d\bar{\lambda}\big)\big)\,.
\end{equation}
Tto arrive at this we just had to evaluate the one forms at the identity in $\mathbb{C}^*\,.$
\item Finally, the van Est map is (again) really a map into the foliated cohomology of $[*/S^1]\,,$ which we can describe using the canonical fibration replacement: 
\begin{equation}
    \begin{tikzcd}
S^1\ltimes\mathbb{C}^*\rtimes\mathbb{C}^* \arrow[d, shift right] \arrow[d, shift left] \arrow[rr, shift right=7] &  & \mathbb{C}^* \arrow[d, shift right] \arrow[d, shift left] \\
\mathbb{C}^*                                                                                                     &  & *                                                        
\end{tikzcd}
\end{equation}We should describe these cocycles here too (using the isomorphism in \Cref{equiv la group}). The result is
    \begin{equation}
     \mathcal{VE}(f)=\bigg(e^{i\theta}\,,\frac{d\lambda}{\lambda}\bigg)\,, \mathcal{VE}(g)=\bigg(1\,,\frac{\gamma}{2}\Big(\frac{d\lambda}{\lambda}+\frac{d\bar{\lambda}}{\bar{\lambda}}\Big)\bigg)\,.
\end{equation}
which is the same as \ref{ve f}, but here $e^{i\theta}\,,\lambda\,,\bar{\lambda}$ refer to the functions $S^1\times\mathbb{C}^*\times\mathbb{C}^*\to\mathbb{C}^*,,\, (e^{i\theta},\lambda,z)\mapsto e^{i\theta}\,,\lambda\,,\,\bar{\lambda}\,,$ respectively, and $d\lambda\,,d\bar{{\lambda}}$ are foliated one forms along the $\mathbb{C}^*$ in the base. These cocycles are equivalent to the ones in \ref{ve ex}.
\item We will now finish the computation of the cohomology $H^*(\mathbb{C}^*,\mathbb{C}^*)\,.$ First note that since $S^1\times\mathbb{R}\cong \mathbb{C}^*\,,$ from \Cref{semidirect} we get that \ref{s1} is Morita equivalent to the following  LA-groupoid, which provides a simpler model for computing the cohomology:
\begin{equation}
    \begin{tikzcd}
S^1\times\mathbb{R} \arrow[r, shift right] \arrow[r, shift left] \arrow[d] & \mathbb{R} \arrow[d] \\
S^1 \arrow[r, shift right] \arrow[r, shift left]                           & *                   
\end{tikzcd}
\end{equation}
Therefore, the cohomology of $\mathbb{C}^*$ is quite literally the cohomology of the normal bundle of $S^1\xhookrightarrow{}\mathbb{C}^*\,.$ Here, the Lie algebroid differentials are all $0$ and the groupoid in the top row is the action groupoid associated to the trivial action of $S^1$ on $\mathbb{R}\,.$ Furthermore, since the cohomology of a proper groupoid vanishes in positive degrees, we can use the exponential seqeuence together with the fact that $B{S^1}=\mathbb{C}P^\infty$ to help us compute the cohomology. Putting this together, we can derive the cohomology groups in all degrees. They are given by:
\begin{equation}
    H^n(\mathbb{C}^*,\mathbb{C}^*)= \begin{cases}
        \mathbb{C}^* & n=0 \\
        \mathbb{C}\times\mathbb{Z} & n=1 \\
        \mathbb{Z} & n>1\text{ is odd} \\
        0 & n>1\text{ is even}\\
    \end{cases}
\end{equation}
\end{itemize}
\end{exmp}
\begin{exmp}
\vspace{5mm}Consider the Lie group $\mathbb{C}^2\rightrightarrows *\,,$ with coordinates $(w,z)$ and consider the subspace $\mathbb{C}\rightrightarrows *$ given by $w=z\,.$ Consider the trivial $\mathbb{C}$-module. We have a degree two cycle given by $f(w_1,z_1,w_2,z_2)=w_1z_2\,,$ whose corresponding extension is the complex Heisenberg group. Let's compute the van Est map here.
\begin{itemize}
    \item First, the true van Est map (as defined in this thesis) is given by pullback with respect to the projection onto the third factor:
\begin{equation}
   \begin{tikzcd}
(\mathbb{C}\times\mathbb{C})\ltimes\mathbb{C}^2\rtimes(\mathbb{C}^2\times\mathbb{C}^2) \arrow[d, shift right=2] \arrow[d] \arrow[d, shift left=2] &  & \mathbb{C}^2\times\mathbb{C}^2 \arrow[d, shift right=2] \arrow[d] \arrow[d, shift left=2] \\
\mathbb{C}\rtimes\mathbb{C}^2\rtimes\mathbb{C}^2 \arrow[d, shift right] \arrow[d, shift left] \arrow[rr, shift left]                              &  & \mathbb{C}^2 \arrow[d, shift right] \arrow[d, shift left]                                 \\
\mathbb{C}^2                                                                                                                                      &  & *                                                                                        
\end{tikzcd}
\end{equation}
Therefore, we have that 
\begin{equation}\label{ve ex2}
    \mathcal{VE}(f)(a_1,a_2,x,y,w_1,z_1,w_2,z_2)=w_1z_2\,.
\end{equation}
\item Now let's compute the van Est map with respect to the LA-groupoid associated to the canonical cofibration (see \ref{differential} for the differential). We have the following diagram:
\begin{equation}
    \begin{tikzcd}
\mathbb{C}\ltimes\mathbb{C}^2\rtimes\mathbb{C}^2 \arrow[d, shift right] \arrow[d, shift left] \arrow[r, shift right] \arrow[r, shift left] & \mathbb{C}^2\rtimes\mathbb{C}^2 \arrow[d, shift right] \arrow[d, shift left] \arrow[r, shift right=7] & \mathbb{C}^2 \arrow[d, shift right] \arrow[d, shift left] \\
\mathbb{C}\ltimes\mathbb{C}^2 \arrow[r, shift right] \arrow[r, shift left]                                                                 & \mathbb{C}^2                                                                                          & *                                                        
\end{tikzcd}
\end{equation}
Now since the cocycle we have lives in degree 2, we must further compute the nerves:
\begin{equation}
  \begin{tikzcd}
                                                                                                           &                                                                                                                                            & \mathbb{C}^2\rtimes(\mathbb{C}^2\times\mathbb{C}^2) \arrow[d, shift right=2] \arrow[d, shift left=2] \arrow[d] & \mathbb{C}^2\times\mathbb{C}^2 \arrow[d, shift right=2] \arrow[d] \arrow[d, shift left=2] \\
                                                                                                           & \mathbb{C}\ltimes\mathbb{C}^2\rtimes\mathbb{C}^2 \arrow[d, shift right] \arrow[d, shift left] \arrow[r, shift right] \arrow[r, shift left] & \mathbb{C}^2\rtimes\mathbb{C}^2 \arrow[d, shift right] \arrow[d, shift left] \arrow[r, "p", shift right=7] & \mathbb{C}^2 \arrow[d, shift right] \arrow[d, shift left]                                 \\
(\mathbb{C}\times\mathbb{C})\ltimes\mathbb{C}^2 \arrow[r, shift right=2] \arrow[r, shift left=2] \arrow[r] & \mathbb{C}\ltimes\mathbb{C}^2 \arrow[r, shift right] \arrow[r, shift left]                                                                 & \mathbb{C}^2                                                                                                   & *                                                                                        
\end{tikzcd}
\end{equation}
Now we have that 
\begin{equation}
p^*f(x,y,w_1,z_1,w_2,z_2)=f(w_1,z_1,w_2,z_2)=w_1z_2=\delta_v^*f_1(x,y,w_1,z_1,w_2,z_2)\,,
\end{equation}
where $f_1(x,y,w,z)=xz\,.$ For the next step, we apply $(-\delta_h^*-d)$ to $-f_1$ (recall that $d$ is the fiberwise de Rham differential).  We get 
\begin{equation}
\delta_h^*f_1(a,x,y,w,z)=az\,,df_1=zdx
\end{equation}
Now $df_1=\delta_v^*\,ydx$ and $\delta_h^*f_1(a,x,y,w,z)=\delta_v^*f_2\,,$ where $f_2(a,x,y)=ay\,.$ Next we have that 
\begin{equation}
(\delta_h^*+d)(-ydx)=-yda-adx-ada+dx\wedge dy\,,
\end{equation}
and $\delta_h^*(-f_2)=-a_1a_2\,,-d(-f_2)=ady+yda\,.$ So in the end, the cocycle we get is 
\begin{equation}\label{cocycle h}
-a_1a_2+a(dy-dx-da)+dx\wedge dy\,.
\end{equation}
\item Now, the LA-groupoid corresponding to the canonical cofibration is
\begin{equation}
\begin{tikzcd}
    \mathbb{C}\times\mathbb{C}\ltimes\mathbb{C}^2 \arrow[d] \arrow[r, shift right] \arrow[r, shift left] & \mathbb{C}^2 \arrow[d] \\
\mathbb{C} \arrow[r, shift right] \arrow[r, shift left]                                              & *                     
\end{tikzcd}
\end{equation}
and further computing the nerve, we have 
\begin{equation}\label{cof h}
    \begin{tikzcd}
(\mathbb{C}\times\mathbb{C})\times(\mathbb{C}\times\mathbb{C})\ltimes\mathbb{C}^2 \arrow[d] \arrow[r, shift right=2] \arrow[r] \arrow[r, shift left=2] & \mathbb{C}\times\mathbb{C}\ltimes\mathbb{C}^2 \arrow[d] \arrow[r, shift right] \arrow[r, shift left] & \mathbb{C}^2 \arrow[d] \\
\mathbb{C}\times\mathbb{C} \arrow[r, shift right=2] \arrow[r] \arrow[r, shift left=2]                                                                  & \mathbb{C} \arrow[r, shift right] \arrow[r, shift left]                                              & *                     
\end{tikzcd}
\end{equation}
and we can identity the cocycle \ref{cocycle h} with a cocycle on \ref{cof h}, where $dx\,,dy$ are the Lie algebra $1$-forms on $\mathbb{C}^2\to *$ and $a$ is the coordinate on the Lie group $\mathbb{C}\rightrightarrows *\,.$
\item Finally, we can describe this cocycle on the foliation of $[*/\mathbb{C}]$ (using \Cref{equiv la group}), which is given by
\begin{equation}
    \begin{tikzcd}
\mathbb{C}\ltimes\mathbb{C}^2\rtimes\mathbb{C}^2 \arrow[d, shift right] \arrow[d, shift left] \arrow[rr, shift right=7] &  & \mathbb{C}^2 \arrow[d, shift right] \arrow[d, shift left] \\
\mathbb{C}^2                                                                                                            &  & *                                                        
\end{tikzcd}
\end{equation}
Forming the next part of the nerve of the groupoid on the left, we have 
\begin{equation}
    \begin{tikzcd}
(\mathbb{C}\times\mathbb{C})\ltimes\mathbb{C}^2\rtimes(\mathbb{C}^2\times\mathbb{C}^2) \arrow[d] \arrow[d, shift left=2] \arrow[d, shift right=2] \\
\mathbb{C}\ltimes\mathbb{C}^2\rtimes\mathbb{C}^2 \arrow[d, shift right] \arrow[d, shift left]                                                      \\
\mathbb{C}^2                                                                                                                                     
\end{tikzcd}
\end{equation}
The cocyce is given by $-a_1a_2+a(dy-dx-da)+dx\wedge dy\,,$ where $-a_1a_2$ lives on the top, $a(dy-dx-da)$ lives in the middle and $dx\wedge dy$ lives on the bottom. This is equivalent to the cocycle given by \ref{ve ex2}.
\end{itemize}
\end{exmp}
\part{Morita Equivalences of Lie Algebroids and LA-Groupoids, Higher Generalized Morphisms, Future Directions}
\section*{Brief Summary of Part 3}By the end of this part (in~\Cref{finale}) we will have combined $n$-equivalences of Lie algebroids with Morita equivalences of Lie groupoids in the category of LA-groupoids (which contains both Lie groupoids and Lie algebroids as objects).\ This category has some \textit{surprising} but natural properties, including generalized morphisms between Lie algebroids and Lie groupoids (and even Morita equivalences between them). These properties include the following (though like much of Part 3, details need to be filled in):
\begin{enumerate}
    \item Manifolds $X\,,Y$ are Morita equivalent in this category if and only if they are diffeomorphic.
    \item More generally, Lie groupoids $G\,,H$ are Morita equivalent in this category if and only if they are Morita equivalent in the category of Lie groupoids.
    \item Two tangent bundles $TX\,,TY$ are Morita equivalent in this category if and only if $X\,,Y$ are homotopy equivalent.
    \item There is a canonical generalized morphism $\mathcal{I}:\mathfrak{g}\to G\,,$ which represents the integration functor. If $G\rightrightarrows G^0$ is source $n$-connected then this morphism is an $n$-equivalence, and dual to this there is a canonical $n$-equivalence $\mathcal{D}:G\to \mathfrak{g}\,,$ which represents the differentiation functor. If $n=\infty$ these generalized morphisms are Morita equivalences.
    \item With regards to the previous two points, the van Est map is given by the pullback $\mathcal{I}^*:H^{\bullet}(G,M)\to H^{\bullet}(\mathfrak{g},M)\,.$ If $G$ is source $n$-connected then integration is given by $\mathcal{D}^*:H^{\bullet}(\mathfrak{g},M)\to H^{\bullet}(G,M)\,,$ for $\bullet\le n\,.$ 
    \item A Lie algebroid $\mathfrak{g}$ is integrable if and only if it is $1$-equivalent to some Lie groupoid $G\,.$
    \item This category induces a notion of homotopy equivalence on the category of Lie groupoids: a generalized morphism $P:G\to H$ is a homotopy equivalence if the induced generalized morphism in the category of LA-groupoids, $TP:TG\to TH\,,$ is a Morita equivalence. In particular, a Morita equivalence of Lie groupoids induces a homotopy equivalence. In addition, we get a natural notion of $n$-equivalence of Lie groupoids (in the homotopy sense).
    \item A finite dimensional classifying space of a Lie groupoid $G$ is just a manifold $BG$ which is homotopy equivalent to $G\,.$
    \item If $EG\to BG$ is finite dimensional, then the Atiyah algebroid $\text{at}(EG)$ is Morita equivalent to $G\,.$ In particular, if $G$ is discrete then $\text{at}(EG)=T(BG)\,,$ therefore $G$ is Morita equivalent to $T(BG)\,.$
    \item Due to points 2, 3 and 4, we get the following result: suppose that $\mathcal{P}$ assigns to each LA-groupoid some property (eg.\ its cohomology) that is invariant under $n$-equivalence. Then if $X$ is homotopy equivalent to $Y\,;$ if $H\rightrightarrows H^0$ is Morita equivalent to $K\rightrightarrows K^0\,;$ if $G\rightrightarrows G^0$ is source $n$-connected, we get that $\mathcal{P}(TX)\cong \mathcal{P}(TY)\,;\mathcal{P}(H)\cong \mathcal{P}(K)\,;\mathcal{P}(G)\cong \mathcal{P}(\mathfrak{g})\,,$ respectively.
    \end{enumerate}
\chapter{Lie Algebroids as Generalized Homotopy Theory}
In this chapter we want to discuss generalized morphisms of Lie algebroids, defined in \Cref{MEA} (see summary of Part 3 on the previous page). Some of what is said will be known to experts. Throughout this thesis, we have made use of Morita equivalences of LA-groupoids, however we don't claim that all Morita equivalences should be of this form, ie. there should be a more general notion. For Lie algebroids, the definition of Morita equivalences we've been using reduces to isomorphisms, but this is all we needed for this thesis. We want to explore a more general notion here for Lie algebroids. First we will motivate the definition, however the idea is that the right notion of Morita equivalences of Lie algebroids should give a generalization of homotopy theory in the smooth category. In particular, we conjecture the following result:
\begin{conjecture}
Suppose $A\to X\,, B\to Y$ are $n$-equivalent Lie algebroids. There is an equivalence of categories between representations up to homotopy (of length at most $n$) of $A$ and $B\,.$ 
\end{conjecture}
\vspace{1mm}We will begin with a discussion and provide evidence for our claims, and then we will state precise definitions. Part 3 of this thesis was written after discussions with Francis Bischoff. At the end of this chapter we make a remark on (perhaps far-fetched) connections with geometric/deformation quantization. Let us emphasize that Part 3 of this thesis is largely conjectural and significant details need to be filled in for completion.
\section{Some Definitions of Morita Equivalences of Lie Algebroids and \texorpdfstring{$\Pi_{\infty}(A)$}{Pi infinity}}
Before getting into the abstract details that are to follow, we wish to prime the reader with a concrete result. It relates to the following conundrum: there is a functor from Lie groupoids to Lie algebroids, however Morita equivalences don't seem to get sent to anything like  Morita equivalences, eg. Pair$(X)$ is Morita equivalent to a point, but $TX$ doesn't seem like it should be Morita equivalent to the zero vector space. Our point of view is the following: it is often more instructive to think of Lie algebroids as structures which quotient to a Lie groupoids, rather than as infinitesimal approximations. This is already familiar in the case of connected abelian groups, where the Lie algebra is identified with the universal cover. This perspective would suggest that, rather than 
\begin{equation}
   \begin{tikzcd}
H \arrow[rr, "\text{Morita map}"] &  & G \arrow[r, Rightarrow] & \mathfrak{h} \arrow[rr, "\text{Morita map}"] &  & \mathfrak{g}\,,
\end{tikzcd}
\end{equation}
the implication should point in the opposite direction, ie.
\begin{equation}\label{oppo}
   \begin{tikzcd}
H \arrow[rr, "\text{Morita map}"] &  & G & \mathfrak{h} \arrow[l, Rightarrow] \arrow[rr, "\text{Morita map}"] &  & \mathfrak{g}
\end{tikzcd}
\end{equation}
That is, Morita equivalences of Lie algebroids should integrate to Morita equivalences of Lie groupoids (we will make a more clear statement later, but $H\,,G$ here should be source simply connected). To support this, let's first give a notion of Morita equivalence due to Fernandes (see~\cite{ruif}, page 198) which we wish to generalize: 
\begin{definition}\label{mlf}
Two Lie algebroids $A\to X\,,B\to Y$ are Morita equivalent if there is a manifold $Z$ with surjective submersions $\pi_X\,,\pi_Y$ to $X$ and $Y\,,$ with simply connected fibers, such that the pullbacks $\pi_X^!A\;,\pi_Y^!B$ are isomorphic. 
\end{definition}
\vspace{1mm}In particular, if $\pi:Y\to X$ is a surjective submersion with simply connected fibers and $A\to X$ is a Lie algebroid, then by this definition $\pi^!A$ is Morita equivalent to $A\,.$ Now we give the following result, which supports\ref{oppo}, and we will come back to this in~\Cref{defml} (see Lemma 1.14 in~\cite{haus} for the proof):
\begin{proposition}\label{morgl}
Let $A\to X$ be an integrable Lie algebroid, and let $\pi:Y\to X$ be a surjective submersion with simply connected fibers. Then the source simply connected groupoids integrating $\pi^!A$ and $A$ are Morita equivalent.
\end{proposition}
This proposition supports~\ref{oppo}. Now we will move on with the discussion, and we will put the preceding result into context a little later. 
\vspace{3mm}\\In the author's opinion, at least in the optimal cases, every Lie algebroid should integrate to a canonical Lie $\infty$-groupoid, and the Lie algebroid and its integrating canonical Lie $\infty$-groupoid should be equivalent in a strong sense. Roughly, the canonical Lie $\infty$-groupoid of a Lie algebroid $A\to X$ should be given by the following simplicial space, denoted $\Pi_{\infty}(A)$ (see page 3 of~\cite{zhuc}, where it is denoted by $S(A)\,;$ see also section 1.4 of~\cite{andre},~\cite{getzler},~\cite{sometitle})\footnote{This part is motivational, we don't intend to be perfectly precise, however this will lead to a concrete definition.}:
\begin{equation}\label{inftyg}
  \Pi_{\infty}^{(i)}(A)=\text{hom}(T\Delta^i,A)\,.
\end{equation}In particular, they should have the same invariants, including cohomology and representations (up to homotopy). Our perspective is that a Lie algebroid $A$ is a finite dimensional model of $\Pi_{\infty}(A)\,.$ We call $\Pi_{\infty}(A)$ the fundamental $\infty$-groupoid of $A\,.$ As a special case, the tangent bundle $TX$ should be equivalent to the fundamental $\infty$-groupoid $\Pi_\infty(X)\,.$ An important property of $\Pi_{\infty}(A)$ is that its source fibers should be weakly contractible. Now the source simply connected integration of $A\,,$ denoted $\Pi_1(A)\,,$ (assuming it exists) should be a quotient (ie. truncation) of $\Pi_{\infty}(A)\,.$ It is in this sense in which a groupoid is a quotient of a Lie algebroid.
\vspace{3mm}\\It was proven that given such a Morita equivalence (see~\Cref{mlf}), the cohomology groups of the Lie algebroids up to degree one agree, as do their category of representations. Note that this definition is the analogue of the definition of Morita equivalence of groupoids (except for the additional simply connected assumption, which we will explain). In the same paper it is mentioned that, instead of asking that the fibers be simply connected, one can ask that they are $n$-connected for some $n\,.$ In the author's opinion, the right value of $n$ is $\infty\,,$ ie. the fibers should be contractible; any other value of $n$ should be thought of as a particular case of an $n$-equivalence of Lie algebroids, which is analogous to the notion of weak $n$-equivalence from topology (see page 144 in~\cite{dieck}). We cite the following two results as evidence to our claim:
\begin{theorem}(Theorem 2 in~\cite{Crainic})\label{crainicmor}
Let $f:X\to Y$ be a surjective submersion with $n$-connected fibers, and let $B\to Y$ be a Lie algebroid. Then the Lie algebroid cohomologies of $f^!B$ and $B$ are isomorphic up to degree $n\,.$
\end{theorem}
\begin{theorem}(see Theorem 4.2 in~\cite{sparano})
Let $f:X\to Y$ be a surjective submersion with $n$-connected fibers, and let $B\to Y$ be a Lie algebroid. Then the Lie algebroids $f^!B$ and $B$ share the same deformation cohomology up to degree $n\,.$\footnote{For more properties invariant under this notion of equivalence, see Remark 3.3 in~\cite{sparano}.}
\end{theorem}
In particular, our definition of Morita equivalence of Lie algebroids will imply that two tangent bundles are Morita equivalent if and only if their underlying manifolds are homotopy equivalent. One piece of evidence supporting this is the fact that the Lie algebroid cohomology of tangent bundles is invariant under homotopy equivalences of the underlying manifolds (this includes the cohomology of constant sheaves). Further evidence supporting this is the following proposition (due to Arias Abad, Quintero V\'{e}lez and  V\'{e}lez V\'{a}squez):
\begin{theorem}(see~\cite{Quintero}, Corollary 5.1)
If $f:X\to Y$ is a smooth homotopy equivalence, then the pullback functor 
\begin{equation}
    f^*:\textbf{Loc}_{\infty}(N)\to \textbf{Loc}_{\infty}(M)
\end{equation}
is a quasi-equivalence (ie. it is a quasi-equivalence between the dg categories of $\infty$-local systems).
\end{theorem}
\vspace{1mm}Let us make a remark which is relevant to this discussion and which we will come back to later: a surjective submersion $f:X\to Y$ with $n$-connected fibers is a weak $n$-equivalence, ie. it defines an isomorphism of homotopy groups up to degree $n\,.$ For $n=\infty\,,$ $f$ is a homotopy equivalence (see Corollary 13 in~\cite{gael}).
\vspace{3mm}\\For another example, we can embed manifolds into the category of Lie algebroids by assigning to a manifold $X$ the $0$ vector bundle over it. Our definition will imply that two manifolds are Morita equivalent if and only if they are diffeomorphic. Certainly, all reasonable properties of manifolds are invariant under diffeomorphisms. Noe that the aformentioned embedding is the one that is consistent with the natural embedding of manifolds into groupoids/stacks, which assigns to $X$ the groupoid containing only identity morphisms (note that, assigning to $X$ the groupoid $\Pi_1(X)\,,$ for example, is not an embedding under Morita equivalences).
\section{Motivation}Before giving our definition of Morita equivalences of Lie algebroids (see~\Cref{MEA}), we will motivative the definition.
\vspace{3mm}\\Let's restrict to the subcategory of Lie algebroids given by tangent bundles. We believe that two tangent bundles $TX\,,TY$ should be Morita equivalent if and only if $X$ and $Y$ are homotopy equivalent. One reason for this is that homotopy equivalence seems to preserve everything that should be an invariant of Lie algerboids, in particular the Lie algebroid cohomologies of $TX\,,TY$ are the same if $X$ and $Y$ are homotopy equivalent (if they are just weak $n$-equivalent, meaning that there is a map $X\to Y$ which is an isomorphism on homotopy groups up to degree $n\,,$ then the Lie algebroid cohomologies will agree up to degree $n$).
\vspace{3mm}\\We will return to Lie algebroids in a moment, but first let's recall the definition of a Morita map for Lie groupoids. A morphism $f:H\to G$ is a Morita equivalence if the following two conditions hold:
\begin{enumerate}
    \item $H^0\times_{G^0}G^{(1)}\to G^0$ is a surjective submersion.
    \item The following diagram is a fiber product:
    \begin{equation}
        \begin{tikzcd}
H^{(1)} \arrow[r, "f"] \arrow[d, "{(s,t)}"'] & G^{(1)} \arrow[d, "{(s,t)}"'] \\
H^0\times H^0 \arrow[r, "{(f,f)}"]       & G^0\times G^0            
\end{tikzcd}
    \end{equation}
\end{enumerate}
Now we want to generalize this to Lie algebroids. The map $H^0\times_{G^0}G^{(1)}\to G^0$ has a Lie algebroid analogue: given two Lie algebroids $A\to X\,,B\to Y$ and a morphism $f:A\to B\,,$ we can form its ``mapping path space". This is given by the fiber product $X\sideset{_f}{_{s}}{\mathop{\times}}P\,,$ where $P$ is the space of algebroid paths. Let's recall the definition:
\begin{definition}\label{nodef1}(see section 1 in~\cite{rui} for more details)
Let $A\to X$ be a Lie algebroid. An algebroid path (or A-path) is given by a $C^1$ curve $\gamma:[0,1]\to A$ such that $\pi\circ\gamma$ is a $C^2$ curve in $X$ and such that \begin{equation}
    \alpha(\gamma(t))=\frac{d}{dt}\pi(\gamma(t))\,.
    \end{equation}
We denote the space of algebroid paths by $P_X\,.$  It is the same space as $\Pi_{\infty}^1(A)$ (see~\ref{inftyg}).
\end{definition}
Now $P_X$ is a Banach manifold (see section 4.2 of~\cite{rui}), and it comes with two maps onto $X\,,$ which we will denote suggestively as follows:
\begin{equation}
s\,,t:P_X\to X\,, s(\gamma)=\pi(\gamma(0))\,,t(\gamma)=\pi(\gamma(1))\,.
\end{equation}
\begin{definition}\label{nodef}
Given a map $f:A\to B$ of Lie algebroids $A\to X\,,B\to Y\,,$ we define the mapping path space to be $P_f=X\sideset{_f}{_{s}}{\mathop{\times}}P_B\,.$
\end{definition}
Now given this definition, a guess for what a Morita equivalence of Lie algebroids might be is the following: given two Lie algebroids $A\to X\,,B\to Y\,,$ a morphism $f:A\to B$ is a Morita equivalence if the following two conditions hold:
\begin{enumerate}
    \item The composition $P_f\to P_Y\xrightarrow[]{t}Y$ is a surjective submersion.
    \item The following diagram is a fiber product:
    \begin{equation}
        \begin{tikzcd}
A \arrow[r, "f"] \arrow[d, "\alpha"'] & B \arrow[d, "\alpha"'] \\
TX \arrow[r, "f_*"]                   & TY                    
\end{tikzcd}
    \end{equation}
\end{enumerate}
Now this definition doesn't work, and to see this, just consider the case of tangent bundles. Assuming $Y$ is path connected, these two conditions will hold automatically. To understand what is going on, let's consider the fundamental groupoid $\Pi_1(Y)\,.$ Given a map $f:X\to Y$ (assume they are connected for simplicity), $f^!\Pi_1(Y)$ is Morita equivalent to $\Pi_1(Y)\,,$ however it isn't necessarily Morita equivalent to the source simply connected integration of $f^!TY=TX\,,$ ie. $\Pi_1(X)\,.$ In order for this to be the case we need that the map $f$ defines an isomorphism on $\pi_1(X)\,.$ Let's emphasize this point: \textbf{the pullback of the source simply connected integration of a Lie algebroid is not necessarily equivalent to the source simply connected integration of the pullback of the Lie algebroid.}
\vspace{3mm}\\Continuing the discussion, suppose that $Y$ has vanishing homotopy groups above degree $1\,,$ so that $\Pi_1(Y)$ has contractible source fibers. In this case, in the author's opinion, $\Pi_1(Y)$ is essentially completely equivalent to $TY$ (and hence $\Pi_{\infty}(Y))\,,$ in particular they have all the same cohomology groups, for all modules and all degrees, by the van Est isomorphism theorem. However, $f^!\Pi_1(Y)$ doesn't necessarily have contractible source (or equivalently target) fibers, in order for this to be true we need $f$ to be a homotopy equivalence. 
\vspace{3mm}\\Now let's consider the case of a map $f:X\to Y$ and let's consider the following question: when is $f^!\Pi_{\infty}(Y)$ equivalent to $\Pi_{\infty}(X)\,?$ In order for this to be true, $f^!\Pi_{\infty}(Y)$ should have contractible target fibers. Explicitly, a target fiber over a point $x_0\in X$ of $f^!\Pi_\infty(Y)$ is given by:
\begin{equation}
    \{(x,\gamma):x\in X\,, \gamma:[0,1]\to Y\,,\gamma(0)=f(x)\,,\gamma(1)=f(x_0)\}\,,
\end{equation}
ie. a target fiber over $x_0$ is the homotopy fiber over $f(x_0)\,.$ Basic results about Serre fibrations from topology then imply that all target fibers are contractible if and only if $f:X\to Y$ is a homotopy equivalence.
\vspace{3mm}\\More generally then, given a map $f:X\to Y$ with a Lie algebroid $B\to Y\,,$ we see that $f^!\Pi_{\infty}(B)$ (see \cref{inftyg}) doesn't need to be equivalent to $\Pi_{\infty}(f^!B)$ (assuming the pullbacks exist). That is, for general $f\,,$ \begin{equation}
f^!\Pi_{\infty}\not\cong \Pi_{\infty}f^!\,.
\end{equation}
The property that $f^!B$ is Morita equivalent to $B$ should be equivalent to the properties that $P_f\to Y$ is surjective submersion and 
\begin{equation}
f^!\Pi_{\infty}(B)\cong \Pi_{\infty}(f^!B)\,.
\end{equation}
If $f$ is a surjective submersion with contractible fibers, then these properties should be satisfied for all $B\,,$ in particular $f^!$ and $\Pi_{\infty}$ should commute. If $f$ is a surjective submersion with (strict) fibers that have vanishing homotopy groups up to degree $n\,,$ then $f^!$ and $\Pi_{\infty}$ should commute ``up to degree $n",$ meaning that $f^!$ and $\Pi_{n}$\footnote{We will not attempt to define $\Pi_n$ here, but it should be an $n$-truncation of $\Pi_{\infty}\,,$ see Theorem 1.2 in~\cite{zhuc}. For $n=1$ it is ``functor" giving the source simply connected integration.} should commute (ignoring the issue of existence of the truncation).
\vspace{3mm}\\Now that we understand why the two preceding conditions are not enough (see the discussion on the previous page, after~\Cref{nodef}) for Lie algebroids to be equivalent, let's discuss what more is needed. In order for $f^!\Pi_{\infty}(B)$ to be equivalent to $\Pi_{\infty}(f^!B)$ the target fibers of $f^!\Pi_{\infty}(B)$ should be (weakly) contractible. These are the analogue of homotopy fibers from topology for general Lie algebroids. 
\section{Definition of Morita Equivalences of Lie Algebroids}\label{defml}
After the discussions in the previous sections, we are now ready to give our definition of Morita equivalences of Lie algebroids (see~\Cref{nodef1}~\Cref{nodef} for the definitions of $P$ and $P_f$). In particular, we can apply this definition to foliations, in which case we obtain a definition of equivalence distinct from the notion of Morita equivalence introduced in~\cite{haus}. One should compare it with the notion of homotopy equivalence of foliations introduced in~\cite{baum}\footnote{It is not clear to the author if our definition agrees with theirs.}:
\begin{definition}\label{MEA}
Let $A\to X\,,B\to Y$ be Lie algebroids. We say that a morphism $f:A \to B$ is a Morita equivalence if the following conditions hold:
\begin{enumerate}
    \item The composition $P_f\xrightarrow[]{\pi} P_Y\xrightarrow[]{t}Y$ is a surjective submersion with (weakly) contractible fibers.\footnote{Here $\pi$ is the natural map $P_f\xrightarrow[]{\pi} P_Y\,.$}
    \item The following diagram is a fiber product:
    \begin{equation}
        \begin{tikzcd}
A \arrow[r, "f"] \arrow[d, "\alpha"'] & B \arrow[d, "\alpha"'] \\
TX \arrow[r, "f_*"]                   & TY                    
\end{tikzcd}
    \end{equation}
\end{enumerate}
\end{definition}
The idea would then be to localize at Morita equivalences to obtain generalized morphisms. We also give the following definition: 
\begin{definition}\label{neq}
Let $A\to X\,,B\to Y$ be Lie algebroids. We say that a morphism $f:A \to B$ is an $n$-equivalence if the following conditions hold:
\begin{enumerate}
    \item The composition $P_f\xrightarrow[]{\pi} P_Y\xrightarrow[]{t}Y$ is a surjective submersion with $n$-connected fibers.
    \item The following diagram is a fiber product:
    \begin{equation}
        \begin{tikzcd}
A \arrow[r, "f"] \arrow[d, "\alpha"'] & B \arrow[d, "\alpha"'] \\
TX \arrow[r, "f_*"]                   & TY                    
\end{tikzcd}
    \end{equation}
\end{enumerate}
\end{definition}
One can localize at $n$-equivalences to obtain $n$-generalized morphisms. Let us emphasize that an $n$-equivalence for $n=\infty$ is a Morita equivalence.
\vspace{3mm}\\Let's \textbf{sketch} a proof that this previous definition generalizes~\Cref{mlf}:
\begin{proposition}
Let $f:X\to Y$ be a surjective submersion with $n$-connected fibers, and let $A\to Y$ be a Lie algebroid. Then $f^!A$ is $n$-equivalent to $A\,.$
\begin{proof}(\textbf{sketch})
The proof essentially follows the proof of Lemma 1.14 in~\cite{haus}. First note that since $f$ is surjective, $P_f\to Y$ is also surjective. Furthermore, $P_X\to Y$ is a submersion\footnote{This was communicated to the author by Rui Loja Fernandes.}, therefore since $f$ is also a submersion it follows that $P_f\to Y$ is a surjective submersion. Now the map $\pi:P_f\to P_Y$ induces a natural map $(t\circ\pi)^{-1}(y)\to t^{-1}(y)$ for $y\in Y\,,$ and the fibers of this map are the fibers of $f\,,$ which are $n$-connected. Since $t^{-1}(y)$ is contractible, it follows that the target fibers of $P_f$ are $n$-connected. The second condition is satisfied by the definition of $f^!A\,.$
\end{proof}
\end{proposition}
\vspace{2mm}Now given a Lie groupoid $G\,,$ there is a natural action of $TG$ on $\mathfrak{g}\,.$ Under the \textbf{assumption} that there is also a natural ``action" of $TP$ on $\mathfrak{g}\,,$ we get the following result:
\begin{conjecture}
Suppose $f:A\to B$ is a map of Lie algebroids which satisfies the hypotheses of~\Cref{neq}. Then $H^*(A,\mathcal{O})=H^*(B,\mathcal{O})\,.$
\end{conjecture}
\begin{proof}(\textbf{idea})
The idea is the following: $P_f$ comes with two natural maps $\pi_X\,,\pi_Y$ to $X\,,Y\,,$ respectively ($\pi_Y$ is what we call $t\circ\pi$ in the definition above, and $\pi_X$ is the map coming from the fact that $P_f=X\times_Y P\,.$) These maps are surjective submersions with $n$-connected fibers, so by~\Cref{crainicmor} (assuming it holds in the appropriate infinite dimensinal setting), the result follows if we can show that $\pi_X^!A\cong \pi_Y^! B\,.$ Note that, by assumption, $A=f\vert_X^!B\,,$ and it follows that an ``action" of $TP$ on $\mathfrak{g}$ would define an isomorphism $\pi_X^!A\cong \pi_Y^! B\,.$ 
\end{proof}
Under the assumption that these definition are correct, and due to the infinite dimensional nature of $\Pi_{\infty}(A)\,,$ we believe it would be natural to generalize~\Cref{mlf} to allow the space $Z$ to be a (paracompact, Hausdorff) Banach manifold, such that the fibers of both maps are $n$-connected, and such that $\pi_X^!A\cong \pi_Y^!B\,.$ We believe that this would also generalize the definitions of equivalence given above, in which case it would be a theorem that a map satisfying the conditions of~\Cref{neq} would be this more general $n$-equivalence. Hence we give the following, more general definition as well:
\begin{definition}(more general)
Let $A\to X\,, B\to Y$ be Lie algebroids. We say $A$ and $B$ are $n$-equivalent if there is a (paracompact, Hausdorff) Banach manifold\footnote{Or perhaps, even more generally, a diffeological space.} $Z$ with surjective submersions $\pi_X:Z\to X\,,\pi_Y:Z\to Y\,,$ both with $n$-connected fibers, such that $\pi_X^!A\cong \pi_Y^!B\,.$ For $n=\infty$ we call this a Morita equivalence.
\end{definition}
\vspace{4mm}Now let's connect all of this to~\Cref{morgl}. Recall that we think of a Lie algebroid $A\to Y$ as being equivalent to $\Pi_{\infty}(A)\,,$ for which we can take quotients (ie. truncations) to get different groupoids $\Pi_n(A)\,.$ Now from this perspective, a Morita equivalence between $A\to Y$ and $B\to X$ should correspond to a Morita equivalence between $\Pi_{\infty}(A)$ and $\Pi_{\infty}(B)\,,$ and taking truncations should preserve this equivalence. Now if $f:Y\to X$ is a surjective submersion with contractible fibers and $A\cong f^!B\,,$ then $A$ and $B$ should be Morita equivalent. In fact,~\Cref{morgl} tells us that, indeed, in this context $\Pi_1(A)$ is Morita equivalent to $\Pi_1(B)\,.$ More generally, since the fibers in the aformentioned proposition are only required to be simply connected, this result suggests that $1$-equivalences between $A$ and $B$ should integrate to Morita equivalences of $\Pi_1(A)$ and $\Pi_1(B)\,.$ This makes sense as there are no higher morphisms.
\begin{remark}
If the comments made in this section are correct, ie. that $A$ is equivalent to $\Pi_{\infty}(A)\,,$ it would imply that van Est is, in a sense, really about the descent of cohomology and there is perhaps a much more general theorem laying around.
\end{remark}
\section{Necessary Conditions for a Morphism to be a Morita Equivalence and Examples}
Now let's determine a necessary condition on $f$ so that the fibers of $t:P_f\to Y$ are weakly contractible. Since the codomain of the map $P_f\to Y$ is finite dimensional the constant rank theorem holds, by theorem F in~\cite{glock}.
Therefore the map $P_f\to Y$ is locally fibered (or what others call a topological submersion), and by Corollary 13 in~\cite{gael}, if such a map has weakly contractible fibers then it is a Serre fibration, hence it also is a weak homotopy equivalence. Now the canonical inclusion $\iota:X\to P_f$ is a deformation retract, and $f=t\circ\pi\circ\iota\,,$ therefore $t\circ\pi$ is a weak equivalence if and only if $f$ is a weak equivalence, which is equivalent to $f$ being a homotopy equivalence by Whitehead's theorem. Therefore, a necessary condition for the map $P_f\to Y$ to have (weakly) contractible fibers is that $f$ is a homotopy equivalence.
\begin{remark}
As a partial converse to the discussion above, if $f$ is a homotopy equivalence and if $P_f\to Y$ is a Serre fibration, then $P_f\to Y$ has (weakly) contractible fibers.
\end{remark}
\vspace{3mm}Now let's look at two extreme examples, one where the Lie algebroid is transitive and the other where the anchor map is $0:$
\begin{exmp}
Let $f:X\to Y$ be a map of manifolds, where we take the Lie algebroid on $Y$ given by $TY\to Y\,.$ Then the surjective submersion property of conditions 1 is satisfied and $TX$ satisfies condition 2. The condition on the fibers is equivalent to the induced map $f:X\to Y$ being a homotopy equivalence. Therefore, we see that a map $TX\to TY$ is a Morita equivalence if and only if the induced map $X\to Y$ is a homotopy equivalence.
\end{exmp}
\begin{exmp}
Consider a surjective submersion $f:X\to Y$ of manifolds, where we take the Lie algebroid on $Y$ to be the $0$-vector bundle, denoted $0_Y\to Y\,.$ We have that $f^!0_Y$ satisfies condition 2. We also have that $P_f=X$ and so $P_f\to Y$ is a surjective submersion with (weakly) contractible fibers if and only if $f$ has contractible fibers. That is, if the induced map $f:X\to Y$ is a surjective submersion, then the morphism $f^!0_Y\to Y$ is a Morita equivalence if and only if $f$ has contractible fibers.
\end{exmp}
\vspace{3mm}Assuming our definition of Morita equivalences of Lie algebroids is correct, it would offer an explanation for why the cohomology of certain sheaves are invariant under homotopy equivalence and others aren't. For example, the sheaf cohomology of constant sheaves is invariant under homotopy equivalence, but the sheaf cohomology of $\mathcal{O}$ is not. From the point of view of Lie algebroids, and restricting to tangent bundles, this can be explained by the fact that it should really the Lie algebroid cohomology which is invariant, and the Lie algebroid cohomology of a constant sheaf happens to just be the sheaf cohomology of the constant sheaf.\footnote{See~\Cref{Chevalley} of Part 1 for the definition of Lie algebroid cohomology that we are using.} However, the Lie algebroid cohomology of $\mathcal{O}$ (with respect to the tangent bundle) is the de Rham cohomology of the underlying manifold, which is indeed a homotopy invariant.
\vspace{3mm}\\On the other hand, $H^*(X,\mathcal{O})$ is the Lie algebroid cohomology of the $0$ vector bundle over $X\,.$ As mentioned earlier, two Lie algebroids consisting of the $0$ vector bundles over $X$ and $Y$ are Morita equivalent if and only if $X$ and $Y$ are diffeomorphic. Certainly, the cohomology of $\mathcal{O}$ is invariant under diffeomorphisms.
\begin{remark}
\begin{enumerate}
    \item Is there a connection between the Morita invariance of Lie algebroid cohomology and the problem of invariance of polarization in geometric quantization? Classical geometric quantization of a symplectic manifold $(M,\omega)$ involves first choosing a line bundle with connection whose curvature is $\omega$ (such a choice isn't unique, but if $M$ is simply connected it is unique up to isomorphism). Once this is done, one must choose a Lagrangian polarization — this in particular equips $M$ with a Lie algebroid and representation. One would normally proceed to then take the quantization to be the degree $0$ cohomology of this representation, however if there is a nontrivial Bohr-Sommerfeld condition this may result in the $0$ vector space, and so one must use ``discontinuous" global sections, given by sections over the Bohr-Sommerfeld leaves. According to~\cite{snia} (chapter 1, or see Theorem 1.2 in~\cite{eva}), the dimension of the cohomology of the representation is (under certain conditions) the number of Bohr-Sommerfeld leaves. Different polarizations can lead to different results, however they do often give the same results.
\vspace{3mm}\\For example, we can quantize $T^*S^1$ and there are two natural choices of polarization: one has quotient space $S^1\,,$ and the other has quotient space $\mathbb{R}\,.$ Computing the foliated cohomologies gives isomorphic results, however the interesting thing is that the isomorphism isn't degree preserving. How does this happen?
\item Can a deformation quantization of a Poisson manifold $(M,\Lambda)$ be obtained from geometric quantization of $\Pi_{\infty}(\Lambda)\,,$ using a generalization of the quantization procedure (due to Hawkins) found in~\cite{eli}? Indeed, the formula for deformation quantization via the path integral approach (due to Cattaneo and Felder, found on the first page of~\cite{catt}) is formally similar to the formula for the twisted convolution product, eg. page 22 in~\cite{eli}. In the former, it looks like the integral is, roughly, over morphisms $T\Delta\to T^*M\,,$ and functions $f:M\to\mathbb{R}$ give ``operators" by pulling them back to $\Pi_{\infty}(\Lambda)$ via the source map. If this is true, it would suggest that the deformation quantization can be obtained via geometric quantization of $(M,\Lambda)\,,$ using the source simply connected groupoid, if its source fibers are contractible. This is consistent with examples 6.2, 6.5 in~\cite{eli}.
\end{enumerate}
\end{remark}
\chapter{Higher Morphisms, Connections with van Est and a No-Go Result About Lie Algebroids}
In the previous chapter we defined Morita equivalences and $n$-equivalences of Lie algebroids, and we conjecture that, after localizing with respect to these $n$-equivalences, we get an equivalence between two categories, one groupoid-like and one algebroid-like, giving a version of Grothendieck's homotopy hypothesis (see~\cite{groth}) for Lie algebroids (or a smooth homotopy hypothesis). Recall that the homotopy hypothesis essentially says that topological spaces, up to weak equivalence, are equivalent to (discrete) $\infty$-groupoids.\footnote{If we use Kan complexes to model $\infty$-groupoids then this becomes a theorem. See~\cite{Quillen}.} The connection between the homotopy hypothesis and the one we will state arises fact that, under our definition of Morita equivalences, two tangent bundles are Morita equivalent if and only if their underlying manifolds are weakly equivalent.\footnote{Furthemore, analogously to how the smooth fundamental groupoid $\Pi_1(X)$ of a space is equivalent to the discrete version, we expect the same result to hold for $\Pi_{\infty}(X)\,.$} In addition, for the case $n=1$ we obtain a generalization of Lie's second and third theorems.
\vspace{3mm}\\To support the claims of the preceding paragraph, we state following result, due to Bloch and Smith, which is a generalization of the Riemann-Hilbert correspondence from $\Pi_{1}(X)$ to $\Pi_{\infty}(X)$\footnote{Recall that the classical Riemann-Hilbert correspondence states that their is an equivalence of categories between flat bundles on a space $X$ and representations of $\Pi_1(X)\,.$}:
\begin{theorem}(see~\cite{block}) 
Let $X$ be a manifold. There is an $A_{\infty}$-quasi-equivalence 
\begin{equation}
    \mathcal{RH}:\mathcal{P}_{\mathcal{A}}\to \textit{Loc}^{\mathcal{C}}(\Pi_{\infty}(X))\,.
\end{equation}
Here, $\mathcal{P}_{\mathcal{A}}$ is the dg-category of graded bundles on $X$ with a flat $\mathbb{Z}$-graded connection, and $\textit{Loc}^{\mathcal{C}}(\Pi_{\infty}(X))$ is the dg-category of $\infty$-local systems on $X\,.$
\end{theorem}
\vspace{1mm}Before getting into an algebroid version of the homotopy hypothesis, we (in particular) aim to do two things: describe how the van Est isomorphism theorem is a generalization of Lie's second theorem, and describe how gerbes can be thought of as a generalized morphisms (analogously to how principal bundles correspond to generalized morphisms). Much of what is said in this section may be known to experts. Some references include~\cite{ginot},~\cite{konrad},~\cite{Murray},~\cite{wockel}.
\vspace{3mm}\\We've already seen that for coefficients in an abelian Lie algebra, the van Est isomorphism theorem generalizes Lie's second theorem. For example, given a morphism $f:\mathfrak{g}\to\mathbb{C}\,,$ both Lie's second theorem and the van Est isomorphism theorem tell us that this morphism integrates to a morphism $G\to\mathbb{C}$ (or any other Lie group integating the Lie algebra $\mathbb{C})\,,$ given that $G$ is source simply connected. 
\vspace{3mm}\\Of course, the van Est isomorphism theorem tells us more, namely it says that a degree $n$-cohomology class for the Lie algebroid integrates to the Lie groupoid as long as the groupoid has $n$-connected source fibers. Our perspective is that the van Est theorem should generalize to the higher integrations $\Pi_n(\mathfrak{g})\,,$\footnote{We won't attempt to define this here, but it should be an $n$-truncation of $\Pi_{\infty}(\mathfrak{g})\,.$ See Theorem 1.2 in~\cite{zhuc}.} where the van Est map should be an isomorphism up to degree $n\,,$ which is why when $n=\infty$ we should get an isomorphism in all degrees.
\vspace{4mm}\\Towards the end we will define generalized morphisms between Lie algebroids and Lie groupoids, which will give us another interpretation of the van Est map.  We then argue that if $G\rightrightarrows G^0$ is source $n$-connected, then $G\rightrightarrows G^0$ is ``$n$-equivalent" to $\mathfrak{g}\to G^0\,.$
\vspace{3mm}\\Once again, Part 3 of this thesis is largely speculative (and again, significant details need to be filled in for completion).
\section{Gerbes as Generalized Morphisms}
Let's flesh out the way in which van Est generalizes Lie's second theorem. Let's think about a manifold $X$ for a moment. Consider the Lie group $\mathbb{C}^*\,.$ Classes in $H^0(X,\mathcal{O}^*)$ correspond to global functions $X\to \mathbb{C}^*\,,$ where here $\mathbb{C}^*$ is thought of as a manifold, not as a group. Classes in $H^1(X,\mathcal{O}^*)$ correspond to principal $\mathbb{C}^*$-bundles over $X\,.$ Now after embedding manifolds into groupoids (up to Morita equivalence, equivalently stacks) we learn that principal $\mathbb{C}^*$-bundles are the same as generalized morphisms $X\to \mathbb{C}^*\,,$ where now $\mathbb{C}^*$ is thought of as a group. A natural question is then: what happens when we embed manifolds into double groupoids (or 2-groupoids)? We will now argue that gerbes give higher generalized morphisms.
\vspace{3mm}\\Consider a class in $H^2(X,\mathcal{O}^*)\,.$ We will think of this geometrically as a bundle gerbe (see~\cite{Murray}), which can be described as follows: we have a surjective submersion $\pi:Y\to X\,,$ together with a central $\mathbb{C}^*$-extension of $Y\times_X Y\rightrightarrows Y\,.$ We will denote this extensions by $P\rightrightarrows Y\,,$ ie. we have the following sequence of groupoids:
\begin{equation}
\begin{tikzcd}
Y\times\mathbb{C}^* \arrow[d, shift right] \arrow[d, shift left] \arrow[r, shift right=7] & P \arrow[d, shift right] \arrow[d, shift left] \arrow[r, shift right=7] & Y\times_X Y \arrow[d, shift right] \arrow[d, shift left] \\
Y                                                                                         & Y                                                                       & Y                                                       
\end{tikzcd}
\end{equation}
Here, the groupoid on the left is just the product of the groupoids $Y\rightrightarrows Y$ with $\mathbb{C}^*\rightrightarrows *\,.$ In particular, $P\to Y\times_X Y$ is a principal $\mathbb{C}^*$-bundle. From this data, we can form the following double groupoid:
\begin{equation}
\begin{tikzcd}
P\rtimes\mathbb{C}^* \arrow[r, shift right] \arrow[r, shift left] \arrow[d, shift right] \arrow[d, shift left] & P \arrow[d, shift right] \arrow[d, shift left] \\
Y \arrow[r, shift right] \arrow[r, shift left]                                                                 & Y                                             
\end{tikzcd}    
\end{equation}
Here, the groupoid in the top row is the action groupoid, corresponding to the fact that $P$ is a principal $\mathbb{C}^*$-bundle over $Y\times_X Y\,,$ and the groupoid in the left column is the product of $P\rightrightarrows Y$ with $\mathbb{C}^*\rightrightarrows *\,.$ Now, this double groupoid is Morita equivalent\footnote{See~\Cref{mordo} for the definition.} to
\begin{equation}
    \begin{tikzcd}
Y\times_X Y \arrow[d, shift right] \arrow[d, shift left] \\
Y                                                       
\end{tikzcd}
\end{equation}
which is Morita equivalent to the manifold $X\,.$ On the other hand, we have a natural morphism 
\begin{equation}
    \begin{tikzcd}
P\rtimes\mathbb{C}^* \arrow[r, shift right] \arrow[r, shift left] \arrow[d, shift right] \arrow[d, shift left] & P \arrow[d, shift right] \arrow[d, shift left] \arrow[r, shift right=7] & \mathbb{C}^* \arrow[r, shift right] \arrow[r, shift left] \arrow[d, shift right] \arrow[d, shift left] & * \arrow[d, shift right] \arrow[d, shift left] \\
Y \arrow[r, shift right] \arrow[r, shift left]                                                                 & Y                                                                       & * \arrow[r, shift right] \arrow[r, shift left]                                                         & *                                             
\end{tikzcd}
\end{equation}
Therefore, using a $\mathbb{C}^*$-gerbe we get a generalized morphism with domain $X\,,$ ie.
\begin{equation}
    \begin{tikzcd}
X \arrow[d, shift right] \arrow[d, shift left] \arrow[r, shift right] \arrow[r, shift left] & X \arrow[d, shift right] \arrow[d, shift left] \arrow[r, shift right=7] & \mathbb{C}^* \arrow[r, shift right] \arrow[r, shift left] \arrow[d, shift right] \arrow[d, shift left] & * \arrow[d, shift right] \arrow[d, shift left] \\
X \arrow[r, shift right] \arrow[r, shift left]                                              & X                                                                       & * \arrow[r, shift right] \arrow[r, shift left]                                                         & *                                             
\end{tikzcd}
\end{equation}
Observe that we are no longer thinking of $\mathbb{C}^*$ as a group, we are thinking of it as a double groupoid. Equivalently, all of the double groupoids we have seen in this section are in one-to-one correspondence with strict 2-groupoids (see Example 3.7 in~\cite{mehtatang}), so rather than using double groupoids we can use strict 2-groupoids. This result is in line with the fact that a bundle gerbe can be thought of as a principal 2-bundle for a 2-group (see~\cite{ginot}), so we might have expected, a priori, that a gerbe can be thought of as a generalized morphisms into a 2-group.
\vspace{3mm}\\Let's look at another example: given a central extension $1\to A\to E\to G\to 1$ of Lie groups (which is described by a class in $H^2(G,A)$), we can form the following double groupoid: 
\begin{equation}
    \begin{tikzcd}
E\rtimes A \arrow[r, shift right] \arrow[r, shift left] \arrow[d, shift right] \arrow[d, shift left] & E \arrow[d, shift right] \arrow[d, shift left] \\
* \arrow[r, shift right] \arrow[r, shift left]                                                       & *                                             
\end{tikzcd}
\end{equation}
where the groupoid in the left column is again the product of the groups $E$ and $A\,.$ This double groupoid is Morita equivalent to 
\begin{equation}
    \begin{tikzcd}
G \arrow[d, shift right] \arrow[d, shift left] \\
*                                             
\end{tikzcd}
\end{equation}
and again, there is a natural morphism
\begin{equation}
    \begin{tikzcd}
E\rtimes A \arrow[r, shift right] \arrow[r, shift left] \arrow[d, shift right] \arrow[d, shift left] & E \arrow[d, shift right] \arrow[d, shift left] \arrow[r, shift right=7] & A \arrow[r, shift right] \arrow[r, shift left] \arrow[d, shift right] \arrow[d, shift left] & * \arrow[d, shift right] \arrow[d, shift left] \\
* \arrow[r, shift right] \arrow[r, shift left]                                                       & *                                                                       & * \arrow[r, shift right] \arrow[r, shift left]                                              & *                                             
\end{tikzcd}
\end{equation}
Therefore, once again we get a generalized morphism with domain $G\rightrightarrows *\,,$ ie,
\begin{equation}
    \begin{tikzcd}
G \arrow[d, shift right] \arrow[d, shift left] \arrow[r, shift right] \arrow[r, shift left] & G \arrow[d, shift right] \arrow[d, shift left] \arrow[r, shift right=7] & A \arrow[r, shift right] \arrow[r, shift left] \arrow[d, shift right] \arrow[d, shift left] & * \arrow[d, shift right] \arrow[d, shift left] \\
* \arrow[r, shift right] \arrow[r, shift left]                                              & *                                                                       & * \arrow[r, shift right] \arrow[r, shift left]                                              & *                                             
\end{tikzcd}
\end{equation}
Observe the pattern: given an abelian Lie group $A\rightrightarrows*\,,$ the cohomology of a groupoid $G\rightrightarrows G^0$ in degree $0$ corresponds to morphisms $G\to A[-1]\,,$ where by $A[-1]$ we mean the groupoid $A\rightrightarrows A\,.$ Cohomology in degree one corresponds to generalized morphisms from $G\to A\,,$ where now we are thinking of $A$ as a Lie group. Cohomology in degree 2 corresponds to generalized morphisms $G\to A[1]\,,$ where by $A[1]$ we mean the following double groupoid (which can equivalently be thought of as a strict 2-groupoid, by example 3.7 in~\cite{mehtatang}):
\begin{equation}
    \begin{tikzcd}
A \arrow[d, shift right] \arrow[d, shift left] \arrow[r, shift right] \arrow[r, shift left] & * \arrow[d, shift right] \arrow[d, shift left] \\
* \arrow[r, shift right] \arrow[r, shift left]                                              & *                                             
\end{tikzcd}
\end{equation}
Note that, the 2-groupoid corresponding to this double groupoid is $*$ in degrees 0 and 1, and $A$ in degree 2 (see~\cite{mehtatang}). We might denote these morphisms, suggestively, by $\text{Hom}^{\bullet}(G,A)$ (where we shift degrees so that $\text{Hom}^{0}(G,A)$ denotes morphisms into $A\rightrightarrows A$).
\vspace{3mm}\\We believe that this phenomenon continues to occur for general cohomology classes after embedding manifolds (or more generally, groupoids) into $n$-fold groupoids\footnote{by $n$-fold groupoids we mean the degree $n$-version of double groupoids, ie. groupoids, double groupoids, triple groupoids, etc.} (or  n-groupoids). If we take $n=\infty\,,$ we should be able to do this for all cohomology classes simultaneously. This would be an $\infty$-category.
\section{Perspectives on van Est and Lie's Second Theorem: The Homotopy Hypothesis}
Now let's connect the discussion in the previous section to van Est and Lie's second theorem. The aformentioned generalized morphisms should be naturally graded, eg. for a manifold $X\,,$ a function has degree $0\,,$ a principal bundle has degree $1\,,$ a gerbe has degree $2\,,$ etc. Equivalently, the generalized morphism has degree $n$ if it maps into an $n$-groupoid. We should be able to think about Lie algebroid cohomology classes in a similar way (eg. associated to a class in $H^2(\mathfrak{g},\mathcal{O})$ is a central extension of Lie algebroids, and associated to such an extension is a Lie $2$-algebroid which is equivalent to $\mathfrak{g}\,.$ This may be related to Example 6.16 in~\cite{Nuiten}). Then Lie's theorem would be about degree $1$ morphisms only, whereas van Est would be about morphisms in all degrees (for coefficients in a representation, for example). 
\vspace{3mm}\\Consider a Lie algebroid $A$ and some $\infty$-groupoid integrating $A$ (eg. $\Pi_n(A)\,,$ for some $n$). A more general van Est isomorphism theorem should imply that, in particular, a degree $n$ generalized morphism with domain $A$ integrates to $\Pi_n(A)\,.$
\vspace{3mm}\\From this point of view, the true mathematical object corresponding to a Lie algebroid $A$ is not a groupoid or source simply connected groupoid, but rather a source $\infty$-connected $\infty$-groupoid, which can be taken to be $\Pi_{\infty}(A)\,.$ The van Est isomorphism theorem would prove that a generalized morphism with domain $A$ (and appropriate codomain) is in one-to-one correspondence with generalized morphisms with domain $\Pi_{\infty}(A)\,,$ where the morphisms can be of any degree.
\vspace{3mm}\\Of course, this raises two questions:
\begin{enumerate}
    \item How can we describe nonabelian gerbes as generalized morphisms? (Presumably by using their descripton in~\cite{bundle}).
    \item Is there a van Est map and isomorphism theorem for these generalized morphisms taking values in general Lie groupoids, rather than abelian Lie groups? Lie's second theorem would suggest this possibility.
\end{enumerate}
Ultimately, a more general van Est isomorphism theorem should imply that there is an equivalence between two appropriate $\infty$-categories, one Lie groupoid-like and one Lie algebroid-like. There should be a Lie algebroid homotopy hypothesis, or smooth homotopy hypothesis, which, in spirit, should imply:\footnote{Technically, we are conjecturing that this is a meaningful statement that expresses truth. We give a related conjecture in~\ref{conjf}.}

\begin{conjecture}
Lie $n$-algebroids up to $n$-equivalence are equivalent to Lie $n$-groupoids which are source $n$-connected\footnote{We are being vague about what source $n$-connected means, but it probably means that the (groupoid) fibers of the inclusion map of the base into the higher groupoid are $n$-connected.} (up to equivalence).\footnote{Due to the failure of Lie's third theorem, in order for such a thing to be true in general we need to use a more general notion of Lie groupoids, eg. see~\cite{zhuc}.} In particular, generalized $n$-morphisms of Lie $n$-algebroids integrate to generalized $n$-morphisms of Lie $n$-groupoids.
\end{conjecture}
The functor inducing th
e equivalence should be Lie differentiation (see~\cite{bona},~\cite{Li},~\cite{jet} for some details about Lie $n$-algebroids/groupoids and differentiaton, see~\cite{zhuc2} for a notion of Morita equivalence of higher groupoids). This should imply that a Lie algebroid $A\,,$ up to $n$-equivalence, is essentially indistinguishable from $\Pi_n(A)\,,$ ``up to degree $n$". For example, since degree $k$ cohomology classes can be viewed as morphisms into a $k$-algebroid/groupoid, this conjecture should imply that they have the same cohomology up to degree $n\,.$
\vspace{3mm}\\When restrcited to Lie algebroids given by tangent bundles, this essentially says that a space $X\,,$ up to $n$-equivalence, is equivalent to $\Pi_n(X)\,.$ Very roughly speaking, we have the following sequence of implications:
\begin{equation}
    \begin{tikzcd}
\text{LA Homotopy Hypothesis} \arrow[r, Rightarrow] & \text{van Est} \arrow[r, Rightarrow] & \text{Lie \RNum{2}}
\end{tikzcd}
\end{equation}
Now from this point of view, why do only degree $1$ cohomology classes (including representions) necessarily integrate to the source simply connected groupoid? It is because higher cohomology classes correspond to morphisms into higher Lie groupoids/algebroids,\footnote{At the algebroid level we should get this correspondence after associating to a cohomology class a higher Lie algebroid} which means that by considering higher cohomology classes we are essentially leaving the category of Lie groupoids/algebroids, and since Lie \RNum{3} is only about $1$-groupoids, one might guess that we don't get a correspondendce between higher degree cohomology groups.
\begin{remark}\label{imrem}
As a final remark of this section we wish to point out that, under the assumed equivalence of $A$ with $\Pi_{\infty}(A)\,,$ the van Est map can be viewed as a pullback map rather than as a differentiation map. Furthermore, integrating classes from Lie algebroid cohomology to Lie groupoid cohomology can be viewed as a question of descent. That is, pullback $\cong$ differentiation, integration $\cong$ descent. 
\vspace{3mm}\\We can already see that this is true in the case of a connected abelian Lie group $G\,,$ since its Lie algebra $\mathfrak{g}$ is naturally identified with the simply connected integation of $\mathfrak{g}\,,$ which is the universal cover of $G$ (since $\mathfrak{g}$ is contractible it should be equivalent to $\Pi_{\infty}(\mathfrak{g})$\footnote{We hope that we aren't confusing the reader too much. We are thinking of $\mathfrak{g}$ in two different ways: 1. as the Lie algebra of $G$ and 2. as a simply connected abelian group.}). In this case we can readily observe that pullback $\cong$ differentiation, integration $\cong$ descent. For example, differentiating a morphism $f:G\to S^1$ to $df\vert_e:\mathfrak{g}\to\mathbb{R}$ is the same as pulling back $f$ to the universal cover of $G\,,$ ie. $\mathfrak{g}\,,$ and lifting this map to $\mathbb{R}\,.$ On the other hand, integrating a morphism $\mathfrak{g}\to\mathbb{R}$ to a morphism $G\to S^1$ is the same as quotienting this morphism to a map $\mathfrak{g}\to S^1$ and then asking if it descends to a map $G\to S^1\,.$
\end{remark}
\section{Differentiating Generalized Morphisms? A No-Go Theorem}
In this section we consider the problem of differentiating generalized morphisms. Perhaps suprisingly, the answer turns out to be negative, ie. generalized morphisms are not, in general, differentiable (even if one doesn't believe in our definition). We prove the following result:
\begin{theorem}\label{comdi}
There is no notion of generalized moprhisms between Lie algebroids with the following properties:
\begin{enumerate}
    \item Associated to any generalized morphism $P:G\to H$ is a generalized morphism $dP:\mathfrak{g}\to\mathfrak{h}\,.$
    \item Generalized morphisms $\mathfrak{g}\to\mathfrak{h}$ induce pullback maps $H^{\bullet}(\mathfrak{h},A)\to H^{\bullet}(\mathfrak{g},A)\,,$\footnote{See Part 1, sections 2 and 3 for the relevant definitions.} for any abelian Lie group $A\,,$ in such a way that the pullback of a trivial class is trivial.
    \item Pullback of cohomology commutes with the van Est map. That is, the following diagram is commutative:
    \begin{equation}
        \begin{tikzcd}
{H^{\bullet}(H,A)} \arrow[r, "P^*"] \arrow[d, "VE"'] & {H^{\bullet}(G,A)} \arrow[d, "VE"] \\
{H^{\bullet}(\mathfrak{h},A)} \arrow[r, "dP^*"]      & {H^{\bullet}(\mathfrak{g},A)}     
\end{tikzcd}
    \end{equation}
\end{enumerate}
\end{theorem}
\begin{proof}
Consider $\Pi_1(S^1)\cong \mathbb{R}\ltimes S^1\rightrightarrows S^1\,,$ whose Lie algebroid is $TS^1\,.$ The universal cover $\mathbb{R}\to S^1$ is a principal $\mathbb{Z}$-bundle, and therefore defines a generalized morphism $P:\Pi_1(S^1)\to\mathbb{Z}\,.$ There is a canonical class in $H^1(\mathbb{Z},\mathbb{R})\,,$ given by $n\mapsto n\,,$ and the pullback of this class to $H^1(\Pi_1(S^1),\mathbb{R})$ can be represented by the natural homomorphism $\mathbb{R}\ltimes S^1\to\mathbb{R}\,.$ After applying the van Est map, we get $d\theta\in H^1_{\text{dR}}(S^1)\,.$ On the other hand, the Lie algebra of $\mathbb{Z}$ is $0\to *\,,$ therefore the three properties imply that $d\theta$ must be exact, which it isn't.
\end{proof}
\begin{remark}
The proof shows that we can relax the conditions so that properties 2 and 3 only need to hold in degree $1\,,$ and we can assume that $A$ is a representation (see~\Cref{strong}). In addition, even if we worked with source connected groupoids only, a similar proof shows that the three properties cannot hold simultaneously. Our resolution is to relax property 1 so that not all generalized morphisms are differentiable.
\end{remark}
Now a conundrum might be given by the following: associated to a rank $n$ representation of $G$ is a generalized morphism into $\text{GL}(n,\mathbb{C})$ (given by the frame bundle associated to the underlying vector bundle), and we can differentiate representations, so why might we not be able to differentiate the associated generalized morphism? The answer is that we can't (in general) \textit{really} differentiate representations. Of course we can, in a sense, differentiate representations, however the point is that the differentiation is only happening on the domain side, not on the codomain side – when we differentiate representations we retain the $\textit{GL}(n,\mathbb{C})$ data which describes the vector bundle. The only thing which gets differentiated is the action of $G\,.$
\vspace{3mm}\\Now, consider a generalized morphism $G\to H$ given by a bibundlle $P\,.$ This doesn't need to differentiate to a generalized morphism $\mathfrak{g}\to\mathfrak{h}\,.$ On the other hand, if $H$ has contractible target fibers the morphism will differentiate to a generalized morphism: in this case the fibers of the map $\pi:P\to G^0$ are contractible, and therefore $\pi^!\mathfrak{g}$ is Morita equivalent to $\mathfrak{g}\,.$ However, we also have that $\pi^!(\mathfrak{g}) \cong\mathfrak{g}\ltimes P\rtimes \mathfrak{h}$ (ie. they are isomorphic, since $\pi^!G\cong G\ltimes P\rtimes H$), and we have a natural morphism $\mathfrak{g}\ltimes P\rtimes \mathfrak{h}\to \mathfrak{h}\,.$ Therefore, we get the following generalized morphism of Lie algebroids:
\begin{equation}
 \begin{tikzcd}
             & \pi^!\mathfrak{g} \arrow[ld,"\text{Morita}"'] \arrow[r, "\text{Iso}", no head, equal] & \mathfrak{g}\ltimes P\rtimes\mathfrak{h} \arrow[rd] &              \\
\mathfrak{g} &                                                                                       &                                                     & \mathfrak{h}
\end{tikzcd}
\end{equation}
More generally, if we work with Lie algebroids up to $1$-equivalence and if we work with source simply connected groupoids only, all generalized morphisms will differentiate to generalized morphisms of algebroids. In this case, property 3 in~\ref{comdi} will hold only up to degree $1\,.$ Even more generally, this is true as long as the codomain is source simply connected.
\vspace{3mm}\\Now let's delve deeper into the theory to see what is going on. What is happening is more transparent if we think from the perspective of $\Pi_{\infty}(\mathfrak{g})\,,$ rather than from the perspective of $\mathfrak{g}$ (we will now only refer to groupoids which are source connected). From this perspective, we have a generalized morphism $G\to H$ and we can always ``lift" this to a generalized morphism $\Pi_{\infty}(\mathfrak{g})\to H\,,$ however these morphisms don't necessarily lift to morphisms $\Pi_{\infty}(\mathfrak{g})\to \Pi_{\infty}(\mathfrak{h})\,,$ and why should they? Lifting problems generally have topological obstructions, and this is no different. Diagrammatically, the problem of differentiation is given by the following:
\begin{equation}
    \begin{tikzcd}
\Pi_{\infty}(\mathfrak{g}) \arrow[d] \arrow[r, "?", dotted] & \Pi_{\infty}(\mathfrak{h}) \arrow[d] \\
G \arrow[r]                                                 & H                                   
\end{tikzcd}
\end{equation}
Now given a diagram of the following form
\begin{equation}
    \begin{tikzcd}
                                                  & Z \arrow[d, "\pi"] \\
X \arrow[r, "f"'] \arrow[ru, "\tilde{f}", dotted] & Y                 
\end{tikzcd}
\end{equation}
there are obstructions to a lift $\tilde{f}$ existing. One such obstruction is the following: given any class $\alpha \in H^*(Y,\mathbb{Z})$ such that $f^*\alpha$ is nontrivial, it must also be the case that $\pi^*\alpha$ is nontrivial.
\vspace{3mm}\\Let's consider to the context of a generlized morphism $G\to \mathbb{C}^*\,,$ which can be described by a degree one cohomology class $H^1(G,\mathbb{C}^*)\,.$ If this cohomology class is trivial when pulled back to the base $G^0\,,$ then this generalized morphism is equivalent to a homomorphism $G\to\mathbb{C}^*\,,$ which differentiates to a morphism $\mathfrak{g}\to\mathbb{C}\,.$ This is always the case if $H^2(G^0,\mathbb{Z})=0\,.$ Recall that $H^*(\mathfrak{g}\,,\mathbb{Z})=H^*(G^0,\mathbb{Z})\,,$ therefore this will always be the case if $H^2(\mathfrak{g}\,,\mathbb{Z})=0\,;$ in the spirit of the smooth homotopy hypothesis, this should be equivalent to $H^2(\Pi_{\infty}(\mathfrak{g}),\mathbb{Z})=0\,.$ Indeed, this is consistent with the fact that the classifying space of $\Pi_{\infty}(\mathfrak{g})$ should be $G^0\,.$ 
\vspace{3mm}\\Let us emphasize this point: there is an obstruction to lifting a generalized morphism $\Pi_{\infty}(\mathfrak{g})\to \mathbb{C}^*$ to a generalized morprhism $\Pi_{\infty}(\mathfrak{g})\to\mathbb{C}\,,$ and this obstruction vanishes if $H^2(\Pi_{\infty}(\mathfrak{g}),\mathbb{Z})=0\,.$\footnote{Since $\mathbb{C}$ is contractible, $\Pi_{\infty}(\mathbb{C})$ should be equivalent to $\Pi_1(\mathbb{C})=\mathbb{C}\rightrightarrows *\,.$} 
\vspace{3mm}\\Why then, do strict morphisms $G\to H$ always differentiate? Again, under the assumption that $H^*(\Pi_{\infty}(\mathfrak{g}),\mathbb{Z})\cong H^*(G^0,\mathbb{Z})\,,$ pulling back a cohomology class
\begin{equation}
    H^*(H,\mathbb{Z})\to H^*(\Pi_{\infty}(\mathfrak{g}),\mathbb{Z})
    \end{equation}
is equivalent to pulling back the induced class in $H^*(H^0,\mathbb{Z})$ to $H^*(G^0,\mathbb{Z})\,.$ The question is, if such a class is nontrivial after being pulled back to $G^0\,,$ can it be trivial when pulled back to $\Pi_{\infty}(\mathfrak{h})\,?$ This is impossible, since the latter condition would imply that the cohomology class, when pulled back to $H^0\,,$ is trivial, which means it must also be trivial when pulled back to $G^0\,.$ Therefore, the obstruction vanishes.
\vspace{3mm}\\Let's emphasize another important point: classes in $H^1(\mathfrak{g}\,,\mathbb{C}^*)$ are essentially generalized morphisms $\mathfrak{g}\to\mathbb{C}^*\,,$ ie.\ we are computing generalized morphisms $\Pi_{\infty}(\mathfrak{g})\to\mathbb{C}^*$ (or generalized morphisms $\Pi_{1}(\mathfrak{g})\to\mathbb{C}^*\,,$ if working up to $1$-equivalence). This leads us into the next section.
\section{Exotic Morphisms Between Algebroids and Groupoids, Differentiability}\label{generalized}
Consider a function $f:X\to Y$ between smooth manifolds, with no assumption about differentiability of $f\,.$ This is equivalent to specifying a homomorphism $\text{Pair}(X)\to\text{Pair}(Y)\,.$ Asking if $f$ is smooth is essentially the same as askng if $f$ lifts to a smooth morphism $\Pi_{\infty}(TX)\to \Pi_{\infty}(TY)\,,$ via the map $\gamma(t)\mapsto f(\gamma(t))\,.$ Inspired by this and the discussion in the previous section, we make the following definition:
\begin{definition}
Let $P:G\to H$ be a generalized morphism of Lie groupoids. We say that $P$ is differentiable if it differentiates to a generalized morphism in the category of Lie algebroids, up to Morita equivalence. We say that $P$ is $n$-differentiable if it differentiates to a generalized morphism, up to $n$-equivalence (we may call a generalized morphism between Lie algebroids, defined up to $n$-equivalence, an $n$-generalized morphism).
\end{definition}
\begin{exmp}
Let $f:G\to H$ be a homomorphism. Then $f$ is differentiable.
\end{exmp}
In particular, non-differentiability of generalized morphisms is related to the fact that Morita maps are not invertible by (smooth) morphisms in general. For if a morphism were invertible (up to Morita equivalence), it would be equivalent to a homomorphism, which is differentiable.
\begin{exmp}
Let $P:G\to H$ be a generalized morphism between source $n$-connected Lie groupoids. Then $P$ is $n$-differentiable.
\end{exmp}
Of course, a $1$-generalized morphism between integrable Lie algebroids will integrate to a generalized morphism between their source simply connected integrations, generalizing Lie's second theorem. 
\vspace{3mm}\\Now, as mentioned in the previous section, we know that classes in $H^1(\mathfrak{g},\mathcal{O}^*)$ describe generalized morphisms $\Pi_1(\mathfrak{g})\to \mathbb{C}^*\,.$ One way to generalize this idea of mapping algebroids into groupoids is to use the embeddings of Lie groupoids and Lie algebroids into LA-groupoids and use generalized morphisms there — these provide a ``wormhole" between algebroids and groupoids. We make the following definition:
\begin{definition}
Consider a Lie algebroid $\mathfrak{g}\to X$ and a Lie groupoid $H\rightrightarrows H^0\,.$ Let $P\to X$ be a principal bundle for $H\,,$ and suppose that $P$ is equipped with an action of $\mathfrak{g}\,.$ We call $P$ a generalized morphism $\mathfrak{g}\to H$ if the following diagram is an LA-groupoid:
\begin{equation}
 \begin{tikzcd}
\mathfrak{g}\ltimes P\rtimes H \arrow[r] \arrow[d, shift right] \arrow[d, shift left] & P\rtimes H \arrow[d, shift right] \arrow[d, shift left] \\
\mathfrak{g}\ltimes P \arrow[r]                                                       & P                                                      
\end{tikzcd}
\end{equation}
Essentially, we are saying that $P$ is a generalized morphism if it is a principal $H$-bundle with a flat $\mathfrak{g}$-connection.
\vspace{3mm}\\Similarly, we can define generalized morphisms $G\to \mathfrak{h}$ as a space $P$ with actions of $G$ and $\mathfrak{h}\,,$ for which the following is an LA-groupoid:
\begin{equation}
 \begin{tikzcd}
G\ltimes P\rtimes \mathfrak{h} \arrow[r, shift right] \arrow[d] \arrow[r, shift left] & P\rtimes \mathfrak{h} \arrow[d] \\
G\ltimes P \arrow[r, shift right] \arrow[r, shift left]                               & P                              
\end{tikzcd}
\end{equation}
In addition, we require that $P\rtimes\mathfrak{h}$ is Morita equivalent to $0_{G^0}\to G^0$ via the canonical morphism (which implies that the Lie algebroid in the left column is Morita equivalent to $0_{G^{(1)}}\to G^{(1)}$). This means that the fibers of $\rho$ should be contractible and that $P\rtimes\mathfrak{h}$ should be the Lie algebroid associated to the foliation $P\to G^0$ (similarly, we can define $n$-generalized morphisms). We may call these \textit{exotic} morphisms.
\end{definition}
\begin{remark}
In Part 2 of this thesis, there was a lot of emphasis places on the double groupoid 
\begin{equation}
\begin{tikzcd}
G\ltimes P\rtimes H \arrow[r, shift right] \arrow[d, shift right] \arrow[r, shift left] \arrow[d, shift left] & P\rtimes H \arrow[d, shift right] \arrow[d, shift left] \\
G\ltimes P \arrow[r, shift right] \arrow[r, shift left]                                                       & P                                                      
\end{tikzcd}
\end{equation}
Most importantly was the LA-groupoid obtained by differentiating in the vertical direction; up until now we had never differentiated in the horizontal direction. We've now come full circle.
\end{remark}
\section{The Category of LA-Groupoids with n-Equivalences — Another Interpretation of the van Est Map}\label{finale}
In Part 2 of this thesis we discussed an interpretation of the van Est map as a map from the cohomology of a groupoid to the foliated cohomology of the space of objects. Here we will discuss another (but not unrelated) interpretation of the van Est map. Previously in this chapter we discussed how, via the smooth homotopy hypothesis, we should be able to view the van Est map as a pullback map
\begin{equation}
    H^*(G,M)\xrightarrow[]{VE}H^*(\Pi_{\infty}(\mathfrak{g}),M)\,.
\end{equation}
In this section, we will make this precise, at the level of $\mathfrak{g}\,.$
\vspace{3mm}\\If one imagines the category consisting of the maximal generalization of Lie algebroids and Lie groupoids, it should contain LA-groupoids. This implies that there is a natural notion of morphisms and equivalences between Lie algebroids and Lie groupoids. Indeed, both Lie algebroids and Lie groupoids have a notion of weak equivalences, and these two notions get mixed in this category of LA-groupoids, which results in some unexpected yet sensible features. In fact, since for each $n\ge 0$ we have a notion of $n$-equivalence for Lie algebroids, for each $n\ge 0$ we get a notion of $n$-equivalence for LA-groupoids (we implicitly used this in the previous section). In particular, the category of LA-groupoids we define unifies the category of differentiable stacks, the category of Lie algebroids and the homotopy category; we nickname it \textbf{Worm}\footnote{Due to its ``wormhole" like property of providing ``paths" (ie.\ morphisms) between different ``universes" (ie.\ categories).} In this category, the differentiation and integration functors between groupoids and algebroids are represented by generalized morphisms.
\vspace{3mm}\\Here we will show that there is a canonical \textit{exotic} morphism $\mathcal{I}:\mathfrak{g}\to G\,.$ One may interpret this morphism as the integration functor. In fact, we will show that if $G$ is source $n$-connected then this morphism is an $n$-equivalence, which is a Morita equivalence if $n=\infty\,.$ 
There are details which need to be filled in, however the idea should (hopefully) be clear.
\begin{definition}
A morphism between LA-groupoids is a Morita morphism if the morphism restricted to the top Lie algebroids is a Morita morphism of Lie algebroids, or if the morphism restricted to the top Lie groupoids is a Morita equivalence of Lie groupoids.\footnote{There are two Lie algebroids and two Lie groupoids appearing in an LA-groupoid — by ``top" we are referring to the Lie algebroid/groupoid in the left column and upper row. Equivalently, the Morita morphism induces an equivalence of both Lie algebroids (or Lie groupoids).} A morphism is an $n$-equivalence if the morphism restricted to the top Lie algebroids is an $n$-equivalence, or if the morphism restricted to the top Lie groupoids is a Morita morphism.\footnote{In the latter case this means it is also a Morita morphism of LA-groupoids.} We can then take the subcategory generated by these two notions of $n$-equivalence (see~\Cref{generate1}) and invert $n$-equivalences to obtain more general equivalences and morphisms. We refer to this as the category of LA-groupoids (with $n$-equivalences). For $n=\infty$ an $n$-equivalence is a Morita equivalence.
\end{definition}
\begin{lemma}
Given a Lie groupoid $G\,,$ there is a canonical generalized morphism $\mathcal{I}:\mathfrak{g}\to G$ in the category of LA-groupoids. If $G$ is source $n$-connected, there is a canonical $n$-generalized morphism $\mathcal{D}:G\to\mathfrak{g}\,.$ 
\end{lemma}
\begin{proof}
We have the following generalized morphism $\mathcal{I}:\mathfrak{g}\to G:$
\begin{equation}\label{mid}
\begin{tikzcd}
\mathfrak{g} \arrow[r] \arrow[d, shift right] \arrow[d, shift left] & G^0 \arrow[d, shift right] \arrow[d, shift left] & \mathfrak{g}\ltimes G\rtimes G \arrow[r] \arrow[d, shift right] \arrow[d, shift left] \arrow[l, "  \;\text{    Morita}"', shift left=7] & G\rtimes G \arrow[d, shift right] \arrow[d, shift left] & 0_G \arrow[r] \arrow[d, shift right] \arrow[d, shift left] & G \arrow[d, shift right] \arrow[d, shift left] \\
\mathfrak{g} \arrow[r]                                              & G^0                                              & \mathfrak{g}\ltimes G \arrow[r]                                                                                                         & G \arrow[r, shift left=8]                               & 0_{G^0} \arrow[r]                                          & G^0                                           
\end{tikzcd}
\end{equation}
Let us remind the reader that the notation $0_X\to X$ refers to the zero Lie algebroid over $X\,,$ and note that we are using the natural embeddings of Lie algebroids and Lie groupoids into LA-groupoids. Now in the case that $G$ is source $n$-connected, the map on the right is an $n$-equivalence, due to the fact that $\mathfrak{g}\ltimes G$ is $n$-equivalent to $0_{G^0}\to G^0\,,$ and $\mathfrak{g}\ltimes G\rtimes G$ is $n$-equivalent to $0_{G}\to G\,.$  Therefore, working up to $n$-equivalence, we get a generalized inverse morphism $\mathcal{D}:G\to\mathfrak{g}\,.$
\end{proof}
Let's point out that the LA-groupoid in the middle of~\ref{mid} previously appeared in this thesis as an important object in Part 2 (see~\ref{exx}, where here the map $H\to G$ is taken to be $G^0\to G$, and where the top right and bottom left corners have been switched), in the form
\begin{equation}
    \begin{tikzcd}
T_sG\rtimes G \arrow[d, shift right] \arrow[d, shift left] \arrow[r] & G\rtimes G \arrow[d, shift right] \arrow[d, shift left] \\
T_sG \arrow[r]                                                       & G                                                      
\end{tikzcd}
\end{equation}
\begin{corollary}\label{CAG}
The van Est map is given by $\mathcal{I}^*:H^{\bullet}(G,M)\to H^{\bullet}(\mathfrak{g},M)\,.$ Integration (up to degree $n$, where $G$ is source $n$-connected) is given by $\mathcal{D}^*:H^{\bullet}(\mathfrak{g},M)\to H^{\bullet}(G,M)\,.$
\end{corollary}
We will now summarize the preceding observations:
\begin{theorem}\label{LALG}
Let $G\rightrightarrows G^0$ be a source $n$-connected Lie groupoid. Then $\mathfrak{g}\to G^0$ is $n$-equivalent to $G\rightrightarrows G^0$ (if $n=\infty$ they are Morita equivalent).
\end{theorem}
This theorem is consistent with the fact that, as far as the author knows, there is no ``significant" way of distinguishing a Lie groupoid from its Lie algebroid (up to ``degree $n$") if the source fibers are $n$-connected (eg.\ they have the same cohomologies up to degree $n$ and we expect them to have equivalent $n$-term representations up to homotopy). In fact, invariance of LA-groupoid cohomology under $n$-equivalence (which we expect to be true up to degree $n$ and which we leave to future work) implies the van Est isomorphism theorem and the homotopy invariance of de Rham cohomology and singular cohomology, and of course the Morita invariance of Lie groupoid cohomology.
\vspace{3mm}\\We emphasize that, under the assumptions that the source fibers are $n$-connected, we expect that $G$ is equivalent to $\Pi_n(\mathfrak{g})$ – via the smooth homotopy hypothesis, this is the higher groupoid interpretation of this theorem.
\vspace{3mm}\\Now that we've defined $n$-equivalences of Lie algebroids, we may discuss ``$n$-Lie algebroids" of stacks (not to be confused with Lie $n$-algebroids):
\begin{definition}
Let $[H^0/H]$ be a differentiable stack. We define its $n$-Lie algebroid to be a Lie algebroid $\mathfrak{g}$ which is  $n$-equivalent to $H\rightrightarrows H^0$ (if it exists).
\end{definition}
Essentially, the $n$-Lie algebroid of $H\rightrightarrows H^0$ is the Lie algebroid of a Morita equivalent Lie groupoid $G\rightrightarrows G^0\,,$ whose source fibers are $n$-connected. Of course, it is unique up to $n$-equivalence.
Really, what one may want to do is define an algebroid version of differentiable stacks. Associated to a Lie algebroid $\mathfrak{g}$ one can assign the category of $n$-generalized morphisms from manifolds into $\mathfrak{g}$ (so there is probably a category of ``algebroid-stacks" for each $n\ge 0$). An object in this category is given by a manifold $M$ together with a surjective submersion $\pi:P\to M$ with $n$-connected fibers, together with an action of $\mathfrak{g}$ on $P\,,$ such that $P\rtimes \mathfrak{g}\cong \pi^!0_M$ (ie. the fiberwise tangent bundle).
\begin{remark}\label{generate1}
This category comes with two notions of weak equivalence, one in the horizontal direction and one in the vertical direction. We should take the smallest subcategory containing all of these weak equivalences to get a category with weak equivalences, or better, a homotopical category. Alternatively, there may be a nicer definition. See~\Cref{generate} for a discussion about essentially the same point.
 \end{remark}
\section{Properties of the Category of LA-Groupoids}
Now that we've seen that there are \textit{exotic} morphisms between Lie algebroids and Lie groupoids (by which we mean generalized morphisms between algebroids and groupoids in \textbf{Worm}, ie. the category of LA-groupoids), we will list some remarkable properties of this category:
\begin{enumerate}
    \item Manifolds $X\,,Y$\footnote{There are two natural embeddings of manifolds into LA-groupoids, one is via the $0$-Lie algebroid, one is via the groupoid containing only identity morphisms. However, $\mathcal{I}$ provides an isomorphism between the two, so it doesn't matter which one we use.} are Morita equivalent in this category if and only if they are diffeomorphic.
    \item More generally, Lie groupoids $G\,,H$ are Morita equivalent in this category if and only if they are Morita equivalent in the category of Lie groupoids.
    \item Two tangent bundles $TX\,,TY$ are Morita equivalent in this category if and only if $X\,,Y$ are homotopy equivalent.
    \item There is a canonical generalized morphism $\mathcal{I}:\mathfrak{g}\to G\,,$ which represents the integration functor. If $G\rightrightarrows G^0$ is source $n$-connected then this morphism is an $n$-equivalence, and dual to this there is a canonical $n$-equivalence $\mathcal{D}:G\to \mathfrak{g}\,,$ which represents the differentiation functor. If $n=\infty$ these generalized morphisms are Morita equivalences.
    \item With regards to the previous two points, the van Est map is given by the pullback $\mathcal{I}^*:H^*(G,M)\to H^*(\mathfrak{g},M)\,.$ If $G$ is source $n$-connected then integration is given by $\mathcal{D}^*:H^{\bullet}(\mathfrak{g},M)\to H^{\bullet}(G,M)\,,$ for $\bullet\le n\,.$  
    \item A Lie algebroid $\mathfrak{g}$ is integrable if and only if it is $1$-equivalent to some Lie groupoid $G\,.$
    \item This category induces a notion of homotopy equivalence on the category of Lie groupoids: a generalized morphism $P:G\to H$ is a homotopy equivalence if the induced generalized morphism in the category of LA-groupoids, $TP:TG\to TH\,,$ is a Morita equivalence.\footnote{There is a notion of homotopy equivalence described in~\cite{noohi1}, page 58, but it's not clear to the author that these notions are the same in the smooth category.} In particular, a Morita equivalence of Lie groupoids induces a homotopy equivalence. In addition, we get a natural notion of $n$-equivalence of Lie groupoids (in the homotopy sense).
    \item A finite dimensional classifying space of a Lie groupoid $G$ is just a manifold $BG$ which is homotopy equivalent to $G\,.$  \footnote{This is in line with Grothendieck's homotopy hypothesis, where geometric realization gives an equivalence of categories between simplicial sets and topological spaces. See chapter 2, section 3 of~\cite{Quillen}. Note that, if $G$ is discrete, then $TG=G\,.$ Once again, we expect that $\Pi_{\infty}(TX)\,,$ as an infinite dimensional higher Lie groupoid, is Morita equivalent to $\Pi_{\infty}(TX)\,,$ as a discrete higher groupoid.}
    \item If $EG\to BG$ is finite dimensional, then the Atiyah algebroid $\text{at}(EG)$ is Morita equivalent to $G\,.$ In particular, if $G$ is discrete then $\text{at}(EG)=T(BG)\,,$ therefore $G$ is Morita equivalent to $T(BG)\,.$
    \item Due to points 2, 3 and 4, we get the following result: suppose that $\mathcal{P}$ assigns to each LA-groupoid some property (eg.\ its cohomology) that is invariant under $n$-equivalence. Then if $X$ is homotopy equivalent to $Y\,;$ if $H\rightrightarrows H^0$ is Morita equivalent to $K\rightrightarrows K^0\,;$ if $G\rightrightarrows G^0$ is source $n$-connected, we get that $\mathcal{P}(TX)\cong \mathcal{P}(TY)\,;\mathcal{P}(H)\cong \mathcal{P}(K)\,;\mathcal{P}(G)\cong \mathcal{P}(\mathfrak{g})\,,$ respectively.
    \end{enumerate}
\begin{itemize}
    \item About point 7: in particular, a homomorphism $f:G\to H$ is a homotopy equivalence if  the induced map $f:G^{(1)}\to H^{(1)}$ is a homotopy equivalence of manifolds, eg. if $K\xhookrightarrow{}G$ is the inclusion of a maximal compact subgroup into a Lie group, then it is a homotopy equivalence. This is consistent with the fact that $BK$ is homotopy equivalent to $BG\,.$ In fact, these kinds of homomorphisms, together with Morita equivalences, generate homotopy equivalences.
    \vspace{3mm}\\In addition, a homotopy equivalence of discrete groupoids is the same as a Morita equivalence. For another example, if the source fibers of $G\rightrightarrows G^0$ are $n$-connected, then the morphism $G^0\to G$ is an $n$-equivalence.
\item About point 8: if $EG\to BG$ is finite dimensional then $EG$ defines a generalized morphism $BG\to G\,,$ given by the following diagram
\begin{equation}
    \begin{tikzcd}
   & EG\rtimes G \arrow[ld, "\text{Morita}"'] \arrow[rd] &   \\
BG &                                              & G
\end{tikzcd}
\end{equation}
By the preceding bullet point, the morphism on the right is a homotopy equivalence, due to the fact that $(EG\rtimes G)^{(1)}\to G^{(1)}$ is a surjective submersion with contractible fibers, and such a map is automatically a homotopy equivalence (see~\cite{gael}, Corollary 13).
\item About point 9: in particular, $\mathbb{Z}\rightrightarrows *$ is Morita equivalent to $TS^1\,,$ while  $\mathbb{Z}\rightrightarrows *$ is homotopy equivalent to $S^1$ (here $S^1$ is just the manifold).
\item About point 10: here a \textit{property} of LA-groupoids can often be formalized as a functor $\mathcal{P}$ with source the category of LA-groupoids. Examples of properties of LA-groupoids that we expect to be invariant under $n$-equivalence are cohomology up to degree $n$ and representations up to homotopy of length at most $n$ (once the latter is defined). In particular, any property of \textbf{Lie groupoids} which factors through a Morita invariant property of \textbf{LA-groupoids}, via taking their tangent LA-groupoids (eg. de Rham cohomology, cohomology of constant sheaves),  will be invariant under homotopy equivalence.
\end{itemize}
\vspace{3mm}Let's connect the properties of this category to Lie's second and third theorems, which tell us that there is an equivalence of categories between source simply connected Lie groupoids and integrable Lie algebroids, given by the functor $\mathfrak{g}\mapsto\Pi_1(\mathfrak{g})\,.$ Of course, given this equivalence, there is still a natural distinction between $\mathfrak{g}$ and $\Pi_1(\mathfrak{g})\,,$ which arises from the existence of morphisms between source simply connected groupoids and non-source simply connected groupoids.
However, after localizing at $1$-equivalences there is no such distinction in the category of LA-groupoids, since $\mathcal{I}$ provides an isomorphism $\mathfrak{g}\to \Pi_1(\mathfrak{g})\,.$
This is the essence of the smooth homotopy hypothesis. We can formulate the following generalization (which we have proved a special case of):\footnote{Again, we are really conjecturing that this is a meaningful statement that expresses truth.}
\begin{conjecture}\label{conjf}
In the category of $\text{L}_\infty$-algebroids over Lie $\infty$-groupoids, $\mathfrak{g}$ is Morita equivalent to $\Pi_{\infty}(\mathfrak{g})\,.$ More generally, if $G$ integrates $\mathfrak{g}$ and if $G$ is source $n$-connected,\footnote{In the appropriate sense of higher groupoids.} then $\mathfrak{g}$ is $n$-equivalent to $G\,.$
\end{conjecture}
In fact, there are different levels to this, since one can now ask about equivalences between LA-groupoids and double groupoids. In order to do this, one can embed both of these into the category of ``LA-groupoids over double Lie groupoids". One can then consider the category of \textit{$\text{L}_\infty$-algebroids over Lie $\infty$-groupoids} \textbf{over} \textit{double Lie $\infty$-groupoids}. This would be ``level 2", and one can continue until ``level $\infty$".
\begin{remark}Once again, if some notion defined with respect to Lie groupoids/algebroids can be interpreted exactly as a functor in the aforementioned higher category, then this notion will automatically be invariant under equivalence (eg. we expect this to be true for cohomology and representations up to homotopy).
\end{remark}
\begin{remark}
This notion of smooth homotopy equivalence naturally generalizes to topological groupoids.\ However, in order for many of the same properties to hold one should put some restrictions on the source and target maps. A nice and fairly large class of maps, which include both fibrations and topological submersions, are homotopic submersions (see~\cite{gael}), or microfibrations (see~\cite{raptis}). Noohi discusses some appropriate classes of maps to use in~\cite{noohi1}.
\end{remark}
\chapter{Conjectures About van Est (for Double Groupoids) and Cohomology}
There are many natural ways to generalize what has been done in this thesis. One can go to higher groupoids ($L_{\infty}$-algebroids), $n$-fold groupoids, or a combination of the two. For example, one can prove a van Est theorem from double groupoids to LA-groupoids, and one from double groupoids to double Lie algebroids.
\begin{conjecture}
There is a van Est map from the cohomology of a double Lie groupoid to the cohomology of its LA-groupoid. If the differentiation takes place in the vertical direction and if the vertical source fiber (a fiber in the (2,1)-category sense of this thesis) is $n$-connected, then the van Est map is an isomorphism up to degree $n$ and is injective in degree $n+1\,.$ 
\end{conjecture}
Let us remark once again that this result would imply the van Est isomorphism theorem given in Part 2 of this thesis. In addition, one can also generalize the results of this thesis to the Bott-Shulman-Stashef complex (see~\cite{BS}). 
\vspace{3mm}\\
As a final remark of this thesis, we wish to make an observation about cohomology which we conjecture generalizes, and is related to the isomorphism of Alexander-Spanier cohomology with de Rham cohomology (see~\cite{Meinrenken}). Following this discussion we will state two conjectures.
\vspace{3mm}\\We consider two examples. 
\begin{enumerate}
    \item First, let $G\rightrightarrows G^0$ be a Lie groupoid and consider the map $\iota:G^0\to G\,.$ We can consider cohomology classes on $G$ with coefficients in $\mathcal{O}\,.$ We can form the nerves and we get an induced map $\iota^{\bullet}:B^{\bullet}G^0\to B^{\bullet}G\,.$ We can then use the inverse image functor to obtain a sheaf on on $B^{\bullet}G^0\,.$ The cohomology of this sheaf, $H^*(B^{\bullet}G^0,\iota^{\bullet-1}\mathcal{O})\footnote{To be clear, the $-1$ here refers to the inverse image functor, not a shift in degree.}\,,$ gives the cohomology of the local groupoid of $G\,,$ which is known to be isomorphic $H^*(\mathfrak{g},\mathcal{O})$ (see~\cite{Meinrenken}). Now the map $\iota:G^0\to G$ is equivalent to the map $\pi_2:G\rtimes G\to G\,,$ and $H^*(B^{\bullet}(G\rtimes G),\pi_2^{\bullet-1}\mathcal{O})$ is also isomorphic to $H^*(\mathfrak{g},\mathcal{O})$ (by arguments shown in the proof of the main theorem in Part 1, for example), therefore 
\begin{equation}
    H^*(B^{\bullet}(G\rtimes G),\pi_2^{\bullet-1}\mathcal{O})\cong H^*(B^{\bullet}G^0,\iota^{\bullet-1}\mathcal{O})\,.
\end{equation}
Now let's emphasize the main point: these two cohomologies agree, and the morphisms $G^0\to G\,,$ $G\rtimes G\to G$ are equivalent (note that, the cohomology group on the right side is really the cohomology of a cochain complex of sheaves on $G^0$).
\item \vspace{3mm}Let's take a look at another example. Let $X$ be a manifold, and consider the sheaf $\mathcal{O}$ over a point $*\,.$ Consider the map $X\to *\,.$ We can compute the inverse image sheaf, and we get the sheaf of locally constant $\mathbb{R}$-valued functions on $X\,.$ The cohomology of this sheaf is $H^*(X,\mathbb{R})\,.$\footnote{We could form the nerves of the spaces and map, still thinking of them as groupoids, however since both groupoids are just manifolds we would obtain the same result.} Now the map $X\to *$ is equivalent to the map $\iota:X\to \text{Pair}(X)\,.$ We can form the nerves to get $\iota^{\bullet}:B^{\bullet}X\to B^{\bullet}\text{Pair}(X)\,,$ and we can take the cohomology of the inverse image sheaf, and what we get is the local cohomology of $\text{Pair}(X)$ (essentially, Alexander-Spanier cohomology), which is isomorphic to the de Rham cohomology of $X\,,$ ie.\begin{equation}
    H^*(X,\mathbb{R})\cong H^*(B^{\bullet}X\,,\iota^{\bullet-1}\mathcal{O})\,.
\end{equation}
Again, the point being that these two cohomologies are isomorphic and the morphisms $X\to *\,,X\to\text{Pair}(X)$ are equivalent (again, the cohomology group on the right side is really the cohomology of a cochain complex of sheaves on $X$). 
\end{enumerate}
Both of these examples have the same properties. We now state the following conjecture:
\vspace{1mm}\begin{conjecture}
Let $f_1:H_1\to G_1\,, f_2:H_2\to G_2$ be morphisms of Lie groupoids with a natural isomorphism $f_1\Rightarrow f_2\,.$ Let $M_1$ be a $G_1$-module and let $M_2$ be the corresponding $G_2$ module. Then the natural isomorphism induces an isomorphism\footnote{Again, to be clear, the $-1$ in~\ref{someq} refers to the inverse image functor, not a shift in degree.}
\begin{equation}\label{someq}
    H^*(B^\bullet H_1, f_1^{\bullet-1}\mathcal{O}(M_1))\cong  H^*(B^\bullet H_2, f_2^{\bullet-1}\mathcal{O}(M_2))\,.
\end{equation}
\end{conjecture}
\vspace{4mm}Note that, in the simple case where $M_1=G_1^0\times A$ for some abelian Lie group $A\,,$ with the trivial $G$-action, we have that $M_2=G_2^0\times A\,,$ with the trivial $G$-action.
\vspace{3mm}\\In particular, together with the work done in Part 2, this would imply the following result (which is essentially a generalization of~\Cref{important}, which is the key result used in proving the van Est isomorphism theorem):
\begin{conjecture}
Let $f:H\to G$ be a map of Lie groupoids such that $H^0\times_{G^0}G^{(1)}\to G^0$ is a surjective submersion, and let $M$ be a $G$-module. Suppose further that the (groupoid) fibers of $f$ are $n$-connected. Then we have an isomorphism up to degree $n\,,$ and an injection in degree $n+1\,,$ given by the following:
\begin{equation}
  H^*(B^{\bullet}G,\mathcal{O}(M))\to H^*(B^{\bullet}H,f^{\bullet-1}\mathcal{O}(M))\,.  
\end{equation}
\end{conjecture}
\newpage
\vspace*{\fill} 
\begin{quote} 
\centering 
``Tell me,” the great twentieth-century philosopher Ludwig Wittgenstein once asked a friend, ``why do people always say it was natural for mankind to assume that the sun went around the Earth rather than that the Earth was rotating?” His friend replied, ``Well, obviously because it just looks as though the Sun is going around the Earth.” Wittgenstein responded, ``Well, what would it have looked like if it had looked as though the Earth was rotating?”
\vspace{2mm}\\― A conversation between Ludwig Wittgenstein and G. E. M. Anscombe.
\end{quote}
\vspace*{\fill}
\appendix\chapter{Appendix}
\section{Derived Functor Properties}\label{derived functor}
In this section we discuss some vanishing results for the derived functors of locally fibered maps. These results are particularly useful
when using the Leray spectral sequence.
\begin{definition}
A map $f:X\to Y$ between topological spaces is called locally fibered if for all points $x\in X$ there exists open sets $U\ni x$ and $V\ni f(x)\,,$ a topological space $F$ and a homeomorphism $\phi:U\to F{\mathop{\times}} V$ such that the following diagram commutes:
\[
\begin{tikzcd}
U \arrow{r}{\phi}\arrow{dr}{f} & F{\mathop{\times}} V\arrow{d}{p_2} \\
& V
\end{tikzcd}
\]
$\blacksquare$\end{definition}
\theoremstyle{definition}\begin{exmp}
If $X\,,Y$ are smooth manifolds and $f:X\to Y$ is a surjective submersion, then $f$ is locally fibered.
\end{exmp}
\begin{proposition}[The Canonical Resolution]\label{canonical resolution}
Let $M\to X$ be a family of abelian groups. Then there is a canonical acyclic resolution of $\mathcal{O}(M)$ that differs from the Godement resolution. It is given by the following: for a sheaf $\mathcal{S}\,,$ let $\mathbb{G}^0(\mathcal{S})$ be the first sheaf in the Godement resolution of $\mathcal{S}\,,$ ie. the sheaf of germs of $\mathcal{S}\,.$ Let $\mathcal{G}^0(M)$ be the sheaf of all sections of $M$ (including discontinuous ones). Let 
\begin{align*}
    \mathcal{G}^{n+1}(M)=\mathbb{G}^0(\textrm{coker}[\mathcal{G}^{n-1}(M)\to \mathcal{G}^n(M)])
\end{align*}
for $n\ge 0\,,$ where $\mathcal{G}^{-1}(M):=\mathcal{O}(M)\,.$ We then have the following acylic resolution of $\mathcal{O}(M):$
\begin{align*}
    0\to\mathcal{O}(M)\to \mathcal{G}^{0}(M)\to \mathcal{G}^1(M)\to\cdots\,.
\end{align*}
\end{proposition}
\begin{definition}[see \cite{Bernstein}]\label{def n acyclic}
A continuous map $f:X\to Y$ is called $n$-acyclic if it satisfies the following conditions:
\begin{enumerate}
    \item For any sheaf $\mathcal{S}$ on $Y$ the adjunction morphism $\mathcal{S}\mapsto R^0f_*(f^{-1}\mathcal{S})$ is an isomorphism and 
    $R^if_*(f^{-1}\mathcal{S})=0$ for all $i=1\,,\ldots\,,n\,.$
    \item For any base change $\tilde{Y}\to Y$ the induced map $f:X{\mathop{\times}}_Y \tilde{Y}\to\tilde{Y}$ satisfies property 1.
\end{enumerate}
$\blacksquare$\end{definition}
\begin{theorem}[see \cite{Bernstein}, criterion 1.9.4]\label{n acyclic}
Let $f:X\to Y$ be a locally fibered map. Suppose that all fibers of $f$ are $n$-acyclic (ie. $n$-connected). Then $f$ is $n$-acyclic. 
\end{theorem} 
\begin{corollary}\label{restriction is zero}
Let $f:X\to Y$ be a locally fibered map and suppose that all fibers of $f$ are $n$-acyclic. Let $M$ be a family of abelian groups. Then if \begin{align*}
\alpha\in R^{n+1}f_*(f^{-1}\mathcal{O}(M))(Y)
\end{align*}
satisfies $\alpha\vert_{f^{-1}(y)}=0$  for all $y\in Y$ $($note that $\alpha\vert_{f^{-1}(y)}\in H^{n+1}(f^{-1}(y),f^{-1}M_y))\,,$ then $\alpha=0\,.$
\end{corollary}
\begin{proof}
By Proposition~\ref{canonical resolution} we have the following resolution of $\mathcal{O}(M)\,:$
\begin{align*}
    0\to\mathcal{O}(M)\to \mathcal{G}^{0}(M)\to \mathcal{G}^1(M)\to\cdots\,.
\end{align*}
Since $f^{-1}$ is an exact functor, we obtain the following resolution of $f^{-1}\mathcal{O}(M):$
\begin{align}\label{resolution}
    0\to f^{-1}\mathcal{O}(M)\to f^{-1}\mathcal{G}^{0}(M)\to f^{-1}\mathcal{G}^1(M)\to\cdots\,.
\end{align}
Hence, \begin{align*}
    R^\bullet f_*(f^{-1}\mathcal{O}(M))=R^\bullet f_*(f^{-1}\mathcal{G}^{0}(M)\to f^{-1}\mathcal{G}^1(M)\to\cdots)\,.
\end{align*}
One can show that $\alpha\vert_{f^{-1}(y)}=0$ for all $y\in Y\iff \alpha\mapsto 0$ under the map induced by $f^{-1}\mathcal{O}(M)\to f^{-1}\mathcal{G}^{0}(M)\,,$
hence we obtain the result by using Theorem~\ref{n acyclic}.
\end{proof}
\begin{theorem}\label{important}
Let $f:X\to Y$ be a locally fibered map such that all its fibers are $n$-acyclic. Let $M\to Y$ be a family of abelian groups. 
Then the following is an exact sequence for $0\le k\le n+1:$
\begin{align*}
& 0\to H^k(Y,\mathcal{O}(M))\overset{f^{-1}}{\to} H^k(X, f^{-1}\mathcal{O}(M))
\\& \to
\prod_{y\in Y} H^k(f^{-1}(y),\mathcal{O}((f^*M)\vert_{f^{-1}(y)}))\,.
\end{align*}
In particular, $H^k(Y,\mathcal{O}(M))\overset{}{\to} H^k(X, f^{-1}\mathcal{O}(M))$
is an isomorphism for $k=0\,,\ldots\,,n$ and is injective for $k=n+1\,.$
\end{theorem}
\begin{proof}
This follows from the Leray spectral sequence (see section $0$) and Corollary~\ref{restriction is zero}.
\end{proof}
Similarly, we can generalize this result to simplicial manifolds:
\begin{theorem}\label{spectral theorem}
Let $f:X^\bullet\to Y^\bullet$ be a locally fibered morphism of simplicial topological spaces such that, in each degree, all of its fibers are $n$-acyclic. Let $M_\bullet\to Y^\bullet$ be a simplicial family of abelian groups. 
Then the following is an exact sequence for $0\le k\le n+1:$
\begin{align*}\label{Short exact}
& 0\xrightarrow{} H^k(Y^\bullet,\mathcal{O}(M_\bullet))\xrightarrow{f^{-1}}H^k(X^\bullet, f^{-1}\mathcal{O}(M_\bullet))
\\ &\xrightarrow{}
\prod_{y\in Y^0} H^k(f^{-1}(y),\mathcal{O}((f^*M_0)\vert_{f^{-1}(y)}))\,.
\end{align*}
In particular, $H^k(Y^\bullet,\mathcal{O}(M_\bullet))\overset{}{\to} H^k(X^\bullet, f^{-1}\mathcal{O}(M_\bullet))$
is an isomorphism for $k=0\,,\ldots\,,n$ and is injective for $k=n+1\,.$\footnote{See Section~\ref{derived functor} for more details.}
\end{theorem}

\subsection{Lie Groupoids}\label{BG functor}
In this section we briefly review some important concepts in the theory of Lie groupoids.
\begin{definition}
A groupoid is a category $G\rightrightarrows G^0$ for which the objects $G^0$ and morphisms $G$ are sets and for which every morphism is invertible. A Lie groupoid is a groupoid $G\rightrightarrows G^0$ such that $G^0\,, G$ are smooth manifolds\begin{footnote}{We allow for the possibility that the manifolds are not Hausdorff, but all structure maps should be locally fibered.}\end{footnote}, such that the source and target maps, denoted $s\,,t$ respectively, are submersions, and such that all structure maps are smooth, ie.
\begin{align*}
    & i:G^0\to G
    \\ & m:G\sideset{_s}{_{t}}{\mathop{\times}} G\to G
    \\& \text{inv}:G\to G
\end{align*}
are smooth (these maps are the identity, multiplication/composition and inversion, respectively). A morphism between Lie groupoids $G\to H$ is a smooth functor between them.
$\blacksquare$\end{definition}
\begin{definition}\label{moritamapjeff}
Let $G\rightrightarrows G^0\,, K\rightrightarrows K^0$ be Lie groupoids. A Morita map $\phi:G\to K$ is a map such that
    \begin{enumerate}
        \item The composite map $G^0\sideset{_{\phi}}{_{s}}{\mathop{\times}}K^{(1)}\xrightarrow[]{p_2} K^{(1)}\xrightarrow[]{t} K^0$ is a surjective submersion.
        \item The following diagram is Cartesian
        \[
\begin{tikzcd}
G \arrow{r}{(s,t)} \arrow{d}{\phi} & G^0{\mathop{\times}} G^0 \arrow{d}{(\phi\,,\phi)} \\
K \arrow{r}{(s,t)} & K^0{\mathop{\times}} K^0
\end{tikzcd}
\]
    \end{enumerate}
   $\blacksquare$ \end{definition}
    \begin{definition}\label{moritaequivjeff}
We say that $G\,,K$ are Morita equivalent Lie groupoids if there is a third Lie groupoid $H$ with Morita maps to both $G$ and $K\,.$
$\blacksquare$\end{definition}
\begin{definition}
There is a functor 
\[\mathbf{B}^\bullet:\text{groupoids}\to \text{simplicial spaces}\,,\,G\mapsto\mathbf{B}^\bullet G\,,
\]
where $\mathbf{B}^0 G=G^0\,,\,\mathbf{B}^1 G=G\,,$ and 
\[\mathbf{B}^n G=\underbrace{G\sideset{_t}{_{s}}{\mathop{\times}} G \sideset{_t}{_{s}}{\mathop{\times}} \cdots\sideset{_t}{_{s}}{\mathop{\times}} G}_{n \text{ times}}\,,
\]
the space of 
$n$-composable arrows. Here the face maps are the source and target maps for $n=1\,,$ and for $(g_0\,,\ldots\,,g_n)\in \mathbf{B}^{n+1}G\,,$
\begin{align*}
& d_{n+1,0}(g_0\,,\ldots\,,g_n)=(g_1\,,\ldots\,,g_n)\,, 
\\& d_{n+1,i}(g_0\,,\ldots\,,g_n)=(g_0\,,\ldots\,,g_{i-1}g_i\,,\hat{g}_i\,,\ldots\,,g_n)\,,\,\;1\le i\le n
\\& d_{n+1,n+1}(g_0\,,\ldots\,,g_n)=(g_0\,,\ldots\,,g_{n-1})\,.
\end{align*}
The degeneracy maps are $\text{Id}:G^0\to G$ for $n=0\,,$ and 
\begin{align*}
\\& \sigma_{n-1,i}(g_0\,,\ldots\,,g_{n-1})=(g_0\,,\ldots\,,g_{i-1}\,,\text{Id}(t(g_i))\,,\hat{g}_i\,,\ldots\,,g_{n-1})\,,\,\;0\le i\le n-1
\\& \sigma_{n-1,n}(g_0\,,\ldots\,,g_{n-1})=(g_0\,,\ldots\,,g_{i-1}\,,\hat{g}_i\,,\ldots\,,g_{n-1}\,,\text{Id}(s(g_{n-1})))\,.
\end{align*}
A morphism $f:G\to H$ gets sent to $\mathbf{B}^\bullet f:\mathbf{B}^\bullet G\to \mathbf{B}^\bullet H\,,$ which acts as $f$ does for $n=0\,,1\,,$ and
\begin{align*}
    \mathbf{B}^n f(g_0\,,\ldots\,,g_{n-1})=(f(g_0)\,,\ldots\,,f(g_{n-1}))
\end{align*}
for $n> 1\,.$
$\blacksquare$\end{definition}
\subsection{Cohomology of Sheaves on Stacks}\label{appendix:Cohomology of Sheaves on Stacks}
In this section we briefly review the Grothendieck topology and sheaves on a differentiable stack, as well as their cohomology. The following definitions are based on~\cite{Kai}, which include further details.
\begin{definition}
We call a family of morphisms $\{P_i\to P\}_i$ in $[G^0/G]$ a covering family if the corresponding family of morphisms on the base manifolds $\{M_i\to M\}_i$ is a covering family for the site of smooth manifolds, ie. a family of \'{e}tale maps such that $\coprod_i M_i\to M$ is surjective. This defines a Grothendieck topology on $[G^0/G]\,,$ thus we can now speak of sheaves on $[G^0/G]\,,$ ie. contravariant functors $\mathcal{S}:[G^0/G]\to\mathbf{Ab}$ such that the following diagram is an equalizer for all covering families $\{P_i\to P\}_i:$
\begin{align*}
\mathcal{S}(P)\to\prod_i \mathcal{S}(P_i)\rightrightarrows \prod_{i,j} \mathcal{S}(P_i{\mathop{\times}}_P P_j)\,.
\end{align*}
A morphism between sheaves $\mathcal{S}$ and $\mathcal{F}$ is a natural transformation from $\mathcal{S}$ to $\mathcal{F}\,.$
$\blacksquare$\end{definition}
\begin{definition}\label{stack cohomology}
Let $\mathcal{S}$ be a sheaf on $[G^0/G]\,.$ Define the global sections functor $\Gamma:\textrm{Sh}([G^0/G])\to \mathbf{Ab}$ by
\begin{align*}
\Gamma([G^0/G],\mathcal{S}):=\text{Hom}_{\text{sh}([G^0/G])}(\mathbb{Z}\,,\,\mathcal{S})\,,
\end{align*}
where $\mathbb{Z}$ is the sheaf on $[G^0/G]$ which assigns to the object
\[
\begin{tikzcd}
P\arrow{r}{}\arrow{d} & G^0
\\M
\end{tikzcd}
\]
the abelian group $H^0(M,\mathbb{Z})\,.$
$\blacksquare$\end{definition}
\begin{definition}
The global sections functor $\Gamma:\textrm{Sh}([G^0/G])\to \mathbf{Ab}$ is left exact and the category of sheaves on $[G^0/G]$ has enough injectives, so we define $H^*([G^0/G],\mathcal{S}):=R^*\Gamma(\mathcal{S})\,.$
$\blacksquare$\end{definition}
\begin{theorem}[see~\cite{Kai}]\label{stack groupoid cohomology}
Let $\mathcal{S}$ be a sheaf on $[G^0/G]\,.$ 
Then 
\begin{align*}
    H^*([G^0/G],\mathcal{S})\cong H^*(\mathbf{B}^\bullet G,\mathcal{S}(\mathbf{B}^\bullet G))\,.
\end{align*}
\end{theorem}
 \subsection{Abelian Extensions}\label{abelian extensions}
Here we review abelian extensions and central extensions of Lie groupoids and Lie Algebroids. 
\begin{definition}Let $M$ be a $G$-module for a Lie groupoid $G\rightrightarrows G^0$. A Lie groupoid extension of $G$ by $M$ is given by a Lie groupoid 
$E\rightrightarrows G^0$ and a sequence of morphisms
\begin{align*}
   1\to M\overset{\iota}{\to} E\overset{\pi}{\to} G\to 1\,,
\end{align*}
such that $\iota\,,\,\pi$ are the identity on $G^0\,;$ such that $\iota$ is an embedding and $\pi$ is a surjective submersion; such that if $m\in M\,,\,e\in E$ satisfy $s(m)=s(e)\,,$ then $e\iota(m)=\iota(\pi(e)\cdot m)e\,;$ in addition, we require that $E\to G$ be principal $M$-bundle with respect to the right action. If $M$ is a trivial $G$-module then $E$ will be called a central extension. If $A$ is an abelian Lie group then associated to it is a canonical trivial $G$-module given by $A_{G^0}\,,$ and by an $A$-central extension of $G$ we will mean an
extension of $G$ by the trivial $G$-module $A_{G^0}\,.$ Furthermore, there is a natural action of $M$ on $E\,,$ and we assume that with this action $E$ is a principal $M$-bundle.
$\blacksquare$\end{definition}
\begin{definition}Let $\mathfrak{m}$ be a $\mathfrak{g}$-representation for a Lie algebroid $\mathfrak{g}\to N$. A Lie algebroid
extension of $\mathfrak{g}$ by $\mathfrak{m}$ is given by a Lie algebroid 
$\mathfrak{e}\to N$ and an exact sequence of the form
\begin{align*}
    0\to \mathfrak{m}\overset{\iota}{\to}\mathfrak{e}\overset{\pi}{\to}\mathfrak{g}\to 0\,,
\end{align*}
such that $\iota\,,\,\pi$ are the identity on $N\,,$ and such that if $X\,,Y$ are local sections over an open set $U\subset N$ of
$\mathfrak{m}\,,\mathfrak{e}\,,$ respectively, 
then $\iota(L_{\pi(Y)}X)=[Y,\iota(X)]\,.$ If $\mathfrak{m}$ is a trivial $\mathfrak{g}$-module then $\mathfrak{e}$ will be called a central extension. Similarly to the previous definition, if $V$ is a finite dimensional vector space then associated to it is a canonical trivial $\mathfrak{g}$-module given by $N{\mathop{\times}} V\,,$ and by a $V$-central extension of $\mathfrak{g}$ we will mean an
extension of $\mathfrak{g}$ by the trivial $\mathfrak{g}$-module $N{\mathop{\times}} V\,.$
$\blacksquare$\end{definition}

\begin{proposition}[see \cite{Kai} and \cite{luk}]
With the above definitions, $H^1_0(G,M)$ classifies extensions of $G$ by $M\,,$ and $H^1_0(\mathfrak{g},M)$ classifies extensions of $\mathfrak{g}$ by $\mathfrak{m}\,.$
\end{proposition}


\begin{thebibliography}{9}
\bibitem{alveraz}
Álvarez, Daniel . \textit{Leaves of stacky Lie algebroids.} Comptes Rendus. Mathématique, Volume 358 (2020) no. 2, pp. 217-226. doi : 10.5802/crmath.37. https://comptes-rendus.academie-sciences.fr/mathematique/articles/10.5802/crmath.37/
\bibitem{BS}
Arias Abad, Camilo, and Crainic, Marius.
\textit{The Weil algebra and the Van Est isomorphism.} 
Ann. Inst. Fourier (Grenoble) 61 (2011), no. 3, 927–970.
\bibitem{abadc}
Arias Abad, Camilo and Crainic, Marius. \textit{Representations up to homotopy and Bott's spectral sequence for Lie groupoids.}
Advances in Mathematics, Volume 248, 2013, Pages 416-452, ISSN 0001-8708,
https://doi.org/10.1016/j.aim.2012.12.022.
\bibitem{abad}
Arias Abad, Camilo and Crainic, Marius. \textit{Representations up to homotopy of Lie algebroids.} Journal für die reine und angewandte Mathematik (Crelles Journal), vol. 2012, no. 663, 2012, pp. 91-126.
\bibitem{Quintero}
Arias Abad, Camilo and Quintero Vélez, Alexander and Vélez Vásquez, Sebastián. \textit{An $A_{\infty}$ version of the Poincar\'{e} lemma.} Pacific J. Math. 302 (2019), no. 2, 385–412. 
\bibitem{bundle}
Aschieri, P. and Cantini, L. and Jur\v{c}o, B. \textit{Nonabelian Bundle Gerbes, Their Differential Geometry and Gauge Theory.} Commun. Math. Phys. 254, 367–400 (2005). https://doi.org/10.1007/s00220-004-1220-6
\bibitem{baez2}
Baez, John C. and Lauda, Aaron D.
\textit{Higher-dimensional algebra. V: 2-Groups.}
Theory and Applications of Categories [electronic only] 12 (2004): 423-491. http://eudml.org/doc/124217.
\bibitem{bailey}
Bailey, Toby N. and Eastwood, M. and Gindikin, S. \textit{Smoothly parameterized Čech cohomology of complex manifolds.} The Journal of Geometric Analysis 15 (2004): 9-23.
\bibitem{bargmann}
Bargmann, Valentine. \textit{On unitary ray representations of continuous groups}, Annals of Mathematics (1954), 59 (1): 1–46, doi:10.2307/1969831, JSTOR 1969831.
\bibitem{baum}
Baum, Paul and Connes, Alain. \textit{Leafwise Homotopy Equivalence and
Rational Pontrjagin Classes.} Advanced Studies in Pure Mathematics 5, (1985), pp.I-14.
\bibitem{Kai}
Behrend, Kai and Xu, Ping. 
\textit{Differentiable Stacks and Gerbes.}
J. Symplectic Geom. 9 (2011), no. 3, 285-341.
\bibitem{Bernstein}
Bernstein, J. and Lunts, V.
\textit{Equivariant Sheaves and Functors.}
LNM 1578, (1991).
\bibitem{block}
Block, Jonathan and Smith, Aaron M. \textit{A Riemann Hilbert correspondence for infinity local systems.} arXiv: Algebraic Topology (2009).
\bibitem{bona}
Bonavolont\`{a}, Giuseppe and  Poncin, Norbert. \textit{On the category of Lie n-algebroids.}
Journal of Geometry and Physics, Volume 73 (2013), Pages 70-90.
\bibitem{brylinski}
Brylinski, Jean-Luc.
\textit{Differentiable Cohomology of Gauge Groups.}
arXiv:math/0011069 [math.DG], (2000).
\bibitem{buch}
Buchdahl, N. \textit{On the Relative De Rham Sequence.} Proceedings of the American Mathematical Society, vol. 87, no. 2, American Mathematical Society, 1983, pp. 363–66, https://doi.org/10.2307/2043718.
\bibitem{cabrera}
Cabrera, Alejandro and Drummond, Thiago. 
\textit{Van Est isomorphism for Homogeneous Cochains.}
Pacific J. Math. 287 (2017), no. 2, 297–336.
\bibitem{Canez}
Cañez,  Santiago. Double groupoids and the symplectic category. Journal of Geometric Mechanics, 2018, 10 (2) : 217-250. doi: 10.3934/jgm.2018009
\bibitem{cegarra}
Cegarra, A. and Remedios, J. \textit{The relationship between the diagonal and the bar constructions on a bisimplicial set.} Topology and its Applications, (2015), 153(1):21-51
\bibitem{catt}
Cattaneo, A. and Felder, G. \textit{A Path Integral Approach¶to the Kontsevich Quantization Formula.} Comm Math Phys 212, 591–611 (2000).
\bibitem{Crainic} 
Crainic, Marius.
\textit{Differentiable and Algebroid Cohomology, van Est Isomorphisms, and Characteristic Classes.} 
Commentarii Mathematici Helvetici, Vol.78, (2003) pp. 681-721.
\bibitem{rui}
Crainic, Marius and Fernandes, Rui Loja.
\textit{Integrabiltiy of Lie Brackets.}
Annals of Mathematics, 157 (2003), 575-620.
\bibitem{zhu}
Crainic, Marius and Zhu, Chenchang.
\textit{Integrability of Jacobi and Poisson structures.}
Annales de l'Institut Fourier, Volume 57 (2007) no. 4, pp. 1181-1216.
\bibitem{davis}
Davis, James F. and Kirk, Paul.
\textit{Lecture Notes in Algebraic Topology.}
American Mathematical Society (2001).
\bibitem{cueca}
Cueca Ten, Miquel and \'{A}lvarez, Daniel.
\textit{Stacky Lie algebroids.                     }https://utrechtgeometrycentre.nl/wp-content/uploads/2020/08/201009-Slides-Miquel-Cueca.pdf, (2014).
\bibitem{hoyor}
del Hoyo, Matias. \textit{On the homotopy type of a cofibred
category.} Cahiers de Topologie et Geometrie Differentielle Categoriques 53 (2012),
82–114).
\bibitem{fernandes}
del Hoyo, Matias and Fernandes, Rui Loja. \textit{Riemannian metrics on differentiable stacks.} Mathematische
Zeitschrift, 292 (2019), 103–132.
\bibitem{del hoyo}
del Hoyo, Matias and Ortiz, Cristian. \textit{Morita Equivalences of Vector Bundles.} International Mathematics Research Notices, Volume 2020, Issue 14, July 2020, Pages 4395–4432, https://doi.org/10.1093/imrn/rny149
\bibitem{Deligne} 
Deligne P. 
\textit{Th\'{e}orie de Hodge. III.} 
Inst. Hautes \'{E}tudes Sci. Publ. Math No. $\mathbf{44}$ (1974),
5-77.
\bibitem{dieck}
tom Dieck, Tammo. \textit{Algebraic Topology.}
European Mathematical Society, Zürich (2008).
https://www.maths.ed.ac.uk/~v1ranick/papers/diecktop.pdf
\bibitem{bo}
Feng, Bo and Hanany, Amihay and He, Yang-Hui and Prezas, Nikolaos.
\textit{Discrete Torsion, Non-Abelian Orbifolds and the Schur Multiplier.}
Journal of High Energy Physics, 5 (2000).
\bibitem{ruif}
Fernandes, Rui. \textit{Invariants of Lie algebroids.}
Differential Geometry and its Applications, Volume 19, Issue 2 (2003), Pages 223-243, ISSN 0926-2245, https://doi.org/10.1016/S0926-2245(03)00032-9.
\bibitem{getzler}
Getzler, Ezra. \textit{Lie theory for nilpotent L$\infty$ algebras.} Annals of Mathematics, 170 (2009), 271–301
\bibitem{ginot}
Ginot, Gr\'{e}gory and Sti\'{e}non, Mathieu. \textit{G-gerbes, principal 2-group bundles and characteristic classes.} J. Symplectic Geom., 13(4):1001–1047, (2015).
\bibitem{glock}
Gl\"{o}ckner, Helge. \textit{Fundamentals of submersions and immersions between infinite-dimensional manifolds.} arXiv: Differential Geometry (2015).
\bibitem{groth}
Grothendieck, Alexander. \textit{Pursuing Stacks.} arXiv:2111.01000 [math.CT], (2010).
\bibitem{grunenfelder}
Grunenfelder, L. and Mastnak, M.\textit{Cohomology of abelian matched pairs and the Kac sequence.}
Journal of ALgebra, Vol. 276 (2013). Pages 706-736.
\bibitem{gen kahler}
Gualtieri, Marco.
\textit{Generalized Kahler Geometry.}
arXiv:1007.3485 [math.DG], (2010).
\bibitem{pym}
Gualtieri, Marco and Li, Songhao and Pym, Brent.
\textit{The Stokes Groupoids.}
Journal für die reine und angewandte Mathematik (Crelles Journal), Vol. 2018 (2013).
\bibitem{luk}
Gualtieri, Marco and Luk, Kevin.
\textit{Logarithmic Picard Algebroids and Meromorphic Line Bundles.}
arXiv:1712.10125 [math.AG], (2017).
\bibitem{eli}
Hawkins, Eli. \textit{A Groupoid Approach to Quantization.}
J. Symplectic Geom. 6 (2008), no. 1, 61-125.
\bibitem{haus}
Garmendia, Alfonso and Marco Zambon. \textit{Hausdorff Morita equivalence of singular foliations.} Annals of Global Analysis and Geometry 55 (2018): 99-132.
\bibitem{andre}
Henriques, André.
\textit{Integrating L$\infty$-algebras.}
Compositio Mathematica (2008), Vol 144.
\bibitem{greg}
Karpilovsky, Gregory.
\textit{The Schur Multiplier.}
Oxford University Press (1987).
\bibitem{krepski}
Krepski, Derek.
\textit{Basic equivariant gerbes on non-simply connected compact simple Lie groups}
Journal of Geometry and Physics
Volume 133, November 2018, Pages 30-41.
\bibitem{Lackman}
Lackman, Joshua. \textit{Cohomology of Lie Groupoid Modules and the Generalized van Est Map.}
International Mathematics Research Notices, rnab027, (2021), https://doi.org/10.1093/imrn/rnab027
\bibitem{Xutu}
Laurent-Gengoux, C., Tu, J.-L. and Xu, P. \textit{Chern-Weil map for principal bundles over
groupoids.} Mathematische Zeitschrift. 255 (2007), 451–491.
\bibitem{Li}
Li, Du. \textit{Higher Groupoid Actions, Bibundles,
and Differentiation.}
Arxiv: Differential Geometry, (2014).
https://arxiv.org/pdf/1512.04209
\bibitem{Meinrenken} 
Li-Bland, David and Meinrenken, Eckhard.
\textit{On the van Est homomorphism for Lie groupoids.} 
L’Enseignement Mathématique,
Vol. 61, (2014).
\bibitem{riehl}
Loregian, Fosco  and Riehl, Emily . Categorical notions of fibration. Expo.
Math., 38(4):496–514, 2020.
\bibitem{Mackenzie}
Mackenzie, K. (2005). General Theory of Lie Groupoids and Lie Algebroids (London Mathematical Society Lecture Note Series). Cambridge: Cambridge University Press. doi:10.1017/CBO9781107325883
\bibitem{mehta}
Mehta, R. \textit{Q-algebroids and their cohomology.} J. Symplectic Geom. 7 (2009), no. 3, 263–293.
\bibitem{mehtatang}
Mehta, R.A. and Tang, X. From double Lie groupoids to local Lie 2-groupoids. Bull Braz Math Soc, New Series 42, 651–681 (2011). https://doi.org/10.1007/s00574-011-0033-4
\bibitem{gael}
Meigniez, Ga\"{e}l. \textit{Submersions, Fibrations and Bundles.} Transactions of the American Mathematical Society. Volume 354, Number 9, Pages 3771–3787.
\bibitem{erbe}
Meinrenken, Eckhard.
\textit{The Basic Gerbe Over a Compact Simple Lie Group.}
Enseign. Math. 49. (2002).
\bibitem{salazar}
Meinrenken, Eckhard and Salazar, María Amelia.
\textit{Van Est Differentiation and Integration.}
Math. Ann. 376 (2020), no. 3-4, 1395–1428.
\bibitem{eva}
Miranda, Eva and Presas, Francisco. \textit{Geometric Quantization of real
polarizations via sheaves.} Journal of
Smplectic Geometry, Volume 13, Number 2, 421–462, (2015)
\bibitem{Moerdijk}
Moerdijk, I. Orbifolds as groupoids: an introduction. In Orbifolds in
mathematics and physics (Madison, WI, 2001), volume 310 of Contemp. Math., pages 205–222. Amer. Math. Soc., Providence, RI, 2002.
\bibitem{Murray}
Murray, Michael. \textit{Bundle gerbes} J. Lond. Math. Soc. 54 (1996) pp. 403-416
\bibitem{nlab}
nlab authors.
\textit{nlab: Lie Group Cohomology} 
Revision 15. (2019).
https://ncatlab.org/nlab/show/Lie+group+cohomology.
\bibitem{noohi1}
Noohi, Behrang. \textit{Foundations of Topological Stacks I.} arXiv: Algebraic Geometry (2005).
\bibitem{noohi}
Noohi, Behrang. \textit{Fundamental Groups of Algebraic Stacks.} Journal of the Institute of Mathematics of Jussieu 3 (2004): 69 - 103.
\bibitem{Nuiten}
Nuiten, J. \textit{Homotopical Algebra for Lie Algebroids.} Appl Categor Struct 27, 493–534 (2019). https://doi.org/10.1007/s10485-019-09563-z
\bibitem{ARANGO}
Ochoa Arango, Jes\'{u}s Alonso and Tiraboschi, Alejandro.  \textit{Double Groupoid Cohomology and Extensions.}
Arxiv: K-Theory and Homology, (2016).
https://arxiv.org/abs/1608.06712
\bibitem{Peters} 
Peters, Chris A.M. and Steenbrink, Joseph H.M.
\textit{Mixed Hodge Structures (A Series of Modern Surveys in Mathematics).} 
Springer (2008).
\bibitem{pym2}
Pym, Brent and Safronov, Pavel. 
\textit{Shifted Symplectic Lie Algebroids.}
International Mathematics Research Notices, rny215 (2018).
\bibitem{Quillen}
Quillen, Daniel. \textit{Homotopical Algebra.} Lecture Notes in Mathematics 43, Springer, (1967).
\bibitem{raptis}
Raptis, G. \textit{On Serre Microfibrations and a Lemma of M. Weirss.} Glasgow Mathematical Journal, 59(3), (2017) 649-657. doi:10.1017/S0017089516000458
\bibitem{schommer}
Schommer-Pries, Christopher J.
\textit{Central Extensions of Smooth 2-Groups and a Finite-Dimensional String
2-Group.}
Geometry and Topology, Vol. 15 (2009).
\bibitem{sometitle}
\v{S}evera, Pavol. \textit{Some title containing the words “homotopy” and “symplectic”, e.g. this one.} Travaux math\'{e}matiques. Fasc. XVI. Vol. 16. Trav. Math. Univ. Luxemb., Luxembourg, (2005), pp. 121–137.
\bibitem{jet}
\v{S}evera, Pavol. \textit{$L_\infty$ algebras as 1-jets of simplicial manifolds (and a bit beyond).} arXiv:math/0612349 [math.DG], (2006).
\bibitem{Severa}
\v{S}evera, Pavol and Weinstein, Alan.
\textit{Poisson Geometry with a 3-Form Background.}
Progress of Theoretical Physics Supplement, Vol. 144, (2002), pp. 145-154.
\bibitem{snia}
\'{S}niatycki, J. \textit{On Cohomology Groups Appearing in Geometric Quantization.} Differential
Geometric Methods in Mathematical Physics I (1975).
\bibitem{sparano}
Sparano, Giovanni and Vitagliano, Luca. \textit{Deformation cohomology of Lie algebroids and Morita equivalence.} C. R. Math. Acad. Sci. Paris 356 (2018), no. 4, 376–381.
\bibitem{str}
Street, Ross. \textit{Fibrations and Yoneda’s lemma in a 2-category.} Category Seminar (Proc. Sem., Sydney, 1972/1973), pp. 104--133. Lecture Notes in Math., Vol. 420, Springer, Berlin, 1974.
\bibitem{Tu} 
Tu, Jean-Louis. 
\textit{Groupoid Cohomology and Extensions.} 
Transactions of the American Mathematical,
Vol. 358, No. 11, (2006), pp. 4721-4747
\bibitem{van est}
van Est, W. T.
\textit{Group cohomology and Lie algebra cohomology in Lie groups. I, II.}
Nederl. Akad. Wetensch. Proc. Ser. A. 56 = Indagationes Math. 15, (1953). 484–492, 493–504.
\bibitem{van est 2}
van Est, W. T.
\textit{On the algebraic cohomology concepts in Lie groups. I, II.}
Nederl. Akad. Wetensch. Proc. Ser. A. 58 = Indag. Math. 17, (1955). 225–233, 286–294
\bibitem{wockel1}
Wagemann, Friedrich and Wockel, Christoph.
\textit{A Cocycle Model for Topological and Lie Group Cohomology.}
Trans. Amer. Math. Soc. 367 (2015), 1871-1909. 
\bibitem{konrad}
Waldorf, Konrad.
\textit{Multiplicative Bundle Gerbes With Connection.}
Differential Geometry and its Applications, Vol. 28, Issue 3 (2010), pp. 313-340.
\bibitem{Waldron}
Waldron, J. \textit{Lie Algebroids over Differentiable Stacks.}
Arxiv: Differential Geometry, (2015).
https://arxiv.org/abs/1511.07366
\bibitem{weinstein}
Weinstein, Alan and Xu, Ping.
\textit{Extensions of symplectic groupoids and quantization.}
Journal für die reine und angewandte Mathematik. Vol. 417, (1991) pp. 159-190.
\bibitem{witten}
Witten, E. \textit{Perturbative Gauge Theory as a String Theory in Twistor Space.} Commun. Math. Phys. 252, 189–258 (2004). https://doi.org/10.1007/s00220-004-1187-3
\bibitem{wockel}
Wockel, C. \textit{Principal 2-bundles and their gauge 2-groups.}
Forum Math. 23 (3) (2011) 565, [arXiv:0803.3692 [math.DG]].
\bibitem{zhuc}
Zhu, Chenchang. \textit{Lie II theorem for Lie algebroids via higher Lie groupoids.}
Arxiv: Differential Geometry, (2010). arXiv:math/0701024v2 [math.DG].
\bibitem{zhuc2}
Zhu, Chenchang. \textit{n-Groupoids and Stacky Groupoids.} International Mathematics Research Notices,(2009), 4087-4141.
\end{thebibliography}
\end{document}